\documentclass[12pt,reqno]{amsart}

\usepackage[colorlinks=true,urlcolor=blue,
bookmarks=true,bookmarksopen=true,citecolor=blue,linkcolor=orange
]{hyperref}

\usepackage{color,soul}
\setstcolor{red}
\usepackage{pdfsync}
\usepackage[all, cmtip]{xy}
\usepackage{graphicx}

\usepackage{subcaption} 
\usepackage{tikz}
\usepackage{tikz-3dplot}
\usepackage{pgfplots}
\usepackage{dsfont}
\usepackage{float}
\usepackage{bm}
\usepackage{mathtools}
\usepackage{amssymb}
\usepackage[T1]{fontenc}

\usepackage{epsfig}
\usepackage{amscd}
\usepackage[mathscr]{eucal}
\usepackage{amssymb}
\usepackage{bm}
\usepackage{amsxtra}
\usepackage{amsmath}
\usepackage{latexsym} 
\usepackage[all]{xy}
\usepackage{enumerate}
\usepackage[normalem]{ulem}

\usepackage{enumitem}

\usepackage{color}
\usepackage{xcolor, soul}
\sethlcolor{green}

\usepackage{bbm}

\theoremstyle{plain}

\newtheorem{mainthm}{Theorem}
\newtheorem{thm}{Theorem}[subsection]

\newtheorem{cor}[thm]{Corollary}

\newtheorem{lem}[thm]{Lemma}
\newtheorem{prop}[thm]{Proposition}

\theoremstyle{definition}
\newtheorem{dfn}[thm]{Definition}

\newtheorem*{claim-nonum}{Claim}

\theoremstyle{remark}

\newtheorem{rem}[thm]{Remark}

\newtheorem{rems}[thm]{Remarks}

\newtheorem{ex}[thm]{Example}

\theoremstyle{plain}

\def\C{\mathcal{C}}
\def\A{\mathcal{A}}

\def\K{\mathcal K}

\def\D{\mathcal{D}}

\def\QH{{\rm QH}}

\def\k{{\bf k}}
\def\F{\mathcal F}

\newcommand{\cobto}{\leadsto}
\newcommand{\id}{\textnormal{id}}

\newcommand{\R}{\mathbb{R}}
\newcommand{\Z}{\mathbb{Z}}

\newcommand{\N}{\mathbb{N}}
\newcommand{\la}{\lambda}
\newcommand{\La}{\Lambda}

\newcommand{\Crit}{\textnormal{Crit\/}}

\newcommand{\fuk}{\mathcal{F}uk}

\renewcommand{\k}{\mathbf{k}}

\def\hfch{H^{0}(\F\mathbf{Ch})}

\newcommand{\tcn}{{\mathcal{C}one}}

\makeatletter
\newcommand{\colim@}[2]{%
  \vtop{\m@th\ialign{##\cr
    \hfil$#1\operator@font colim$\hfil\cr
    \noalign{\nointerlineskip\kern1.5\ex@}#2\cr
    \noalign{\nointerlineskip\kern-\ex@}\cr}}%
}
\newcommand{\colim}{%
  \mathop{\mathpalette\colim@{\rightarrowfill@\scriptscriptstyle}}\nmlimits@
}
\renewcommand{\varprojlim}{%
  \mathop{\mathpalette\varlim@{\leftarrowfill@\scriptscriptstyle}}\nmlimits@
}
\renewcommand{\varinjlim}{%
  \mathop{\mathpalette\varlim@{\rightarrowfill@\scriptscriptstyle}}\nmlimits@
}
\makeatother

\newcommand{\pbred}[1]{{\textbf{{\color{red} #1}}}}
\newcommand{\ch}{\mathbf{Ch}}

\newcommand{\Ob}{\text{Obj}}

\newcommand{\md}{\text{mod}}

\newcommand{\dint}{{d_{\text{int}}}}
\newcommand{\dret}{{d_{\text{r-int}}}}

\newcommand{\sdint}{{\bar{d}_{\text{int}}}}
\newcommand{\sdret}{{\bar{d}_{\text{r-int}}}}

\newcommand{\lag}{{\mathcal{L}\text{ag}}}
\newcommand{\lagex}{{\lag}^{\text{(ex)}}}
\newcommand{\lagwex}{{\lag}^{\text{(wex)}}}
\newcommand{\lagmon}{\lag^{\text{(mon,} \mathbf{d} \text{)}}}

\newcommand{\ziso}{{(\text{\tiny{iso}},0)}}

\newcommand{\msc}{{\mathscr{C}}}
\newcommand{\msd}{{\mathscr{D}}}
\newcommand{\msf}{{\mathscr{F}}}
\newcommand{\msh}{{\mathscr{H}}}
\newcommand{\msj}{{\mathscr{J}}}
\newcommand{\msg}{{\mathscr{G}}}

\newcommand{\pbaddress}{biran@math.ethz.ch}
\newcommand{\ocaddress}{cornea@dms.umontreal.ca}
\newcommand{\gaaddress}{giovanni.ambrosioni@math.ethz.ch}
       
\begin{document}
 
\title{Approximability for Lagrangian submanifolds}

\date{January 17, 2026}

\author{Giovanni Ambrosioni, Paul Biran and Octav Cornea}

\address{Giovanni Ambrosioni, Department of Mathematics, ETH-Z\"{u}rich,
  R\"{a}mistrasse 101, 8092 Z\"{u}rich, Switzerland}
\email{\gaaddress}

\address{Paul Biran, Department of Mathematics, ETH-Z\"{u}rich,
  R\"{a}mistrasse 101, 8092 Z\"{u}rich, Switzerland}
\email{\pbaddress}

\address{Octav Cornea, Department of Mathematics and Statistics,
  University of Montreal, C.P. 6128 Succ.  Centre-Ville Montreal, QC
  H3C 3J7, Canada}
 \email{\ocaddress}

%\keywords{Triangulated category, Persistence module, Symplectic
 % manifold, Lagrangian submanifold, Floer homology, Fukaya category.}
%\subjclass[2020]{ 53D12 (Primary); 53D37, 55N31 (Secondary)}

\bibliographystyle{plain}
%\bibliographystyle{alphanum}

% ----------------------------------------------------------------------
%

% ----------------------------------------------------------------------
%
% Abstract

\begin{abstract}
  This paper introduces a notion of categorical approximability for metric spaces that can be viewed as a categorification of approximability for metric groups, as defined by Turing in 1938. Approximability as introduced here is a property of metric spaces that is more general than  precompactness.  It is  shown that several classes of Lagrangian submanifolds - closed Lagrangian submanifolds in a cotangent disk bundle; equators on the sphere; weakly exact Lagrangians on the torus - endowed with the spectral metric are approximable in this sense.  Among other geometric applications, we show that there are such examples of spaces of Lagrangians that are approximable but are not precompact.  \end{abstract}

\maketitle

% ----------------------------------------------------------------------
%
% Beginning of text
%

\tableofcontents 
% !TEX root = approx8.tex

\section{Introduction} \label{sec-intro}

\subsection{Overview of results} \label{subsec:overview}
Let $(M^{2n},\omega)$ be a symplectic manifold. In this paper we study
metric spaces $(\mathcal{L}ag(M), d_{\gamma})$ consisting of classes $\mathcal{L}ag(M)$ of Lagrangian submanifolds in $M$ endowed with the so-called spectral Lagrangian metric $d_{\gamma}$ (see \S\ref{subsec:nearby-TPC} and \S\ref{subsec:proof_ThmAii} for definitions of this metric in the setting of the paper). The metric space $(\mathcal{L}ag(M), d_{\gamma})$ is separable \cite{Chasse:metrics2} 
but neither complete, nor compact or locally compact. 

Nonetheless we will see that in several meaningful cases this space satisfies a property more general than compactness called  {\em  categorical (metric)  approximability} that we introduce in this paper and that has some interesting symplectic consequences. 

To explain this property assume that $X$ is a subset of a metric space $(Y,d)$, and, more generally,  $d$ may be only a pseudo-metric, and possibly take infinite values.  For $\epsilon>0$, an $\epsilon$-approximation  (also called an $\epsilon$-net) of $X$ in $Y$ is a subspace $X_{\epsilon}\subset Y$ such that for all $x\in X$ there is $y\in X_{\epsilon}$ with $d(x,y)<\epsilon$. Recall that $(X,d)$ is called {\em totally bounded} if, for all $\epsilon > 0$, there is a {\em finite} $\epsilon$-net.  This is equivalent to the space $(X,d)$ being {\em precompact} in the sense that its completion is compact. Obviously, a compact metric space is totally bounded, but not vice-versa.

Categorical approximability in the sense used here is weaker than the totally bounded property in  that  the $\epsilon$-nets no longer need to be finite but we only require that they be {\em generated} by a finite set in a sense made more explicit in the next definition.  

\begin{dfn}\label{def:approx} Assume that $(X,d)\subset (Y,d)$ is an inclusion of metric spaces such that the points in $Y$ are the objects of a triangulated category $\mathcal{C}$. For $\epsilon > 0$, we say that 
$X$ is   {\em categorically metric $\epsilon$-approximable} in $\mathcal{C}$ if there exists a subspace $X_{\epsilon}\subset Y$ such that the elements of $X_{\epsilon}$ are the objects of a {\em finitely generated} (in the sense of triangular generation) subcategory of $\mathcal{C}$ and, additionally, $X_{\epsilon}$ is an $\epsilon$-net for $X$.  If this property is satisfied for all $\epsilon >0$ we say that $X$ is {\em categorically metric approximable} in $\mathcal{C}$. \end{dfn}

To put this definition in context see Remark \ref{rem:turing} below. Of course, producing categories $\C$ that are triangulated and also carry a natural (pseudo) metric is not obvious. One source of such structures is a type of categories called {\em triangulated persistence categories} (TPC) that
were introduced in \cite{BCZ:tpc} (the main relevant definitions are recalled in \S\ref{sec:alg1}).  Indeed, we will use in the paper a more precise version of approximability that is specific to TPCs. This version, in Definition \ref{def:TPC-approx} (see also Definition \ref{def:simple_approx}), is called {\em TPC approximability} and is applied to TPC refinements of the derived Fukaya category.  In Definition \ref{def:approx} the triangular structure and the metric are unrelated.  However, TPC approximability  ties these two structures in a natural fashion (we could have tied the two structures already in
Definition \ref{def:approx} through some simple axioms but we preferred to keep the presentation simpler at this point and we have skipped this aspect). 
An important notion associated with Definition \ref{def:approx}  (see again Definition \ref{def:simple_approx} for the TPC case) is that of an $\epsilon$-approximating family. This is a finite family of objects  of $\C$, often denoted by $\F_{\epsilon}$, with the property that $X_{\epsilon}$ from Definition \ref{def:approx} is included in the triangulated completion of the family $\F_{\epsilon}$.

\begin{rem}\label{rem:usual_triang}
In triangulated category theory there exists a different notion of approximability,
unrelated to any underlying metric structures, like in \cite{Neeman:approx_tri}. This is why we use the word {\em metric} in our definition. We will drop this qualifier frequently in this paper as the only meaning used here is that of Definition \ref{def:approx} and its TPC variants.
\end{rem}

\

We will also use a weaker notion of approximability, called {\em retract categorical (metric) $\epsilon$-approximability} which is obtained by replacing in Definition \ref{def:approx} the condition that $X_{\epsilon}$ is an $\epsilon$-net for $X$ by the weaker requirement that for each $x\in X$
there exists an object $k_{x}$ in $\C$ such that $x\oplus k_{x}$ satisfies
$d(x\oplus k_{x}, X_{\epsilon}) < 2\epsilon$. To emphasize, the stronger form, in Definition \ref{def:approx}, is equivalent to being able to pick $k_{x}=0, \ \forall x$, and replacing  $2\epsilon$ by $\epsilon$ (the factor $2$ here is irrelevant as we are interested in $\epsilon$-approximability for all $\epsilon >0$; in practice, it is a consequence of some of the conventions in the paper, see Remark \ref{rem:relation-TPCapp}). 

Retract approximability has the advantage to be easier to estimate algebraically, as will be discussed below, while at the same time being sufficiently robust to bound from above the complexity (in a sense that will be made precise) of the objects of $X$ by the complexity of the elements in $X_{\epsilon}$. A corresponding notion of $TPC$ {\em retract approximability} is  defined  in \S\ref{subsubsec:approx_TPC} and it is this form that will be used in our applications.

The main result of this paper is that several natural classes of Lagrangians
are categorically approximable.

\begin{mainthm} \label{thmmain1} The following three classes of Lagrangians $\mathcal{L}ag(M)$ are approximable, as below.
\begin{itemize} 
\item[i.] Closed, exact Lagrangians in unit cotangent disk bundles $M=D^{\ast} N$ for any closed manifold $N$
(and any metric on $N$) are  TPC approximable.
\item[ii.] Equators on the $2$-sphere, $M=S^{2}$, are  TPC retract approximable.
\item[iii.] Non-contractible Lagrangians on the $2$-torus, $M=\mathbb{T}^{2}$, are  TPC retract approximable.
\end{itemize} 
In all three cases  TPC approximability takes place in appropriate TPC refinements of the 
certain derived Fukaya categories and the relevant spaces of Lagrangian submanifolds $\mathcal{L}ag(M)$ are endowed with the metric structure given by the spectral metric $d_{\gamma}$ in the cases i and ii and with a metric that restricts to the spectral metric on each Hamiltonian isotopy class at point iii. 
\end{mainthm}

The symplectic form on the cotangent disk bundles $D^{\ast}N$ is the derivative of the 
Liouville form $\la$ and the exactness of Lagrangians is understood with respect to $\la$.
The relevant Fukaya categories $\C$ are made precise in the body of the paper. The $\epsilon$-approximating families in the three cases consist of: fibers of the cotangent-disk bundle in the first case; great circles on the sphere connecting the north and south poles in the second; one latitude and a (finite) collection of longitudes on the torus in the third. 

\

To place the next result in context notice that if $(\mathcal{L}ag(M),d_{\gamma})$ would be totally bounded in the case $M=D^{\ast}N$, then the space of exact, closed Lagrangian submanifolds in the unit cotangent bundle would have bounded diameter in the spectral metric (because totally bounded implies bounded). This statement was conjectured by Viterbo and it was proven in some cases by Shelukhin  \cite{Shl:vit-2},\cite{Shl:vit-1}  but is not known in full generality. Our next result claims that, in general, these spaces of Lagrangian submanifolds are not totally bounded. 

\begin{cor}\label{cor:no-t-b} The spaces $(\mathcal{L}ag(M),d_{\gamma})$ are not totally bounded when
$M=D^{\ast}N$ and $N$ is endowed with a hyperbolic metric and $\mathcal{L}ag(M)$ consists of closed exact Lagrangians, as well as when $M=S^{2}$ and $\mathcal{L}ag(S^{2})$ consists of equators.
\end{cor}

The metric completion  of the space $(\mathcal{L}ag(M),d_{\gamma})$ has been studied under 
the name of Humili\`ere completion (see, for instance, \cite{AGHI:micro}). Thus, Corollary \ref{cor:no-t-b} can be reformulated to say that, in the cases listed, this completion is not compact.

\

The proof of Corollary \ref{cor:no-t-b} is related to a natural application of categorical approximability that is of interest in itself and is worth pointing out here. Namely, one can use categorical approximability of $(X,d)\subset (Y,d)$ to define and study notions of complexity for both $X$ and for its points. There are several variants but the general recipe is the same. One starts with some notion of complexity for triangulated categories, particularly for finitely generated ones,  as well as for the objects of such categories. We will make some of these choices explicit later in the paper but they include, in the case of categories, the Rouquier dimension, generation time, the minimal number of triangular generators, and, for objects, the cone-length (originating in homotopy theory \cite{Gan:cat} \cite{Cor:cone-LS}), complexity in the sense of Dimitrov-Haiden-Katzarkov-Kontsevich  \cite{DHKK:dyn_cat}.  
 Assuming that $c(-)$ is such a choice of complexity for the objects of a 
 triangulated category one can then define the corresponding $\epsilon$-complexity of an element $x\in X$ by: 
 $$\bar{c}(x; \epsilon)=\inf \{ c(y) \ | \  y\in X_{\epsilon}\ ,\  d(x,y)\leq \epsilon\} .$$ The complexity of $x$, without reference to $\epsilon$, can then be analyzed through the asymptotic behaviour of $\bar{c}(x; \epsilon)$  when $\epsilon \to 0$.  A similar construction can be applied to $X$ itself starting  with some notion of complexity for $X_{\epsilon}$. These notions of complexity can be applied to study the complexity of self maps $\phi : X\to X$ by analyzing how 
 $\bar{c}(\phi^{k}x; \epsilon)$ changes with 
 respect to $k$. 
 
 \
 
 One application to the spaces $(\mathcal{L}ag(M), d)$ as in the statement of 
 Theorem \ref{thmmain1},  uses in the place of $c(y)$ the cone length of $y$ which
 is  the minimal number of iterated exact triangles formed with elements from the family $\F_{\epsilon}$ that
 are needed to obtain $y$ starting from the $0$ object. The resulting $\epsilon$-complexity is denoted by 
 $N(- ; \F_{\epsilon},\epsilon)$ (to keep the notation shorter we will denote
 it here by $N(-; \epsilon)$) and is called the $\epsilon$-weight cone-length with linearization in $\F_\epsilon$, see Definition
 \ref{def:w-cl}.  We relate this $N(-;\epsilon)$ to counts of bars of length at least $\epsilon$ in the barcodes of certain Floer persistence modules. This relation is used to analyze the entropy of symplectic diffeomorphisms by relating the resulting $\epsilon$-weighted categorical entropy - which is a weighted version of a notion introduced by Dimitrov-Haiden-Katzarkov-Kontsevich \cite{DHKK:dyn_cat} -  to the barcode entropy introduced by  \c{C}ineli-Ginzburg-Gurel \cite{CGG_bar_ent}. When $M=D^{\ast}N$ and $N$ is endowed with a hyperbolic metric we can relate the growth of $N(\phi^{k}L; {\epsilon})$, $k\in\N$, for well chosen
 Hamiltonian diffeomorphism $\phi:M\to M$ and $L\in\mathcal{L}ag(M)$, to the exponential growth of numbers of geodesics of increasing length in $N$ and we deduce exponential growth of  $N(\phi^{k}L; {\epsilon})$ with respect to $k$.
 As a byproduct, we obtain in this case examples of Lagrangians in $\mathcal{L}ag(M)$ of arbitrarily high complexity $N(-;\epsilon)$ which implies immediately that the space $(\mathcal{L}ag(M),d_{\gamma})$  is not totally bounded, as is claimed in 
 Corollary \ref{cor:no-t-b}. A similar argument also proves  Corollary \ref{cor:no-t-b} for the case
 $M=S^{2}$ by using a complexity variant $N^{r}(-;\epsilon)$ that is adapted to retract approximability.

\

In each of the three  cases in Theorem \ref{thmmain1}, by construction, the Lagrangians in the respective $\epsilon$-approximating families are distributed inside the ambient manifolds in a somewhat uniform way. Indeed, this property is necessary for a family $\F_{\epsilon}$ to be an $\epsilon$- aproximating family, when
$\epsilon$ is small enough. Making this statement  rigorous goes beyond the scope of 
this paper but some intuition in this direction can be gained from the following two results,
Corollaries \ref{cor:quasi-rig} and \ref{cor:link}, that establish some geometric properties of these families.

The first such result only applies to the case $M=D^{\ast}N$ and is a quasi-rigidity property relative to the $\epsilon$-approximating families $\F_{\epsilon}$.  To formulate this property,  for an exact Lagrangian  $L\subset  D^{\ast}N$, denote by $h_{L}:L\to \R$ a choice of a primitive of the restriction of the Liouville form $\la$ to $L$ (see \S\ref{sbsb:fil-ex} for the rest of our  conventions). 

\begin{cor}\label{cor:quasi-rig} Let $M=D^{\ast}N$ and assume that  $\phi :D^{\ast}N\to D^{\ast}N$ is an exact symplectomorphism with support in the
interior of $D^{\ast}N$. Let $\F_{\epsilon}$ be a finite family of fibers of $D^{\ast}N$ that is $\epsilon$-approximating in the sense of Theorem \ref{thmmain1} i. Let
$$\chi (\phi; \F_{\epsilon}) =\max_{F\in\F_{\epsilon}}\ \{ \max h_{\phi(F)} - \min h_{\phi(F)}\}~.~$$
Then, for any $L\in \mathcal{L}ag(M)$, we have $$d_{\gamma}(L,\phi(L)) \leq 4(\epsilon + N(L;\F_{\epsilon}, \epsilon) \chi(\phi;\F_{\epsilon}))~.~$$
In particular, if $\phi(F)=F$ for all $F\in \F_{\epsilon}$, we have $d_{\gamma}(L,\phi(L)) \leq 4\epsilon$, $\forall L$ closed exact $\subset D^{\ast}N$.
\end{cor}

 The next result is an easy consequence of Corollary 6.13  in \cite{Bi-Co-Sh:LagrSh}.
To state it,  recall \cite{Bar-Cor:Serre} \cite{Bi-Co:rigidity} that the relative Gromov width of $L$ relative to some subset $K$ of $M$ is the supremum
of $\frac{\pi r^{2}}{2}$ for all embeddings $e: B_{r}\to M$ such that $e^{-1}(L)=B_{r}\cap \R^{n}$, $e^{\ast}\omega=\omega_{0}$,  and $e(B_{r})\cap K=\emptyset$ (here and later in the paper $B_{r}\subset \R^{2n}$ is the ball of radius $r$, centered at the origin and $\omega_{0}$ is the standard symplectic form on $\R^{2n}$).

Assume that $(\mathcal{L}ag (M), d_{\gamma})$ is TPC retract $\epsilon$-approximable  and assume additionally that there  are $\epsilon$-approximating families $\F_{\epsilon}$ that are geometric in the sense that they are Yoneda modules in the relevant TPC category $\C$. 

\begin{cor}\label{cor:link} In the setting above, let $L\in\mathcal{L}ag(M)$. The  Gromov width of $L$ relative to the union of the submanifolds in $\F_{\epsilon}$ is bounded above by $2\epsilon$. In particular, this bound applies to the three classes of Lagrangians in Theorem \ref{thmmain1}. 
\end{cor}
An obvious consequence of this Corollary is that when $\epsilon \to 0$, the set $\cup_{F\in\F_{\epsilon}} F$
becomes dense in $M$.

\

In a less precise form,  Corollaries  \ref{cor:link} and \ref{cor:quasi-rig} can be read independently of approximability and of the machinery that serves in their proof. To focus ideas consider only the case $M=D^{\ast}N$. These corollaries imply the following statement: {\em for  every $\epsilon >0$, there exists a  finite family   of fibers, $\F_{\epsilon}$, of $D^{\ast}N$ with both the properties in Corollary \ref{cor:link} and in Corollary \ref{cor:quasi-rig}}.  While this statement appears new, it was noted by Egor Shelukhin and  by Mark Gudiev - see Remark \ref{rem:direct_cor} - that certain parts of it can be established by making use of the types of functions constructed in \S\ref{subsec:Morse}, followed by more direct arguments, that do not appeal to Fukaya category machinery. 

\

\begin{rem}\label{rem:turing} a. Categorical approximability, as introduced here, is related to a notion due to Turing from 1938 \cite{Turing:approx}. Turing's definition is concerned with a metric group $(G,d)$. This is approximable in Turing's sense if there are  $\epsilon$-nets  $X_{\epsilon}$ (for each $\epsilon>0$) that are finite groups with a group structure that is $\epsilon$ close to the one of $G$  (see \cite{Thom:approx} for a recent overview of developments in geometric group theory that originate in this notion). Turing proved that a Lie group that is approximable in this sense is compact abelian.
Categorical metric approximability  is a categorification of a  variant of Turing's notion.  This variant, called here {\em finite type approximable},   no longer  requires that  $X_{\epsilon}$ be a finite group but,
instead, $X_{\epsilon}$ is a {\em finitely generated} subgroup of $Y$. For instance, any subset of $\R^{2}$, with $\R^{2}$ viewed as a group and endowed with the usual euclidean distance,  admits $1/m$- approximations by lattices $L_{1/m}\subset \R^{2}$ generated by two vectors $(1/2m,0), (0,1/2m)\in\R^{2}$ for each $m\in \N^{*}$, thus it is finite type approximable. Any compact metric group is trivially finite type approximable. On the other hand, it is easy to find examples of functional spaces that are not  approximable in this sense. For instance, the space of smooth real functions on $[0,1]$ is not finite type approximable inside the space of continuous functions on $[0,1]$ with the {\em sup} metric (and the obvious group structure). Similarly, the space of step functions on $[0,1]$, with only a finite number of discontinuities, endowed with the {\em sup} metric, is not finite type approximable in itself. 

b. It is easy to see that if $(X,d)$ is  categorically approximable in $\C$ in the sense of Definition \ref{def:approx}, then the image of $X$ inside the Grothendieck $K$-group of $\C$, $K(\C)$, is finite type approximable with respect to the pseudo-metric $\bar{d}$  induced on $K(\C)$ from $d$. In other words,
categorical metric approximability is a categorification of finite type approximability for groups.

c. The approximability approach to complexity of maps, as described above, bypasses two difficulties encountered when trying to define directly  invariants - such as entropy - for a map $\phi : X \to X$. The first is that, generally, the metric space $X=(\mathcal{L}ag(M), d_\gamma)$ is not totally bounded (as shown in
Corollary \ref{cor:no-t-b}), and thus defining directly invariants analogous to topological entropy for  continuous self-maps $\phi :X\to X$ induced by symplectic diffeomorphisms of $M$ is not possible. A  second difficulty is that, in general, such $\phi$ do not preserve $\epsilon$-approximations, in the sense that for a fixed $\delta>0$ and for any $0<\epsilon \leq \delta $ there are some elements of $X_{\epsilon}$ that are mapped  outside $X_{\delta}$, and thus approximating $\phi$ by presumably simpler maps $\phi_{\epsilon}:X_{\epsilon}\to X_{\epsilon}$ is not practical. The formalism above avoids this second difficulty because the complexity $\bar{c}_{\epsilon}(\phi^{k}x)$
is well defined for all $k\in \Z$, $\epsilon>0$, $x\in X$ as soon as $\phi$ preserves $X$, as is our assumption.

d. The notion of approximability discussed here is quite strong. Consider again the case $M=D^{\ast}N$.  Approximability not only  means that any $L\in\lag (D^{\ast}N)$ can be approached as close as needed by an iterated cone of fibers but, moreover, for fixed $\epsilon$, we can get $\epsilon$-close to {\em any} $L$  by only using fibers from the {\em same fixed, finite family} $\F_{\epsilon}$, independent of $L$. 
\end{rem}

Our proof of Theorem \ref{thmmain1} is based on certain properties of relevant filtered Fukaya categories and the main ingredients in our arguments are outlined in the next subsection.  Before proceeding, we mention  a recent result due to Guillermou-Viterbo-Zhang \cite{Gu-Vi-Zh:approx} that shows approximability for the cotangent bundle (part i in Theorem \ref{thmmain1}) through a different approach, by using as receptacle category  the Tamarkin category.

\subsection{Overview of the main ingredients, technical tools and novelties}
Much of this paper makes use of the theory of triangulated persistence categories as developed in \cite{BCZ:tpc}. This framework, inspired by symplectic considerations but independent of them,  turned out to be essential to formulate the correct definitions and statements related to approximability and it is further developed here, particularly with regards to notions of complexity in \S\ref{subsec:complex_alg} which can be read and used independently of the applications in this paper.  

Another essential tool is provided by the filtered version of the construction of the Fukaya category and of its derived versions, in a way that is precise enough for applications and so that it fits naturally  with the TPC formalism. The starting point for the construction is Seidel's setup \cite{Se:book-fukaya-categ}. The filtered version uses ingredients from 
\cite{Bi-Co-Sh:LagrSh}, \cite{BCZ:tpc} with an essential addition provided by the perturbative methods introduced in \cite{Amb:fil-fuk}. The paper brings some novelties
in this direction. In particular,  the process of choosing perturbative data with increasing degrees of accuracy is formalized in \S\ref{s:pdfuk} in a way that is likely to be of use beyond the applications to approximability. More importantly, the paper contains, in \S\ref{sec:split-app}, a persistence version of Abouzaid's split generation result \cite{Ab:geom-crit}, including persistence versions of Hochschild homology and cohomology and of the open-closed and closed-open maps. We will see that persistence split-generation is  extremely tightly related to retract-approximability (see Proposition \ref{gio:algfilabo}).

A third important tool appears in the study of approximability for cotangent bundles in \S\ref{sec:nearby} and has 
to do with some Lefschetz fibration considerations. The root of this part lies in the Fukaya-Seidel-Smith approach \cite{FSS:ex} as adjusted from the perspective  of Lagrangian cobordism in \cite{Bi-Co:lefcob-pub} and \cite{Bi-Co:mspectral}. An essential  ingredient in this part, that allows sufficient precision for the quantitative estimates required,
is provided by Giroux's work \cite{Gir:Lef-cot}.

The definition of categorical approximability  is sufficiently flexible and natural that it is highly likely that this notion is applicable beyond symplectic topology.  This said,
symplectic topology is our main focus in the paper and approximability appears to be a deep new feature of symplectic rigidity.  A concurring sign pointing in this direction is provided by 
the recent result in \cite{Gu-Vi-Zh:approx} that was already mentioned before.

 \subsection{Structure of the paper} In \S\ref{sec:alg1} we set up the necessary technical foundations to be able to analyze approximability in the setting of TPC refinements of Fukaya categories. We recall in this section the main required definitions and properties of triangulated persistence categories. We then discuss how this formalism applies to Fukaya categories.    The starting point is the construction due to Ambrosioni \cite{Amb:fil-fuk} that provides filtered, strictly unital $A_{\infty}$-Fukaya categories by making use of Seidel's classical scheme but using a special,  well-chosen system of perturbations.  Continuing from there, one encounters several delicate points, some of them of technical nature, but also some more conceptual. One of them is that a choice of perturbations used in constructing  the filtered Fukaya categories comes with a certain ``size''.  It is possible to pick perturbations of arbitrarily small size, of course. Nonetheless, once a choice of perturbation is picked
the associated category can only be used to study approximability for $\epsilon$ bigger than that size. This is a somewhat obvious point, similar to saying that a measuring tool needs to have  a precision margin better than the scale of the phenomena studied, but it has significant consequences for the involved formalism. It implies that we need to relate as tightly as possible the various categories and associated constructions for systems of perturbations of decreasing size and  we also need to formulate approximability and related notions for TPCs,
taking these sizes into account. 

In \S\ref{sec:nearby} we prove the first point in Theorem \ref{thmmain1}. The proof makes use of a remarkable Lefschetz fibration structure constructed by Giroux on cotangent bundles \cite{Gir:Lef-cot}. Some important elements of this construction are recalled \S\ref{subsec:Giroux}.   Giroux's construction makes possible some quantitative refinements of some of the arguments in \cite{Bi-Co:lefcob-pub} that studied Lagrangian cobordisms in Lefschetz fibrations. These are recalled in \S\ref{subsec:Lef-Dehn}. Finally, the  other important ingredient in the proof is a  ``soft'' construction, having to  do with producing certain Morse functions with small variation but big differential almost everywhere, that appears in \S\ref{subsec:Morse}. The various ingredients are put togehter in \S\ref{subsec:finish-proof}.

Section \ref{sec:split-app} shows 
the points ii and iii in Theorem \ref{thmmain1}. The main step is to establish a persistence
version of the Abouzaid split generation result. To achieve this, several significant algebraic aspects need to be developed, in particular persistence Hochschild homology and cohomology. The persistence split generation result is then used to show  by explicit calculations Theorem \ref{thmmain1} in the case of $S^{2}$ and of $\mathbb{T}^{2}$.

The last section, \S\ref{sec:complex}, is focused on weighted numerical complexity measurements for Lagrangian
submanifolds and their applications. In particular, in this section are included the arguments necessary to deduce Corollary \ref{cor:no-t-b}. The section also contains the proofs of 
Corollaries  \ref{cor:quasi-rig} and \ref{cor:link}.

 \

 \noindent {\bf Acknowledgements.} We thank Viktor Ginzburg who suggested to the third author to investigate relations between TPCs and the categorical entropy of Dimitrov-Haiden-Katzarkov-Kontsevich as well as for several useful discussions related to barcode entropy and relevant examples, see also Remark \ref{rem:thanks_VG}. We  thank  Emmanuel Giroux  for useful discussions and for sharing with us an early version of his preprint \cite{Gir:Lef-cot}. We thank Claude Viterbo for sharing with us an early version of the preprint \cite{Gu-Vi-Zh:approx} as well as for many useful discussions with all three of us, in particular in relation to metric aspects of 
 persistence categories. We thank Egor Shelukhin  and Mark Gudiev for 
 sharing with us their more direct approach to some of the geometric corollaries in the paper, see Remark \ref{rem:direct_cor}. The first named author thanks  Nick Sheridan for references in relation to Hochschild homology and to him and to Kenji Fukaya for helpful discussions.   

Part of this work was accomplished while  members of the team were graciously hosted by several 
research institutions for short and medium length research visits:  the CRM in Montr\'eal, 
the FIM in Z\"urich, the IHP in Paris, Universit\'e Paris-Saclay (Orsay). 
We thank them all for their support.

% !TEX root = approx8.tex

\section{Triangulated persistence refinements of Fukaya
  categories} \label{sec:alg1}
The aim of this section is to fix the basic techniques and notation to
deal with triangulated persistence categories and their applications
to Fukaya categories. Persistence categories and some of their
properties are discussed in \S\ref{subsec:PC}. One important feature
mentioned here is that the set of objects of such a category carries a
class of natural pseudo-distances called {\em interleaving}, see
\eqref{eq:d-int}. Triangulated persistence categories are briefly
recalled in \S\ref{subsec:TPC_approx} where is also introduced TPC
approximability in Definition \ref{def:TPC-approx}. This is the
version of approximability appearing in our main theorem, Theorem
\ref{thmmain1}. In \S\ref{subsec:A_{infty}} we start to discuss
filtered $A_{\infty}$ categories. The key section in this respect is
\S\ref{sb:pdc} where are discussed derived categories associated to
filtered $A_{\infty}$ categories. Section \ref{s:sys-tpc} makes
explicit the formal properties of a system of $A_{\infty}$ categories
(and the associated derived categories) that one obtains by making
choices of perturbations that become smaller and smaller. In
\S\ref{s:approximability} a new, more precise, version of
approximability is defined taking into account the system of
categories with increasing precision discussed before.  This is a
stronger, but more technical, version of Definition
\ref{def:TPC-approx} which is shown later in the paper to be satisfied
in the cases of geometric interest. Finally, \S\ref{s:pdfuk} contains
the main steps in the construction of the filtered Fukaya categories
to which can be applied the algebraic tools described earlier in the
section. We refer to \cite{BCZ:tpc} for further details on the
filtered algebra background material in this section.

\subsection{Persistence categories}\label{subsec:PC}
In brief, a persistence category $\mathscr{C}$ is a category enriched
over the category of persistence modules.  We fix notations and give the main
definitions below.
\subsubsection{Persistence modules}
A persistence module
$$M = \Bigl( \{M^{\alpha}\}_{\alpha \in \mathbb{R}}, \; i_{\alpha,
  \beta}: M^{\alpha} \longrightarrow M^{\beta}, \; \forall \ \alpha
\leq \beta \Bigr)$$ over a ring $R$ (for us $R$ will be often
  either a field or the positive Novikov ring) is a collection of
$R$-modules $M^{\alpha}$ for each $\alpha\in\R$ and
morphisms $i_{\alpha,\beta}$ such that
$i_{\beta,\gamma}\circ i_{\alpha,\beta}=i_{\alpha,\gamma}$ whenever
$\alpha\leq\beta\leq\gamma$ and $i_{\alpha,\alpha}=id$ (sometimes
additional constraints are imposed on these modules, as needed).  Let $u, v \in M^{\alpha}$ and $r \geq 0$. We write
  $u \underset{r}{=}v$ to indicate that $i_{\alpha,\alpha+r}(u) = i_{\alpha,\alpha+r}(v)$.
%\item For $\alpha \leq \beta$ we write
%  $m_{\beta}: M^{\alpha} \longrightarrow M^{\beta}$ for the map
%  $i_{\alpha, \beta}$.
%\item For $u \in M^{\alpha}$, $v \in M^{\beta}$. We write
%  $u \underset{\text{per}}{=} v$ to indicate that
%  $m_{\gamma}(u) = m_{\gamma}(v)$, where
%  $\gamma := \max \{\alpha, \beta \}$.
% For $u \in M^{\alpha}$ we denote 
%  $\ell(u) := \alpha$ and call $\ell(u)$ the persistence level of $u$.
%
We  denote by
$$M_{\text{tot}} := \coprod_{\alpha \in \mathbb{R}} M^{\alpha}$$ the underlying set
of $M$ %Note that the persistence level $\ell$ can be viewed as a function $\ell: M_{\text{tot}} \longrightarrow \mathbb{R}$.  
and we define $$M^{\infty} := \colim_{\alpha \to \infty} M^{\alpha},$$
where the colimit (or direct limit) is taken with respect to the
structural maps $i_{\alpha, \beta}$, $\alpha \leq \beta$, and call it
the $\infty$-limit of $M$.

\begin{rem}
  Let $C$ be a filtered chain complex, with an increasing filtration
  $C^{\alpha} \subset C^{\beta} \subset C$,
  $\forall \; \alpha \leq \beta$. We will sometimes view $C$ as a
  persistence module with
  $i_{\alpha, \beta}: C^{\alpha} \longrightarrow C^{\beta}$ being the
  inclusion maps. Let $u \in C^{\alpha}$ and consider the same
  element $u$ but now viewed in $C^{\beta}$ for some $\beta > \alpha$.
  For many practical purposes these two elements are viewed as one
  element that belongs to $C$. However, from the persistence module
  viewpoint $u \in C^{\alpha}$ and its image
  $i_{\alpha, \beta}(u) \in C^{\beta}$ should be viewed as distinct
  elements. \end{rem}

\subsubsection{Persistence categories} \label{sb:pc} (PC's in
short) are (small) categories in which the morphisms between objects form
persistence modules and the composition of morphisms is compatible
with the persistence structures. The general theory of persistence
categories is described in detail in~\cite[Section~2]{BCZ:tpc} and
here we will only recall several key concepts.

Unless otherwise stated, from now on we will always assume all the PC
categories to be endowed with a shift functor (or, more precisely, a
system of shift functors) $\Sigma = \{\Sigma^r\}_{r \in \mathbb{R}}$
(see~\cite[Section~2.2.3]{BCZ:tpc} for the definition), and all the PC
functors to commute with $\Sigma^r$ for every $r \in \mathbb{R}$. Part
of the shift functor structure are the natural transformations
  $\eta_{r,s}: \Sigma^r \longrightarrow \Sigma^s$,
  $s, r \in \mathbb{R}$.  The special case $s=0$ will be especially
  important and we denote it by
  $\eta_r := \eta_{r,0}: \Sigma^r \longrightarrow \id$.
Moreover, all persistence functors
  $\msf: \msc \longrightarrow \msd$ between PC's will be implicitly
  assumed to be compatible with the shift functors of these categories
  in the following sense: $\msf \circ \Sigma^r = \Sigma^r \circ \msf$
  for all $r \in \mathbb{R}$, and
  $(\eta_{r,s})_{\msf A} = \msf ((\eta_{r,s})_A)$ for all
  $A \in \Ob(\msc)$ and $r, s \in \mathbb{R}$.

Let $\mathscr{C}$ be a PC category. Denote by $\mathscr{C}^0$ the
$0$-level category associated with $\mathscr{C}$ and by
$\mathscr{C}^{\infty}$ the limit (or $\infty$-level) category of
$\mathscr{C}$. Isomorphisms in $\msc^0$ will be called 
$0$-isomorphisms and the property of being $0$-isomorphic will be 
sometimes denoted by $\cong_0$. Similarly, the existence of an isomorphism 
in $\msc^{\infty}$ will be denoted by $\cong_{\infty}$.

Most of the time we will assume our PC's $\mathscr{C}$ to be graded in
the sense that the persistence modules $\hom_{\mathscr{C}}(A,B)$ are
$\mathbb{Z}$-graded. We will mostly use cohomological grading
conventions and denote by $\hom_{\mathscr{C}}(A,B)^k$,
$k \in \mathbb{Z}$ the degree-$k$ component of
$\hom_{\mathscr{C}}(A,B)$. For $\alpha \in \mathbb{R}$ we denote by
$\hom^{\alpha}_{\mathscr{C}}(A,B)^k$ the $\alpha$-persistence level of
$\hom_{\mathscr{C}}(A,B)^k$. We assume composition of morphisms
to be degree-preserving and that identity morphisms are in degree $0$.

Our PC's $\mathscr{C}$ will be often endowed with a {\em translation
  functor} $T: \mathscr{C} \longrightarrow \mathscr{C}$. This is a PC
isomorphism functor with the property that for all
$A, B \in \Ob(\mathscr{C})$, $k \in \mathbb{Z}$, we have PC-natural
isomorphisms
$$\hom_{\mathscr{C}}(TA,B)^k \cong \hom_{\mathscr{C}}(A,B)^{k-1},
\quad \hom_{\mathscr{C}}(A,TB)^k \cong
\hom_{\mathscr{C}}(A,B)^{k+1}.$$ As mentioned earlier, since $T$ is a
PC functor, we will implicitly assume that $T$ commutes with the shift
functors, i.e.~$T \circ \Sigma^{r} = \Sigma^r \circ T$,
$\forall \, r \in \mathbb{R}$.

\subsubsection{Acyclic objects} \label{sbsb:acyclic} Let $\mathscr{C}$
be a PC, $A \in \Ob(\mathscr{C})$ and $r \geq 0$. We say that $A$ is
$r$-acyclic if the morphism $\eta_r^A: \Sigma^r A \longrightarrow A$
equals $0$, or, equivalently, if the identity $id_{A}$ satifies
$i_{0,r}(id_{A})=0$.  Here,
$\eta^A_{r} \in \hom_{\mathscr{C}^0}(\Sigma^{r}A, A)$ is given by
$\eta^{A}_{r}=\eta_{r}(A)$. An object $A$ is called acyclic if it is
$r$-acyclic for some $r \geq 0$. Equivalently, $A$ is acyclic if and
only if $A$ is isomorphic to $0$ in the limit category
$\mathscr{C}^{\infty}$.  We denote by $A\mathscr{C}$ the full
subcategory of $\mathscr{C}$ whose objects are the acyclic ones.

\subsubsection{Pseudo metrics associated with
  PC's} \label{sbsb:pc-metrics}

Let $\mathscr{C}$ be a PC with a shift functor $\Sigma$. The
collection of objects $\Ob(\mathscr{C})$ carries a natural pseudo-metric
$\dint$, called the interleaving distance, which is defined as
follows.  For $X, Y \in \Ob(\mathscr{C})$ define:
\begin{equation} \label{eq:d-int}
  \begin{aligned}
    \dint(X, Y) = \inf \Bigl\{r \geq 0 & \bigm| \exists \, \varphi:
    \Sigma^rX \longrightarrow Y, \, \exists \, \psi: \Sigma^rY
    \longrightarrow X \\
    & \;\;\; \text{such that} \; \psi \circ \Sigma^r \varphi =
    \eta_{2r}^X, \; \varphi \circ \Sigma^r \psi = \eta_{2r}^Y \Bigr\}.
  \end{aligned}
\end{equation}
Here, $\eta^X_{2r} \in \hom_{\mathscr{C}^0}(\Sigma^{2r}X, X)$ are the
standard maps associated with the persistence structure of
$\mathscr{C}$ and its shift functor and similarly for $\eta^Y_{2r}$.

Note that $\dint$ is in general not a genuine metric but only a
pseudo-metric and moreover it may have infinite values. Indeed, if $X$
and $Y$ are isomorphic in the subcategory $\mathscr{C}^0$, then
$\dint(X, Y) = 0$. Similarly, if $X$ and $Y$ are not isomorphic in
$\mathscr{C}^{\infty}$ then $\dint(X,Y) = \infty$. This is compatible
with the convention that $\inf \emptyset = \infty$ which we use
in~\eqref{eq:d-int}.

There is another variant of $\dint$ which measures how far an object
$R$ is from being a retract of another object $X$. More precisely, for
$R, X \in \Ob(\mathscr{C})$ we define
\begin{equation} \label{eq:d-rint} \dret(R,X) = \inf \Bigl\{r \geq 0
  \bigm| \exists \, \varphi: \Sigma^rR \longrightarrow X, \, \exists
  \, \psi: \Sigma^rX
  \longrightarrow R \\
  \;\; \text{such that} \; \psi \circ \Sigma^r \varphi = \eta_{2r}^R
  \Bigr\}.
\end{equation}
In contrast to $\dint$, the measurement $\dret(-,-)$ is in general not
symmetric (such a structure is sometimes called a quasi-pseudometric).

The pseudo-metric $\dint$, admits a shift-invariant version $\sdint$ that will play 
an important role further on:
\begin{equation}\label{eq:sdint}
\sdint (X,Y)=\inf_{r,s} \dint (\Sigma^{r}X,\Sigma^{s}Y)
\end{equation}
It is easily seen that this is also a pseudo-metric. A similar
construction applies also to $\dret$ and produces $\sdret$.

Here is an additional notation to be used later in the paper. Let $Z$
be a set endowed with a (not necessarily symmetric) pseudo-metric $d$
and $X, Y \subset Z$. We denote by

$$d(X,Y) := \sup_{x \in X} \inf_{y \in Y} d(x,y)$$
%$$\vec{d}(X,Y) := \sup_{x \in X} \inf_{y \in Y} d(x,y)$$ 
the longest distance
between the points of $X$ and those of $Y$. 
%(The arrow above the $d$ indicates that
Note that this is not symmetric in $(X,Y)$.
We will sometimes use this measurement for $d=\dint$  and $d=\dret$.
%and denote it $\vec{d}_{\text{int}}$. Similarly
%we can also define also $\vec{d}_{\text{r-int}}$ (although $\dret$ is
%not a symmetric pseudo-metric to start with).

To end this subsection, we list a few basic metric properties of PC
functors in the following lemma which is immediate to establish. In
the statement the same notation $\dint$ is used for the inteleaving
distance in various PC categories.
\begin{lem} \label{l:d-func} Let $\msc', \msc''$ be PC's and
  $\msf: \msc' \longrightarrow \msc''$ a PC-functor.
  \begin{enumerate}
  \item For every $A, B \in \Ob(\msc')$ we have
    $\dint(\msf A, \msf B) \leq \dint(A,B)$.
  \item If there exists a PC functor
    $\msg: \msc'' \longrightarrow \msc'$ such that
    $\dint(\msg \circ \msf, \id_{\msc'}) \leq s$ then for every
    $A, B \in \Ob(\msc')$ we have
    $|\dint(\msf A, \msf B) - \dint (A,B)| \leq 2s$.
  \item If there is a PC-functor $\msh: \msc'' \longrightarrow \msc'$
    such that $\dint(\msf \circ \msh, \id_{\msc''}) \leq r$ then
    $$\dint(\msc'', \textnormal{image\,}\msf) \leq r.$$
  \end{enumerate}
\end{lem}

\subsection{Triangulated persistence categories and TPC
  approximability}\label{subsec:TPC_approx}
We discuss here approximability in the setting of triangulated
persistence categories.

\subsubsection{TPCs} \label{sb:tpc} A {\em Triangulated Persistence
  Category} (TPC in short) $\mathscr{C}$ is a PC with an additional
structure which makes its $0$-level category $\mathscr{C}^0$ a
triangulated category. The precise definition of a TPC involves
several other axioms and we refer the reader to~\cite{BCZ:tpc} for a
detailed exposition of the subject.

Let $\msc$ be a TPC and $S \subset \Ob(\msc)$. We write
$\langle S \rangle^{\Delta}$ for the minimal full sub-TPC of $\msc$
which contains all the objects of $S$ and is closed under
$0$-isomorphisms. (Note that in particular the shift and translation
functors of $\msc$ restrict to respective functors on
$\langle S \rangle^{\Delta}$.) 

%We denote by
%$(\langle S \rangle^{\Delta})^0$ the $0$-level category of the TPC
%$\langle S \rangle^{\Delta}$.

\subsubsection{TPC approximability} \label{subsubsec:approx_TPC} The
key approximability definition for the paper is the following.
\begin{dfn}\label{def:TPC-approx} A pseudo-metric space $(X,d)$ is
  approximable through triangulated persistence categories (in short,
  TPC-approximable) if there exists a constant $A\geq 1$ such that for
  each $\epsilon >0$ the following structure exists. For each
  $0< \eta <\epsilon$ there is a TPC, $\mathscr{Y}_{\epsilon,\eta}$, with two
  properties:
  \begin{itemize}
  \item[i.] There exists an $(A,\eta)$-quasi-isometric embedding (see Remark \ref{rem:def-TPC-approx}):
    $$\Phi_{\epsilon,\eta}:(X,d) \to (\Ob(\mathscr{Y}_{\epsilon,\eta}), \sdint^{\ \epsilon,\eta})$$
  \item[ii.] There exists a {\em finite} family
    $\mathcal{F}_{\epsilon,\eta} \subset \Ob(\mathscr{Y}_{\epsilon,\eta})$ with the
    property:
    $$\sdint^{\ \epsilon,\eta}(\Phi_{\eta}(x), \Ob( \langle\mathcal{F}_{\epsilon,\eta}
    \rangle ^{\Delta}) ) < \epsilon, \ \forall x\in X$$
  \end{itemize} 
  where $\sdint^{\ \epsilon,\eta}$ is the shift invariant interleaving pseudo-metric associated to
  $\mathscr{Y}_{\epsilon,\eta}$ (see \eqref{eq:sdint}).
\end{dfn}

For fixed
$\epsilon$, the data $(\Phi, \F)$ with  $\Phi= \{\Phi_{\epsilon,\eta}\}_{0<\eta <\epsilon}$,
$\F=\{\F_{\epsilon,\eta} \}_{0<\eta<\epsilon}$ is called 
TPC $\epsilon$-{\em approximating data} for $(X,d)$. Thus, $(X,d)$ is
TPC-approximable if $\epsilon$-TPC-approximating data for $(X,d)$
exists for each $\epsilon>0$. It can certainly happen that multiple
choices of such data exist for the same $(X,d)$. When one of the two parameters $\epsilon$ or $\eta$ are clear from the context we omit it from the notation. For instance, for fixed $\epsilon$, we write $\Phi_{\eta}$ for the relevant quasi-isometric embeddings , $\sdint^{\ \eta}$ for the respective pseudo-metrics and so forth.

\

This definition  is the TPC  version of Definition \ref{def:approx}. 
We will also  use in the paper the notion of  {\em TPC retract approximability}.  TPC retract approximability is defined just as above
without any modifications for point i. above but by using $ \sdret$ at
point ii.  (we emphasize that the {\em only} change is at point ii. in
Definition \ref{def:TPC-approx}). Pairs $(\Phi, \F)$ as above but ensuring retract approximability (in other words, using $\sdret$ instead of $\sdint$ for condition ii in the definition) will be called TPC {\em retract $\epsilon$-approximating data}.

\

Point c. in Remark \ref{rem:relation-TPCapp} below explains more precisely how  TPC  approximability, as just defined, fits with the notion of categorical metric approximability introduced in Definition \ref{def:approx}.

\begin{rem}\label{rem:def-TPC-approx}
  a. Recall that an $(A,B)$-quasi-isometric embedding
  $$(X,d) \to (Y,d')$$ between two metric spaces is a
  map $\Phi : X\to Y$ such that:
  $$\frac{1}{A}\  d(x,y)-B \leq d'(\Phi(x),\Phi(y))\leq A\ d(x,y)+B~.~$$
  Thus, the first point of the Definition \ref{def:TPC-approx} shows that, when $\epsilon$ is
  fixed, the metrics
  $\sdint^{\ \eta}$ defined on the objects of $\mathscr{Y}_{\eta}$, when
  restricted to $X$, get closer and closer to a (pseudo)-metric
  equivalent to $d$, when $\eta\to 0$. More precisely, we can put for
  each $x,y\in X$
  $$\hat{d}(x,y)=\limsup_{\eta\to 0} \sdint^{\ \eta} (x,y)~.~$$
  This gives a pseudo-metric defined on $X$ and point i. of the
  Definition implies that this $\hat{d}$ is equivalent to $d$, the
  original metric on $X$.

  b. A $(A,\eta)$-quasi-isometric embedding is also a
  $(A,\eta')$-quasi-isometric embedding whenever $\eta'\geq \eta$. As
  a result, to verify the property in the definition, one only needs
  to show the existence of $\mathscr{Y}_{\eta}$, $\Phi_{\eta}$, and of
  the families $\F_{\epsilon} =\{\F_{\epsilon,\eta}\}$ whenever $\eta$ is small enough.

  c. This definition will be applied to the case when $(X,d)$ is a
  space of Lagrangian submanifolds endowed with the spectral metric
  $d=d_{\gamma}$ and the categories $\mathscr{Y}_{\eta}$ are TPC
  refinements of derived Fukaya categories (we assume again $\epsilon$ fixed here). 
  The role of $\eta$ is to
  keep track of the size of the perturbations needed to define these
  Fukaya categories. While this parameter is irrelevant in the
  non-filtered case, for purposes of approximations it is crucial to
  keep track of it because any choice of perturbations needed to
  define the Fukaya category determines some $\eta$ which constraints what the approximating accuracy $\epsilon$ can be. Controlling systematically the various TPC refinements of  Fukaya categories  relative to the choices of perturbation data of varying size requires significant elaboration which is contained in \S\ref{s:sys-tpc}. In particular, the metric $\hat{d}$ above is a version  of $\widehat{d}_{\mathrm{int}}$ from Lemma \ref{l:lim-dint}. In this setting of families of filtered $A_{\infty}$-categories it is operationally useful to use a more refined and precise variant of the definition above. This is formulated  in Definition \ref{d:approx-sys} and its retract analogue is  in Definition \ref{d:approx-sys-ret}. See also Remark \ref{rem:different_approx} for a comparison between these notions.
  
  d. The constant $A$
  will be equal to $2$ in the Lagrangian topology application. The  application to Lagrangian topology was determinant for our choice of the shift invariant
 metric $\sdint$ in Definition \ref{def:TPC-approx} (as opposed to
  $\dint$ which could as well have been used, in principle). The
  reason is that the spectral pseudo-metric is shift invariant and
  thus the limit $\hat{d}$ from point a above needs to also be shift
  invariant (see also Remark \ref{rem:relation-TPCapp} ).

  e. In certain contexts, when choices of perturbations are not
  required, and $\epsilon$ is fixed, one presumably can replace the family $\mathscr{Y}_{\eta}$ by a single $PC$ (or TPC) $\mathscr{Y}_{ 0}$ - in other words include
  $\eta=0$ - in the definition above.

  f. In Definition \ref{def:TPC-approx} we have not imposed any
  relation tying the categories $\mathscr{Y}_{\epsilon,\eta}$, when
  $\eta \to 0$, nor did we assume any relation among the respective
  families $\F_{\epsilon,\eta}$.  In practice, our constructions in
  the case of Fukaya categories produce categories that satisfy a
  property more precise than the one in Definition
  \ref{def:TPC-approx}, in the sense that, for fixed $\epsilon >0$, the
  categories $\mathscr{Y}_{\epsilon,\eta}$, with $\eta\to 0$, are related by
  certain comparison functors and the corresponding families
  $\F_{\epsilon,\eta}$ correspond one to the other through these
  functors. This more refined property is more technical to state and
  is formulated in Definition \ref{d:approx-sys}. Further below we
  will also discuss the dependence of our constructions relative to $\epsilon$.
  
\end{rem}

We will also make use of a simplified notion of TPC-approximability
that reformulates point ii of Definition \ref{def:TPC-approx}.

\begin{dfn} \label{def:simple_approx} Fix a triangulated persistence
  category $\msc$ and $\mathcal{X}\subset \Ob(\msc)$.  For $\epsilon >0$, we say that $\mathcal{X}$ is $\epsilon$-{\em approximable in} $\msc$ if there
  exists a {finite} family $\F_{\epsilon}\subset \Ob(\msc)$, called an
  $\epsilon$-{\em approximating family} for $\mathcal{X}$ in $\msc$, such that
  $$\dint\left( x, \Ob (\langle \F_{\epsilon}\rangle^{\Delta})\right)
  < \epsilon, \forall x \in \mathcal{X}.~$$ The set $\mathcal{X}$ is called approximable
  in $\msc$ if it is $\epsilon$-approximable for each $\epsilon >0$.
 We say that $\mathcal{X}$ is  {\em retract $\epsilon$-approximable in}
 $\msc$ if the same property as above holds but with $\dret$ in the place of $\dint$. 
  \end{dfn} 

\begin{rem}\label{rem:relation-TPCapp} a. Recall from \S\ref{sbsb:pc-metrics} that 
both $\dint$ and $\dret$ have
  shift-invariant versions.  Thus, one could use these shift invariant metrics,  $\sdint$ and $\sdret$, instead of $\dint$ and, respectively, $\dret$ in the definition above. However, this has no impact on the notion defined. Indeed, the set $\Ob (\langle \mathcal{S}\rangle ^{\Delta})$ where
  $\langle \mathcal{S} \rangle ^{\Delta}$ is the smallest sub-TPC of
  $\C$ containing $\mathcal{S}\subset \Ob (\C)$ is shift and
  translation invariant. Thus,
  $\dint(x, \Ob (\langle \F_{\epsilon}\rangle^{\Delta})) =\sdint(x,
  \Ob (\langle \mathcal{F}_{\epsilon}\rangle^{\Delta}))$ for all
  $x\in Y$, and similarly for $\dret$.  As a result, the regular and shift
  invariant versions of approximability in the sense of Definition
  \ref{def:simple_approx} coincide and the same is true for retract
  approximability. 
  
  b. Notice that, with the terminology in Definition \ref{def:simple_approx}, point ii in Definition \ref{def:TPC-approx} is equivalent to the fact that $\mathcal{X}=\Phi_{\eta}(X)$ is $\epsilon$-approximable in $\mathscr{Y}_{\eta}$ and similarly for retract approximability ($\epsilon$ is fixed here).

  c. The notion of $\epsilon$-approximability
 in Definition \ref{def:simple_approx} obviously implies metric categorical
 $\epsilon$-approximability of $\mathcal{X}$ in $\msc$, as  introduced in Definition \ref{def:approx} (the difference between the two definitions is that in  \ref{def:simple_approx} the metric on the objects of the category $\msc$ is fixed to be the interleaving metric associated with the persistence structure on $\msc$).
 
 d. Assume, as above, that $\msc$ is a TPC.  
  The following statement is an exercise in manipulating exact triangles in the triangulated category $\msc^{0}$:  
  \begin{equation}\label{eq:dint-dret}
  \mathrm{For\ }\  x, y\in \Ob(\msc)\ \mathrm{if}\ \sdret (x,y) <\epsilon \ , \mathrm{then}\ \exists\ k_{x}\in\Ob(\msc),\ \mathrm{such\ that}\  \sdint(x\oplus k_{x},y) < 2\epsilon ~.~
  \end{equation}
It is immediate to see that we also have for all $x,y, k\in \Ob(\msc)$ , $\sdret(x,y)\leq \sdint (x\oplus k,y)$. From (\ref{eq:dint-dret}) we deduce that if $\mathcal{X}$ is retract $\epsilon$-approximable in $\msc$, in the sense of Definition \ref{def:simple_approx}, then it also is retract categorically $\epsilon$-approximable in the sense of \S\ref{subsec:overview}.  In other words, if $\mathcal{X}$ is retract 
$\epsilon$-approximable in $\msc$, then for each $x\in \mathcal{X}$ there exist $k_{x}\in \msc$ and $y\in \Ob (\langle \mathcal{F}_{\epsilon}\rangle^{\Delta}))$ such that $\dint (x\oplus k_{x},y) < 2\epsilon$.

e. Idempotent completion in TPCs is a subtle matter, much more so than its counterpart in triangulated categories. This topic is studied by Miller in \cite{Miller:idempotents} but we will not appeal in this paper to the results there.

 f. In \cite{BCZ:tpc} were introduced a class of so-called fragmentation pseudo-metrics $d^{\F}(-,-)$ on the objects of a TPC that depend on the choice of a family $\F\subset \Ob(\C)$. These pseudo-metrics are closely related to approximability.  Given $\mathscr{C}$ as above, $d^{\F_{\epsilon}}(0,Y)\leq \epsilon/4$ implies that  $\F_{\epsilon}$ is $\epsilon$-approximating. Conversely, if $\F_{\epsilon}$ is $\epsilon$-approximating, then $d^{\F_{\epsilon}}(0,Y)\leq \epsilon$.  Understanding the behaviour of these fragmentation pseudo-metrics when the family $\mathcal{F}$ changes was one of the main motivations leading to the definition of approximability in this paper. \end{rem}

\subsection{Conventions for PC's and filtered $A_{\infty}$
  categories}\label{subsec:A_{infty}}

\subsubsection{Filtered $A_{\infty}$-categories} \label{sbsb:fai} Most
of the PC's in this paper arise as persistence homological categories
of filtered $A_{\infty}$-categories. Unless otherwise stated we will
always assume such categories, as well as $A_{\infty}$-functors
between them, to be strictly unital and endowed with a shift and
translation functors. See~\S\ref{a:fil-ai} for the precise
definitions.

Let $\mathcal{A}$ be a filtered $A_{\infty}$-category.  We will
sometimes abbreviate $\mathcal{A}(X,Y) :=
\hom_{\mathcal{A}}(X,Y)$. For $X, Y \in \Ob(\mathcal{A})$,
$\alpha \in \mathbb{R}$, we denote by
$\hom_{\mathcal{A}}^{\alpha}(X,Y) \subset \hom_{\mathcal{A}}(X,Y)$, or
sometimes by $\mathcal{A}^{\alpha}(X,Y)$ for short, the
$\alpha$-filtration level of $\hom_{\mathcal{A}}(X,Y)$. Whenever we
need to combine these with (cohomological) grading we will write
$\hom_{\mathcal{A}}^{\alpha}(X,Y)^i$ (or
$\mathcal{A}^{\alpha}(X,Y)^i$) for the degree-$i$ component of
$\hom_{\mathcal{A}}^{\alpha}(X,Y)$. The notation in case of
homological grading is similar, by making the degree $i$ a subscript.

Next, we denote by $\mathcal{A}^0$ the $0$-level category of
$\mathcal{A}$, namely the (unfiltered) $A_{\infty}$-category with the
same objects as $\mathcal{A}$ and
$\hom_{\mathcal{A}^0}(X,Y) = \hom_{\mathcal{A}}^0(X,Y)$.

Passing to homology, we write $PH(\mathcal{A})$ for the persistence
homological category of $\mathcal{A}$. The shift and translation
functors induce respective functors on $PH(\mathcal{A})$.

As in the case of PC's, we say that an object $A \in \mathcal{A}$ is
$r$-acyclic if $A$ is $r$-acyclic in the persistence category
$PH(\mathcal{A})$, and similarly for acyclic objects (without
reference to a specific $r$). The full $A_{\infty}$-subcategory of
acyclic objects will be denoted by $A\mathcal{A}$.

Finally, we import the interleaving (and r-interleaving)
  pseudo-metrics also to the realm of $A_{\infty}$-categories
$\mathcal{A}$, by setting both of them to be equal to the respective
distances (as defined in~\S\ref{sbsb:pc-metrics}) in the persistence
homological category $PH(\mathcal{A})$.

\subsubsection{Various completions} \label{sbsb:comp} Let $\msc$ be a
PC. Given subsets $S_1, S_2, S_3 \subset \Ob(\mathscr{C})$ we define
their completions with respect to shifts, translations and
$0$-isomorphisms, respectively, as follows:
\begin{equation}
  \begin{aligned}
    & S_1^{\Sigma} := \{ \Sigma^r(A) \mid A \in S_1, \, r \in
    \mathbb{R}\}, \quad
    S_2^{T} := \{T^j(A) \mid A \in S_2, \, j \in \mathbb{Z}\}, \\
    & S_3^{\ziso} := \{ A \in \Ob(\mathscr{C}) \mid A \text{ is
      isomorphic in } \mathscr{C}^0 \text{ to an object from } S_3\}.
  \end{aligned}
\end{equation}
We will also need the completion of a subset
$S \subset \Ob(\mathscr{C})$ with respect to several of the above
procedures, for example $S^{\Sigma, T} := (S^{\Sigma})^T$ etc.

Another important completion is with respect to the interleaving
distance. Given a subset $S_4 \subset \Ob(\msc)$ we define
$$S_4^{\text{int}} :=
\{ A \in \Ob(\mathscr{C}) \mid \dint(A, B) < \infty \text{ for some }
B \in S_4\}.$$ To simplify the notation, for $S \subset \Ob{\msc}$ we
will write $S^{c} := (S^{T})^{\text{int}}$.
%$$S^{c,0} := S^{\Sigma, T, \ziso}, \quad S^{c} := (S^{T})^{\text{int}}.$$
Note that $S^c$ is automatically complete with respect to both shifts
and translations.
% completing with respect to the interleaving distance already
% includes the completion with respect to shifts and to
% $\msc^0$-isomorphisms, hence we can omit them when defining $S^c$.
One can also define $S^{c,\text{r-int}}$ in a similar way to $S^c$,
but with $\dint$ replaced by $\dret$.

If $\mathscr{D} \subset \mathscr{C}$ is a full sub-PC we define its
various completions, e.g.~$\mathscr{D}^{\Sigma}$, $\mathscr{D}^{T}$,
$\mathscr{D}^{(\text{iso}, 0)}$, $\mathscr{D}^{c}$ etc.~by taking the
full sub-PC's of $\mathscr{C}$ with the respective completed sets of
objects $\Ob(\mathscr{D})^{\Sigma}$, $\Ob(\mathscr{D})^{T}$,
$\Ob(\mathscr{D})^{\ziso}$, $\Ob(\mathscr{D})^{c}$ etc. In the case of
$A_{\infty}$-categories $\mathcal{A}$, we define analogous
completions, by completing the subsets of objects using the respective
structures in $PH(\mathcal{A})$.

%\subsection{TPC's} \label{sb:tpc} A {\em Triangulated Persistence
%  Category} (TPC in short) $\mathscr{C}$ is a PC with an additional
%structure which makes its $0$-level category $\mathscr{C}^0$ a
%triangulated category. The precise definition of TPC involves several
%other axioms and we refer the reader to~\cite{BCZ:tpc} for a detailed
%foundation of the subject.
%
%Let $\msc$ be a TPC and $S \subset \Ob(\msc)$. We write
%$\langle S \rangle^{\Delta}$ for the minimal full sub-TPC of $\msc$
%which contains all the objects of $S$ and is closed under
%$0$-isomorphisms. (Note that in particular the shift and translation
%functors of $\msc$ restrict to respective functors on
%$\langle S \rangle^{\Delta}$.) We denote by
%$(\langle S \rangle^{\Delta})^0$ the $0$-level category of the TPC
%$\langle S \rangle^{\Delta}$.

\subsection{Persistence derived categories associated with a filtered
  $A_{\infty}$-category} \label{sb:pdc} We will present below three
variants of TPC's that can be viewed each as a generalization of the
derived category of $\mathcal{A}$ to the realm of persistence
categories. However before we go into this we need a quick preparation
about $A_{\infty}$-functors in the filtered context.

\subsubsection{Filtered functors and functors with linear
  deviation} \label{sbsb:fil-func}

Let $\mathcal{A}$, $\mathcal{B}$ be two filtered
  $A_{\infty}$-categories. A filtered $A_{\infty}$-functor
  $\mathcal{F}: \mathcal{A} \longrightarrow \mathcal{B}$ is an
  $A_{\infty}$-functor whose action on morphisms (in all orders)
  preserves the filtration levels. More specifically, for every
  $d \geq 1$, $X_0, \ldots, X_d \in \Ob(\mathcal{A})$ and
  $\alpha \in \mathbb{R}$ the $d$'th order term
  $\mathcal{F}_d: \mathcal{A}(X_0, \ldots, X_d) \longrightarrow
  \mathcal{B}(\mathcal{F}X_0, \mathcal{F}X_d)$ of $\mathcal{F}$
  satisfies:
  $$\mathcal{F}_d \bigl(\mathcal{A}^{\alpha}(X_0, \ldots, X_d)\bigr) \subset
  \mathcal{B}^{\alpha}(\mathcal{F}X_0, \mathcal{F}X_d).$$ Unless
  otherwise stated, in what follows we will implicitly assume our
  $A_{\infty}$-functors to be strictly unital.

 Motivated by examples coming from Fukaya categories we will
  also need the concept of $A_{\infty}$-functors with linear
  deviation. These are defined as follows. Given a tuple
  $\vec{X} = (X_0, \ldots, X_d)$ of objects from $\mathcal{A}$ we
  define its reduced tuple
  $\vec{X}_R := (X_{i_0}, \ldots, X_{i_{d_R}})$ by ommitting from
  $\vec{X}$ subsequent (in the cyclic order) objects that are equal up
  to a shift. The objects forming $\vec{X}_R$ are well defined only up
  to shifts, but the length of $\vec{X}_R$,  $0 \leq d_R \leq d$, is well defined. We call it the reduced length
  of $\vec{X}$ and denote it by $d_R$ or $d_R(\vec{X})$ whenever we
  want to emphasize its dependence on $\vec{X}$. Note that in case
  every two consecutive objects in $\vec{X}$ are different (up to
  shifts) then $d_R(\vec{X}) = d$. At the other extreme, if all the
  objects in $\vec{X}$ are equal up to shifts, then
  $d_R(\vec{X}) = 0$.  

An $A_{\infty}$-functor
  $\mathcal{F}: \mathcal{A} \longrightarrow \mathcal{B}$ is said to
  have linear deviation rate $s \geq 0$ if for every $d \geq 1$, every
  tuple of objects $\vec{X} = (X_0, \ldots, X_d)$ from
  $\Ob(\mathcal{A})$ and every $\alpha \in \mathbb{R}$ we have:
$$\mathcal{F}_d \bigl(\mathcal{A}^{\alpha}(X_0, \ldots, X_d)\bigr)
\subset \mathcal{B}^{\alpha + d_R(\vec{X}) s}(\mathcal{F}X_0,
\mathcal{F}X_d).$$ We will often refer to such functors as LD (Linear
Deviation) functors.

\subsubsection{Filtered modules} \label{sbsb:fmod} Let
$\mathcal{A}$ be a filtered $A_{\infty}$-category. Denote by $F
\md_{\mathcal{A}}$ the category of filtered strictly unital (left)
$\mathcal{A}$-modules. This is a filtered 
$A_{\infty}$-category. The Yoneda embedding
$$\mathcal{Y}: \mathcal{A} \longrightarrow F\md_{\mathcal{A}}$$
is a filtered strictly unital $A_{\infty}$-functor which is
homologically full and faithful.  Denote by
$\mathcal{Y}(\mathcal{A}) \subset F \md_{\mathcal{A}}$ the image
category of $\mathcal{A}$ under $\mathcal{Y}$. Denote by
$\mathcal{Y}(\mathcal{A})^{\Delta} \subset F \md_{\mathcal{A}}$ the
triangulated completion of $\mathcal{Y}(\mathcal{A})^{\Delta}$,
namely the minimal full subcategory of $F \md_{\mathcal{A}}$ which is
closed under mapping cones over morphisms with filtration level
$\leq 0$ and quasi-isomorphisms of filtration level $\leq 0$
(these are maps of non-positive filtration level that induce an isomorphism in the $0$-level homological category). Note
that due to our assumptions on the filtered $A_{\infty}$ category $\mathcal{A}$, the category
$\mathcal{Y}(\mathcal{A})$ is closed under shifts and translations,
and therefore the same holds for $\mathcal{Y}(\mathcal{A})^{\Delta}$.
The construction of $\mathcal{Y}(\mathcal{A})^{\Delta}$ can be
explicitly carried out by iteratively taking mapping cones over
morphisms of filtration level $\leq 0$ and adding at each stage
$0$-quasi-isomorphic modules. The resulting
$A_{\infty}$-category $\mathcal{Y}(\mathcal{A})^{\Delta}$ is filtered
and strictly unital $A_{\infty}$ and carries a translation and a shift
functor. The persistence-homological category of
$\mathcal{Y}(\mathcal{A})^{\Delta}$
$$PD(\mathcal{A}) := PH(\mathcal{Y}(\mathcal{A})^{\Delta})$$
is a TPC which we call the {\em persistence derived} category of
$\mathcal{A}$.

Let $\mathcal{A}$, $\mathcal{B}$ be two filtered
$A_{\infty}$-categories. Let
  $\mathcal{F}: \mathcal{A} \longrightarrow \mathcal{B}$ be an
  $A_{\infty}$-functor with a given linear deviation rate
  (see~\S\ref{sbsb:fil-func} and~\S\ref{a:func-LD}). Then the
  push-forward operation from~\S\ref{ap:pshf} gives rise to a {\em
    filtered} $A_{\infty}$-functor
  $\mathcal{F}_*: F\md_{\mathcal{A}} \longrightarrow
  F\md_{\mathcal{B}}$.  More importantly for us, it sends
$\mathcal{Y}(\mathcal{A})^{\Delta}$ to
$\mathcal{Y}(\mathcal{\mathcal{B}})^{\Delta}$ hence it gives rise to a
filtered $A_{\infty}$-functor
$$\mathcal{F}_* : \mathcal{Y}(\mathcal{\mathcal{A}})^{\Delta}
\longrightarrow \mathcal{Y}(\mathcal{\mathcal{B}})^{\Delta}.$$ Passing
to homology we obtain a TPC functor
$$PD(\mathcal{F}): PD(\mathcal{A}) \longrightarrow PD(\mathcal{B}).$$

\subsubsection{A larger derived category} \label{sbsb:fmod-2} As in
the previous sections, let $\mathcal{A}$ be a filtered
$A_{\infty}$-category. We will now construct a somewhat larger version
of the category $PD(\mathcal{A})$ from~\S\ref{sbsb:fmod} by adding
also acyclic modules.

Denote by $AF\md_{\mathcal{A}} \subset F\md_{\mathcal{A}}$ the full
$A_{\infty}$-subcategory of acyclic modules (see~\S\ref{sbsb:fai} for
the definition of the acyclic subcategory). Consider now the minimal
$A_{\infty}$ full subcategory of $F \md_{\mathcal{A}}$ which contains
both the image of the filtered Yoneda embedding
$\mathcal{Y}(\mathcal{A})$ as well as $AF\md_{\mathcal{A}}$ and is
closed under mapping cones over morphisms with filtration level
$\leq 0$ and quasi-isomorphisms of filtration level $\leq 0$.  Denote
this filtered $A_{\infty}$-category by
$\langle \mathcal{Y}(\mathcal{A}), AF\md_{\mathcal{A}}
\rangle^{\Delta}$. Define now the complete persistence derived
category of $\mathcal{A}$ to be the persistence homology category of
the latter, namely:
$$PD^{c}(\mathcal{A}) := PH \Bigl( \langle \mathcal{Y}(\mathcal{A}),
AF\md_{\mathcal{A}}\rangle^{\Delta} \Bigr).$$ The discussion
from~\S\ref{sbsb:fmod} on functors carries over to this case. More
specifically, if $\mathcal{A}$, $\mathcal{B}$ are filtered
$A_{\infty}$-categories and
$\mathcal{F}: \mathcal{A} \longrightarrow \mathcal{B}$ is an
$A_{\infty}$-functor with linear deviation then we obtain a TPC
functor:
$$PD^c(\mathcal{F}): PD^c(\mathcal{A})
\longrightarrow PD^c(\mathcal{B})$$ induced by $\mathcal{F}$.

There is an alternative description of $PD^c(\mathcal{A})$, by
completing $PD(\mathcal{A})$ with respect to the interleaving distance, or even by applying  the same
type of completion to $\mathcal{Y}(\mathcal{A})^{\Delta}$ and then
passing to persistence homology. The result turns out to be the same:
\begin{lem} \label{l:PDc}
  $PD^{c}(\mathcal{A}) = \bigl(PD(\mathcal{A})\bigr)^{c} = PH \Bigl(
  (\mathcal{Y}(\mathcal{A})^{\Delta})^{c} \Bigr)$.
\end{lem}

\subsubsection{$r$-isomorphisms} \label{sbsb:r-iso} In the context of
TPC's there is an important notion of $r$-isomorphism that leads to
interesting measurements and can also be used to arrive at the same
completion $PD^c(\mathcal{A})$ described in~\S\ref{sbsb:fmod-2}.

Let $\mathscr{C}$ be a TPC, $A,B \in \Ob(\mathscr{C})$,
$\phi: A \longrightarrow B$ a morphism in $\hom_{\mathscr{C}^0}(A,B)$
and $r \geq 0$. We say that $\phi$ is an $r$-isomorphism if $\phi$ can
be completed to an exact triangle in the (triangulated) category
$\mathscr{C}^0$
$$A \xrightarrow{\ \phi \ } B \xrightarrow{\ \ \ }
K \xrightarrow{\ \ \ } TA$$ with $K \in \Ob(\mathscr{C})$ $r$-acyclic
(see~\S\ref{sbsb:acyclic}). For example, the standard morphism
$\eta^A_r : \Sigma^r A \longrightarrow A$ is always an $r$-isomorphism
(this is in fact one of the axioms defining TPC's), but of course
there are many other examples. We refer the reader to~\cite{BCZ:tpc}
for more details on the concept of $r$-isomorphisms.

Completing with respect to $r$-isomorphisms is somewhat subtle because
of the following. The existence of an $r$-isomorphism
$\phi: A \longrightarrow B$ does not imply that there exists an
$r$-isomorphism $B \longrightarrow A$, hence this does not lead to a
symmetric relation on pairs of objects $(A,B)$. Nor is this
``relation'' transitive (if $\phi: A \longrightarrow B$ is an
$r_1$-isomorphism and $r_2$-isomorphism $\psi: B \longrightarrow C$ is
an $r_2$-isomorphism then $\psi \circ \phi: A \longrightarrow C$ is in
general only an $(r_1+r_2)$-isomorphism). However by the results
of~\cite{BCZ:tpc}, if $\phi: A \longrightarrow B$ is an
$r$-isomorphism then there exists a $2r$-isomorphism
$\Sigma^r B \longrightarrow A$. This leads to the following
definition: two objects $A, B \in \Ob(\msc)$ are called {\em almost
  isomorphic} if there is an $r\geq 0$ and $s \in \mathbb{R}$ and an
$r$-isomorphism $\phi: A \longrightarrow \Sigma^s B$. By the results
of~\cite{BCZ:tpc}, being ``almost isomorphic'' is an equivalence
relation. Moreover, we have:
\begin{lem} \label{l:a-iso} Let $\msc$ be a TPC and Let
  $A, B \in \Ob(\msc)$. Then $A$ is almost isomorphic to $B$ if and
  only if $\dint(A,B) < \infty$.
\end{lem}
It follows that completing $PD(\mathcal{A})$ with respect to almost
isomorphisms gives precisely the same category as $PD^c(\mathcal{A})$.

To end this discussion, let us also mention the following useful fact:
completing a TPC with respect to the interleaving distance always
yields a TPC. More precisely
\begin{lem} \label{l:cmp-TPC} Let $\mathscr{C}$ be a TPC and
  $\mathscr{D} \subset \mathscr{C}$ a full sub-TPC of
  $\mathscr{C}$. Then $\mathscr{D}^{c} \subset \mathscr{C}$ is also a
  sub-TPC.
\end{lem}

\subsubsection{Filtered twisted complexes} \label{sbsb:ftw} Let
$\mathcal{A}$ be a filtered $A_{\infty}$-category. There is another
variant of a TPC associated with $\mathcal{A}$ which is based on
twisted complexes (see~\cite[Chapter~I,
Section~3(l)]{Se:book-fukaya-categ} for unfiltered
$A_{\infty}$-categories and~\cite[Section~2.5]{BCZ:tpc} for the case
of filtered dg-categories. The case of filtered
$A_{\infty}$-categories is treated in detail in~\S\ref{ap:ftc}.

Roughly speaking, one defines first the additive enlargement
$\mathcal{A}^{\oplus}$ of $\mathcal{A}$ as
in~\cite{Se:book-fukaya-categ} (where it is denoted by
$\Sigma \mathcal{A}$, but in our case $\Sigma$ is already used for
other purposes), and defines a filtration structure on it by an
obvious extension of the filtrations on $\mathcal{A}$. Recall that, by
assumption, $\mathcal{A}$ is closed under shifts and translations.
Next one forms a new filtered $A_{\infty}$-category $F Tw \mathcal{A}$
whose objects are all twisted complexes $(X, q_X)$ over $\mathcal{A}$
with $X \in \Ob(\mathcal{A}^{\oplus})$ and differential
$q_X \in \hom^{0}_{\mathcal{A}^{\oplus}}(X,X)$ lying at filtration
level $\leq 0$. The morphisms in this category come from
$\mathcal{A}^{\oplus}$ and are thus filtered. One then defines the
$A_{\infty}$-operations on $FTw \mathcal{A}$ in the standard way and
the fact that $\mathcal{A}$ itself is filtered implies that
$FTw \mathcal{A}$ is filtered too. Finally note that since
$\mathcal{A}$ is strictly unital the same holds for
$F Tw \mathcal{A}$.

The filtered $A_{\infty}$-category $\mathcal{A}$ embeds into
$Tw \mathcal{A}$ in an obvious way, the embedding being a filtered
$A_{\infty}$-functor which is full and faithful (on the chain
level). Moreover, $Tw \mathcal{A}$ is pre-triangulated in the filtered
sense (which in particular means that it is closed under formation of
filtered mapping cones). We denote by $PH(F Tw \mathcal{A})$ the
persistence homological category of $F Tw \mathcal{A}$. Standard
arguments show that $PH(F Tw \mathcal{A})$ is a TPC with translation
and shift functors induced from those of $\mathcal{A}$.

An important property of $PH(F Tw \mathcal{A})$ is the following.  Let
$\mathcal{A}$, $\mathcal{B}$ be two filtered
$A_{\infty}$-categories. Let
$\mathcal{F}: \mathcal{A} \longrightarrow \mathcal{B}$ be a filtered
$A_{\infty}$-functor. Then there is a canonical extension of
$\mathcal{F}$ to a filtered, strictly unital, $A_{\infty}$-functor
$Tw \mathcal{F}: FTw \mathcal{A} \longrightarrow FTw \mathcal{B}$.
Moreover, $Tw \mathcal{F}$ induces a homological functor
$$PH(Tw \mathcal{F}): PH(F Tw \mathcal{A}) \longrightarrow PH(FTw
\mathcal{B})$$ which is a TPC functor. 
% !TEX root = approx8.tex

\subsection{Systems of categories with increasing
  accuracies} \label{s:sys-tpc}

In what follows we will deal with families of TPC's that should be
thought of as approximations of a (currently not yet defined) limit TPC. The
TPC's occurring in the family are parametrized by a space controlling
the accuracy-level of their approximation. This situation naturally
occurs when dealing with Fukaya categories that can be parametrized by
different choices of perturbation data. The setting and definitions
below were conceived with the example of Fukaya categories in
mind. Nevertheless we believe that a general axiomatic framework may
prove useful in other cases too.

We will axiomatize these notions in three different settings below,
starting with the most general one - a system of PC's.

\subsubsection{The case of PC's} \label{sb:sys-pc}
Let $(\mathcal{P}, \preceq)$ be a directed set (i.e.~a preordered set
such that
$\forall \; p', p'' \in \mathcal{P}, \; \exists \; q \in \mathcal{P}$
with $p', p'' \preceq q$), together with a function
$\nu: \mathcal{P} \longrightarrow \mathbb{R}_{\geq 0}$ that has the
following properties:
\begin{enumerate}
\item For every $p \preceq q$ we have $\nu(p) \geq \nu(q)$ (i.e.~$\nu$
  is decreasing with respect to $\preceq$).
\item For every $\delta>0$ there exists $p \in \mathcal{P}$ such that
  $\nu(p) \leq \delta$.
\end{enumerate}

We will refer to $(\mathcal{P}, \preceq , \nu)$ as the parameter space
for our system of PC's. For brevity we will omit $\preceq$ and $\nu$
from the notation and simply write $\mathcal{P}$ for the triple. We
will refer to $\nu(p)$ as the norm or size of the parameter
$p \in \mathcal{P}$.

Consider now a family
$\widehat{\msc} = \{\msc_p\}_{p \in \mathcal{P}}$ of PC's $\msc_p$
parametrized by $p \in \mathcal{P}$. We will denote the shift and
translation functors of all the $\msc_p$'s by the same notation,
$\Sigma$ and $T$, respectively. Assume further that for every
$p \preceq q$ we are given two sets
$\msj_{p,q} = \mathscr{J}_{p,q}(\widehat{\msc})$,
$\msj_{q,p} = \mathscr{J}_{q,p}(\widehat{\msc})$, of PC functors
$\msc_p \longrightarrow \msc_q$, $\msc_q \longrightarrow \msc_p$,
respectively, with the following properties: \label{pp:sys-functors}
\begin{enumerate}
\item $\msj_{p,p} = \{\id_{\msc_p}\}$ for every $p \in \mathcal{P}$.
\item \label{i:0-iso-sys-1} If $p \preceq q$ then all the functors in
  $\msj_{p,q}$ are mutually $0$-isomorphic. More specifically, for
  every two functors $\msf, \msg \in \msj_{p,q}$ there is a natural
  isomorphism $\msf \cong \msg$ of persistence level $0$ (in other
  words, $\msf$ and $\msg$ are isomorphic in the $0$-level category
  $\text{pc-fun}^0(\msc_p, \msc_q)$ of the PC of PC-functors
  $\text{pc-fun}(\msc_p, \msc_q)$). Here and in what follows we will
  abbreviate this by writing $\msf \cong_0 \msg$. Furthermore, we
  require the same thing also for any two functors in $\msj_{q,p}$.
\item \label{i:0-iso-sys-2} If $p \preceq q \preceq r$ then for every
  $\msf \in \msj_{p,q}$, $\msg \in \msj_{q,r}$, $\msh \in \msj_{p,r}$
  we have $\msg \circ \msf \cong_0 \msh$.
\item \label{i:d-int-func} If $p \preceq q$ then for every
  $\msf \in \msj_{p,q}$ and $\msh \in \msj_{q,p}$ we have
  $$\dint(\msh \circ \msf, \id_{\msc_p}), \, \dint(\msf \circ
  \msh, \id_{\msc_q}) \leq C | \nu(p) - \nu(q)|,$$ for some constant
  $C$ that depends only on the entire family $\widehat{\msc}$ but
  neither on $p, q$ nor on $\msf, \msg$. Here the two $\dint$'s
  stand for the interleaving distances on the PC's of PC functors
  $\msc_p \longrightarrow \msc_p$ and $\msc_q \longrightarrow \msc_q$.
\end{enumerate}
In what follows we will denote by $\widehat{\msc}$ the entire data
above, namely both the family of categories
$\{\msc_p\}_{p \in \mathcal{P}}$ as well the sets of functors
$\msj_{p,q}$ and refer to $\widehat{\msc}$ as a {\em system of PC's
  with increasing accuracy}, or sometimes a {\em system of PC's} for
short. We will call the functors in $\msj_{p,q}$ comparison functors
and typically denote such a functor by
$\msh_{p,q}: \msc_p \longrightarrow \msc_q$.

\begin{rem} \label{rem:pc-limits}
  \begin{enumerate}
  \item Given a system $\widehat{\msc}$ of PC's with increasing accuracy,
    all the $\infty$-level categories $\msc^{\infty}_p$, corresponding
    to the PC's $\msc_p$ in the system $\widehat{\msc}$, are mutually
    equivalent. From the persistence viewpoint, as $p$ ``grows'' (with
    respect to $\preceq$) the comparison between $\msc_p$ and
    $\msc_q$ for $p \preceq q$ becomes more and more
    precise. \label{i:rem-PC-1}
  \item A single PC $\msc$ makes a special case of the definition
    above. We can view it as a system of PC's parametrized by a
    $1$-point parameter space $\mathcal{P} = \{*\}$. In this case we
    set the norm function $\nu$ to be trivial, $\nu(*) =0$.
  \end{enumerate}
\end{rem}

\subsubsection{Functors of systems of PC's} \label{sbsb:func-sys-pc}
Next we discuss functors between systems of PC's with increasing
accuracy. Let $\widehat{\msc}$ and $\widehat{\msd}$ be two systems of
PC's with increasing accuracy. Whenever we need to distinguish between
the additional structures of each of the systems $\widehat{\msc}$ and
$\widehat{\msd}$ we will use $\widehat{\msc}$ and $\widehat{\msd}$
subscripts or superscripts, depending on notational convenience (for
example, the parameter space of $\widehat{\msc}$ will be denoted
$\mathcal{P}_{\widehat{\msc}}$, the comparison functors
$\msh_{p,q}^{\widehat{\msc}}$, etc. but we will denote the preorders
and size of parameters in both systems by $\preceq$ and $\nu$
respectively).

A functor
$\widehat{\msf} : \widehat{\msc} \longrightarrow \widehat{\msd}$
consists of the following structures:
\begin{enumerate}
\item A map
  $\widehat{\msf}: \mathcal{P}_{\widehat{\msc}} \longrightarrow
  \mathcal{P}_{\widehat{\msd}}$ between the parameter spaces of
  $\widehat{\msc}$ and of $\widehat{\msd}$ (which for simplicity we
  have denoted by $\widehat{\msf}$ too). 
\item A family of PC functors
  $\bigl\{ \msf_p: \msc_p \longrightarrow
  \msd_{\widehat{\msf}(p)}\bigr\}_{p \in
    \mathcal{P}_{\widehat{\msc}}}$, parametrized by
  $\mathcal{P}_{\widehat{\msc}}$.
\end{enumerate}
These two structures are required to have the following additional
properties. The map $\widehat{\msf}$ from~(1) above should satisfy
that for every $p,q \in \mathcal{P}_{\widehat{\msc}}$ with
$p \preceq q$ we have $\widehat{\msf}(p) \preceq
\widehat{\msf}(q)$. Moreover, we require that
$\nu(\widehat{\msf}(p)) \leq \nu(p)$ for every
$p \in \mathcal{P}_{\widehat{\msc}}$. The functors $\msf_p$ are
required to satisfy the following two conditions. First, for every
$p \preceq q$ in $\mathcal{P}_{\widehat{\msc}}$ and any choices of
comparison functors $\msh^{\widehat{\msc}}_{p,q}$,
$\msh^{\widehat{\msd}}_{\widehat{\msf}(p), \widehat{\msf}(q)}$ the
following diagram:
\begin{equation} \label{sys-fun-diag-1} \xymatrixcolsep{3pc}
  \xymatrixrowsep{3pc} \xymatrix{ \msc_p \ar[r]^-{\msf_p}
    \ar[d]_-{\msh^{\widehat{\msc}}_{p,q}} & \msd_{\widehat{\msf}(p)}
    \ar[d]^-{\msh^{\widehat{\msd}}_{\widehat{\msf}(p), \widehat{\msf}(q)}}\\
    \msc_q \ar[r]^-{\msf_q} & \msd_{\widehat{\msf}(q)}}
\end{equation}
commutes up to natural isomorphisms in the $0$-level category
$\text{pc-fun}^0(\msc_p, \msd_{\widehat{\msf}(q)})$ of the PC of
PC-functors $\text{pc-fun}(\msc_p, \msd_{\widehat{\msf}(q)})$.
Similarly, we also require that for every $p \preceq q$ in
$\mathcal{P}_{\widehat{\msc}}$ and every choice of comparison functors
$\msh^{\widehat{\msc}}_{q,p}$,
$\msh^{\widehat{\msd}}_{\widehat{\msf}(q), \widehat{\msf}(p)}$ the
diagram:
\begin{equation} \label{sys-fun-diag-2} \xymatrixcolsep{3pc}
  \xymatrixrowsep{3pc} \xymatrix{ \msc_p \ar[r]^-{\msf_p} &
    \msd_{\widehat{\msf}(p)}
    \\
    \msc_q \ar[u]^-{\msh^{\widehat{\msc}}_{p,q}} \ar[r]^-{\msf_q} &
    \msd_{\widehat{\msf}(q)}
    \ar[u]_-{\msh^{\widehat{\msd}}_{\widehat{\msf}(p),
        \widehat{\msf}(q)}} }
\end{equation}
commutes up to natural isomorphisms in the $0$-level category
$\text{pc-fun}^0(\msc_q, \msd_{\widehat{\msf}(p)})$ of the PC of
PC-functors $\text{pc-fun}(\msc_q, \msd_{\widehat{\msf}(p)})$.

It remains to define natural transformations between functors of
systems of PC's. Let $\widehat{\msc}, \widehat{\msd}$ be two systems of PC's and
$\widehat{\msf}, \widehat{\msg}: \widehat{\msc} \longrightarrow
\widehat{\msd}$ two functors having the same action on the parameter
spaces (i.e.~the two maps
$\widehat{\msf}, \widehat{\msg}: \mathcal{P}_{\widehat{\msc}}
\longrightarrow \mathcal{P}_{\widehat{\msd}}$ coincide). A natural
transformation
$\widehat{\theta}: \widehat{\msf} \longrightarrow \widehat{\msg}$ of
shift $s \in \mathbb{R}$, consists of a family of persistence natural
transformations $\theta_p : \msf_p \longrightarrow \msg_p$ for every
$p \in \mathcal{P}$ each of which of shift $s$. In addition the
$\theta_p$ are required to have the following weak compatibility with
the comparison functors. For every $p \preceq q$ and every
$X \in \Ob(\msc_p)$ there exist $0$-isomorphisms
$$\tau_{p,q}: \msh_{\widehat{\msf}(p), \widehat{\msf}(q)}(\msf_p (X))
\longrightarrow \msf_q (\msh_{p,q}(X)), \quad \sigma_{p,q}:
\msh_{\widehat{\msf}(p), \widehat{\msf}(q)}(\msg_p(X)) \longrightarrow
\msg_q(\msh_{p,q}(X)),$$ such that
$$\sigma_{p,q} \circ \msh_{\widehat{\msf}(p), \widehat{\msf}(q)}(\theta_p(X)) = 
\theta_q(\msh_{p,q}(X)) \circ \tau_{p,q}$$ in
$\hom_{\msd_{\widehat{\msf}(q)}} \bigl( \msh_{\widehat{\msf}(p),
  \widehat{\msf}(q)}(\msf_p (X)), \msg_q(\msh_{p,q}(X))
\bigr)$. Finally, we require a similar compatibility also with the
comparison functors $\msh_{q,p}$.

\subsubsection{The case of TPC's} \label{sb:sys-tpc} A system of TPC's
with increasing accuracy is simply a system $\widehat{\msc}$ of PC's
as in~\S\ref{sb:sys-pc} with the following additional
assumptions. Each of the categories $\msc_p$, $p \in \mathcal{P}$, in
the system is assumed to be a TPC. Moreover, all the functors in the
collections $\msj_{p,q}$ and $\msj_{q,p}$ are assumed to be TPC
functors. The notion of functors of PC's (and natural transformations
between them) from~\S\ref{sbsb:func-sys-pc} extends to the setting of
systems of TPC's in a straightforward way, by just requiring the
functors ${\msf}_p$ to be TPC functors.

Note also that point~\eqref{i:rem-PC-1} of Remark~\ref{rem:pc-limits}
carries over to systems
$\widehat{\msc} = \{\msc_p\}_{p \in \mathcal{P}}$ of TPC's, namely the
$\infty$-levels $\msc^{\infty}_p$, $p \in \mathcal{P}$, are mutually
equivalent in the triangulated sense.

\subsubsection{Systems of filtered
  $A_{\infty}$-categories} \label{sb:sys-ainfty} There is also a
notion of system of filtered $A_{\infty}$-categories with increasing
accuracy. The definition is analogous to that given for a system of
PC's in~\S\ref{sb:sys-pc} above, but with one significant change
regarding the comparison functors between the $q$th-category and the
$p$th one when $p \preceq q$.

In the $A_{\infty}$-case a system of categories with increasing
accuracy consists of the following. A family
$\widehat{\mathcal{A}} = \{\mathcal{A}_p\}_{p \in \mathcal{P}}$ of
filtered $A_{\infty}$-categories $\mathcal{A}_p$, parametrized by
$p \in \mathcal{P}$. Note that we do not assume these
$A_{\infty}$-categories to be pretriangulated. The parameter space
$\mathcal{P}$ is endowed with the same structures $\nu$ and $\preceq$
as in~\S\ref{sb:sys-tpc}. For $p, q \in \mathcal{P}$ with
$p \preceq q$ we have two collections $\mathcal{J}_{p,q}$ and
$\mathcal{J}_{q,p}$ of comparison $A_{\infty}$-functors with the
following properties. The functors
$\mathcal{H}_{p,q} \in \mathcal{J}_{p,q}$ from the first collection
are filtered $A_{\infty}$-functors. The functors
$\mathcal{H}_{q,p} \in \mathcal{J}_{q,p}$ from the second collection
are assumed to be $A_{\infty}$-functors
$\mathcal{H}_{q,p}: \mathcal{A}_q \longrightarrow \mathcal{A}_p$ with
linear deviation rate $\leq C | \nu(p) - \nu(q)|$, for some constant
$C$ that depends only on $\widehat{\mathcal{A}}$ (and not on $p,
q$). Moreover, the functors from $\mathcal{J}_{p,q}$ and
$\mathcal{J}_{q,p}$ are assumed to satisfy the $A_{\infty}$-analogs of
the properties listed in~\S\ref{sb:sys-tpc} on
page~\pageref{pp:sys-functors}, with the following obvious
modifications. Instead of $0$-isomorphisms $\cong_0$ between two
functors mentioned at points~\eqref{i:0-iso-sys-1}
and~\eqref{i:0-iso-sys-2} on page~\pageref{pp:sys-functors} we will
now require that we have $0$-isomorphisms between the respective functors in the persistence homology category of the
filtered $A_{\infty}$-functors
$PH (F\text{fun}_{A_{\infty}}(\mathcal{A}_p, \mathcal{A}_q))$, and
similarly for the comparison functors in
$PH(\text{fun}^{\text{LD}}(\mathcal{A}_q, \mathcal{A}_p))$. The
assumptions in point~\eqref{i:d-int-func} on
page~\pageref{pp:sys-functors} will be now replaced by requiring 
$\dint(\mathcal{F} \circ \mathcal{H}, \id) \leq C | \nu(p) - \nu(q)|$
for the interleaving distance on
$PH (\text{fun}_{A_{\infty}}(\mathcal{A}_p, \mathcal{A}_p))$, and
similarly for $\dint(\mathcal{H} \circ \mathcal{F}, \id)$.

The notion of functors between systems of PC's (as well as natural
transformations between them) has an $A_{\infty}$-analog, following
the preceding principle.

\subsubsection{Modules and Yoneda embeddings of
  systems} \label{sbsb:yoneda-sys}

Let $\widehat{\mathcal{A}} = \{\mathcal{A}_p\}_{p \in \mathcal{P}}$ be
a system of filtered $A_{\infty}$-categories. Denote by
$F\md_{\mathcal{A}_p}$ the $A_{\infty}$-category of filtered
$\mathcal{A}_p$-modules. Recall that this is a filtered
$A_{\infty}$-category (in fact a filtered dg-category). These
categories fit into a system
$$F\md(\widehat{\mathcal{A}}) := \{ F\md_{\mathcal{A}_p} \}_{p \in
  \mathcal{P}}$$ of filtered $A_{\infty}$-categories, where for
$p \preceq q$ the comparison functors
$$(\mathcal{H}_{p,q})_* : F\md_{\mathcal{A}_p} \longrightarrow
F\md_{\mathcal{A}_q}$$ are just the push-forward functors induced by
the comparison functors $\mathcal{H}_{p,q}$ of the system
$\widehat{\mathcal{A}}$, and similarly for $(\mathcal{H}_{q,p})_*$.
Since the comparison functors $\mathcal{H}_{p,q}$, for $p \preceq q$,
are filtered the same holds for $(\mathcal{H}_{p,q})_*$. As for the
functors in the opposite direction, interestingly the pushforward
functors $(\mathcal{H}_{q,p})_*$ behave somewhat better than the
$\mathcal{H}_{q,p}$ in the following sense. While $\mathcal{H}_{q,p}$
are not really filtered functors but only LD functors, the induced
comparison functors $(\mathcal{H}_{q,p})_*$ are genuinely
filtered. This is a very nice feature of the filtered push-forward
construction (see~\S\ref{ap:pshf}). However, note that it is still the
case that $(\mathcal{H}_{p,q})_* \circ (\mathcal{H}_{q,p})_*$ is only
$r$-quasi-isomorphic to $\id$, for $r = C | \nu(p) - \nu(q)|$, and the
same for $(\mathcal{H}_{q,p})_* \circ (\mathcal{H}_{p,q})_*$.

We now turn to Yoneda embeddings of systems. Let
$\widehat{\mathcal{A}}$ be a system of filtered
$A_{\infty}$-categories as above. Denote by 
$$\mathcal{Y}_p: \mathcal{A}_p \longrightarrow F\md_{\mathcal{A}_p}$$ 
the filtered Yoneda embedding (see~\S\ref{gio:fyon-def}), and let
$(\mathcal{Y}_p(\mathcal{A}_p))^{(\text{q-iso},0)}$ be the minimal
full subcategory of $F\md_{\mathcal{A}_p}$ which contains the objects
$\mathcal{Y}_p(\Ob(\mathcal{A}_p))$ and is also closed with respect to
$0$-quasi-isomorphisms. Note that the comparison functors
$(\mathcal{H}_{p,q})_*$ and $(\mathcal{H}_{q,p})_*$ preserve these
subcategories. The main point here is that for every
$A \in \Ob(\mathcal{A}_p)$ and $p \preceq q$ the modules
$(\mathcal{H}_{p,q})_* \mathcal{Y}_p(A)$ and
$\mathcal{Y}_q(\mathcal{H}_{p,q}(A))$ are
$0$-quasi-isomorphic. Similarly, for every $B \in \Ob(\mathcal{A}_q)$
the modules $(\mathcal{H}_{q,p})_*(\mathcal{Y}_q(B))$ and
$\Sigma^r \mathcal{Y}_p(\mathcal{H}_{q,p}(B))$ are
$0$-quasi-isomorphic, where $r \leq C |\nu(p) - \nu(q)|$ is the
deviation rate of the LD functor $\mathcal{H}_{q,p}$. Recall that our
categories $\mathcal{A}_p$ are endowed with a shift functor $\Sigma$,
hence the categories $\mathcal{Y}(A_p)$ and
$(\mathcal{Y}_p(\mathcal{A}_p))^{(\text{q-iso},0)}$ are both closed
under shifts.

It follows that the categories
$(\mathcal{Y}_p(\mathcal{A}_p))^{(\text{q-iso},0)}$,
$p \in \mathcal{P}$, together with the restrictions of the
push-forward comparison functors $(\mathcal{H}_{p,q})_*$,
$(\mathcal{H}_{q,p})_*$ to them, form a system of filtered
$A_{\infty}$-categories (which can be viewed as a subsystem of
$F\md(\widehat{\mathcal{A}})$). We denote this system by
$\mathcal{Y}(\widehat{\mathcal{A}})$ and call it the Yoneda system of
$\widehat{\mathcal{A}}$.

\subsubsection{Systems of PC's associated with systems of
  $A_{\infty}$-categories} \label{sbsb:sys-ai-pc} Given a system of
filtered $A_{\infty}$-categories
$\widehat{\mathcal{A}} = \{ \mathcal{A}_p\}_{p \in \mathcal{P}}$ one
can obtain a system of PC's by passing to persistence
homology. However, there is a small subtlety in this procedure. The
corresponding system $\{PH(\mathcal{A}_p)\}_{p \in \mathcal{P}}$ of
persistence homological categories is not really a system of PC's,
because for $p \preceq q$ the $A_{\infty}$-functors from
$\mathcal{J}_{q,p}$ are not filtered but only LD functors, hence do not
induce PC functors
$PH(\mathcal{A}_q) \longrightarrow PH(\mathcal{A}_p)$. However this
problem can be rectified by passing to the Yoneda system introduced
earlier in~\S\ref{sbsb:yoneda-sys}. More precisely, consider the
Yoneda system $\mathcal{Y}(\widehat{\mathcal{A}})$ of
$\widehat{\mathcal{A}}$. Recall from~\S\ref{sbsb:yoneda-sys} that this
is (genuine) system of filtered $A_{\infty}$-categories. Define
$$PH(\widehat{\mathcal{A}}) :=
\{PH((\mathcal{Y}_p(\mathcal{A}_p))^{(\text{q-iso},0)})\}_{p \in
  \mathcal{P}}$$ to be system consisting of the persistence homology
categories of $(\mathcal{Y}_p(\mathcal{A}_p))^{(\text{q-iso},0)}$
forming the Yoneda system $\mathcal{Y}(\widehat{\mathcal{A}})$. The
comparison functors $\msh_{p,q}$ are defined to be the persistence
homology functors induced by the $(\mathcal{H}_{p,q})_*$,
i.e.~$\msh_{p,q} := PH((\mathcal{H}_{p,q})_*)$ and similarly for
$\msh_{q,p}$.

% More precisely, denote by $\mathcal{Y}$ the Yoneda embedding functor
% and consider the $A_{\infty}$-categories
% $\mathcal{Y}(\mathcal{A}_p)^{\ziso}$, $p \in \mathcal{P}$, of filtered
% Yoneda modules (namely, the images of $\mathcal{A}_p$ under the Yoneda
% embedding followed by a completion with respect to to
% $0$-quasi-isomorphisms; see~\S\ref{sbsb:comp}).  For $p \preceq q$,
% the comparision functors $\mathcal{H}_{p,q}$ and $\mathcal{H}_{q,p}$
% induce now push forward functors
% $(\mathcal{H}_{p,q})_*: \mathcal{Y}(\mathcal{A}_p)^{\ziso}
% \longrightarrow \mathcal{Y}(\mathcal{A}_q)^{\ziso}$ and
% $(\mathcal{H}_{q,p})_*: \mathcal{Y}(\mathcal{A}_q)^{\ziso}
% \longrightarrow \mathcal{Y}(\mathcal{A}_p)^{\ziso}$. The important
% point now is that both of these functors are filtered. This is clear
% for $(\mathcal{H}_{p,q})_*$ (since $\mathcal{H}_{p,q}$ is filtered),
% but it also holds for $(\mathcal{H}_{q,p})_*$ since
% $\mathcal{H}_{q,p}$ has linear deviation.

% We now define
% $PH(\widehat{A}) := \{PH(\mathcal{Y}(\mathcal{A}_p)^{\ziso})\}_{p \in
%   \mathcal{P}}$ and obtain a system of PC's. It has the same parameter
% space $\mathcal{P}$, the same preorder $\preceq$, and size function
% $\nu$. The comparison functors
% $\msh_{p,q} = PH((\mathcal{H}_{p,q})_*)$ are the persistence homology
% functors induced from $(\mathcal{H}_{p,q})_*)$ and the same for
% $\msh_{q,p}$.

The assignment
$\widehat{\mathcal{A}} \longmapsto PH(\widehat{\mathcal{A}})$ extends
to a functor
$\text{SYS}_{F A_{\infty}} \longrightarrow \text{SYS}_{\text{PC}}$
between the category of systems of filtered $A_{\infty}$-categories
and the category of systems of PC's. This functor assigns to a functor
$\widehat{\mathcal{F}} : \widehat{\mathcal{A}} \longrightarrow \widehat{\mathcal{B}}$
between systems of filtered $A_{\infty}$-categories a functor
$PH(\widehat{\mathcal{F}}): PH(\widehat{\mathcal{A}}) \longrightarrow
PH(\widehat{\mathcal{B}})$ between the corresponding systems of PC's and this
correspondence behaves as expected on natural transformations between
such functors.

Finally, we remark that if
$\widehat{\mathcal{B}} = \{\mathcal{B}_p\}_{p \in \mathcal{P}}$ is a systems of
pre-triangulated filtered $A_{\infty}$-categories then the
corresponding system $PH(\widehat{\mathcal{B}})$ of PC's is in fact a
system of TPC's. For the rest of this paper, the most important
example in this context will be the following. Let
$\widehat{\mathcal{A}} = \{\mathcal{A}_p\}_{p \in \mathcal{P}}$ be a
system of filtered $A_{\infty}$-categories. Consider the Yoneda system
$\mathcal{Y}(\widehat{\mathcal{A}})$ as defined
in~\S\ref{sbsb:yoneda-sys}. By applying the construction
from~\S\ref{sbsb:fmod} to each category that participates in the
latter system we obtain a new system
$\mathcal{Y}(\mathcal{\widehat{A}})^{\Delta} =
\{\mathcal{Y}(\mathcal{A}_p)^{\Delta} \}_{p \in \mathcal{P}}$ of
filtered pre-triangulated categories. The comparison functors for this
system are induced from those of the Yoneda system
(see~\S\ref{sbsb:yoneda-sys}). This works since filtered and LD
$A_{\infty}$-functors preserve triangulated completions as defined
in~\S\ref{sbsb:fmod}. We now pass to persistence homology and define
the following system of TPC's:
\begin{equation} \label{eq:PDA-sys} PD(\widehat{\mathcal{A}}) :=
  \{PH(\mathcal{Y}(\mathcal{A}_p)^{\Delta})\}_{p \in \mathcal{P}}.
\end{equation}
We call this system the {\em persistence derived system of}
$\widehat{\mathcal{A}}$.

\subsubsection{Systems with a coherent base of
  objects} \label{sbsb:sys-base} Let
$\widehat{\msc} = \{\msc_p\}_{p \in \mathcal{P}}$ be a system of PC's
with increasing accuracy. A coherent base of objects $\{mathscr\}{B}$ for
$\widehat{\msc}$ is a family of subsets
  $\mathscr{B} = \{\mathscr{B}_p\}_{p \in \mathcal{P}}$ of objects
  $\mathscr{B}_p \subset \Ob(\msc_p)$ for every $p \in \mathcal{P}$,
satisfying the following properties:
\begin{enumerate}
\item For every $p \preceq q$, every two comparison functors
  $\msh'_{p,q}, \msh''_{p,q} \in \mathscr{J}_{p,q}$ and every
  $A \in \mathscr{B}_p$ we have $\msh'_{p,q}(A) = \msh''_{p,q}(A)$ and
  $\msh'_{p,q}(A) \in \mathscr{B}_q$. We require the analogous
  properties to hold also for every two comparison functors
  $\msh'_{q,p}, \msh''_{q,p} \in \mathscr{J}_{q,p}$ and every
  $B \in \mathscr{B}_q$.
\item For every $p \preceq q \preceq r$, every two comparison functors
  $\msh_{p,q} \in \mathscr{J}_{p,q}$,
  $\msh_{q,r} \in \mathscr{J}_{q,r}$ and every $A \in \mathscr{B}_p$
  we have $\msh_{q,r}(\msh_{p,q}(A)) = \msh_{p,r}(A)$. We require the
  analogous properties to hold also for all comparison functors going
  in the other direction, namely for $\msh_{r,q}$, $\msh_{q,p}$,
  $\msh_{r,p}$ and
  every object $B \in \mathscr{B}_r$.
\item For every $p \preceq q$, every choice of comparison functors
  $\msh_{p,q}$, $\msh_{q,p}$ and every $A \in \mathscr{B}_p$,
  $B \in \mathscr{B}_q$ we have $\msh_{q,p}(\msh_{p,q}(A)) = A$,
  $\msh_{p,q}(\msh_{q,p}(B)) = B$.
\end{enumerate}
It follows from the definition above that for every
$p', p'' \in \mathcal{P}$ we have a canonical bijection
$\mathscr{B}_{p'} \longrightarrow \mathcal{B}_{p''}$, namely
$\msh_{q, p''} \circ \msh_{p', q}$ where $q \in \mathcal{P}$ is any
element with $q \succeq p', p''$.

One can define functors between systems of PC's with a coherent base
of objects in a straightforward way.

The concept of systems with a coherent base of objects (as well
functors between such) carries over without any modification also to
the case of systems of filtered $A_{\infty}$-categories.

A special case of a system of categories (filtered $A_{\infty}$ or
PC's) with a coherent base of objects that will appear later is when
$\mathscr{B}_p = \Ob(\msc_p)$ for every $p \in \mathcal{P}$. We will
refer to such a situation as a {\em coherent full base of objects}.
In the following we will encounter an even simpler situation, where
all the categories $\msc_p$, $p \in \mathcal{P}$, have the same set of
objects (i.e.~$\Ob(\mathcal{C}_{p'}) = \Ob(\mathcal{C}_{p''})$ for
every $p', p'' \in \mathcal{P}$) and all the comparison functors act
like the identity on objects, namely $\msh_{p,q}(X) = X$ for every
$p \preceq q$, every comparison functor $\msh_{p,q}$, and
$X \in \Ob(\mathcal{C}_{p})$. In this case we will view the coherent
base $\mathscr{B}$ as a single set rather than a family, and set
$\mathscr{B} = \Ob(\mathcal{C}_p)$ for any $p \in \mathcal{P}$.  We
will call this type of systems by the name {\em system of categories
  (PC's or filtered $A_{\infty}$-categories) with a fixed full base of
  objects}.

\begin{rem} \label{r:coh-base} Consider a system of filtered
  $A_{\infty}$-categories
  $\widehat{\mathcal{A}} = \{\mathcal{A}_p\}_{p \in \mathcal{P}}$ with
  a fixed full base of objects $\mathcal{B}$. Let
  $F\md(\widehat{\mathcal{A}})$ be the corresponding system of
  filtered modules over the $\mathcal{A}_p$'s and
  $\mathcal{Y}(\widehat{\mathcal{A}})$ be the Yoneda system of
  $\widehat{\mathcal{A}}$ (see~\S\ref{sbsb:yoneda-sys}). Unlike
  $\widehat{\mathcal{A}}$ these other two systems do not have a fixed
  (or even coherent) full base of objects. The reason is that the
  comparison functors for $F\md(\widehat{\mathcal{A}})$ and
  $\mathcal{Y}(\widehat{\mathcal{A}})$ are the ``module push-forward''
  functors $\mathcal{H}_*$ induced from the comparison functors
  $\mathcal{H}$ of $\widehat{\mathcal{A}}$ and these do not satisfy
  the conditions required for a coherent base of objects. More
  specifically, for $p \preceq q \preceq r$, the functors
  $(\mathcal{H}_{q,r})_* \circ (\mathcal{H}_{p,q})_*$ and
  $(\mathcal{H}_{q,r})_*$ do not act in the same way on filtered
  modules (though the images of a module under these two functors are
  $0$-quasi-isomorphic). Similarly, the push forward
  $(\mathcal{H}_{p,q})_* \mathcal{Y}(X)$ of a Yoneda module is not a
  Yoneda module but only $0$-quasi-isomorphic to a Yoneda module.
\end{rem}

\subsubsection{Homotopy systems} \label{sbsb:homotopy-sys}

Let $\widehat{\mathcal{A}} = \{\mathcal{A}_p\}_{p \in \mathcal{P}}$ be
a system of filtered $A_{\infty}$-categories with a coherent full base
of objects. We say that the system $\widehat{\mathcal{A}}$ is a {\em
  homotopy system} if the following conditions hold:
\begin{enumerate}
\item For every $p \preceq q$ in $\mathcal{P}$, and every two
  comparison functors
  $\mathcal{H}'_{p,q}, \mathcal{H}''_{p,q} \in \mathcal{J}_{p,q}$
  there is a homotopy
  $$T \in \hom_{F \text{fun}(\mathcal{A}_p,
    \mathcal{A}_q)}^0(\mathcal{H}'_{p,q}, \mathcal{H}''_{p,q})$$
  between $\mathcal{H}'_{p,q}$ and $\mathcal{H}''_{p,q}$ that does not
  shift filtrations. From now on we will call such $T$'s
  $0$-homotopies and say that $\mathcal{H}'_{p,q}$ and
  $\mathcal{H}''_{p,q}$ are $0$-homotopic, and the superscript in
  $\hom^0$ stands for pre-natural transformations of filtration level
  $0$ (i.e.~they preserve filtrations). We refer
  to~\cite[Section~(1h)]{Se:book-fukaya-categ} for the definition of
  homotopy between $A_{\infty}$-functors in the unfiltered case.
\item For every $p \preceq q$ in $\mathcal{P}$,  every two
  comparison functors
  $\mathcal{H}'_{q,p}, \mathcal{H}''_{q,p} \in \mathcal{J}_{q,p}$ are
  $0$-homotopic as LD-functors with deviation rate
  $C|\nu(p) - \nu(q)|$, where $C$ is the constant associated with the
  family $\widehat{\mathcal{A}}$ that appears in the definition of
  systems of filtered $A_{\infty}$-categories
  in~\S\ref{sb:sys-ainfty}. Specifically, this means that there exists
  a homotopy $T$ between the LD-functors
  $(\mathcal{H}'_{q,p}, C|\nu(p) - \nu(q)|)$ and
  $(\mathcal{H}''_{q,p}, C|\nu(p) - \nu(q)|)$ of shift $0$, namely its
  $d$th order $T_d$ shifts filtrations by $\leq d C|\nu(p) -
  \nu(q)|$. See~\S\ref{a:func-LD}.
\item For every $p \preceq q \preceq r$ in $\mathcal{P}$ and every
  choices of comparison functors $\mathcal{H}_{p,q}$,
  $\mathcal{H}_{q,r}$, $\mathcal{H}_{p,r}$ we have that
  $\mathcal{H}_{q,r} \circ \mathcal{H}_{p,q}$ and $\mathcal{H}_{p,r}$
  are $0$-homotopic.
\end{enumerate}

\subsubsection{Invariants of systems} \label{sb:sys-inv} Systems of
categories $\widehat{\msc} = \{\msc_p\}_{p \in \mathcal{P}}$ with
increasing accuracy call for considering their limits (or rather
colimits) when the parameter $p \in \mathcal{P}$ goes to ``infinity''
(or equivalently when the accuracy-levels $\nu(p)$ converges to
$0$). In other words, it would be desirable to be able to define a
colimit category $\colim_{p \in \mathcal{P}} \msc_p$ which will have
the same structures (e.g.~PC or TPC structures, filtered
$A_{\infty}$-structure etc.) as the categories $\msc_p$ forming the
family $\widehat{\msc}$. This seems a non-trivial foundational problem
and we will not attempt to solve it in this paper. Instead, we will
show that some algebraic invariants and measurements associated with
the categories $\msc_p$ do admit colimits, even without having a
colimit category. We will concentrate on three such invariants (each
of which is defined in a different settings): the interleaving
(pseudo)-distance on the set of objects of a PC, the Grothendieck
groups $K_0$ of triangulated categories and Hochschild homologies of
$A_{\infty}$-categories.

We begin with a useful equivalence relation on the totality of
  objects of a system of categories.  \label{p:equiv-ob} Let
  $\widehat{\msc} = \{\msc_p\}_{p \in \mathcal{P}}$ be a system of
  PC's with increasing accuracy. Consider
  $$\Ob^{\text{tot}}(\widehat{\msc}) := \coprod_{p\in \mathcal{P}}
  \Ob(\msc_p)$$ and denote elements of this set by $(A,p)$ with
  $p \in \mathcal{P}$ and $A \in \msc_p$. Define an equivalence
  relation $\sim_{\widehat{\msc}}$ on
  $\Ob^{\text{tot}}(\widehat{\msc})$ as follows:
  $(A', p') \sim_{\widehat{\msc}} (A'', p'')$ if there exists
  $q \in \mathcal{P}$ with $p',p'' \preceq q$ such that
  $\msh_{p',q}(A') \cong_0 \msh_{p'',q}(A'')$ for some comparison
  functors $\msh_{p',q} \in \msj_{p',q}$ and
  $\msh_{p'',q} \in \msj_{p'',q}$. Here $\cong_0$ means an isomorphism
  in $\msc_q^0$. (Since all the functors in each of $\msj_{p',q}$,
  $\msj_{p'',q}$, are mutually $0$-isomorphic, requiring that
  $\msh_{p',q}(A') \cong_0 \msh_{p'',q}(A'')$ for some
  $\msh_{p',q} \in \msj_{p',q}$, $\msh_{p'',q} \in \msj_{p'',q}$, is
  the same as to ask that the same property holds for {\em all}
  $\msh_{p',q} \in \msj_{p',q}$, $\msh_{p'',q} \in \msj_{p'',q}$.)
  Clearly if $A', A'' \in \Ob(\msc_p)$ are $0$-isomorphic then
  $(A', p) \sim_{\widehat{\msc}} (A'',p)$, but in general the
  equivalence relation $\sim_{\widehat{\msc}}$ may relate more objects
  belonging to categories parametrized by different $p$'s.
Denote by $\widetilde{\Ob}(\widehat{\msc})$ the equivalence
  classes of $\sim_{\widehat{\msc}}$.  Given
  $\mathcal{S} \subset \widetilde{\Ob}(\widehat{\msc})$ and
  $p \in \mathcal{P}$ we write
$$\mathcal{S}_p := \{ A \in \Ob(\msc_p) \mid [A]_{\widehat{\msc}}
\in \mathcal{S}\},$$ where $[A]_{\widehat{\msc}}$ stands for the
equivalence class of $A$ in $\widetilde{\Ob}(\widehat{\msc})$.  We
will use the same notation $\mathcal{S}_p$ also when we deal with just
one element $\mathcal{S} \in \widetilde{\Ob}(\widehat{\msc})$.

We will define now a pseudo distance on
  $\widetilde{\Ob}(\widehat{\msc})$ induced by the interleaving
  distances on the $\msc_p$'s.
  Denote by $d^p_{\text{int}}$
  the interleaving (pseudo) distance on $\Ob(\msc_p)$. We define the
  following measurement on pairs of elements
  $\widetilde{X}, \widetilde{Y} \in \widetilde{\Ob}(\widehat{\msc})$:
\begin{equation} \label{eq:dist-lim}
  \widetilde{d}_{\text{int}}(\widetilde{X}, \widetilde{Y}) = \inf
  \bigl\{ d^p_{\text{int}}(X',Y') \mid p \in \mathcal{P}, X' \in
  \widetilde{X}_p, Y' \in \widetilde{Y}_p \bigr\},
\end{equation}
and call it the {\em limit interleaving distance} (the wording "limit"
will be justified shortly, in Lemma~\ref{l:lim-dint} below). Note that
the term $d^p_{\text{int}}(X',Y')$ that appears in~\eqref{eq:dist-lim}
depends only on the $0$-isomorphism classes of
$X', Y' \in \Ob(\msc_p)$. It is straightforward to see that
$\widetilde{d}_{\text{int}}$ is a pseudo-distance. Note that due to
the $\inf$ in~\eqref{eq:dist-lim},
$\widetilde{d}_{\text{int}}(\widetilde{X}, \widetilde{Y})$ may vanish
even if $\widetilde{X} \neq \widetilde{Y}$. It may also assume the
value $\infty$ in case $\widetilde{X}$ and $\widetilde{Y}$ do not
represent isomorphic objects in the persistence $\infty$-level
categories $\msc^{\infty}_p$ for ``large'' $p$'s. The pseudo-metric
$\widetilde{d}_{\text{int}}$ can be viewed as a more robust version of
the pseudo-distance $\hat{d}$ from Remark
\ref{rem:def-TPC-approx}. Indeed, as its name suggests, the limit
distance can also be defined in terms of the limit of
$d^p_{\text{int}}$ as $\nu(p) \longrightarrow 0$.
More
specifically:
\begin{lem} \label{l:lim-dint} Let $\{p_k\}_{k \in \mathbb{N}}$ be a
  sequence of elements from $\mathcal{P}$ with
  $\nu(p_k) \longrightarrow 0$ as $k \longrightarrow \infty$. Then for
  every
  $\widetilde{X}, \widetilde{Y} \in \widetilde{\Ob}(\widehat{\msc})$
  we have
  $$\lim_{k \to \infty} d^{p_k}_{\text{int}}(X'_k,Y'_k) =
  \widetilde{d}_{\text{int}}(\widetilde{X},
  \widetilde{Y}),$$ where $X'_k, Y'_k$ are any sequences with $X'_k
  \in \widetilde{X}_{p_k}$, $Y'_k \in \widetilde{Y}_{p_k}$.
\end{lem}

\begin{proof}
  This follows from elementary arguments, using the properties of the
  comparison functors listed on page~\pageref{pp:sys-functors}
  together with Lemma~\ref{l:d-func}.
\end{proof}

  \begin{rem} \label{r:d-limit} Let
    $\widehat{\msf}: \widehat{\msc} \longrightarrow \widehat{\msd}$ be
    a functor between two systems of PC's
    (see~\S\ref{sbsb:func-sys-pc}).  It follows directly from the
    definitions that $\widehat{\msf}$ induces a well defined map
    $\widetilde{\msf}: \widetilde{\Ob}(\widehat{\msc})
    \longrightarrow\widetilde{\Ob}(\widehat{\msd})$ that satisfies
    $\widetilde{\msf}([A]_{\widehat{\msc}}) =
    [\msf_p(A)]_{\widehat{\msd}}$ for every $A \in \Ob(\msc_p)$,
    $p \in \mathcal{P}$. Moreover, $\widetilde{\msf}$ is non-expanding
    with respect to the limiting interleaving distances on
    $\widetilde{\Ob}(\widehat{\msc})$ and
    $\widetilde{\Ob}(\widehat{\msd})$ in the sense that
    $\widetilde{d}^{\widehat{\msd}}_{\text{int}}
    (\widetilde{\msf}(\widetilde{X}), \widetilde{\msf}(\widetilde{Y}))
    \leq \widetilde{d}^{\widehat{\msc}}_{\text{int}}(\widetilde{X},
    \widetilde{Y})$ for every
    $\widetilde{X}, \widetilde{Y} \in
    \widetilde{\Ob}(\widehat{\msc})$.
\end{rem}

We now turn to $K$-groups. Let
$\widehat{\msc} = \{\msc_p\}_{p \in \mathcal{P}}$ be a system of TPC's
(see~\S\ref{sb:sys-tpc}). Recall that the $0$-persistence level
$\msc_p^0$ of each $\msc_p$ is a triangulated category. Denote by
$K_0(\msc_p^0)$ the Grothendieck (or $K$-) group of $\msc_p^0$. Let
$p \preceq q$. Recall that all the comparison functors
$\msh_{p,q}: \msc_p \longrightarrow \msc_q$ are TPC-functors, hence
they restrict to triangulated functors
$\msh^0_{p,q}: \msc^0_p \longrightarrow \msc^0_q$. We thus obtain an
induced homomorphism
\begin{equation} \label{eq:HK-pq} \mathscr{H}^K_{p,q}: K_0(\msc^0_p)
  \longrightarrow K_0(\msc^0_q).
\end{equation}
This homomorphism turns out to be independent of the specific choice
of the comparison functor $\mathcal{H}_{p,q}$. To see this, recall
that for $p \preceq q$ every two comparison functors
$\msh'_{p,q}, \msh''_{p,q} \in \mathscr{J}_{p,q}$
are $0$-isomorphic, and
that if two objects $X, Y \in \Ob(\mathscr{E})$ in a triangulated
category $\mathscr{E}$ are isomorphic, then their classes in
$K_0(\mathscr{E})$ are equal.

A similar argument shows that the $K$-group homomorphisms
from~\eqref{eq:HK-pq} satisfy
$\msh^K_{q,r} \circ \msh^K_{p,q} = \msh^K_{p,r}$ for every
$p\preceq q \preceq r$, and clearly $\msh^K_{p,p} = \id$ for every
$p$. It follows that $\{K_0(\msc_p^0)\}_{p \in \mathcal{P}}$ together
with the maps $\msh^{K}_{p,q}$ form a directed system of abelian
groups.  We define the $K_0$-group of $\widehat{\msc}$ to be the
colimit of this system:
\begin{equation} \label{eq:K-colim} K_0(\widehat{\msc}^0) := \colim_{p
    \in \mathcal{P}} K_0(\msc_p^0).
\end{equation}
The assignment $\widehat{\msc} \longmapsto K_0(\widehat{\msc}^0)$ is
functorial in the sense that TPC-functors
$\widehat{\mathscr{F}}: \widehat{\msc} \longrightarrow \widehat{\msd}$
between systems of TPC's
(see~\S\ref{sbsb:func-sys-pc},~\S\ref{sb:sys-tpc}) induce
homomorphisms
$\widehat{\mathscr{F}}^K: K_0(\widehat{\msc}^0) \longrightarrow
K_0(\widehat{\msd}^0)$.

\begin{rem} \label{r:K-Nov} Denote by $\Lambda^P_{\mathbb{Z}}$ the
  ring of Novikov polynomials (i.e.~finite Novikov series) with
  coefficients in $\mathbb{Z}$. Its elements can be written as finite
  sums $\sum n_k t^{a_k}$ with $n_k \in \mathbb{Z}$,
  $a_k \in \mathbb{R}$ ($t$ stands for the formal Novikov
  variable). As explained in~\cite{BCZ:PKth}, the Grothendieck group
  $K_0(\msc^0)$ of the $0$-persistence level category $\msc^0$ of a
  TPC $\msc$ is a $\Lambda^P_{\mathbb{Z}}$-module, where the action of
  $t^a$ on the class $[X]$ of an object $X \in \Ob(\msc^0)$ is defined
  by $t^{a}[X] := [\Sigma^a X]$ (here $\Sigma$ is the shift functor of
  $\msc$). It is straightforward to check that the homomorphisms
  $\msh^K_{p,q}$ are $\Lambda^P_{\mathbb{Z}}$-linear and therefore the
  colimit $K_0(\widehat{\msc})$ inherits the structure of a
  $\Lambda^P_{\mathbb{Z}}$-module. Similarly, the homomorphisms
  $\widehat{\mathscr{F}}^K$ induced by functors $\widehat{\msf}$ of TPC-systems
  are maps of $\Lambda^P_{\mathbb{Z}}$-modules.
\end{rem}

Finally, we discuss persistence Hochschild invariants of
  systems. Let
$\widehat{\mathcal{A}} = \{\mathcal{A}_p\}_{p \in \mathcal{P}}$ be a
systems of $A_{\infty}$-categories with increasing accuracy, and with a
coherent full base of objects (see~\S\ref{sbsb:sys-base}). Assume
further that $\widehat{\mathcal{A}}$ is a homotopy system as defined
in~\S\ref{sbsb:homotopy-sys}.

For every $p \in \mathcal{P}$ we have the persistence Hochschild
homology $PHH(\mathcal{A}_p, \mathcal{A}_p)$ of $\mathcal{A}_p$ (with
coefficients in the diagonal bimodule of $\mathcal{A}_p$). Let
$p, q \in \mathcal{P}$ with $p \preceq q$. Each comparison functor
$\mathcal{H}_{p,q} : \mathcal{A}_p \longrightarrow \mathcal{A}_q$
induces a map of persistence modules
$$\mathcal{H}^{\text{PHH}}_{p,q}: PHH(\mathcal{A}_p, \mathcal{A}_p)
\longrightarrow PHH(\mathcal{A}_q, \mathcal{A}_q).$$ Since every two
comparison functors from $\mathcal{J}_{p,q}$ are homotopic it follows
from Proposition~\ref{p:homotopy-HH} that
$\mathcal{H}^{\text{PHH}}_{p,q}$ is independent of the specific choice
of $\mathcal{H}_{p,q}$. Moreover, for every $p \preceq q \preceq r$ we
have that
$\mathcal{H}^{\text{PHH}}_{q,r} \circ \mathcal{H}^{\text{PHH}}_{p,q} =
\mathcal{H}^{\text{PHH}}_{p,r}$, and
$\mathcal{H}^{\text{PHH}}_{p,p} = \id$ for every $p \in \mathcal{P}$.
In other words,
$\{PHH(\mathcal{A}_p, \mathcal{A}_p)\}_{p \in \mathcal{P}}$ together
with the maps $\mathcal{H}^{\text{PHH}}_{p,q}$, $p \preceq q$, forms a
directed system of persistence modules, parametrized by $\mathcal{P}$.
We define the persistence Hochschild homology of
$\widehat{\mathcal{A}}$ to be the colimit of this system:
\begin{equation} \label{eq:HH-colim} PHH(\widehat{\mathcal{A}}) :=
  \colim_{p \in \mathcal{P}} PHH(\mathcal{A}_p, \mathcal{A}_p).
\end{equation}
Note that since colimits of directed systems of persistence modules
are persistence modules, $PHH(\widehat{\mathcal{A}})$ is a persistence
module too.

\begin{rem} \label{r:K-HH} Let $\mathcal{A}$ be a strictly unital
  $A_{\infty}$-category and denote its derived category by
  $D(\mathcal{A})$. There is a canonical map
  \begin{equation} \label{eq:K-HH-1} \kappa: K_0(D(\mathcal{A}))
    \longrightarrow HH_*(\mathcal{A}, \mathcal{A})
  \end{equation}
  which satisfies $\kappa([X]) = [e_X]$ for every
  $X \in \Ob(\mathcal{A})$. Here $[X] \in K_0(D(\mathcal{A}))$ stands
  for the $K_0$-class of $X$, $e_X \in \hom_{\mathcal{A}}(X,X)$ is the
  unit and $[e_X] \in HH(\mathcal{A}, \mathcal{A})$ is its Hochschild
  homology class. The map $\kappa$ is the composition of two maps,
  $\kappa = j^{-1} \circ \kappa'$, where
  $\kappa' : K_0(D(\mathcal{A})) \longrightarrow
  HH(\mathcal{Y}(\mathcal{A})^{\Delta},
  \mathcal{Y}(\mathcal{A})^{\Delta})$ and
  $j: HH(\mathcal{A}, \mathcal{A}) \longrightarrow
  HH(\mathcal{Y}(\mathcal{A})^{\Delta},
  \mathcal{Y}(\mathcal{A})^{\Delta})$. The map $\kappa'$ is defined by
  $\kappa([X]) = [e_X]$ for every $X \in \Ob(D(\mathcal{A}))$ and its
  well-definedness follows from simple considerations (e.g.~by
  using~\cite[Proposition~3.8]{Se:book-fukaya-categ}). The map $j$ is
  the one induced in homology by the chain level inclusion of the
  Hochschild complex of $\mathcal{A}$ into that of
  $\mathcal{Y}(\mathcal{A})^{\Delta}$. A version of Morita invariance
  for Hochschild homology states that $j$ is an isomorphism, hence we
  can define $\kappa = j^{-1} \circ \kappa'$. (For Morita invariance
  of Hochschild homology see e.g.~\cite{To:dg-cat} in the case when
  $\mathcal{A}$ is a dg-category, and the discussion
  in~\cite[Section~5]{Se:homology-spheres}; the case of $A_{\infty}$-categories is treated in~\cite{Sher:form-hodge}. See
  also~\cite[Section~5.4]{Bi-Co:lefcob-pub} for the significance of
  the map~\eqref{eq:K-HH-1} in some geometric situations.)

  We expect the same considerations to continue to hold also in our
  persistence setting. More specifically, let
  $\widehat{\mathcal{A}} = \{\mathcal{A}_p\}_{p \in \mathcal{P}}$ be a
  system of filtered $A_{\infty}$-categories. We expect to have for
  every $p \in \mathcal{P}$ a homomorphism
  $\kappa_p : K_0(PD(\mathcal{A}_p)^0) \longrightarrow
  PHH^0(\mathcal{A}_p, \mathcal{A}_p)$, where the $0$-superscript on
  the left stands for the $0$-persistence level category of
  $PD(\mathcal{A}_p)$ and the one on the right denotes the
  $0$-persistence level of the persistence module
  $PHH(\mathcal{A}_p, \mathcal{A}_p)$. The maps $\kappa_p$ are then
  expected to be compatible with the respective directed systems
  parametrized by $p \in \mathcal{P}$ and to induce one map
  $$\widehat{\kappa}: K_0(PD(\widehat{\mathcal{A}})^0) \longrightarrow
  PHH^0(\widehat{\mathcal{A}}).$$
\end{rem}

\subsubsection{Approximability revisited} \label{s:approximability}
%\subsubsection{Approximability within a system} \label{sb:approx-sys}

Let $\widehat{\msc}$ be a system of TPC's with increasing accuracy, as
defined above. We will use here the equivalence relation
  $\sim_{\widehat{\msc}}$ on
  $\Ob^{\text{tot}}(\widehat{\msc}) = \coprod_{p\in \mathcal{P}}
  \Ob(\msc_p)$ as defined in~\S\ref{sb:sys-inv} on
  page~\pageref{p:equiv-ob} and the notation from that section.
\begin{dfn} \label{d:approx-sys} Let
  $\mathcal{L} \subset \widetilde{\Ob}(\widehat{\msc})$ and let
  $\mathcal{F}_1, \ldots, \mathcal{F}_l \in
  \widetilde{\Ob}(\widehat{\msc})$ be a finite collection (of equivalence
  classes of objects), and let $\epsilon>0$. We say that $\mathcal{L}$
  is $\epsilon$-approximable by
  $\{\mathcal{F}_1, \ldots, \mathcal{F}_l\}$ if there exists
  $\delta>0$ and $c_{\epsilon}>0$ such that for every
  $p \in \mathcal{P}$ with $\nu(p) < \delta$ we have
  $${d}_{\text{int}} \Bigl(\mathcal{L}_p,
  \Ob \langle \{F_1, \ldots, F_l\}_p \rangle^{\Delta} \Bigr)
  < \epsilon + c_{\epsilon}\nu(p).$$
  Let $\mathcal{F} \subset \widetilde{\Ob}(\widehat{\msc})$. We say that
  $\mathcal{L}$ is approximable by $\mathcal{F}$ if for every
  $\epsilon>0$ there exists a finite collection
  $\mathcal{F}_1, \ldots, \mathcal{F}_l \in \mathcal{F}$ which
  $\epsilon$-approximates $\mathcal{L}$ and the constants $c_{\epsilon}$ 
  are bounded, independently of $\epsilon$.

  \end{dfn}
  Recall that 
  %$\{F_1, \ldots, F_l\}_p^{\Sigma}$ stands
  %for completing the set of objects
  %$\{F_1, \ldots, F_l\}_p \subset \Ob(\msc_p)$ with respect to all
  %shifts, and 
  $\langle \mathcal{O} \rangle^{\Delta}$ stands for the
  smallest sub-TPC of $\msc_p$ containing the set of
  objects $\mathcal{O}$. Sometimes we will use the wording {\em
    $\{\mathcal{F}_1, \ldots, \mathcal{F}_l\}$ $\epsilon$-approximates
    $\mathcal{L}$} and {\em $\mathcal{F}$ approximates $\mathcal{L}$}.

\

Similarly we also define retract approximability as follows. 
\begin{dfn} \label{d:approx-sys-ret} Let $\widehat{\msc}$,
  $\mathcal{L}$, $\mathcal{F}_1, \ldots, \mathcal{F}_l$ be as in
  Definition~\ref{d:approx-sys} and let $\epsilon>0$. We say that
  $\mathcal{L}$ is $\epsilon$-retract-approximable by
  $\{\mathcal{F}_1, \ldots, \mathcal{F}_l\}$ if there exists
  $\delta>0$, $c_{\epsilon}>0$, such that for every $p \in \mathcal{P}$ with
  $\nu(p) < \delta$ we have
  $${d}_{\text{r-int}} \Bigl(\mathcal{L}_p,
  \Ob \langle \{F_1, \ldots, F_l\}_p \rangle^{\Delta} \Bigr)
  < \epsilon +c_{\epsilon}\nu(p).$$

  Let $\mathcal{F} \subset \widetilde{\Ob}(\widehat{\msc})$. We say
  that $\mathcal{L}$ is retract-approximable by $\mathcal{F}$ if for
  every $\epsilon>0$ there exists a finite collection
  $\mathcal{F}_1, \ldots, \mathcal{F}_l \in \mathcal{F}$ which
  $\epsilon$-retract-approximates $\mathcal{L}$ and the constants
  $c_{\epsilon}$ are uniformly bounded.
\end{dfn}

\begin{rem}\label{rem:different_approx} It is useful at this point to
  relate our two notions of approximability, in Definitions
  \ref{def:TPC-approx} and \ref{d:approx-sys}.  In our geometric
  applications, the set $\mathcal{L}$ will be a set of Lagrangian
  submanifolds $\mathcal{L}ag(M)$ and the category $\msc_{p}$ will be
  a certain TPC associated with a filtered Fukaya category defined using
  perturbation data that is identified by the parameter $p$.  The
  inclusion $\mathcal{L}ag(M)\subset \msc_{p}$ is the Yoneda
  embedding. The set $\mathcal{L}ag(M)$ is endowed with the spectral
  metric $d_{\gamma}$ (that will be recalled later) and this metric is
  closely related to the stabilized interleaving metric $\sdint$ from
  \eqref{eq:sdint} restricted to $\mathcal{L}ag(M)$.  In view of this,
  we will see that $\epsilon$ - approximability in the sense of
  Definition \ref{d:approx-sys} implies $\epsilon'$-approximability
  for all $\epsilon'>\epsilon$ in the sense of Definition
  \ref{def:TPC-approx}. The reason for this discrepancy in the
  parameters $\epsilon,\epsilon'$ is related to the presence of the
  term $c_{\epsilon}\nu(p)$ in Definition \ref{d:approx-sys}.  Given
  that these properties are supposed to be satisfied for all
  $\epsilon,\epsilon' >0$ and $\nu(p)$ can be taken as small as
  desired, this discrepancy is essentially irrelevant. However, for a
  fixed $\epsilon$, including the term $c_{\epsilon} \nu(p)$ in the
  upper bound in Definition \ref{d:approx-sys} is useful for questions
  having to do with minimizing the cardinality of the set
  $\{\mathcal{F}_{i}\}$ for that fixed $\epsilon$.  A similar remark
  applies to retract-approximability.
\end{rem}

\subsubsection{Stabilizing functors} \label{sbsb:sfunc} Let
  $\widehat{\msc} = \{ \msc_p\}_{p \in \mathcal{P}}$ be a system of
  PC's with increasing accuracy and let
  $\widehat{\msf}: \widehat{\msc} \longrightarrow \widehat{\msc}$ be
  an endo-functor of such systems (see~\S\ref{sbsb:func-sys-pc}). Let
  $\mathscr{B} \subset \Ob^{\text{tot}}(\widehat{\msc})$. We say that
  $\widehat{\msf}$ is {\em $\mathscr{B}$-stabilizing} if the following
  holds for every $p \in \mathcal{P}$:
  \begin{enumerate}
  \item \label{i:stab-1} For every $A \in \mathscr{B}_p$ we have
    $(A,p) \sim_{\widehat{\msc}} (\msf_p(A), \widehat{\msf}(p))$,
    where $\sim_{\widehat{\msc}}$ is the equivalence relation
    from~\S\ref{sb:sys-inv}, page~\pageref{p:equiv-ob}.
  \item \label{i:stab-2} For every $A', A'' \in \mathscr{B}_p$ there
    exists $q \in \mathcal{P}$ with $p, \widehat{\msf}(p) \preceq q$
    and two $0$-isomorphisms
    $\sigma_{A'}: \msh_{\widehat{\msf}(p), q} \msf_p A'
    \longrightarrow \msh_{p,q}A'$,
    $\sigma_{A''}: \msh_{\widehat{\msf}(p), q} \msf_p A''
    \longrightarrow \msh_{p,q}A''$ such that for every
    $u \in \hom_{\msc_p}(A',A'')$ we have:
    \begin{equation} \label{eq:stab-2} \sigma_{A''} \circ
      \msh_{\widehat{\msf}(p), q} \circ \msf_p(u) = \msh_{p,q}(u)
      \circ \sigma_{A'}.
    \end{equation}
  \end{enumerate}

\begin{rem} \label{r:stab-func} 
Condition~\eqref{i:stab-1}
    above is equivalent to requiring that the map
    $\widetilde{\msf}: \widetilde{\Ob}(\widehat{\msc}) \longrightarrow
    \widetilde{\Ob}(\widehat{\msc})$ (see Remark~\ref{r:d-limit})
    stabilizes $\mathscr{B}$ (i.e.~$\widetilde{\msf}$ sends
    $\mathscr{B}$ to itself and restricts to the identity map
    $\mathscr{B} \longrightarrow \mathscr{B}$).

  Condition~\eqref{i:stab-2} can be rephrased as follows. By a
    slight abuse of notation denote for every $p \in \mathcal{P}$ by
    $\mathscr{B}_p \subset \msc_p$ the full persistence subcategory
    with objects $\mathscr{B}_p$. We require that given
    $p \in \mathcal{P}$, for large enough (with respect to $\preceq$)
    $q \in \mathcal{P}$ the two restricted PC functors
    $$\msh_{\widehat{\msf}(p), q} \circ \msf_p|_{\mathscr{B}_p}, \, 
    \msh_{p,q}|_{\mathscr{B}_p}: \mathscr{B}_p \longrightarrow
    \mathscr{B}_q$$ are naturally $0$-isomorphic in the category of
    persistence functors
    $\mathscr{B}_p \longrightarrow \mathscr{B}_q$.
  \end{rem}

  In the presence of approximating families in TPC's,
    stabilizing functors have a remarkable metric property.
    \begin{prop} \label{p:stab-func} Let $\widehat{\msc}$ be a system
      of TPC's, $\mathcal{L} \subset \widetilde{\Ob}(\widehat{\msc})$
      and
      $\mathcal{F}_1, \ldots, \mathcal{F}_l \in
      \widetilde{\Ob}(\widehat{\msc})$ an $\epsilon$-approximating
      family for $\mathcal{L}$, as in Definition~\ref{d:approx-sys}.
      Let
      $\widehat{\Phi}: \widehat{\msc} \longrightarrow \widehat{\msc}$
      be a TPC endo-functor of systems and assume that
      $\widehat{\Phi}$ is
      $\{\mathcal{F}_1, \ldots, \mathcal{F}_l\}$-stabilizing. Then for
      every $\widetilde{L} \in \mathcal{L}$ we have
      $\widetilde{d}_{\text{int}}(\widetilde{L},
      \widetilde{\Phi}(\widetilde{L})) \leq \epsilon$.
  \end{prop}

% !TEX root = approx8.tex

\subsection{Filtered and persistence derived Fukaya categories}
\label{s:pdfuk} Let $(X, \omega)$ be a symplectic manifold of one of
the following types:
\begin{enumerate}
\item $X$ is closed.
\item $(X, \omega)$ is symplectically convex at infinity.
\item $(X, \omega = d \lambda)$ is a Liouville manifold with a
  prescribed primitive $\lambda$ of the symplectic structure $\omega$,
  and such that $X$ is symplectically convex at infinity with respect
  to these structure).
\item $(X, \omega = d \lambda)$ is a compact Liouville domain.
\end{enumerate}

In what follows we will work with several variants of Fukaya
categories of $X$, and we will specify below in each case the class of
admissible Lagrangians for this purpose. In each of these cases we
will use the following convention. Unless otherwise stated, whenever
$X$ is a manifold with boundary, the closed Lagrangian submanifolds
$\overline{L} \subset X$ will be implicitly assumed to lie in the {\em
  interior} of $X$.

\subsubsection{The exact case} \label{sb:fuk-ex} Here we assume that
$X$ is Liouville (with a given Liouville form $\lambda$). In order to
obtain a graded theory we add the following assumption: $2c_1(X)=0$,
where $c_1(X)$ stands for the 1'st Chern class of the tangent bundle
of $X$, viewed as a complex vector bundle by endowing $X$ with any
$\omega$-compatible almost complex structure.  We also fix a nowhere
vanishing quadratic complex $n$-form (where
$n = \dim_{\mathbb{C}} X$), namely a nowhere vanishing section
$\Theta$ of the bundle $\Omega^n(X, J)^{\otimes 2}$.

Denote by $\lagex$ the collection of closed, graded, marked, exact
Lagrangian submanifold $L = (\overline{L}, h_L, \theta_L)$. Here
$\overline{L} \subset X$ is a closed $\lambda$-exact Lagrangian
submanifold, which we call the {\em underlying Lagrangian of $L$},
$h_L: \overline{L} \longrightarrow \mathbb{R}$ is a primitive of
$\lambda|_{\overline{L}}$ and
$\theta_L: \overline{L} \longrightarrow \mathbb{R}$ is a grading of
$\overline{L}$, see
e.g.\cite{Se:graded},~\cite[Section~3.2.2.2]{BCZ:tpc} for more
details. Sometimes we will write $\lagex(X)$ or $\lagex_{\lambda}(X)$,
instead of simply $\lagex$, in order to keep track of the additional
data $X$ or $\lambda$.

We can now form the Fukaya category of exact Lagrangians in $X$. We
will follow here a variant of the standard construction
from~\cite{Se:book-fukaya-categ} due to Ambrosioni~\cite{Amb:fil-fuk}
(see also~\cite{BCZ:tpc} for a somewhat simpler implementation in a special case) which
yields a filtered $A_{\infty}$-category. The objects of our Fukaya
category are the elements of $\lagex$. To define the morphisms and
higher structures we need perturbation data. The space of admissible
perturbation data $\mathcal{P}$ is described in detail
in~\cite{Amb:fil-fuk}. Once we fix $p \in \mathcal{P}$ we can define
the Fukaya category $\fuk(\lagex; p)$ as
in~\cite{Amb:fil-fuk}. (Sometimes we will write $\fuk(X, \lagex; p)$
to emphasize the ambient manifold $X$.) For $L_0, L_1 \in \lagex$ we
will sometime denote $\hom_{\fuk(\lagex;p)}(L_0,L_1)$ by
$CF(L_0, L_1; p)$ or  by $\hom(L_0,L_1; p)$. We use here
$\mathbb{Z}_2$ for the coefficients in these $\hom$'s. The grading on
$CF(L_0, L_1; p)$ is defined using the given grading functions
$\theta_{L_0}, \theta_{L_1}$ with which $L_0$, $L_1$ are endowed,
respectively.

\subsubsection{Filtered Fukaya categories in the exact
  case} \label{sbsb:fil-ex} Given two geometrically distinct elements
$L_0, L_1 \in \lagex$ (i.e.~$\overline{L}_0 \neq \overline{L}_1$) and
$p \in \mathcal{P}$, denote by
$H^{L_0,L_1}_p: [0,1] \times X \longrightarrow \mathbb{R}$ the
Hamiltonian function prescribed by $p$ for the Floer datum of
$(L_0, L_1)$. Let $x \in CF(L_0, L_1; p)$ be a generator (i.e.~an
$H^{L_0,L_1}_p$-Hamiltonian chord with endpoints on $L_0, L_1$). Its
action is defined by
\begin{equation} \label{symp-act-1} \mathcal A(x) := \int_0^1
  H_p^{L_0,L_1}(t, x(t))dt - \int_0^1 \lambda(\dot{x}(t)) dt +
  h_{L_1}(x(1)) - h_{L_0}(x(0)).
\end{equation}

We use the action to filter the chain complexes $CF(L_0, L_1; p)$,
$L_0, L_1 \in \lagex$. The grading on $CF(L_0, L_1;p)$ follows the
standard recipe~\cite{Se:graded}, using the grading functions
$\theta_{L_0}, \theta_{L_1}$.

In case $\overline{L}_0 = \overline{L}_1$ the Floer complex
$CF(L_0, L_1; p)$ coincides (up to a shift in grading) with the Morse
complex of $\overline{L}:= \overline{L}_0 = \overline{L}_1$ and we set
the action level $\mathcal{A}(x)$ of all the (non-zero) elements
$x \in CF(L_0, L_1; p)$ to be $h_{L_1} - h_{L_0}$ (note that the
latter is a constant since $\overline{L}_0 = \overline{L}_1$). The
(cohomological) grading of a generator $x \in CF(L_0, L_1; p)$ is
defined as $|x| = (n- \textnormal{ind}_f(x)) + \kappa$, where
$f: \overline{L} \longrightarrow \mathbb{R}$ is the Morse function on
$\overline{L}$ presribed by $p$, $\textnormal{ind}_f(x)$ is the Morse
index of the critical point $x$, $n := \dim \overline{L}$, and
$\kappa = \theta_{L_1} - \theta_{L_0}$ (which is again an integer).

The work of~\cite{Amb:fil-fuk} introduces a special class
$\mathcal{P}$ of perturbation data such that for every
$p \in \mathcal{P}$, the Fukaya category $\fuk(\lagex;p)$ is a genuine
filtered and strictly unital $A_{\infty}$-category.  We endow the
category $\fuk(\lagex;p)$ with translation and shift functors as
follows. Let $L = (\overline{L}, h_L, \theta_L) \in \lagex$. We define
the translation and shift functors, respectively, on objects by:
\begin{equation} \label{eq:T-Sigma-ex-1} TL := (\overline{L}, h_L,
  \theta_L+1), \quad \Sigma^r L := (\overline{L}, h_L+r, \theta_L), \;
  \forall r \in \mathbb{R}.
\end{equation}
These extend to maps on $\hom_{\fuk(\lagex;p)}$'s in the obvious way
and furthermore to filtered $A_{\infty}$-functors
$\fuk(\lagex;p) \longrightarrow \fuk(\lagex;p)$ with trivial higher
order terms.  With these definitions we have natural isomorphisms:
\begin{equation} \label{eq:T-Sigma-ex-2}
  \begin{aligned}
    & CF(T^kL_0, T^lL_1; p)^i \cong
    CF(L_0, L_1; p)^{i-k+l}, \; \forall \; k,l \in \mathbb{Z}, \\
    & CF^{\alpha}(\Sigma^rL_0, \Sigma^sL_1;p) \cong
    CF^{\alpha+r-s}(L_0, L_1; p), \; \forall \; \alpha, r, s \in
    \mathbb{R}.
  \end{aligned}
\end{equation}
The superscripts $i, i-k+l$ in the first line stand for degrees and
the superscripts $\alpha, \alpha+r-s$ in the second line are action
levels.  In the first line we have used cohomological grading
conventions. In homological grading conventions the isomorphism from
the first line takes the following form:
$CF(T^kL_0, T^lL_1; p)_j \cong CF(L_0, L_1; p)_{j+k-l}, \; \forall \;
k,l \in \mathbb{Z}$. In case no confusion between grading and  action levels may arise
we will sometimes also write $CF^i(L_0, L_1;p)$ instead of
$CF(L_0, L_1; p)^i$ and similarly in homology.

Before we proceed we remark that the filtered Fukaya category
$\fuk(\lagex;p)$ depends on the choice of $\lambda$ in two
ways. Firstly, the set of objects $\lagex$ depends on $\lambda$ and
secondly the filtrations on the hom's depend  on $\lambda$.
Whenever we need to emphasize the dependence on $\lambda$ we will
write $\lagex_{\lambda}$ for the set of objects and denote this
category by $\fuk(\lagex_{\lambda}, \lambda ;p)$. (The second
$\lambda$ in the notation indicates that the action is measured with
respect to $\lambda$.) Note that if $\lambda' = \lambda + dG$, where
$G: X \longrightarrow \mathbb{R}$ is a smooth function, compactly
supported in the interior of $X$, then there is a canonical
isomorphism of filtered $A_{\infty}$-categories
\begin{equation} \label{eq:can-iso-fuk-lam} \fuk(\lagex_{\lambda},
  \lambda;p) \longrightarrow \fuk(\lagex_{\lambda'}, \lambda';p).
\end{equation}
Its action on objects is given by
$(\overline{L}, h, \theta) \longmapsto (\overline{L},
h+G|_{\overline{L}}, \theta)$. Its action on morphisms is induced by
the identity map and it has zero higher order terms. We will often
identify all the categories $ \fuk(\lagex_{\lambda'}, \lambda';p)$
with $\lambda'$ as above.

However if $\lambda'$ differs form $\lambda$ by a non-exact closed
form then the collection $\lagex$ entirely changes, so in that case
$\fuk(\lagex_{\lambda}, \lambda ;p)$ and
$\fuk(\lagex_{\lambda'}, \lambda' ;p)$ have completely different
objects and there is no general way to compare these categories.

We will discuss further important properties of the categories
$\fuk(\lagex;p)$ in~\S\ref{sb:sys-fuk} below.

\subsubsection{The monotone case} \label{sb:fuk-mon} In case $X$ is
closed, or symplectically convex at infinity we will consider monotone
Lagrangian submanifolds. (Note that the existence of monotone
Lagrangians in $X$ is not obvious, there are  global
obstructions, e.g.~that the ambient manifold $X$ is itself monotone.)

To obtain a graded Floer theory we follow Seidel's approach for graded
Lagrangians submanifolds~\cite{Se:graded}. Fix
$N \in \mathbb{Z}_{\geq 2}$ for which the image of $2c_1(X)$ in
$H^2(X; \mathbb{Z}_N)$ vanishes (we assume that such an $N$ does
exist). Fix an $N$-fold Maslov covering
$Gr^N_{\text{Lag}}(X, \omega) \longrightarrow Gr_{\text{Lag}}(X,
\omega)$ of the Lagrangian Grassmannian bundle of $(T(X), \omega)$.

Denote by $\Lambda$ the Novikov field with coefficients in
$\mathbb{Z}_2$, namely:
\begin{equation} \label{eq:Nov-ring} \Lambda = \Bigl\{
  \sum_{k=0}^{\infty} a_k T^{\lambda_k} \mid a_k \in \mathbb{Z}_2,
  \lim_{k \to \infty} \lambda_k = \infty \Bigr\},
\end{equation}
and let $\Lambda_0 \subset \Lambda$ be the positive Novikov ring:
\begin{equation} \label{eq:Nov-ring-0} \Lambda_0 = \Bigl\{
  \sum_{k=0}^{\infty} a_k T^{\lambda_k} \mid a_k \in \mathbb{Z}_2,
  \lambda_k\geq 0, \lim_{k \to \infty} \lambda_k = \infty \Bigr\}.
\end{equation}

Denote by $\lagmon$ the collection of closed, graded, monotone
Lagrangian submanifolds whose Maslov-$2$ pseudo-holomorphic disk count (through a generic point and for a generic $J$ that tames $\omega$; see for instance \cite{Bi-Co-Sh:LagrSh} \S 3.5 page 66 for a definition) is
$\mathbf{d} \in \Lambda_0$. We write these elements as triples
$L = (\overline{L}, a_L, s_L)$, where:
\begin{itemize}
\item $\overline{L} \subset X$ is a monotone Lagrangian submanifold
  with Maslov-$2$ disk count $\mathbf{d}$.
\item $s_L$ is a $Gr^N_{\text{Lag}}(X, \omega)$-grading of $L$. More
  specifically,
  $s_L: \overline{L} \longrightarrow Gr^N_{\text{Lag}}(X, \omega)$ is
  a lift of the canonical section
  $\overline{L} \longrightarrow Gr_{\text{Lag}}(X,
  \omega)$. See~\cite{Se:graded} for more details.
\item $a_L \in \mathbb{R}$ is a real number that will be used for
  action filtration purposes in~\S\ref{sbsb:fil-mon} below.
\end{itemize}

\subsubsection{Filtered Fukaya categories in the monotone
  case} \label{sbsb:fil-mon} Let
$L_0 = (\overline{L}_0, a_{L_0}, s_{L_0})$,
$L_1 = (\overline{L}_1, a_{L_1}, s_{L_1})$ be two geometrically
distinct elements in $\lagmon$. Let $p \in \mathcal{P}$, and
$H^{L_0,L_1}_p: [0,1] \times X \longrightarrow \mathbb{R}$ the
Hamiltonian function prescribed by $p$ for the Floer datum of
$(L_0, L_1)$. Recall that the Floer complex $CF(L_0, L_1;p)$ is a free
$\Lambda$-module generated by the $H^{L_0,L_1}_p$-Hamiltonian chords
$x$ with endpoints on $\overline{L}_0, \overline{L}_1$.

Let $0 \neq P(T) \in \Lambda$ and $x$ be an
$H^{L_0,L_1}_p$-Hamiltonian chord as above. We define the action of
$P(T)x$ to be:
\begin{equation} \label{eq:action-mon} \mathcal{A}\bigl(P(T)x\bigr) :=
  -\lambda_0 + \int_0^1 H_p^{L_0,L_1}(t, x(t))dt + a_{L_1} - a_{L_0},
\end{equation}
where $\lambda_0 \in \mathbb{R}$ is the minimal exponent that appears
in the formal power series of $P(T) \in \Lambda$, i.e.
$P(T) = a_0 T^{\lambda_0} + \sum_{i=1}^{\infty} a_i T^{\lambda_i}$
with $a_0 \neq 0$ and $\lambda_i > \lambda_0$ for every $i\geq 1$.  We
now extend $\mathcal{A}$ to $CF(L_0,L_1;p)$ as follows. For a
non-trivial element
$c = P_1(T)x_1 + \cdots + P_l(T)x_l \in CF(L_0, L_1;p)$ we define
$$\mathcal{A}(c) = \max \{ \mathcal{A}(P_k(T) x_k) \mid 1\leq k \leq
l\},$$ and finally we set $\mathcal{A}(0) = -\infty$. The filtration
on $CF(L_0,L_1;p)$ is then defined by:
$$CF^{\alpha}(L_0,L_1;p) := 
\{c \in CF(L_0,L_1; p) \mid \mathcal{A}(c) < \alpha \}.$$ Standard
arguments show that the Floer differential $\mu_1$ preserves this
filtration. It is important to note however, that the filtration
levels $CF^{\alpha}(L_0,L_1;p)$ of $CF(L_0,L_1;p)$ are {\em not}
$\Lambda$-modules but rather $\Lambda_0$-modules.

The grading on $CF(L_0, L_1;p)$, when
$\overline{L}_0 \neq \overline{L}_1$, is defined using the $N$-fold
maslov covering $Gr^N_{\text{Lag}}(X, \omega)$ and the gradings
$s_{L_0}, s_{L_1}$ following the recipe from~\cite{Se:graded}. In
contrast to the case described in~\S\ref{sbsb:fil-ex}, here we only
obtain a $\mathbb{Z}_N$-grading.

As in the exact case, following~\cite{Amb:fil-fuk} there is a class of
perturbation data $\mathcal{P}$ such that for every
$p \in \mathcal{P}$ the Fukaya category $\fuk(\lagmon;p)$ is a genuine
filtered and strictly unital $A_{\infty}$-category. The total category
is $\Lambda$-linear, however when viewed as a filtered
$A_{\infty}$-category, $\fuk(\lagmon;p)$ is only $\Lambda_0$-linear.

Similarly to the exact case, here too we have translation and shift
functors. For the translation functor, recall from~\cite{Se:graded} that the $N$-fold
Maslov covering
$Gr^N_{\text{Lag}}(X, \omega) \longrightarrow Gr_{\text{Lag}}(X,
\omega)$ comes with a $\mathbb{Z}_N$-action
$\rho: \mathbb{Z}_N \longrightarrow
\text{Aut}\Bigl(Gr^N_{\text{Lag}}(X, \omega) \longrightarrow
Gr_{\text{Lag}}(X, \omega) \Bigr)$.  We now
define for $L = (\overline{L}, a_L, s_L) \in \lagmon$:
\begin{equation} \label{eq:T-Sigma-mon-1} TL := (\overline{L}, a_L,
  \rho(1) \circ s_L), \quad \Sigma^r L := (\overline{L}, a_L+r, s_L),
  \; \forall r \in \mathbb{R}.
\end{equation}
Both $T$ and $\Sigma^r$ extend to filtered $A_{\infty}$-functors
$\fuk(\lagmon;p) \longrightarrow \fuk(\lagmon;p)$ and the identities
from~\eqref{eq:T-Sigma-ex-2} continue to hold (with the grading being
taken modulo $N$).

\subsubsection{The weakly exact case} \label{sb:fuk-wex} Assume that
$X$ is closed, or symplectically convex at infinity. we will consider
in this case also weakly exact closed Lagrangian submanifolds
$\overline{L} \subset X$. By weakly exact we mean that
$\langle [\omega], A \rangle = 0$ for all classes
$A \in \textnormal{image\,} (\pi_2(X,L) \longrightarrow
H_2(X,L))$.  A necessary condition for such Lagrangians
to exist is that $X$ is symplectically aspherical,
i.e.~$\langle [\omega], A \rangle = 0$ for all spherical classes
$A \in \textnormal{image\,} (\pi_2(X) \longrightarrow H_2(X))$. Also
note that if $(X, \omega)$ is exact then every exact Lagrangian is
automatically weakly exact, regardless of the chosen primitive
$\lambda$ of $\omega$.)

Denote by $\lagwex$ the collection of triples
$L = (\overline{L}, a_L, \theta_L)$, where $\overline{L} \subset X$ is
a closed weakly exact Lagrangian submanifold, $a_L \in \mathbb{R}$ and
$\theta_L: \overline{L} \longrightarrow \mathbb{R}$ is a grading of
$\overline{L}$ as in~\S\ref{sb:fuk-ex}. Similarly to the monotone
case, we will work here with coefficients in the Novikov ring
$\Lambda$ and its positive version $\Lambda_0$. The Floer complexes
$CF(L_0, L_1; p)$ are defined exactly as in the monotone
case~\S\ref{sbsb:fil-mon}. For the action filtration we follow the
recipe from the monotone case, while for the grading we use the recipe
described in the exact case. More specifically, the action filtration
is defined by~\eqref{eq:action-mon} and the grading is defined as
in~\S\ref{sbsb:fil-ex}. The defintion of the translation functor is
the same as in~\eqref{eq:T-Sigma-ex-1} and the shift functors
$\Sigma^r$ are defined as in~\eqref{eq:T-Sigma-mon-1}. The identities
from~\eqref{eq:T-Sigma-ex-2} continue to hold.

\subsubsection{Systems of Fukaya categories and their
  TPC's} \label{sb:sys-fuk} Let $(X, \omega)$ be a symplectic manifold
as at the beginning of~\S\ref{s:pdfuk} and let $\lag$ be a subset of
one of the collections $\lagex$, $\lagmon$ or $\lagwex$ (depending on
the type of $X$). We assume that $\lag$ is closed under all
translations and all shifts.

Let $\mathcal{P}$ be the space of perturbation data, as
in~\S\ref{sbsb:fil-ex} and~\S\ref{sbsb:fil-mon}, and denote for every
$p \in \mathcal{P}$ by $\fuk(\lag;p)$ the full
$A_{\infty}$-subcategory (of the respective Fukaya categories
$\fuk(\lagex;p)$, $\fuk(\lagmon;p)$ or $\fuk(\lagwex;p)$) with the
objects $\lag$. The categories $\fuk(\lag;p)$ inherit the structure of
filtered $A_{\infty}$-categories from their ambient categories.

We will further restrict the space of perturbation data $\mathcal{P}$
to satisfy the following condition. Fix a constant $B > 0$, and
consider only perturbation data $p \in \mathcal{P}$ such that for
every pair of geometrically distinct objects $L_0, L_1 \in \lag$ with
(i.e.~$\overline{L}_0 \neq \overline{L}_1$) we have
\begin{equation} \label{eq:H-smaller-B} H_p^{L_0,L_1}(t,x) \leq B, \;
  \forall \; t \in [0,1], x \in X.
\end{equation}
Recall from \cite{Amb:fil-fuk} that for every $p \in \mathcal{P}$ and $L_0, L_1$ as above, the
function $H_p^{L_0, L_1}$ is strictly positive. As we will see later,
the particular choice of the constant $B$ will not play an important
role for our considerations (intuitively, it should be a small
number), and its main purpose is to have a uniform bound on all the
Hamiltonian functions that appear in the perturbation data. By a
slight abuse of notation we will denote the space of perturbation data
satisfying the additional condition~\eqref{eq:H-smaller-B} also by
$\mathcal{P}$.

Define now $\nu: \mathcal{P} \longrightarrow \mathbb{R}_{>0}$ by:
\begin{equation} \label{eq:nu-fuk}
  \nu(p) := \sup \Bigl\{H_p^{L_0, L_1}(t,x) \bigm| L_0, L_1 \in \lag
  \text{\ are geometrically distinct, }
  (t,x) \in [0,1]\times X\Bigr\}.
\end{equation}

Next, define a preorder on $\mathcal{P}$. For $p, q \in \mathcal{P}$
we define $p \preceq q$ if for every two geometrically distinct
$L_0, L_1 \in \lag$ we have
$$H_q^{L_0, L_1}(t,x) \leq H_p^{L_0, L_1}(t,x), \; \forall \; (t,x) \in
[0,1] \times X.$$ Note that when endowed with $\preceq$, the set
$\mathcal{P}$ becomes directed (i.e.~in addition to being a preorder
we have that
$\forall \; p', p'' \in \mathcal{P}, \; \exists \; q \in \mathcal{P}$
with $p', p'' \preceq q$).

Let $p, q \in \mathcal{P}$ with $p \preceq q$. From ~\cite{Amb:fil-fuk},
there exists a distinguished class of filtered, strictly unital,
continuation $A_{\infty}$-functors
$\mathcal{H}_{p,q}: \fuk(\lag;p) \longrightarrow
\fuk(\lag;q)$. Different functors in this class are $0$-homotopic
Moreover, we also have a distinguished class of strictly unital
continuation $A_{\infty}$-functors
$\mathcal{H}_{q,p}: \fuk(\lag;q) \longrightarrow \fuk(\lag;p)$ with
linear deviation rate $\leq 2(\nu(p) - \nu(q))$ (see~\S\ref{a:func-LD}
for the precise definition). Here too, any two functors in this class
are $0$-homotopic (the homotopy being a homotopy of LD-functors of a
given deviation rate; see~\S\ref{a:homotopy}).

By~\cite{Amb:fil-fuk} the functors $\mathcal{H}_{p,q}$ staisfy the
conditions listed in~\S\ref{sb:sys-ainfty}, hence the family of
categories $\{\fuk(\lag;p)\}_{p \in \mathcal{P}}$ together with the
comparison functors $\mathcal{H}_{p,q}$, $\mathcal{H}_{q,p}$,
$p \preceq q$, become a system of filtered $A_{\infty}$-categories
which we denote by $\widehat{\fuk}(\lag)$.

The system $\widehat{\fuk}(\lag)$ has in fact two additional important
properties. The first is that all the categories $\fuk(\lag;p)$ have
the same set of objects $\lag$ and the comparison functors act like
the identity on this set. In the terminology of~\S\ref{sbsb:sys-base},
the system $\widehat{\fuk}(\lag)$ has a fixed full base of objects.
The second property is that $\widehat{\fuk}(\lag)$ is in fact a
homotopy system, according to the definitions
from~\S\ref{sbsb:homotopy-sys}. This follows immediately from the
discussion above on the homotopy properties of the comparision
functors.

Note however that although $\widehat{\fuk}(\lag)$ has a fixed full
base of objects, this does not hold anymore for the associated system
of filtered modules $F\md(\widehat{\fuk}(\lag))$ or the Yoneda system
$\mathcal{Y}(\widehat{\fuk}(\lag))$ (in particular, they cannot even
be considered as homotopy systems). This has nothing to do with Fukaya
categories and the reason is purely algebraic - see
Remark~\ref{r:coh-base}. Nevertheless, the fact that
$\widehat{\fuk}(\lag)$ is a homotopy system is still useful, as it
allows to define some of the limit invariants introduced
in~\S\ref{sb:sys-inv}, e.g.~the persistence Hochschild homology
$PHH(\widehat{\fuk}(\lag))$ of the system $\widehat{\fuk}(\lag)$ as
defined in~\eqref{eq:HH-colim}.

Returning to $\widehat{\fuk}(\lag)$, we will turn it into a genuine
system of TPC's with increasing accuracy. There are two ways to do
this, each leading to a slightly different system. The first way is to
use the persistence derived Fukaya category using filtered
modules as described in~\S\ref{sbsb:fmod}, and the second one is to
use the construction from~\S\ref{sbsb:fmod-2}.

We begin with the first implementation. We use the parameter space
$\mathcal{P}$, the size function
$\nu: \mathcal{P} \longrightarrow \mathbb{R}_{>0}$ and the pre-order
$\preceq$ discussed above. The categories in the system of TPC's will
be the persistence derived categories $PD(\fuk(\lag; p))$,
$p \in \mathcal{P}$. The comparison functors
for $p \preceq q$ are defined to be
$$\msh_{p,q} := PD (\mathcal{H}_{p,q}): PD(\fuk(\lag; p))
\longrightarrow PD(\fuk(\lag; q)),$$ where we use here the notation
and construction from~\S\ref{sbsb:fmod}. For $p \preceq q$ we define
$\msh_{q,p}: PD(\fuk(\lag; q)) \longrightarrow PD(\fuk(\lag;
p))$ in the same way. Note that although $\mathcal{H}_{q,p}$ has
linear deviation, its induced functor on the persistence derived
categories is still a TPC-functor. See~\S\ref{sbsb:fmod} for more
details. We will denote this system of TPC's by
$PD(\widehat{\fuk}(\lag))$.

The second implementation is very similar, only that it appeals to the
construction from~\S\ref{sbsb:fmod-2}. The categories in the system
are now $PD^c(\fuk(\lag;p))$, $p \in \mathcal{P}$, and the comparison
functors are $\msh^c_{p,q} := PD^c(\mathcal{H}_{p,q})$ and similarly
for $\msh^c_{q,p}$. We denote this system of TPC's by
$PD^c(\widehat{\fuk}(\lag))$.

Finally, we remark that although $\widehat{\fuk}(\lag)$ is a system
with a fixed full base of objects (and even a homotopy system),
neither $PD(\widehat{\fuk}(\lag))$ nor $PD^c(\widehat{\fuk}(\lag))$
has a coherent full base of objects.

\subsubsection{Functors associated with
  symplectomorphisms} \label{sbsb:symplecto-fun} Let $(X, \omega)$ be
a symplectic manifold of one of the types listed at the beginning
of~\S\ref{s:pdfuk} and $\lag$ be a subset of the collection of
Lagrangians in $X$ as at the beginning of~\S\ref{sb:sys-fuk}. Denote
by $\mathcal{P}$ the space of admissible perturbation data
parametrizing the system of filtered Fukaya categories with objects
$\lag$.

Let $\phi : X \longrightarrow X$ be a symplectic diffeomorphism
compactly supported in the interior of $X$. We will define below an endo-functor of each of the systems of TPC's $PD(\widehat{\fuk}(\lag))$ and
$PD^c(\widehat{\fuk}(\lag))$, that is induced by $\phi$.

In order to define such a functor we need several additional
assumptions on $\phi$. First of all we assume that $\phi$ preserves
the collection of underlying Lagrangians $\overline{L}$ with
$L \in \lag$. The other assumptions on $\phi$ depend on the collection
of Lagrangians (exact, monotone or weakly exact) which $\lag$ is
associated with.

We begin with the case of exact Lagrangians. Here we first need to
keep in mind the Liouville form $\lambda$. Let
$\lag_{\lambda} \subset \lagex_{\lambda}$ be a subset which is closed
under translations and shifts. We will make the following additional
assumptions: $\phi$ is assumed to be exact, namely
$\phi^*\lambda = \lambda + dF$ for some smooth function
$F: X \longrightarrow \mathbb{R}$ which is compactly supported inside
the interior of $X$. We also assume that $\phi$ preserves the homotopy
class of the quadratic complex $n$-form $\Theta$ (among non-vanishing
forms of this type, and after indentifying $(T(X),J)$ with
$(T(X), \phi^*J)$). We also assume that $\phi$ is graded
(see~\cite{Se:graded} for the definition of a graded
symplectomorphism) and denote its action on graded Lagrangians by
$(\overline{L}, \theta) \longmapsto (\phi(\overline{L}),
\phi_*(\theta))$.

Under the above assumptions we obtain a bijective map,
\begin{equation} \label{eq:phi-obj-lam-1} \Ob(\fuk(\lag, \lambda; p))
  \longrightarrow \Ob(\fuk(\lag, \lambda';q)), \quad (\overline{L}, h,
  \theta) \longmapsto (\phi(\overline{L}), h \circ
  \phi^{-1}|_{\phi(\overline{L})}, \phi_*(\theta)),
\end{equation}
where $\lambda' := (\phi^{-1})^*\lambda$. Here $p,q$ are any two
choices of perturbation data (indeed, below we will need to work with
$p \neq q$).

By the assumptions on $\phi$ we have
$\lambda' = \lambda - d(F \circ \phi^{-1})$, hence after composing the
preceding map with the canonical
isomorphism~\eqref{eq:can-iso-fuk-lam} we obtain a bijection
\begin{equation} \label{eq:phi-obj}
  \begin{aligned}
    & \phi: \Ob(\fuk(\lag, \lambda;p)) \longrightarrow \Ob(\fuk(\lag,
    \lambda;q)),
    \\
    & (\overline{L}, h, \theta) \longmapsto (\phi(\overline{L}), h
    \circ \phi^{-1}|_{\phi(\overline{L})} - F|_{\overline{L}} \circ
    \phi^{-1}|_{\phi(\overline{L})}, \phi_*(\theta)),
  \end{aligned}
\end{equation}
which by abuse of notation we still denote by $\phi$.

Next we define a map
$\phi_* : \mathcal{P} \longrightarrow \mathcal{P}$ by pushing forward
the structures in the perturbation datum. Specifically, let
$p \in \mathcal{P}$ be a perturbation data. Recall that $p$ assigns to
each tuple of Lagrangians in $\lag$ some auxiliary structures.  In the
simplest case, when this tuple $\vec{L} = (L_0, \ldots, L_d)$ has no
consecutive entries (in the cyclic sense) with the same underlying
Lagrangians, the perturbation $p$ assigns to $\vec{L}$ a tuple
$(K^p, J^p)$ with two components: the first is a family
$K^p =\{K^p_S \in \Omega^1(S; C_0^{\infty}(X))\}_{S \in \mathcal{S}}$
of $1$-forms parametrized by the space $\mathcal{S}$ of
boundary-punctured disks (with any number $\geq 2$ of punctures). Each
member of the family $K^p$ is a $1$-form $K_S^p$ on the punctured disk
$S$ with values in the space of compactly supported smooth functions
on $X$ (which should be viewed as Hamiltonian functions). The second
component $J^p$ is a family $\{J_S^p\}_{S \in \mathcal{S}}$ of (domain
dependent) almost complex structure, or more precisely each $J^p_S$ is
itelsf an $S$-parametrized family $\{J^p_{S,z}\}_{z \in S}$ of
$\omega$-compatible almost complex structures on $X$ which are also
compatible with the symplectic convexity of $(X, \omega)$ at infinity
(in case $X$ is not closed).  The push forward $\phi_*(p)$ of $p$ by
$\phi$ assigns to $\phi(\vec{L}) := (\phi(L_0), \ldots, \phi(L_d))$
the perturbation data $\phi_*(p) = (\phi_* K^p, \phi_* J^p)$, where
\begin{equation} \label{eq:push-forward}
%  \begin{aligned}
  (\phi_*K^p)_S (\xi) :=  K^p_S(\xi) \circ \phi^{-1}, \; \forall
  \xi \in T(S), \quad (\phi_* J^p)_{S,z} := D \phi \circ J_{S,z}
  \circ D \phi^{-1}, \; \forall S \in \mathcal{S}, z \in S.
%  \end{aligned}
\end{equation}
This completes the definition of $\phi_*(p)$ in the case when
$\vec{L}$ does not contain consecutive entries whose underlying
Lagrangians coincide.

When the tuple of Lagrangians does contain consecutive entries with
the same underlying Lagrangians the perturbation data $p$ are defined
over the space clusters (see~\cite{Amb:fil-fuk}, and
also~\cite[Section~3.3]{BCZ:tpc} where these are called decorated
clusters of punctured disks). In essence these combine punctured disks
and some trees attached to some of the boundary arcs of the punctured
disks. Over each punctured disk in the cluster the perturbation data
are as before (namely, they consits of a pair $(K^p,J^p)$ as described
above). Over each tree the perturbation data consists of a Morse data,
which is a collection of paramter-depending Morse functions and
Riemannian metrics assigned to each edge of the graph (and depending
on a parameter varying along each edge). The push forward by $\phi$
of the Morse data (in the tree part of a cluster) is simply done by
composing the Morse functions with $\phi^{-1}$ and similarly for the
Riemannian metrics. Finally, over each punctured disk in a cluster we
use the same recipe as in~\eqref{eq:push-forward}.

The space of admissible perturbation data
$\mathcal{P}$ is closed under the action of $\phi$, namely for every
$p \in \mathcal{P}$ we have $\phi_*(p) \in \mathcal{P}$. This follows
from the definition of $\mathcal{P}$ in~\cite{Amb:fil-fuk}. Moreover,
it follows from~\eqref{eq:nu-fuk} that for every $p \in \mathcal{P}$ we
have $\nu(\phi_*(p)) = \nu(p)$ and that for every $p \preceq q$ we
have $\phi_*(p) \preceq \phi_*(q)$.

The diffeomorphism $\phi$ also induces obvious chain isomorphisms,
defined for every $L_0, L_1 \in \lag_{\lambda}$ and
$p \in \mathcal{P}$:
\begin{equation} \label{eq:phi-chain} \phi^{CF} : CF(L_0, L_1;
  \lambda; p) \longrightarrow CF(\phi(L_0), \phi(L_1); \lambda;
  \phi_*(p)).
\end{equation}
Note that we are using here the same Liouville form $\lambda$ both in
the domain and target of $\phi^{CF}$, as we have done also
in~\eqref{eq:phi-obj}. Since
$(\phi^{-1})^*\lambda = \lambda - d(F \circ \phi^{-1})$ it follows
that the chain map $\phi^{CF}$ preserves action filtrations.

The maps $\phi^{CF}$ extend to a filtered $A_{\infty}$-functor (in
fact an $A_{\infty}$-isomorphism)
\begin{equation} \label{eq:phi-fun-1} \phi^{\fuk}_p: \fuk(\lag,
  \lambda; p) \longrightarrow \fuk(\lag, \lambda; \phi_*(p))
\end{equation}
whose first order terms are $\phi^{CF}$ and with vanishing higher
order terms. These functors fit together to a functor of systems of
filtered $A_{\infty}$-categories,
$$\widehat{\phi}^{\fuk}: \widehat{\fuk}(\lag, \lambda) \longrightarrow
\widehat{\fuk}(\lag, \lambda),$$ as defined at the end
of~\S\ref{sb:sys-ainfty} (see also~\S\ref{sbsb:func-sys-pc}).

Passing to the persistence derived level, we obtain by~\S\ref{sb:pdc}
TPC functors:
\begin{equation} \label{eq:phi-fun-2} PD(\phi^{\fuk}_p): PD(\fuk(\lag,
  \lambda, p)) \longrightarrow PD(\fuk(\lag, \lambda, \phi_*(p))).
\end{equation}

The collection of functors $PD(\phi_p^{\fuk})$, $p \in \mathcal{P}$,
and the push forward map
$\phi_*: \mathcal{P} \longrightarrow \mathcal{P}$, form together one
functor of systems of TPC's, according to the definitions
from~\S\ref{sb:sys-tpc}:
\begin{equation} \label{eq:phi-fun-3} PD(\widehat{\phi}^{\fuk}) :
  PD(\widehat{\fuk}(\lag, \lambda)) \longrightarrow
  PD(\widehat{\fuk}(\lag, \lambda)).
\end{equation}
The compatibility with respect to the comparison functors, as required
by diagrams~\eqref{sys-fun-diag-1} and~\eqref{sys-fun-diag-2}, follows
by standard arguments.

There is also a similar functor
$PD^c(\widehat{\phi}^{\fuk}) : PD^c(\widehat{\fuk}(\lag, \lambda))
\longrightarrow PD^c(\widehat{\fuk}(\lag, \lambda))$, defined
analogously to~\eqref{eq:phi-fun-3}, in which we just replace $PD$ by
$PD^c$ everywhere in~\eqref{eq:phi-fun-2}.

The definition of $\phi^{\fuk}_{p}$, $\widehat{\phi}^{\fuk}$ and their
persistence derived versions $PD(\phi^{\fuk}_{p})$,
$PD(\widehat{\phi}^{\fuk})$ in the monotone and the weakly exact cases
is very similar to the above. In these cases, the action on the
perturbation data $p \mapsto \phi_*(p)$ is precisely the same as in
the exact case described above. The action of $\phi$ on the grading in
the weakly exact case is the same as in the exact case, while in the
montone case it follows the receipe
from~\cite[Section~2.b]{Se:graded}. Finally, the action of $\phi$ on
the middle parameter $a_L \in \mathbb{R}$ of an object
$L = (\overline{L}, a_L, -)$ is the identity in both the monotone and
the weakly exact cases.

\subsection{Ungraded setting} \label{sb:fuk-ungraded} Sometimes it
will be more convenient to consider Fukaya categories without any
grading. This can be done in each of the settings described
in~\S\ref{sb:fuk-ex} -~\S\ref{sb:fuk-wex}, by simply omitting the
grading component from the objects $L \in \lag$ and viewing the hom's
as ungraded chain complexes. In the ungraded context, whenever we deal
with triangulated structures we just set the translation functor to be
$\id$. Note also that all the considerations from~\S\ref{sb:sys-fuk}
carry over to the ungraded setting in a straightforward way.

% !TEX root = approx8.tex

\section{Nearby Approximability}\label{sec:nearby}
 
 The aim of this section is to prove an approximability result, Theorem \ref{thm:nearby},  which is formulated in \S \ref{subsec:state}. In the last subsection, \S\ref{subsec:nearby-TPC},  the first point  of Theorem  \ref{thmmain1} is  seen to easily follow  from Theorem \ref{thm:nearby}. 

 \subsection{Setting}\label{subsec:back} We start by providing some necessary notation and background.
 
 Let $(N^{n},\mathtt{g})$ be a closed $n$-manifold with Riemannian metric $\mathtt{g}$. 
 We denote by $D_{r}^{\ast}N$ the closed $r$-disk cotangent bundle of $N$, with respect to  the  metric corresponding to $\mathtt{g}$ (under the identification $T^{\ast}N\cong TN$, induced by $\mathtt{g}$) $$D_{r}^{\ast}(N)\subset T^{\ast}N~.~$$
 We endow $T^{\ast}N$ with the canonical symplectic form
 $\omega=d\lambda$ with $\lambda$ the Liouville $1$-form on $T{^\ast}N$.
 
 We consider the class $\lagex(D^{\ast}_{r}N)$ of all closed, {\em exact} Lagrangian submanifolds in the interior of $D_{r}^{\ast}N$, see \S\ref{sb:fuk-ex}. For $r=1$, we simplify the notation by omitting the subscript $r$. There are filtered Fukaya categories with objects the elements in $\lagex(D^{\ast}_{r}N)$ denoted by $\fuk(\lagex(D^{\ast}_{r}N); p)$ where $p\in\mathcal{P}$ is a choice of perturbations, as described in  
\S\ref{sbsb:fil-ex}.  The categories associated with various perturbation choices are related through 
comparison functors, as described in
\S\ref{sb:sys-fuk}.  To simplify the various constructions below we denote $\mathcal{A}_{p}(r)=\fuk(\lagex(D^{\ast}_{r}N); p)$. We  omit $p$ from the notation, when this parameter is fixed and no confusion is possible and in case $r=1$ we write $\mathcal{A}_{p}$ instead of $\mathcal{A}_{p}(r)$. 
 
The considerations below are independent of $r$ and thus we fix $r=1$.
 The Yoneda functor embeds $\mathcal{A}_{p}$ inside the category 
 of filtered $A_{\infty}$ modules, $
 F\md_{\mathcal{A}_{p}}$ (see \ref{sbsb:fmod}),
 $$\mathcal{Y}_{p}: \mathcal{A}_{p}\hookrightarrow F\md_{\mathcal{A}_{p}}~.~$$
 
 For a point $x\in N$, let $\overline{F}_{x}$ be the cotangent fiber through $x$,
 $$\overline{F}_{x}=\pi^{-1}(x)$$ where  $\pi : T^{\ast}N\to N$  is the projection of the cotangent bundle. Such a fiber also admits markings and we denote by $F_{x}= (\overline{F}_{x}, h_{F_{x}}, \theta_{F_{x}})$ a marked fiber, with underlying Lagrangian submanifold $\overline{F}_{x}$ ($h_{F_{x}}$ 
 is a constant function in this case). Each $F_{x}$ defines a filtered $A_{\infty}$-module still denoted  
 $$F_{x}\in F\md_{\mathcal{A}_{p}}~.~$$  
 For this it is necessary to define the perturbations required for
 defining $\mu_{d}$ operations involving tuples of closed Lagrangians and one copy of $\overline{F}_{x}$. 
 Denoting marked Lagrangians and the  associated Yoneda modules  by the same symbols simplifies notation and it is generally non-problematic, however it is important to keep track of the precise context, namely of the underlying $A_{\infty}$-category.
 With these conventions, we have $F_{x}(L)=CF(L,F_{x}; p)$ as filtered chain complexes.  
 
\

We denote by ${\mathcal{A}^{F}_{p}}$ the filtered full sub-category of $F\md_{\mathcal{A}_{r}}$ that contains the image  of $\mathcal{Y}_{p}$, and all the modules ${F}_{x}$, $x\in N$ (with all possible markings on the fibers $\overline{F}_{x}$). As before,
in case we want to render explicit both the perturbation parameter $p$ and the radius of the 
disk bundle $D^{\ast}_{r}N$ we write $\mathcal{A}^{F}_{p}(r)$. 

\

The next step is to complete each category ${\mathcal{A}^{F}_{p}}$ with respect to mapping cones over morphisms with filtration level $\leq 0$ and quasi-isomorphisms of filtration level $\leq 0$ and denote the result by $[\mathcal{A}^{F}_{p}]^{\Delta}$. Finally, we let 
$\msc_{p}$ be the associated TPC, $\msc_{p}=PH([\mathcal{A}^{F}_{p}]^{\Delta})$. The categories 
$\mathcal{A}^{F}_{p}$ fit into a system of filtered categories for varying $p$, as in \S\ref{sb:sys-ainfty}.  Thus we have functors $\mathcal{H}_{p,q}:\mathcal{A}^{F}_{p}\to \mathcal{A}^{F}_{q}$ defined using the 
push-forward construction of filtered modules. More specifically,
recall that for $p \preceq q$ we have a continuation functor
$\mathcal{H}^{\text{cont}}_{p,q} : \mathcal{A}_p \longrightarrow
\mathcal{A}_q$. The functor $\mathcal{H}_{p,q}$ is defined as
$\mathcal{H}_{p,q} := (\mathcal{H}^{\text{cont}}_{p,q})_*$, the
induced push-forward of modules. The functors $\mathcal{H}^{\text{cont}}_{p,q}$ preserve filtration whenever $p\preceq q$. We also have functors $\mathcal{H}^{\text{cont}}_{q,p}$, however these are $A_{\infty}$-functors with a possibly non-vanishing linear deviation  (see \S\ref{sb:sys-ainfty}).

Summing up,  we obtain comparison functors, still denoted  
$\mathcal{H}_{p,q}: \msc_{p}\to \msc_{q}$. The categories $[\mathcal{A}^{F}_{p}]^{\Delta}$ are pre-triangulated and, as in \S\ref{sb:sys-fuk}, as a result the system  $\widehat{\msc}(D^{\ast}N) =(\msc_{p}, \mathcal{H}_{p,q})$ forms a system of TPCs with increasing accuracy in the sense of \S\ref{s:sys-tpc}.

 \subsection{The main statement, and the idea of the proof}\label{subsec:state}

Using the system $\widehat{\msc}(D^{\ast}(N))$ of TPCs we can now formulate the main result of the section.

\

To proceed, recall the equivalence relation in \S\ref{sb:sys-inv} defined on 
$$\Ob^{\text{tot}}(\widehat{\msc}(D^{\ast}N)) := \coprod_{p\in \mathcal{P}} \Ob(\msc_p)~.~$$
 The set of associated equivalence classes
is denoted by $\widetilde{\Ob}(\widehat{\msc}(D^{\ast}N))$. 
All exact, marked, Lagrangians $L\in \mathcal{L}ag^{(ex)}(D^{\ast}N)$ correspond to well-defined equivalence classes in $\widetilde{\Ob}(\widehat{\msc}(D^{\ast}N))$  because $\mathcal{H}_{p,q}$ takes Yoneda modules to modules $0$-equivalent to Yoneda modules whenever $p\preceq q$. We denote
by $\mathcal{L}(D^{\ast}N)\subset \widetilde{\Ob}(\widehat{\msc}(D^{\ast}N))$ the set of equivalence classes of all such, closed, exact, marked $L$'s. We formulate a similar property for a family of marked fibers $\F=\{F_{x_{1}},\ldots, F_{x_{i}},\ldots\}$ of $D^{\ast}N$. We say that the family
$\F$ satisfies property ($\ast$) if:

\begin{itemize}
\item[($\ast$)]
 {\it for all $p \preceq q$, for every marked fiber
$F_{x} \in \F$ we have $\mathcal{H}_{p,q}(F_{x}) \cong_0 F_{x}$, where the $F_{x}$ on the right-hand side of $\cong_0$ stands for the  module corresponding to 
$F_{x}$ in $\Ob(\msc_{q})$.} 
\end{itemize}
If this property is satisfied by the family $\F$, then  each marked fiber $F_{x}\in \F$ corresponds to a well-defined equivalence class $\mathscr{F}_{x}\in\widetilde{\Ob}(\widehat{\msc}(D^{\ast}N))$.

\begin{thm}\label{thm:nearby}
Fix $\epsilon>0$. It is possible to construct the system $\widehat{\msc}(D^{\ast}(N))$ such that  there is a finite family of marked fibers $\F_{\epsilon}=\{F_{x_{1}}, \ldots, F_{x_{l}}\}$ that satisfies property {\rm ($\ast$)} and, moreover, the set $\mathcal{L}(D^{\ast}N)$ is $\epsilon$-approximable in the sense of Definition \ref{d:approx-sys} by $\{\mathscr{F}_{x_{1}},\ldots, \mathscr{F}_{x_{l}}\}$.
\end{thm}

In other words, for some $\beta = \beta(\epsilon) >0$ we have that for each $L\in\mathcal{L}ag^{(ex)}(D^{\ast}N)$ and each $p$ with 
$\nu(p)\leq \beta$ there is an iterated cone $C$ inside $\msc_{p}$ constructed by using  the marked fibers $(\overline{F}_{x_{i}},  \alpha_{x_{i}})$, for some constants $\alpha_{x_{i}}\in \R$ (and some choice of grading),  such that $C$ is at interleaving  distance at most $\epsilon$ from $L$ in $\msc_{p}$.  Of course, we identify here a Lagrangian $L$ with its corresponding Yoneda module. Moreover, fitting the categories $\msc_{p}$ in the system $\widehat{\msc}(D^{\ast}N)$
renders the relevant construction independent of the choice of $p$, as long as $\nu(p)$ is small enough.

\

We will see in \S\ref{subsec:nearby-TPC} that this result immediately implies approximability in the sense of Definition \ref{def:TPC-approx} and thus  point i of Theorem \ref{thmmain1} is established.

\begin{rem}\label{rem:dep_on_e} The dependence of the statement 
on $\epsilon$ is subtle. It is possible to construct the system $\widehat{\msc}(D^{\ast}N)$ such that each category $\msc_{p}$ is independent of $\epsilon$ (for $\nu(p)$ small) and  all the disk bundle fibers, and in particular  all the elements of  $\F_{\epsilon}$, for all $\epsilon$,  are represented as modules in $\msc_{p}$. However, the system
$\widehat{\msc}(D^{\ast}N)$ also depends on the functors $\mathcal{H}_{p,q}$ and these functors need to be so that they satisfy property ($\ast$) relative to the family of fibers $\F_{\epsilon}$. In our argument, this property follows from a geometric construction that is the central part of the proof and that requires a fixed $\epsilon$ as starting point. As a result, when $\epsilon$ changes, apriori, the system $\widehat{\msc}(D^{\ast}N)$ changes too, along with the family
$\F_{\epsilon}$ whose number of elements increases when 
$\epsilon$ tends to $0$.  See also Remark \ref{rem:compare_data}.
\end{rem}

\

\subsubsection{Outline of the proof of Theorem \ref{thm:nearby}.} 

The  proof proceeds roughly as follows.  We first establish the existence
of some special Morse functions that have small variation and large gradient away from their critical points. More precisely, denote by $B_{\eta}(x)$ the disk in $N$ (with respect to the metric $g$) of radius $\eta$ and center at $x\in N$.  
For any values $K, \delta >0$ and sufficiently small $\eta >0$ there
exists a Morse function $f:N\to [0, K]$ such that $||d f_{x}||\geq \delta$ for all points $x\in N\backslash (\cup_{z\in Crit(f)} B_{\eta}(z))$.  We then consider a Lefschetz 
fibration structure on $T^{\ast}N$, $p:T^{\ast}N\to \mathbb{C}$, 
as constructed by  Giroux \cite{Gir:Lef-cot}, such that the real part of the projection $p$ coincides with a Morse function
 $f$ with the properties above, with $1>\delta>0$ fixed and $K$ small. We then use a method introduced in \cite{Bi-Co:lefcob-pub} to surger each exact Lagrangian with a certain number of (compactifications) of the fibers through $Crit(f)$ and  use a Hamiltonian isotopy to push the result away from $D^{\ast}N$.  This part of the construction takes place in the total space $E$ of a modified Lefschetz fibration that is constructed along the approach in \cite{Bi-Co:lefcob-pub}.  Algebraically,
this process  expresses the Lagrangian as an iterated cone of fibers and an $s$-acyclic term, where $s$ is at most equal to the Hofer energy needed to generate the previous Hamiltonian displacement. The key point is to control this energy. This is where  the fact that $K$ is small (and $\delta$ is relatively big) comes into play:  through some precise estimates that appear in Giroux's construction we show that this energy  can be related to the area of a certain planar region corresponding to the projection on $\mathbb{C}$ of $D^{\ast}N$, and, as a result, it is bounded uniformly, independently of $L$, by a constant times $K$.

The  steps required for the proof of Theorem \ref{thm:nearby} are each contained in a subsection below, as follows.  In \S\ref{subsec:Morse} we outline the construction of the special Morse functions with small variation but big differential away from the critical points. This is a soft result, but not completely trivial. In \S\ref{subsec:Giroux} we review the main ingredients in Giroux's construction. Next, in \S\ref{subsec:Lef-Dehn} we
recall the approach from \cite{Bi-Co:lefcob-pub} and adjust it to Giroux's set-up. Most of the proof -  in \S\ref{subsec:Morse}, \S\ref{subsec:Giroux}, \S\ref{subsec:Lef-Dehn} and \S\ref{subsec:finish-proof} up to \S\ref{subsubsec:back_to_alg} -  is entirely geometric and serves to establish the statement in Proposition \ref{lem:disj_fibers} which, essentially,  claims that:

{\em  Given $\epsilon >0$ there exists a Lefschetz fibration $p:E\to \mathbb{C}$ that contains $D^{\ast}N$ and such that there are exact Lagrangian spheres $\hat{S}_{1},\ldots, \hat{S}_{l}$ in $E$, each of them intersecting $D^{\ast}N$ along a (different) fiber of $D^{\ast}N$, with the property that, for any exact, closed Lagrangian $L\subset D^{\ast}N$, the iterated Dehn twist $\tau_{\hat{S}_{l}}(\tau_{\hat{S}_{l-1}} (\ldots \tau_{\hat{S}_{1}}(L)).. )$ can be disjoined from $D^{\ast}N$ by a Hamiltonian diffeomorphism of Hofer norm at most $\epsilon$.}

The last part 
of the proof of Theorem \ref{thm:nearby}, in \S\ref{subsubsec:back_to_alg}, translates this geometric result in an appropriate system  $\widehat{\msc}(D^{\ast}N)$. 

\subsection{Special Morse functions }\label{subsec:Morse}
\begin{prop}\label{prop:Morse}
Let $(N^{n},\mathtt{g})$ be a smooth, closed riemannian manifold . For any $K>0$, $1>\delta >0$ and  
$\eta>0$ there exists a Morse function $\varphi=\varphi_{N, K,\delta}:N\to [0,K]$ with the property
that $ ||d\varphi(x)||\geq\delta$ for all $x\in N\backslash (\cup_{z\in Crit(\varphi_{N, K,\delta})} B_{\eta_{z}}(z))$, where $0<\eta_{z}\leq \eta$, $\forall z\in Crit(\varphi_{N, K,\delta})$.
\end{prop} 
Notice that we assume here that the radius $\eta_{z}$ of the ball $B_{\eta_{z}}(z)$ around 
the critical points $z$ of $f$ depends on the critical point.  We also assume that all these balls $B_{\eta_{z}}$ are pairwise disjoint. 
\begin{proof}  To start the proof we fix a Morse function $f_{0}:N\to \R$
as well as some small quantity $\alpha>0$. We assume that the function $f_{0}$
and the metric  $\mathtt{g}$ are standard inside $B_{\alpha}(z)$ in the neighbourhoods $B_{\alpha}(z)$ for each $z\in Crit(f_{0})$. These two conditions
are not essential for the argument, they just simplify some of the calculations
in the proof. With these assumptions,  the gradient (calculated with respect to $\texttt{g}$) 
$\nabla f_{0}$ is standard, of the form $k(\ldots, \pm x_{i}\partial/\partial x_{i},\ldots)$ with respect to coordinates $(x_{1},\ldots, x_{i},\ldots, x_{n})$
inside each $B_{\alpha}(z)$; $k$ is a constant that depends on $z$. We also assume that $f_{0}$ has $m$ critical points $z_{0},\ldots, z_{m-1}$ and $f_{0}(z_{j})=jB$ for some positive constant $B$.
After possibly
multiplying $f_{0}$ with a large positive constant we may assume that $||df_{0}||\geq \delta$ outside the union of all the balls $B_{\alpha}(z)$.  We denote by $V=(m-1)B$ the maximal value of $f_{0}$.  We let $U_{z}(\mu)=f_{0}^{-1}([f_{0}(z)-\mu, f_{0}(z)+\mu])$ and 
$U'_{z}(\mu)=f_{0}^{-1}([f_{0}(z)-\mu/2, f_{0}(z)+\mu/2])$ and assume that $B_{\alpha}(z)$
is contained in $U'_{z}(\mu)$. We fix $\mu>0$ so  that all
these sets $U_{z}(\mu)$ are pairwise disjoint and to simplify notation below
we denote $U_{z}=U_{z}(\mu)$, $U'_{z}=U'_{z}(\mu)$.

We will modify $f_{0}$ with two aims:  diminish the variation $V$ and decrease the size of the neighbourhoods $B_{z}(\alpha)$, in both cases as much as needed. Of course, the price to pay is that we will add many new additional critical points
to those of $f_{0}$. The argument uses induction on the dimension of $N$. Thus, we assume the Proposition established for all manifolds of dimension at most $n-1$ (it is an easy exercise to prove the result in dimension $1$).

The modifications of the function $f_{0}$ are of two types. 

The first is a change inside a small
neighbourhood of each critical point $z$ of $f_{0}$. We will later refer to this step as ``shrinking''. To simplify notation we fix such a critical
point and its critical value $v=f(z)$. We consider $\eta <\alpha/2$ and a ball $B_{\eta}(z)$ as well as the larger ball $B_{2\eta}$.
We then consider a diffeomorphism $\psi_{z}:U'_{z}\to U_{z}$ with the following properties:
\begin{itemize}
\item[-] $\psi_{z}(f_{0}^{-1}(v\pm\mu/2))=f_{0}^{-1}(v\pm\mu)$.
\item[-] $\psi_{z}$ preserves  $f^{-1}(v)$.
\item[-] $\psi_{z}$ maps  $B_{2\eta}(z)$ to $B_{\alpha}(z)$ while keeping $z$ fixed and on $B_{\eta}(z)$ is acts by $x \to rx$ where 
$r>1$ is given by $r^{2} =\frac{\delta}{\eta k}$ (this obviously assumes $\eta$ small enough, see also the next point).
\item[-] we assume $\eta$ small enough such that that $2r\eta = 2\sqrt{\frac{\delta\eta}{k}} <\alpha$.
\item[-]  $\psi_{z}$ 
transports $f_{0}^{-1}(v-\mu/2)$ along the flow of $-\nabla f_{0}$, mapping it to
$f_{0}^{-1}(v-\mu)$, and it transports 
 $f_{0}^{-1}(v+\mu/2)$ along the flow of $+\nabla f_{0}$, sending it to $f_{0}^{-1}(v+\mu)$.
\end{itemize}
It is easy to construct such a diffeomorphism, see Figure \ref{fig:expand}. 
\begin{figure}
\includegraphics[scale=0.9]{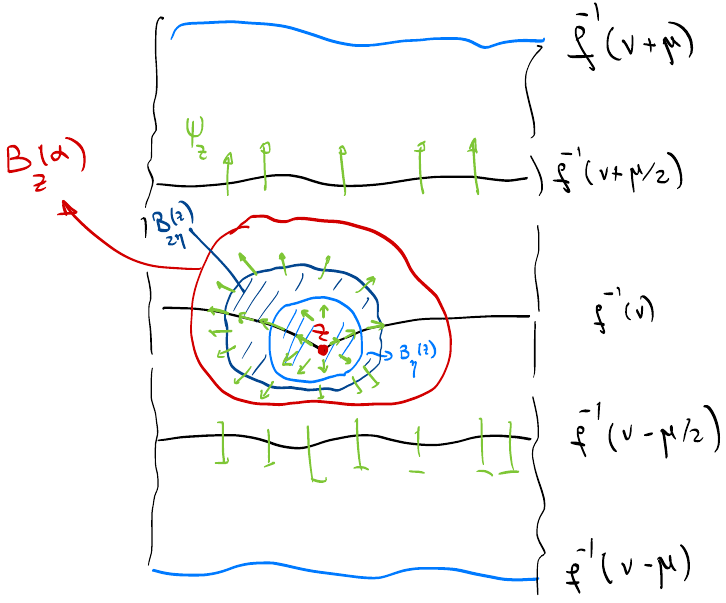}
  \centering
  \caption{The diffeomorphism $\psi_{z}$.} \label{fig:expand}
\end{figure}

We now define
a new function $f_{1}$ as follows: on the set $U'_{z}$ it is given by $f_{0}\circ \psi_{z}$; on the set $f_{0}^{-1}(-\infty, v-\mu/2]$ it is equal to $f_{0}-\mu/2$; on the set $f_{0}^{-1}[v+\mu/2,\infty)$ it is equal to $f_{0}+\mu/2$.
Notice that this function is continuous, of total variation $V+\mu$,
but not smooth along $f_{0}^{-1}(v\pm \mu/2)$. However,
it can be smoothed along these two hypersurfaces, and by this smoothing - still denoted $f_{1}$ -  does not add any new critical points and  $||df_{1}(x)||\geq \delta$ for all points $x$ outside $B_{2\eta}(z)$. 

The total variation of the  perturbed function $f_{1}$ is smaller than $V+2\mu$, by taking the smoothing sufficiently small. 
We now discuss the role of two of the  properties of $\psi_{z}$. The role of the third  property is to ensure that $||df_{1}(x)||\leq \delta$ for $x\in B_{\eta}(z)$ (this property will be important for us later in the proof of the proposition). 
This is easily seen as $||\nabla f_{0}(x)||= k ||x||$ inside $B_{\alpha}(z)$
and thus $||\nabla f_{1}(x)||=r k ||r x||= r^{2}k ||x||=\delta\frac{||x||}{\eta}$ for $x\in B_{\eta}(z)$. The role of the fourth property is to allow for
enough ``room'' so that $||df_{1}(x)||\geq \delta $ for the points $x\in B_{2\eta}(z)\backslash B_{\eta}(z)$.

We apply this procedure for each critical point of $f_{0}$ and we call the resulting function
still $f_{1}$. Its variation is  $ < V+2m\mu$.   
It is useful to add an appropriate constant to the function $f_{1}$ so that its minimum is $0$. 
With this convention, we notice that the critical values of the function $f_{1}$ are (as close as desired, by taking the smoothings small enough) to, in order, $0$,  $B+\mu/2$, $2B+3\mu/2$ and so forth. In particular, these critical values
are separated by $\approx B+\mu$. To fix ideas we will assume also $\mu << B$. 

We denote $C=V+2m\mu$ and notice that this constant is independent of $\eta$.  Therefore, if further arguments require us to reduce $\eta$, this does not affect this upper bound for the variation of $f_{1}$, nor for the difference between successive critical values which remains $\approx B+\mu$.  We denote by $V_{1}$ the maximal value of  $f_{1}$.
To summarize, the function $f_{1}$ is Morse and satisfies:
\begin{equation}\label{eq:fone}
||df_{1}(x)||\geq \delta, \ \forall x\notin \cup_{z\in \Crit(f_{1})}B_{\eta}(z)\,  \  \mathrm{and}\ \ \min(f_{1})=0,\ \  \max (f_{1})= V_{1} < C ~.~
\end{equation}
Moreover, as explained before we also have:
\begin{equation}\label{eq:fone'}
||df_{1}(x)||\leq \delta \ \ \forall x \in \cup_{z\in \Crit(f_{1})}B_{\eta}(z)\ , B\leq f(z_{i+1})-f(z_{i})\leq B + 3\mu/2~.~
\end{equation}
An important consequence of property (\ref{eq:fone'}) is that the variation
of $f_{1}$ over each of the balls $B_{\eta}(z)$ is bounded from above by
$2\eta\delta$, and thus it can be reduced as needed by making $\eta$ small.

\

The second step is a ``folding'' procedure around regular hypersurfaces of the form $f_{1}^{-1}(a)$. This procedure produces a function of small variation, but  adds
many new additional critical points.  

We first explain the prototype of this  folding construction, and its impact on the variation of the function. The description of folding is as follows. Pick a value $a\in [C/2, C)$. To describe the folding we assume that $a$ is a regular value of $f_{1}$ and that $W_{1}=f^{-1}_{1}(a)$ does not intersect any of the balls $B_{\eta}(z)$.

We consider the following function $\phi_{1}:[0, C]\to [0,a]$
defined by $$\phi_{1}(t)=t \ \ \mathrm{when}\  t \leq a \ \ , \ 
\phi_{1}(t)=a-t \ \ \mathrm{when}\  t\geq a~.~$$
Clearly, this function is continuous but not smooth. Nonetheless, we define $f'_{2}=\phi_{1}\circ f_{1}$. We notice that the variation of $f'_{2}$ is $a$ and that outside the 
balls $B_{\eta}(z)$ we still have $||df'_{2}||\geq \delta$ except along $W_{1}$, where 
$f'_{2}$ is not differentiable. 

Our aim is to perturb the function $f'_{2}$ to obtain a Morse
function $f_{2}$ such that, for a fixed $\xi>0$  we have:
\begin{equation}\label{eq:ftwo}
||df_{2}(x)||\geq \delta, \ \forall x\notin \cup_{z\in \Crit(f_{2})}B_{\eta_{z}}(z)\ , \ \min(f_{2})=0,\ \max (f_{2})= V_{2} < \xi + a \ \ 
  ~.~
  \end{equation}
 First, for $\epsilon>0$
we let $K_{\epsilon, a}=f^{-1}_{1}([a-\epsilon, a+\epsilon])$. We pick $\epsilon_{1}>0$
small enough such that $K_{\epsilon_{1}, a}\cap B_{\eta}(z)=\emptyset$ for all critical points
$z$ of $f_{1}$. The function $f_{2}$
has the following form:
\begin{equation}\label{eq:perturb}
f_{2}=\phi_{2}\circ f_{1}+(\beta \circ f_{1})\cdot f_{W_{1}}~.~
\end{equation}
Here $\phi_{2}: [0,C]\to [0,a]$ is smooth, it coincides with $\phi_{1}$ 
outside the interval $(a-\epsilon_{1}/2, a+\epsilon_{1}/2)$,  $\phi_{2}$ is Morse 
with a unique maximum at $a$, and $\phi_{2}\leq \phi_{1}$ - see Figure \ref{fig:first-fold}.
\begin{figure}
\includegraphics[scale=0.9]{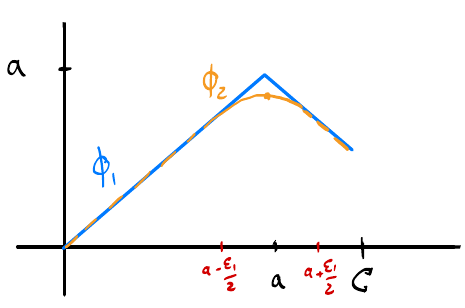}
  \centering
  \caption{The graphs of the functions $\phi_{1}$ and $\phi_{2}$.} \label{fig:first-fold}
\end{figure}

 The function $\beta$ is a smooth bump function $\beta:\R\to [0,1]$, which vanishes outside
$( a-\epsilon_{1}, a+\epsilon_{1})$ is increasing on $(a-\epsilon_{1}, a)$, decreasing on $(a, a+\epsilon_{1})$
and $\beta(a)=1$. The function $f_{W_{1}}: W_{1}\to \mathbb{R}$ is obtained through the inductive hypothesis, by applying  the Proposition to $W_{1}$. More precisely, we 
fix on $W_{1}$ the metric induced from $\mathtt{g}$ and construct the function $f_{W_{1}}$ such that it is Morse, with minimal value equal to $0$, and with maximal value $C_{1}< \xi$, and 
with $||df_{W_{1}}(z)||\geq 2\delta$, outside the union of balls $\hat{B}_{\eta_{z}}(z)$, $\eta_{z}\leq \eta$, for each critical point $z$ of $f_{W_{1}}$. The notation $\hat{B}_{-}(-)$
 indicates that the respective balls are in $W_{1}$. It is easy to see that by adjusting conveniently
 the profiles of the functions $\phi_{2}$ and $\beta$ we may ensure the desired properties for
 $f_{2}$. Notice that the critical points of $f_{2}$ are of two types,  critical points of $f_{1}$,
 and critical points of $f_{W_{1}}$,  with their Morse index raised by $1$ when viewed
 as critical points of $f_{2}$.
 
 \
 
 We will now apply this folding construction for not just one hypersurface such as $W_{1}$ before but for more hypersurfaces at the same time. First, recall the constant $K$ in the statement of the proposition and  fix a natural number  $l$ such that $ \frac{B+3\mu/2}{2^{l}}<K/4$. 

We consider a function $f_{0}$ as in the shrinking procedure and apply that construction for $\eta < \frac{B}{2^{l+1}\delta}$. We denote by $f_{1}$ the resulting function.

Consider the interval $J_{i}=[f_{1}(z_{i}), f_{1}(z_{i+1})]$, $0\leq i < m-1$, 
and divide it into $2^{l}$ equal pieces denoted
$I_{k,i}=(f_{1}(z_{i})+\frac{k S_i}{2^{l}}, f_{1}(z_{i})+\frac{(k+1)S_{i}}{2^{l}})$
where $S_{i}= f_{1}(z_{i+1})-f_{1}(z_{i})$.
Consider a continuous function:
$\varphi_{i}: J_{i}\to [0,\infty)$ with the following properties: 
\begin{itemize}
\item[-] $\varphi_{i}$ is  smooth  on each of the intervals  $I_{k,i}$ for $k=0, 1,\ldots 2^{l}-1$, with constant slope equal to $+1$ on $I_{k,i}$  when  $k$ is even and with slope equal to $-1$ when $k$ is odd. 
\item[-] We let $a_{k,i}= \frac{kS_{i}}{2^{l}}$,  $0<k\leq 2^{l}-1$. The minimal value of ${\varphi}_{i}$ is $0$ and is attained at the points $a_{k,i}$ with 
$k$ even. The maximal value of ${\varphi}_{i}$ is $\frac{S_{i}}{2^{l}}$ and is attained at the points $a_{k,i}$ with $k$ odd.
\end{itemize}
Given that $S_{i}\geq B$ we have $S_{i}/2^{l}\geq B/2^{l}>2\eta \delta$. This implies by (\ref{eq:fone'}) that all the values $a_{k,i}$ are regular values
for $f_{1}$ and that the hypersurfaces $W_{k,i}=f^{-1}(a_{k,i})$ do not 
intersect $B_{\eta}(z_{i})$ and $B_{\eta}(z_{i+1})$. We  piece together
the functions ${\varphi}_{i}$ to a new continuous function $\hat{\phi}_{1}$
defined as follows: $\hat{\phi}_{1}$ is equal to $\varphi_{i}$ on each 
interval $J_{i}$ for $i$ even and is equal to $-\varphi_{i}$ on the intervals
$J_{i}$ for $i$ odd. This means that the function $\hat{\phi}_{1}$ is smooth
at each value $f(z_{i})$, with slope alternating between $+1$ and $-1$, depending 
on the parity of $i$. Moreover, the hypersurfaces $W_{k,i}$ do not intersect
any of the balls $B_{\eta}(z)$, $z\in Crit(f_{1})$. Finally, the variation of 
$\hat{\phi}_{1}$ is at most $2 \max {\frac{S_{i}}{2^{l}}}\leq 2\frac{B+3\mu/2}{2^{l}}<K/2$, see Figure \ref{fig:saw}.
\begin{figure}
\includegraphics[scale=1.3]{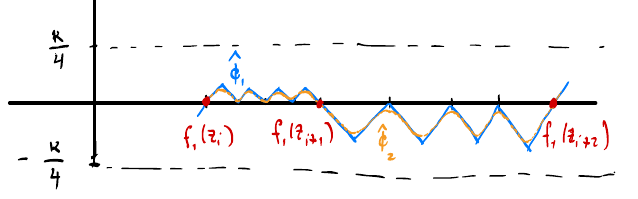}
  \centering
  \caption{The graphs of the functions $\hat{\phi}_{1}$ and $\hat{\phi}_{2}$.} \label{fig:saw}
\end{figure}

Next, we consider a Morse smoothing $\hat{f}_{2}$ of $\hat{\phi}_{1}\circ f_{1}$ that coincides
with $\hat{\phi}_{1}\circ f_{1}$ except on the sets 
$K_{\epsilon_{2}, a_{k,i}}=f^{-1}_{1}([a_{k,i}-\epsilon_{2}, a_{k,i}+\epsilon_{2}])$. More precisely, $\hat{f}_{2}$
is given by a formula similar to (\ref{eq:perturb}) with some straightforward modifications:  the place of $\phi_{2}$ is taken
by a Morse smoothing $\hat{\phi}_{2}$ of $\hat{\phi}_{1}$ with maxima at $a_{k,i}$ for $ki$ odd, and minima
at $a_{k,i}$ for $ki$ even; on each $K_{\epsilon_{2}, a_{k,i}}$ the place of the function $f_{W_{1}}$ is taken by a corresponding 
function $f_{W_{a_{k,i}}}: W_{a_{k,i}}\to \mathbb{R}$ of variation smaller than $\frac{S_{i}}{2^{l+1}}$; the bump function
$\beta:\R\to [-1,1]$ is supported in the union of the intervals $(a_{k,i}-\epsilon_{2}, a_{k,i}+\epsilon_{2})$, it is monotone on each of the intervals $(a_{k,i}-\epsilon_{2}, a_{k,i})$ and $(a_{k,i}, a_{k,i}+\epsilon_{2})$, it attains its maximum equal
to $1$ at the points $a_{k,i}$ for $ki$ odd and its minimum, equal to $-1$, at the points $a_{k,i}$ for $ki$ even ($0< k \leq 2^{l}-1$, $0\leq i < m-1$). 

As a result of all these choices the function $\hat{f}_{2}$ has a variation bounded from above by $K$ and the function $\varphi_{N,K,\delta}$ in the statement is defined
by adding a constant to $\hat{f}_{2}$ such that the resulting
function, denoted $\varphi_{N,K,\delta}$, is positive, with minimal value equal to $0$ and with a maximum value smaller than $K$. This function satisfies the other properties claimed in the statement.
\end{proof}

\begin{rem}\label{rem:conf_metric} In the arguments appearing later in the paper it is useful to assume, in addition to the properties in Proposition \ref{prop:Morse}, that the norm $||d\varphi_{N,K,\delta}||\leq 1$ over all of $N$. This is not so easy to achieve while keeping the metric $\texttt{g}$ fixed. However, it will be enough for us that there is a metric 
$\bar{\texttt{g}}$ on $N$, conformal to $\texttt{g}$, with $\bar{\texttt{g}}=\alpha \texttt{g}$, with $\alpha:N\to [1,\infty)$ such that the rest of the properties claimed in the proposition remain true, relative to $\bar{\texttt{g}}$,
and in addition $||d\varphi_{N,K,\delta}||_{\bar{\texttt{g}}}\leq 1$.  This is very easy to see, by adjusting the metric $\texttt{g}$, outside the neighbourhoods $B_{\eta_{z}}(z)$, at those points $x\in N$ where the norm $||d\varphi_{N,K,\delta}(x)||$ (in the metric $\texttt{g}$) is greater than $1$.
\end{rem}

\subsection{Review of Giroux's construction}\label{subsec:Giroux} We will review here the parts of \cite{Gir:Lef-cot} that will be relevant later on in the proof of Theorem \ref{thm:nearby}. The type of Lefschetz fibration considered (Definition 1 in \cite{Gir:Lef-cot}) is a Liouville manifold $(W, d\lambda)$ endowed
with a smooth map 
$$h\colon W\to \mathbb{C}$$
such that $h$ has the following properties: 
\begin{itemize}
\item[1.] The singularities of $h$ are of the form 
$$h(z)=h(0)+\sum_{j}z^{2}_{j} $$ in coordinates $(z_{1},\ldots, z_{n})$ centered at the singularity and in which $\omega=d\lambda$ is a positive $(1,1)$ form at $0$.
\item[2.] The distribution $\ker  dh \subset TW$ consists of symplectic subspaces and the singular connection provided by its symplectic orthogonal complement is complete.
\item[3.]  The manifold $W$ is exhausted by Liouville domains $(W_{k}, \lambda|_{W_{k}})$ such that for
every $w\in \mathbb{C}$ and for all sufficiently large $k\geq k_{w}$, the fiber $F_{w}=h^{-1}(w)$
intersects $\partial W_{k}$ transversely along a positive contact submanifold of $\partial W_{k}$, and
the Liouville field on $F_{w}$ dual to $\lambda|_{F_{w}}$ is complete.
\end{itemize}

The main result in Giroux's paper is the following.

\begin{thm}\label{thm:gir}[Giroux \cite{Gir:Lef-cot}] Let $N$ be a closed manifold, $\varphi:N\to \mathbb{R}$ a Morse function and $\nu$ an adapted gradient of $\varphi$ satisfying the Morse-Smale condition. Then $\varphi$ extends to a Lefschetz fibration (in the sense above) $h=f+i g\colon T^{\ast}N\to \mathbb{C}$ whose imaginary part is the function:
$$ g\colon T^{\ast}N\to \mathbb{R}\ ,\  (q,p) \mapsto\ g(q,p):=  \langle p, \nu (q)\rangle$$
and whose real part $f$ coincides with $\varphi$ on the $0$-section and is homogenous of degree $1$ in the variable $p$ near infinity.
\end{thm}

A vector field $\nu$ is called adapted to $\varphi$ if $d\varphi(\nu) >0$ away from the critical points of $\varphi$ and near each critical point $a$ there exists
a local coordinate chart that expresses $\varphi$ in Morse form, $\varphi(x)=\varphi(a)+\sum_{j} \epsilon_{j} x^{2}_{j}$, and $\nu$ in linear 
form, $\nu(x)=2\sum_{j}\epsilon_{j}x_{j}\partial_{x_{j}}$, $\epsilon_{j}\in \{-1,+1\}$.

In our case, it is convenient to assume that $\nu$ is an actual gradient 
vector field of $\varphi$ with respect to a riemannian metric $\texttt{g}$ on 
$M$, \begin{equation}\label{eq:nu-grad}
\nu= grad _{\texttt{g}}(\varphi)~.~
\end{equation}
We will also assume that each critical point of $\varphi$ is unique on its critical level.

A few properties of Giroux's construction are important in our proof and we will review them now.

\begin{itemize}
\item[i.] The function $f:T^{\ast}N\to \mathbb{R}$ has the form $f=f_{0}+f_{1}$ with: 

\begin{equation}\label{eq:expression_re}
f_{0}(q,p)=\varphi(q)-\frac{1}{2}\nabla^{2}\varphi (q)(p,p)~.~
\end{equation}
Here $\nabla^{2}\varphi (q)$ is the  covariant second derivative of $\varphi$ at $q$ viewed as bilinear map on $T^{\ast}N$. The function $f_{1}:T^{\ast}N\to \mathbb{R}$ is supported away from the $0$ section $N\subset T^{\ast}N$. The precise construction of this perturbative term  $f_{1}$ requires much of the effort in \cite{Gir:Lef-cot}. Our notation is somewhat different from Giroux's and, to facilitate the correspondence
between the two papers, we indicate the differences here - our function $f_{0}$ is denoted $f^{0}$ in \cite{Gir:Lef-cot} and the function $f_{1}$ here has the form:
$$ f_{1}= \tau_{1} f^{1} - (1-\tau_{0}) f^{0}$$
where the notation on the right side of the equality is that in \cite{Gir:Lef-cot}. Namely, $\tau_{i}$ are functions only depending on the radial distance $r$ from the $0$-section such that $\tau_{1}$ vanishes for small values of $r$ and equals a large positive constant $C$ for $r$ large enough, while $\tau_{0}$ equals $1$ for small $r$  and vanishes away from the $0$-section. The function $f^{1}$ is of the form $f^{1}=r f^{\infty}$ and $f^{\infty}$ is a well chosen function defined on the sphere cotangent bundle $f^{\infty}:ST^{\ast}N\to \R$.

\item[ii.] The only critical points of $h$ are along the $0$ section and they
coincide with the critical points of $\varphi$.

\item[iii.] For each critical point $a\in Crit(\varphi)$, the vertical thimble  along the vertical straight half-line originating at $h(a)$ and pointing upwards coincides with the  unstable manifold of the Hamiltonian vector field  $X^{f}$  at $a$.

\item[iv.] The Poisson bracket $\{f, g\}$ is positive at each point $x\in T^{\ast}(N)\backslash Crit(\varphi)$ 
\end{itemize}
Our arguments require a quantitative complement to property iv which we state next:

\begin{itemize}
\item[v.] Assume that $d\varphi (\nu)\geq \delta$ outside $\cup_{a\in Crit(\varphi)}B_{\epsilon}(a)$, for some small $\epsilon$ such that the balls $B_{\epsilon}(a)$, $B_{\epsilon}(b)\subset N$ are disjoint for all $a\not=b\in Crit(\varphi)$.  
With this assumption, the construction of $f$ can be made such that 
$$\{f,g\}(x)\geq \delta,  \ \forall x\in T^{\ast}N\backslash \cup_{a\in Crit(\varphi)} B'_{\epsilon}(a)$$
where we denote by $B'_{\epsilon}(-)$ the respective balls of radius $\epsilon$ in $T^{\ast}N$.
\end{itemize}

A reformulation of property v will play an important role in our argument.

\begin{cor}\label{cor:size_proj}
Under the assumptions in property \textnormal{v} above, we have that $dh(X^{f})$ is purely imaginary and $Im(dh(X^{f}(x)))\geq \delta$ for each $x\in T^{\ast}N$ outside any of the balls $B'_{\epsilon}(a)$, $a\in Crit(h)$.
\end{cor}

Properties i, ii, iii, and iv appear explicitly in \cite{Gir:Lef-cot}. Property 
v follows easily from the proof of property iv as given in \S E \cite{Gir:Lef-cot}. To give a few more details, here are our conventions: $\omega (Y,X^{f})=df(Y)$, $\{f,g\}=\omega (X^{f},X^{g})=dg(X^{f})=-df(X^{g})$. The arguments in \cite{Gir:Lef-cot} make use of the vector field $\tilde{\nu}= -X^{g}$ that is easily seen to be a lift of $\nu$ to $T^{\ast}N$. 
It is shown that the quantity $df(\tilde{\nu})=\{f,g\}$ is positive at all points different from the critical points of $h$, which is the claim at  iv. The proof of this positivity depends on the construction of the perturbative term $f_{1}$ and appeals to a choice of the constant $C>0$ that has appeared at point i above. When this constant is taken sufficiently big, $df(\tilde{\nu})$ is  seen to be non-negative and to only vanish at the critical points of $h$. By possibly taking $C$ even bigger,  the same argument shows that property v is also true. 

Starting from property v the statement in the corollary is immediate. Indeed,
$\{f,g\}=dg (X^{f})= d(p_{2}\circ h)(X^{f})=
dp_{2}\circ (dh (X^{f}))=dh( X^{f})$ where $p_{2}:\mathbb{C}\to \R$
is the projection on the second coordinate. We have used the fact that for each
point $x$, $dh( X^{f}(x))$ is a vertical vector in $\mathbb{C}$. This is 
true because $X^{f}(x)$ is tangent to the level hypersurface  $f^{-1}(c)=(p_{1}\circ h)^{-1}(c) =h^{-1} (\{(c,y) : \ y\in \R\})$ with $c=f(x)$ and $p_{1}:\mathbb{C}\to \R$ the projection on the first coordinate. 

\subsection{Disjunction through Dehn twists}\label{subsec:Lef-Dehn} In this section we consider  Giroux's Lefschetz fibration from \cite{Gir:Lef-cot}, as recalled in \S\ref{subsec:Giroux},  and apply to it the construction in \cite{Bi-Co:lefcob-pub}, in particular the arguments in Section 4.4 in that paper.  The methods in \cite{Bi-Co:lefcob-pub}  as reflected 
in Proposition 2.3.1 (see also \S  2.3 \cite{Bi-Co:lefcob-arxiv} for more details on the relevant construction), together with the doubling of singularities in \S 4.4.2 \cite{Bi-Co:lefcob-pub} (see  Figure 17 there) applied to the Lefschetz fibration $h\colon T^{\ast}N\to \mathbb{C}$ from \S\ref{subsec:Giroux} lead to a new Lefschetz fibration $\bar{h}: E\to \mathbb{C}$ which is schematically represented in Figure \ref{fig:new_fibr}.  
\begin{figure} [h]
  \includegraphics[scale=0.9]{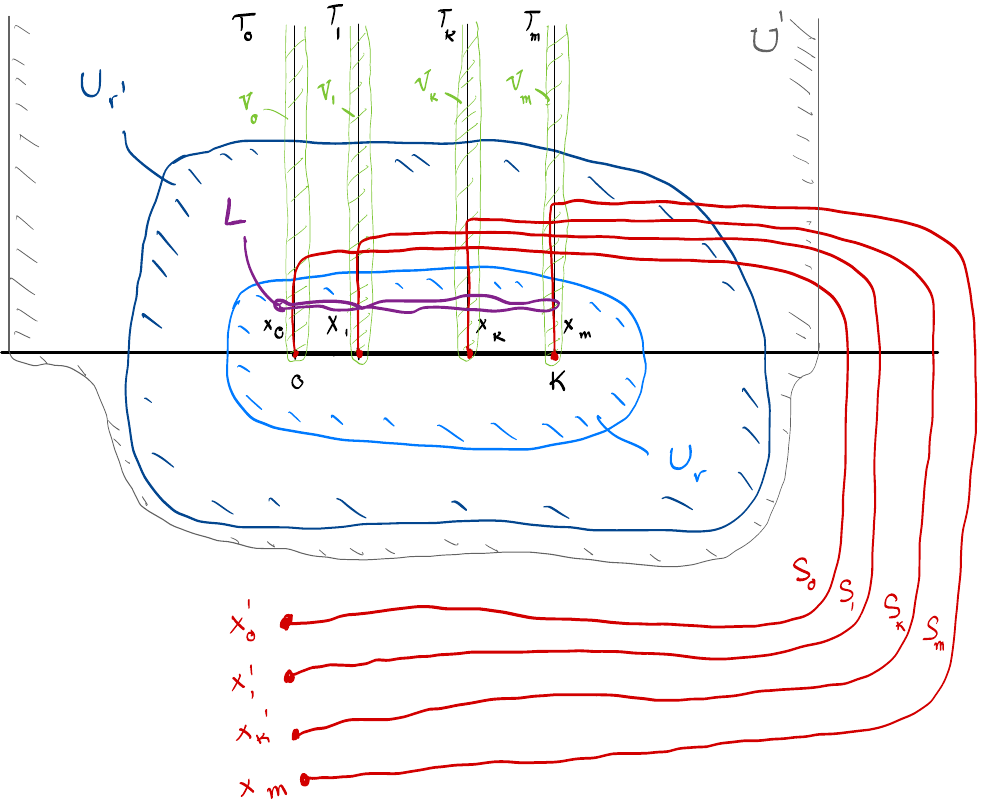}
  \centering
  \caption{The fibration $\bar{h}:E\to \mathbb{C}$.} \label{fig:new_fibr}
\end{figure}
Here are the main relevant properties of $\bar{h}$.  We assume the setting 
in Theorem \ref{thm:gir}. In particular, we have the 
Morse function $\varphi: N\to \mathbb{R}$ with a unique critical point on each critical level and that is such that the pair $(\varphi, \texttt{g})$ is  Morse-Smale where $\texttt{g}$ is a fixed Riemannian metric on $N$ and the properties  i.-v. from \S\ref{subsec:Giroux} are satisfied. Furthermore, we have:
\begin{itemize}
\item[a.] The Morse function $\varphi: N\to \mathbb{R}$ has values in $[0,K]$, having $0$ as minimal value and $K$ for maximal value. 
\item[b.]  We put $U_{r}=h(D_{r}^{\ast}N)$  where the disk cotangent bundle  $D_{r}^{\ast}N$ is defined relative to the metric $\texttt{g}$, and 
similarly $U_{r'}=h(D_{r'}^{\ast}N)$ - both these sets are represented in Figure \ref{fig:new_fibr}.
\item[c.]  Fix a constant $R>0$ big enough such that $U_{r'}\subset [-R,R]\times\mathbb{R}$. There is a neighbourhood $U'\subset \mathbb{C}$ of $U_{r'}\cup ([-R, R]\times [0,\infty))$ such that $h$  and  $\bar{h}$ agree on $h^{-1}(U')$.
\item[d.] The fibration $\bar{h}$ has one additional critical point $x'_{i}$ for 
each critical point $x_{i}$ of $h$. For each pair $x_{i}, x'_{i}$ of such critical points, $x_{i}$ and $x'_{i}$ are related through a matching cycle $S_{i}\subset E$ (which is a Lagrangian sphere) whose image onto $\mathbb{C}$ is like in Figure \ref{fig:new_fibr}. 
\item[e.] The vertical thimbles pointing up and originating at the critical points 
$x_{i}$ are denoted by $\tau_{i}$. They are unstable manifolds of the Hamiltonian vector field $X^{\bar{f}}$ at each critical point $x_{i}$ with $\bar{f}= Re (\bar{h})$ (recall from iii. that $\bar{f}= f$ over $U'$).
\item[f.] Fix some $\delta >0$ and assume that $h= f+i g$ satisfies $\{f,g\}(x)\geq \delta$ at all points $x$ outside
small balls $B'_{\epsilon}(x_{i})$, $x_{i}\in Crit(\varphi)$, as at point v. in \S\ref{subsec:Giroux}. Under this assumption we fix small neighbourhoods $V_{i}\subset E$  for each of the thimbles $\tau_{i}$ such that $V_{i}\supset B'_{\epsilon}(x_{i})$. The projection of the $V_{i}$'s onto $\mathbb{C}$ is as in the picture. 
\end{itemize}
In our arguments the constants $\delta>0$, $K>0$ are fixed in advance with 
$\delta < 1$ and the function $\varphi=\varphi_{N,K,\delta}$ is given  by Proposition \ref{prop:Morse}. The key property that will be used further below in the proof is provided by Corollary \ref{cor:size_proj}: 
\begin{equation}\label{eq:Ham_big}
 \ \  d\bar{h}(X^{\bar{f}}(x))\in i [\delta,+\infty) \subset  i\mathbb{R}\subset \mathbb{C}\ \   , \forall\ x\in \bar{h}^{-1}(U')\backslash (\cup_{i} V_{i}) ~.~
 \end{equation} 
 
 We will also assume that, after possibly a conformal change of metric, as in  Remark \ref{rem:conf_metric}, passing from the arbitrary metric $\mathtt{g}$, to a conformally equivalent metric $\bar{\mathtt{g}}=\alpha \mathtt{g}$ with $\alpha:N\to [1,\infty)$, we have $||d\varphi||\leq 1$.  We will use this metric $\bar{\mathtt{g}}$ throughout from this point on and discuss this modification further below in \S\ref{subsec:general}. 
 This means that
$||\nu ||\leq 1$ - see (\ref{eq:nu-grad}) - and further implies by the description of $g$ in Theorem \ref{thm:gir} that:
\begin{equation}\label{eq:width}
U_{r}\subset \mathbb{R}\times [-r,r], \ \  \forall r \in \mathbb{R}
\end{equation} 

The description of $f_{0}$ from the point i, in \S\ref{subsec:Giroux} implies that, for $r$ sufficiently small, we  have 
\begin{equation}\label{eq:length}
U_{r}\subset [- C_{\varphi} r , K + C_{\varphi} r]\times [-r,r]
\end{equation}
where $C_{\varphi}$ is a positive constant depending on the $C^{2}$ norm
of $\varphi$.

In this context, the main result that we need from \cite{Bi-Co:lefcob-pub} is next.

\begin{prop}[\cite{Bi-Co:lefcob-pub}] \label{prop:disj}For any Lagrangian $L\subset D^{\ast}_{r}N$, and any arbitrarily small neighbourhoods $N_{i}$ of the  matching cycles $S_{i}$,  there exist models of the Dehn twists $\tau_{S_{i}}$ that are  supported in $N_{i}$ and  such that  $\tau_{S_{m}} \circ \ldots  \circ \tau_{S_{2}}\circ \tau_{S_{1}}\circ \tau_{S_{0}}(L)$ does not intersect any of the thimbles $\tau_{i}$.
\end{prop}

\subsection{Proof of Theorem \ref{thm:nearby}}\label{subsec:finish-proof}

The proof of the theorem will be completed in four steps. In \S\ref{sububsec:disjunction-we} we start from the setting in \S\ref{subsec:Lef-Dehn} and use (\ref{eq:Ham_big}), (\ref{eq:length}) and Proposition \ref{prop:disj} to show that, if $r$ is sufficiently small,  the iterated Dehn twist $$\tau_{m,\ldots, 1,0}L= \tau_{S_{m}} \circ \ldots  \circ \tau_{S_{2}}\circ \tau_{S_{1}}\circ \tau_{S_{0}}(L)$$ can be disjoined from $D^{\ast}_{r}(N)$ through a Hamiltonian isotopy of energy 
at most $8Kr/\delta$.  In \S\ref{subsec:exact-seq} we reformulate the disjunction result in Lemma \ref{lem:disj-energy} in terms of Lagrangian spheres that restrict to fibers $F_{x_{i}}\subset D^{\ast}_{r}(N)$. In \S\ref{subsubsec:back_to_alg} we establish an approximability type result in Proposition \ref{prop:estimates-in-r} that still involves the constants $K$, $r$, $\delta$. Finally, the proof concludes in \S\ref{subsec:general} where we get rid of the restriction of $r$ being sufficiently small, showing that the conformal change of metric   $\texttt{g}$ does not affect the argument and that
the constants $K$, $\delta$ in Proposition \ref{prop:estimates-in-r} can be picked as needed to deduce the statement of  Theorem \ref{thm:nearby}.
\subsubsection{Disjunction with controlled energy.}\label{sububsec:disjunction-we}
The aim in this subsection is to prove the following Lemma. To state it we fix
the constants $0<\delta < 1$, $K>0$
 and the function $\varphi=\varphi_{N,K,\delta}$ given  by Proposition \ref{prop:Morse}, just as discussed in \S\ref{subsec:Lef-Dehn}.  We
also assume that points a. - f. there are satisfied as well as the inclusions in (\ref{eq:Ham_big}) and (\ref{eq:length}). 

\begin{lem}\label{lem:disj-energy} For $r$ sufficiently small, for any (marked) exact Lagrangian $L\subset D^{\ast}_{r} N$ the iterated Dehn twist $\tau_{m, \ldots, 1,0}L$ can be disjoined from $D^{\ast}_{2r} (N)$ inside $E$ with less  than $8Kr/\delta$ energy. 
\end{lem}
\begin{proof} To simplify notation denote $L'= \tau_{m, \ldots, 1,0}L$. Recall
that the constant $C_{\varphi}$ from (\ref{eq:length}) only depends on $\varphi$. We start by choosing $r>0$ such that $rC_{\phi}<K/4$. We also pick the constant $r'$ such that $r'=2r$. Therefore,
we deduce from (\ref{eq:length}) that 
$$U_{r}\subset [-K/4, 5K/4]\times [-r,r]$$
and, similarly $$U_{2r}\subset [-K/2, 3K/2]\times [-2r,2r]~.~$$
Assume for a moment that 
\begin{equation}\label{eq:assumption-dis}
L' \subset [\bar{h}^{-1}(U') \backslash \ (\cup_{i} V_{i}) ] \cup (\cup_{i} N'_{i})
\end{equation}
where $N'_{i}\supset N_{i}$ are neighbourhoods of the matching cycles $S_{i}$, possibly slightly larger than $N_{i}$. In that case, inspecting Figure \ref{fig:new_fibr} and  using (\ref{eq:Ham_big}), we see that the Hamiltonian flow $- X^{\bar{f}}$ moves $L'$ outside of $W_{K,r}=\bar{h}^{-1}([-K/2, 3K/2]\times [-2r,\infty))$ in time 
less than $4r/\delta$. At the same time the variation of $\bar{f}$ over $W_{K,r}$ is at most $2K$. It follows  that the Hofer displacement energy
of $L'$ from $W_{K,r}$, and, in particular, from $D^{\ast}_{2r}(N)\subset W_{K,r}$ 
 is less than $8Kr/\delta$.

We now intend to show that assumption (\ref{eq:assumption-dis}) is satisfied by some Lagrangian $L''=\psi (L')$ for some Hamiltonian diffeomorphism $\psi$ that will be described next. To this purpose, we notice that the closures of the neighbourhoods $V_{i}$ of the thimbles $\tau_{i}$ are included in the interior of $W_{K,r}$. Around each thimble $\tau_{i}$ we can find a Hamiltonian isotopy $\psi_{i}$ supported in a slightly larger neighbourhood $V'_{i}\subset W_{K,r}$ that is ``repelling'' away from
the thimble $\tau_{i}$ and such that $\psi_{i}$ leaves the thimble $\tau_{i}$ invariant (as a set) and moves the intersection $L'\cap V_{i}$  outside of $V_{i}$ , without creating any new intersections with $V_{i}$. For different $i$'s these Hamiltonian isotopies have disjoint support and thus they can be incorporated in a single isotopy whose time-one map, $\psi$, has
the property that $L''=\psi (L')$ satisfies (\ref{eq:assumption-dis}). 

We next apply to $L''$ the estimate for the displacement Hofer energy relative to $W_{K,r}$. Thus, there exists a Hamiltonian isotopy $\phi$ of energy less than
$8Kr/\delta$ such that  $\phi (L'')\cap W_{K,r}=\emptyset$. As the support of
$\psi$ is inside $W_{K,r}$ we deduce that $\psi^{-1}\phi\psi (L') \cap W_{K,r}=\emptyset$ and, by the bi-invariance of the Hofer norm, we also have $||\psi^{-1}\phi\psi||_{H}=||\phi||_{H}< 8Kr/\delta$ which concludes the proof of the lemma.

\end{proof}

\begin{rem}\label{rem:small_Dehn}
Notice that, as in Proposition \ref{prop:disj}, the support of the Dehn twists can be assumed to be contained in neighbourhoods of the $\hat{S}_{i}$ that are as small as desired.

\end{rem}
\subsubsection{Replacing matching cycles by Lagrangian spheres that restrict to fibers}\label{subsec:exact-seq} In this subsection we rewrite the disjunction energy estimate from Lemma \ref{lem:disj-energy} in terms of Dehn twists where the matching cycles $S_{i}$ are replaced by perturbations  $\hat{S}_{i}$ such that the intersection of $\hat{S}_{i}$ with $D^{\ast}_{r}(N)$ coincides with the cotangent fiber $F_{x_{i}}$.

We assume that $r$ is sufficiently small such that for each 
fiber $\overline{F}_{x_{i}}\subset D^{\ast}_{r}(N)$ (we recall $\{x_{i}\}_{i}=Crit(\varphi)$)  the image $\bar{h}(\overline{F}_{x_{i}})$ is included in $V_{i}$, see 
(\ref{eq:expression_re}),  and recall the expression of $g= Im(h)$ from Theorem \ref{thm:gir}. In particular, all these images are disjoint. We then 
consider a Hamiltonian diffeomorphism $\eta$ that is supported in $D^{\ast}_{\frac{3}{2}r}(N)$ and that deforms $S_{i}\cap D^{\ast}_{r}(N)$ to $F_{x_{i}}$. We will 
assume that the ``bends'' of $S_{i}$ in Figure \ref{fig:new_fibr} take place 
outside of $D^{\ast}_{\frac{3}{2}r}(N)$. We denote $\hat{S}_{i}=\eta(S_{i})$. We may also assume that $\eta$ is constructed such that $\hat{S}_{i}\cap D^{\ast}_{r}(N) = \overline{F}_{x_{i}}$ and $\eta(D^{\ast}_{r}(N))= D^{\ast}_{r}(N)$. Constructing such an $\eta$ is a simple exercise  by comparing the local expressions of the fiber $x_{i}$ and that of the thimble
$\tau_{i}$. Obviously, the submanifolds $\hat{S}_{i}$ are also Lagrangian spheres and we pick models for the Dehn twists relative to $\hat{S}_{i}$ such 
that $\tau_{\hat{S}_{i}}=\eta \circ \tau_{S_{i}}\circ \eta^{-1}$. This means, in particular,
that for any Lagrangian $L\subset D^{\ast}_{r}(N)$ we can write: 
$$\tau_{\hat{S}_{i}}(L)= \eta ( \tau_{S_{i}} (\eta^{-1}(L)))~.~$$
By iterating this relation we deduce that the iterated Dehn twist
$$\hat{\tau}_{m,\ldots, 1,0} (L) = \tau_{\hat{S}_{m}}\circ \ldots  \circ \tau_{\hat{S}_{1}}\circ\tau_{\hat{S}_{0}} (L)$$
satisfies
$$\hat{\tau}_{m,\ldots, 1,0} (L) =\eta( \tau_{m,\ldots,1,0} (\eta^{-1}(L))~.~$$
We know from the proof of Lemma \ref{lem:disj-energy} that there exists a Hamiltonian diffeomorphism $\phi'$ that, for any Lagrangian $L''$ (in our class), displaces $\tau_{m,\ldots,1,0}(L'')$ from $W_{K,r}$ and is of energy less than $8Kr/\delta$.  For a fixed Lagrangian $L$ we apply this property to $L''=\eta^{-1}(L)$  and, using the fact that the support of $\eta$ is included in $W_{K,r}$, we deduce that $\eta \circ \phi'\circ \eta^{-1}$ displaces $\hat{\tau}_{m,\ldots, 1, 0}(L)$ from  $W_{K,r}$. The Hofer norm of $\eta\circ \phi'\circ \eta^{-1}$ equals that of $\phi'$ and is thus smaller than $8Kr/\delta$.
In summary:
\begin{prop}\label{lem:disj_fibers} With the notation above, and for $r$ sufficiently small,  the conclusion  of Lemma \ref{lem:disj-energy} also applies to the iterated Dehn twist $\hat{\tau}_{m,\ldots, 1,0}(L)$ with each Dehn twist $\tau_{\hat{S}_{i}}$ having support in a neighbourhood of $\hat{S}_{i}$ that can be assumed as small as desired. \end{prop}

\subsubsection{Translating geometry into algebra}\label{subsubsec:back_to_alg}
In this section, we interpret the geometric disjunction result in 
Proposition \ref{lem:disj_fibers} in homological terms in
a Fukaya TPC  system $\widehat{\msc}(D^{\ast}_{r}(N))$.
The purpose here is to translate homologically the statement in Proposition \ref{lem:disj_fibers}. 
Two systems of related Fukaya categories will be important for us here, with components associated with the perturbation choice $p\in\mathcal{P}$ as below:
\begin{equation}\label{eq:components_cat}
\fuk(\mathcal{L}ag^{(ex)}(D^{\ast}_{r}N), p)\hookrightarrow \fuk(\mathcal{L}ag^{(ex)}E,p)~.~
\end{equation}
The category on the left is the one appearing in \S\ref{subsec:back} and the one on the right is
constructed just as in \S\ref{sbsb:fil-ex}. Its objects are marked exact, closed Lagrangian submanifolds in
$E$, the total space $E$ of the Lefschetz fibration  from \S\ref{subsec:Lef-Dehn}. An important point has to do with the choice of perturbations $p$: they are first picked on $D^{\ast}_{r}N$ and then extended to $E$. This is why there is an inclusion of $A_{\infty}$-categories relating the two sides. These categories fit into systems of filtered $A_{\infty}$-categories as in \S\ref{sb:sys-ainfty}. Further, as in \S\ref{sb:sys-fuk}, we 
obtain the system of TPCs $\widehat{PD}(\fuk(\mathcal{L}ag^{(ex)}(E))$. We also have the system of TPCs from \S\ref{subsec:back}, $\widehat{\msc}(D^{\ast}_{r}N) =(\msc_{p}(r), \mathcal{H}_{p,q})$. Notice that we include in the notation the radius $r$ of the relevant disk bundle. The categories $\msc_{p}(r)$ are homotopy categories of filtered  modules over the filtered Fukaya category $\mathcal{A}_{p}(r)=\fuk(\lagex(D^{\ast}_{r}N); p)$ and we have an inclusion of 
TPCs: $$ J_{p,r}:\msc_{p}(r)\hookrightarrow  H(F\md_{\mathcal{A}_{p}(r)})~.~$$

Because the perturbations on $E$ extend the perturbations on $D^{\ast}_{r}N$, the inclusion (\ref{eq:components_cat}) induces  
 pull-back TPC functors:
%$$\xi:\widehat{PD}(\fuk(\mathcal{L}ag^{(ex)}(E)) \to \widehat{\msc}(D^{\ast}_{r}N) $$
$$\xi_{p}:PD(\fuk(\mathcal{L}ag^{(ex)}(E), p)) \to H(F\md_{\mathcal{A}_{p}(r)})~.~$$
These functors commute with the comparison functors $\mathcal{H}_{p,q}$ if we  choose the perturbations used to define the comparison $\mathcal{H}_{p,q}$ functors for $D^{\ast}_{r}N$ by restriction of the perturbations used to define the corresponding functors for $E$. Notice also that, because the Lagrangian spheres $\hat{S}_{i}$ from \S\ref{subsec:exact-seq} intersect 
$D^{\ast}_{r}N$ along the fibers $\overline{F}_{x_{i}}$ it follows
that the module $F_{x_{i}}$ corresponding to the fiber $\overline{F}_{x_{i}}$ (see \S\ref{subsec:back}) is the pull-back of the Yoneda module $\mathcal{Y}(\hat{S}_{i})$. In particular, this module belongs to the image of $J_{p,r}$.  In $E$, the comparison functors $\mathcal{H}_{p,q}$ preserve the (Yoneda) modules associated with $\hat{S}_{i}$ for $p\preceq q$ and thus we deduce that, with these choices of perturbations, the family
$\mathcal{F}=\{F_{x_{i}}\}_{i}$ satisfies property ($\ast$)
from  the statement of Theorem \ref{thm:nearby}.
The next result brings us very close to the statement of this  theorem .

\begin{prop}\label{prop:estimates-in-r}
Assuming the setting and notation above, for $r$ sufficiently small and for a choice of perturbations $p$ with $\nu(p)$ sufficiently small, any exact, marked Lagrangian $L\subset D^{\ast}_{r}(N)$ satisfies in $\msc_{p}(r)$:
$$d_{int}\left( L, \Ob\  \langle\{F_{x_{0}},\ldots, F_{x_{m}}\} \rangle^{\Delta}\right) \leq 48 Kr/\delta~.~$$  
\end{prop}

\begin{proof} To simplify notation denote in this proof $c= 8Kr/\delta$, and 

\begin{equation} \label{eq:amb_Fuk}
\mathscr{D}_{p}(E):=PD(\fuk(\mathcal{L}ag^{(ex)}(E), p))~.~
\end{equation}

We first notice that to prove the claim it is enough to show:
\begin{lem}\label{lem:in_large} There are strict exact, weighted triangles in
$\mathscr{D}_{p}(E)$ of the form:
\begin{equation}\label{eq:iterated_cones1}
\Delta_{i} \ : \ Z_{i}\longrightarrow X_{i}\longrightarrow X_{i+1} \  , \ 0\leq i \leq m-1
\end{equation}
 such that $X_{0}$ is the Yoneda module of a Lagrangian disjoint from $D^{\ast}_{r}N$,  $ X_{m}= L$, and with the $Z_{i}$ of the form $\Sigma^{p(i)}\hat{S}_{x_{q(i)}}$ for some integers $p(i), q(i)$,  or possibly $Z_{i}=0$, and of total weight not more than $3c$. 
\end{lem}

Indeed, in that case, by the construction in Lemma 2.87
in \cite{BCZ:tpc} we obtain that there is a similar sequence  of triangles
\begin{equation*}\label{eq:iterated_cones2}
\Delta'_{i} \ : \ Z'_{i}\longrightarrow X'_{i}\longrightarrow X'_{i+1}
\end{equation*}
with $Z'_{j}$ again of the form $\Sigma^{l(j)}\hat{S}_{x_{s(j)}}$ or $=0$,
 with each $\Delta'_{i}$  exact in
 $(\mathscr{D}_{p}(E))^{0}$, with, $X'_{0}=X_{0}$, and with the last term $X'_{m}$ carrying a $6c$-isomorphism $X'_{m}\to L$. This means
by Lemma 2.85  \cite{BCZ:tpc} that $d_{int}^{\mathscr{D}_{p}(E)}(L,X'_{m})\leq 6c$. We now pull-back the triangles 
$\Delta'_{i}$ to  $H(F\md_{\mathcal{A}_{p}(r)})$. This produces triangles: 
$$
\Delta''_{i} \ : \ Z''_{i}\longrightarrow X''_{i}\longrightarrow X''_{i+1}
$$
that are exact in $(H(F\md_{\mathcal{A}_{p}(r)}))^{0}$, with $X''_{0}=\xi_{p}(X'_{0})=0$, with $Z''_{j}$ of the form
$\Sigma^{l(j)}F_{x_{s(j)}}$ (or $=0$) and with $\dint (X''_{m},L)\leq 6c$. Given that $\msc_{p}(r)$ contains the triangular completion of all the fiber modules $F_{x_{i}}$ we deduce that $X''_{i}\in \Ob (\msc_{p}(r))$ for all $i$ and this ends the proof of 
the Proposition \ref{prop:estimates-in-r} up to showing the Lemma \ref{lem:in_large}.
\end{proof}

\begin{rem} \label{rem:pull}
It is useful to emphasize that the pull-back $\xi_{p}$ is not necessarily full and faithful. In particular, 
 $\xi_{p}: HF(\hat{S}_{x_{i}},\hat{S}_{x_{j}})\to \hom_{\msc_{p}(r)}(F_{x_{i}}, F_{x_{j}})$ is not necessarily either injective nor surjective.
\end{rem}

\begin{proof}[Proof of Lemma \ref{lem:in_large}]
To construct the desired sequence of strict exact triangles we first fix the notation
$L_{m+1}=L$,  $L_{m-i}=\tau_{\hat{S}_{i}}(L_{m-i+1})$ so that $L_{0}=\hat{\tau}_{m,\ldots, 1,0}(L)$.
The statement of Proposition \ref{lem:disj_fibers} shows that $L_{0}$ can be  disjoined from $D^{\ast}_{r}(N)$ by a Hamiltonian isotopy $\Phi$ of Hofer norm $||\Phi||_{H} < c=\frac{8Kr}{\delta}$. 
By taking $p$ such that $\nu(p) \ll c- ||\Phi||_{H}$ we deduce, in TPC terminology,  that there is a weighted strict exact triangle in $\mathscr{D}_{p}(E)$:
\begin{equation}\label{eq:first-ex} 
0\longrightarrow L_{0}' \longrightarrow L_{0}\longrightarrow 0
\end{equation}
of weight $2c$, with $L_{0}'$ the Yoneda module of a Lagrangian  $L'_{0}\subset E$ disjoint from $D^{\ast}_{r}(N)$. This will be taken as the first triangle in our sequence (in other words we take $X_{0}=L_{0}'$). 

To construct the next triangles, 
we consider the homological interpretation of the Dehn twist, as
in Seidel's work  \cite{Se:book-fukaya-categ}, with the exception that we want to express the output  in a persistence context.  
Namely, we claim that in $\mathscr{D}_{p}(E)$ we have weighted strict exact triangles
\begin{equation}\label{eq:exact-tr}\hat{S}_{x_{m-i}}\otimes HF(\hat{S}_{x_{m-i}}, L_{i})\longrightarrow L_{i-1}\longrightarrow L_{i}~.~
\end{equation}
Here the Floer homology group $HF(\hat{S}_{x_{m-i}},L_{i})$ is a persistence module and the tensor product is a tensor product in the persistence realm. Moreover, to avoid complicating the notation, we identify the Lagrangians $L_{i}$ with their associated filtered Yoneda modules. 

The  triangle (\ref{eq:exact-tr}) is a persistence refinement of  Seidel's
classical Dehn-twist exact triangle.  This refinement is a strict exact triangle in the terminology of TPC's, of a weight denoted by $\epsilon_{i}$ (see \cite{BCZ:tpc}) and, additionally, $\epsilon_{i}$ depends on the size of the support of the Dehn twist $\tau_{\hat{S}_{i}}$, and can thus be made arbitrarily small.  This  feature of (\ref{eq:exact-tr})  can be seen most easily through the cobordism approach to the Dehn twist  due to Mak-Wu \cite{Mak-Wu:Dehn-twist} which, combined with the Lagrangian cobordism results in \cite{Bi-Co-Sh:LagrSh}, provides an upper bound for the weight of the triangle in terms of the shadow of the cobordism constructed in \cite{Mak-Wu:Dehn-twist}. This shadow can be made as close to $0$ as needed by diminishing the  neighbourhood of $\hat{S}_{i}$ that contains the support of $\tau_{\hat{S}_{i}}$.

To proceed, we pick the Dehn twists $\tau_{\hat{S}_{i}}$ with
sufficiently small support such that the weights $\epsilon_{i}$ of the triangles
(\ref{eq:exact-tr}) sum up to less than $c$.  We now notice that the persistence module $HF(\hat{S}_{x_{m-i}}, L_{i})$ can be viewed as the homology of a direct sum  $S$ of (translates of) elementary filtered complexes
$E_{2}(a, b)=\k(a, b:  da=0, db=a)$ and $E_{1}(c)=\k(c : dc=0)$ with the generators $x$ in both cases filtered with values $v(x)\in \R$, and $v(b)> v(a)$. The generators have degrees $0$ for $a$ and $c$, and $1$ for $b$.
The first type of filtered module can be written as a cone
$$E_{2}(a,b)=Cone \left( E_{1}(b)[-1]\to  E_{1}(a) \right)$$ over the obvious isomorphism, $b\to a$. Making use of this decomposition of $HF(\hat{S}_{x_{m-i}}, L_{i})$, the triangle (\ref{eq:exact-tr})
can be refined to a sequence of exact triangles in $(\mathscr{D}_{p}(E))^{0}$ of the form
\begin{equation}\label{eq:tr-ref}
\Delta_{i,j} \ \ \ : \ \ \ \hat{S}_{x_{i}}\otimes  E_{1}(d_{j})\to L_{i,j}\to L_{i,j+1}
\end{equation}
for $0\leq j\leq m(i)$
where $L_{i,0}=L_{i}$ and $d_{j}$ is one of the generators of type $a,b, c$ that appear in the sum $S$, followed by a last strict exact triangle of weight $\epsilon_{i}$ of
 the form 
 $$0\longrightarrow L_{i,m(i)}\longrightarrow L_{i+1}~.~$$
 This refinement is an immediate consequence of the octahedral axiom.

Notice that $\hat{S}_{x_{i}}\otimes E_{1}(x)$ is isomorphic to
$\Sigma^{v(x)}\hat{S}_{x_{i}}$. Therefore, by splicing together
all these exact triangles, starting with (\ref{eq:first-ex}) and following in order with $\Delta_{0, j}$, $0\leq  j\leq m(0)$ followed by   $\Delta_{1,j}$, \ $0\leq  j \leq m(1)$ and so forth, we obtain a sequence of exact triangles of the form required and of total weight not more than $3c$ which concludes the proof of the Lemma.
\end{proof}

\subsubsection{Conclusion of the proof.}\label{subsec:general} 
Theorem \ref{thm:nearby} is formulated in terms of the unit disk
bundle $D^{\ast}N$ with respect to some metric $\mathtt{g}$ but 
Proposition \ref{prop:estimates-in-r} is established for a disk bundle
$D^{\ast}_{r}N$ with $r$ small and with respect to a metric 
$\overline{\mathtt{g}}$ that is conformally equivalent to $\mathtt{g}$, as at the end of \S\ref{subsec:Lef-Dehn}.
 
To show Theorem \ref{thm:nearby} it suffices to show that for {\em any $r$}, and {\em any} metric $\mathtt{g}$ and any $L\subset D^{\ast}_{r}(N)$ we have 
 $$d_{int}(L, \Ob \langle \mathcal{F}_{\epsilon}
 \rangle^{\Delta})) < \epsilon+c_{\epsilon}\nu(p)$$ where $\mathcal{F}_{\epsilon}$ is a finite family of (modules of) fibers  depending on $\epsilon$ and on $r$. The interleaving distance $d_{int}(-)$ is defined on the 
 triangulated persistence category $\msc_{p}(r)$ for the perturbation parameter $p$ such that $\nu(p)$ is small enough. Moreover, these perturbation parameters, and the respective TPCs, are supposed to fit into a coherent family of TPCs as in Definition \ref{d:approx-sys} and 
 the assumption ($\ast$) from Theorem \ref{thm:nearby} should be satisfied.
  
  \
  
 We fix $\epsilon > 0$ and the metric $\mathtt{g}$. We keep the constant $0<\delta < 1$ as in the previous subsections.  
 We intend to use  Proposition \ref{prop:estimates-in-r}. We need to address the following points: 
 \begin{itemize}
 \item[i.] The proposition applies to the metric $\bar{\mathtt{g}}=\alpha \mathtt{g}$ with $\alpha:N\to [1,\infty)$ (chosen such  that we have $||d\varphi||\leq 1$, see Remark \ref{rem:conf_metric}) and not directly to $\mathtt{g}$.
 \item[ii.] The proposition applies to only values of $r$ that are sufficiently small.
 \end{itemize}
 
 For i. we denote by $|| - ||_{\mathtt{g}}$ and respectively by
 $||-||_{\bar{\mathtt{g}}}$ the norms with respect to the two metrics and 
 notice that for $a\in T^{\ast}(N)$ we have $||a||_{\bar{\mathtt{g}}}=\frac{1}{\sqrt{\alpha}}||a||_{\mathtt{g}}$. As a result $D^{\ast}_{r,\mathtt{g}}(N)\subset D^{\ast}_{r, \bar{\mathtt{g}}}(N)$ (where we add the metrics in the notation for the respective disk bundles).  Therefore,  the statement for the metric $\bar{\mathtt{g}}$ implies automatically  the corresponding result for the metric $\mathtt{g}$. 
 
 \
 
 It remains to discuss  point ii.  We want to show the statement for an arbitrary  value $R>0$. We pick $K>0$ such that  $48 K R/\delta \leq \epsilon$ and a Morse function $\varphi=\varphi_{N,K,\delta}$, as in Proposition \ref{prop:Morse}.  This determines the family $\mathcal{F}_{\epsilon}$ of fibers $F_{x_{i}}$, one for each critical point $x_{i}$ of $\varphi$.  We also consider the metric $\bar{\mathtt{g}}=\alpha \mathtt{g}$ obtained from $\mathtt{g}$ by rescaling as in Remark \ref{rem:conf_metric} with $\alpha :N\to [0,\infty)$ such that with respect to $\bar{\mathtt{g}}$ we have $||d\varphi||\leq 1$.

In case $R$ is small enough, the statement of Proposition  \ref{prop:estimates-in-r} directly applies and there is nothing further to prove. Assuming this is not the case, consider a sufficiently small $r$ such that the estimate in the statement of Proposition  \ref{prop:estimates-in-r}
\begin{equation}\label{eq:ineq3}d_{int}(L, \ \Ob \langle \{F_{x_{0}},\ldots, F_{x_{m}}\} \rangle^{\Delta} ) \leq 48 Kr/\delta 
\leq \epsilon \frac{r}{R}\end{equation}
is valid in $\msc_{p} (r)$ - the TPC associated with the exact, marked Lagrangians in $D^{\ast}_{r, \bar{\mathtt{g}}}(N)$ with respect
to perturbation data $p$ with $\nu(p)$ sufficiently small.
 
We consider the following rescaling map $\psi_{R,r}: T^{\ast} N\to T^{\ast} N$, $\psi_{R,r}(q,p)=(q, \frac{r}{R}p)$.
 Obviously, $\psi_{R,r}$ maps symplectomorphically $(D^{\ast}_{R,\bar{\mathtt{g}}}(N), \frac{r}{R}\omega)$ to $(D^{\ast}_{r,\bar{\mathtt{g}}}(N), \omega)$ (where $\omega$ is the standard form).  The map $\psi_{R,r}$  preserves the fibers $F_{x_{i}}$.  Using the rescaling $\psi_{R,r}$, the estimate (\ref{eq:ineq3})  applies as well 
to $(D^{\ast}_{R,\bar{\mathtt{g}}}(N), \frac{r}{R}\omega)$.  We emphasize that  this unit disk bundle is viewed as symplectic manifold with the non-canonical symplectic form $\frac{r}{R}\omega$. There is a corresponding TPC that we will denote by $\msc_{\bar{p}} (R; \frac{r}{R}\omega)$ where the perturbation data $\bar{p}$ is also obtained by pull-back through the rescaling map. In this TPC we have the inequality: 
\begin{equation}\label{eq:ineq-res} d_{int}(L, \ \Ob \langle \{ F_{x_{0}},\ldots, F_{x_{m}}\}\rangle^{\Delta}  ) \leq  \epsilon \frac{r}{R}~.~
\end{equation} 
The TPC that interests us is $\msc_{q}(R)$, associated with the symplectic manifold $(D^{\ast}_{R,\bar{\mathtt{g}}}(N),\omega)$ and with  perturbation data $q$ such that $\nu(q)$ is sufficiently small.
 It is easy to see that there is an isomorphism of categories (but not of TPC's):
 $$\Phi :   \msc_{\bar{p}} \left(R; \frac{r}{R}\omega\right) \longrightarrow \msc_{\bar{p}}(R)$$
 such that, on geometric objects, $\Phi $ rescales the relevant primitives,   $$\Phi ( (L, f_{L})) = \left(L, \frac{R}{r} f_{L}\right)$$ and, for morphisms, it identifies $\hom^{\alpha}_{\msc_{\bar{p}} (R; \frac{r}{R}\omega)}(X,Y)$ with 
 $\hom^{\alpha\frac{R}{r}}_{\msc_{\bar{p}}(R)}(\Phi (X),\Phi (Y))$.
 In brief, $\Phi$ preserves all algebraic structures except that it rescales all action values and shifts by $\frac{R}{r}$.  Consequently, the interleaving distance is also rescaled by $\frac{R}{r}$. Hence, the inequality (\ref{eq:ineq-res}) is also rescaled by  the same constant when passing to $\msc_{\bar{p}}(R)$. Thus we get:
 $$d_{int}(L,\ \Ob \langle \{ F_{x_{0}},\ldots, F_{x_{m}}\}\rangle ^{\Delta } ) \leq \epsilon $$ in $\msc_{\bar{p}}(R)$.
 We now consider the system of TPCs $\msc_{\bar{p}}(R)$ for varying
 $p$ and notice that, by pushing-forward also the comparison 
 functors from $\widehat{\msc}(D^{\ast}_{r,\bar{\mathtt{g}}}(N))$,
 these categories fit into a system $\widehat{\msc}(D^{\ast}_{R}N)$
 such that property $(\ast)$ is satisfied relative to the same family
 $\mathcal{F}_{\epsilon}$. To conclude, the claim of Theorem \ref{thm:nearby} is satisfied in $\widehat{\msc}(D^{\ast}_{R}N)$
 for $R=1$.
 
 \begin{rem}\label{rem:loose} The inequalities that we obtain in the proof
 of this theorem are not optimal in the sense that we do not try to minimize the number
 of fibers $F_{x_{0}}, \ldots, F_{x_{m}}$ in the argument and we do not try to optimize the other choices in the construction.  As a result,  at the end of the proof we obtain an inequality claiming that, with our choices and $\nu(p)$ small enough, the relevant interleaving distance is not larger than $\epsilon$, while with optimal choices, the expected inequality would be that the interleaving distance has $\epsilon +c_{\epsilon}\nu(p)$ as upper bound (for some universal constant $c_{\epsilon}$ and for $\nu(p)$ small enough). 
 \end{rem}
 
 \subsection{Theorem \ref{thm:nearby} implies Theorem \ref{thmmain1} i. and more}\label{subsec:nearby-TPC}
 
 Theorem \ref{thmmain1} i. claims that the (pseudo)-metric space $(\lag^{(ex)}(D^{\ast}N), d_{\gamma})$ is TPC-approximable in the sense of  Definition \ref{def:TPC-approx}.  Here $d_{\gamma}$ is the spectral metric whose 
 definition is recalled below, in Remark \ref{rem:spectral_d}.
 
 \
 
To show this statement we need to show that for any $\epsilon>0$ there exists $\epsilon$-TPC-approximating
data $\Phi=\{\Phi\}_{\eta}$, $\F=\{\F_{\epsilon,\eta}\}$ as in Definition \ref{def:TPC-approx}. 
By Remark \ref{rem:def-TPC-approx} b it is enough to show this for $\eta < \epsilon$ sufficiently small.
Of course, we will
 deduce the existence of this data out of Theorem \ref{thm:nearby}. 
 We will actually produce two choices of such data, each with its own advantages. 
  
 \begin{rem}\label{rem:spectral_d}
 For completeness we provide a simple description of the spectral pseudo-metric $d_{\gamma}$ defined on 
 $\mathcal{L}ag^{(ex)}(D^{\ast}N)$. For $L, L'\in \mathcal{L}ag^{(ex)}(D^{\ast}N)$ we consider the persistence
module $HF(L,L'; H^{L,L'})$ where $H^{L,L'}:[0,1]\times D^{\ast} N \to \R$ is an  admissible Hamiltonian function in the sense of standard Floer theory (see also \S \ref{sbsb:fil-ex}). We define 
$$v (L,L';  H^{L,L'})=b_{max}(HF(L,L'; H^{L,L'}))-b_{min}(HF(L,L'; H^{L,L'}))~.~$$ Here, if $\mathcal{M}$ is a persistence module with a finite number of bars  such that $\mathcal{M}=\oplus_{i\in I} [a_{i}, b_{i}) \oplus_{j\in J} [c_{j}, \infty)$ (with $a_{i}, b_{i}, c_{j}\in \R$, $a_{i}<b_{i}$) we put $b_{max}(\mathcal{M})= \max_{j\in J} c_{j}$ and $b_{min}(\mathcal{M})=\min_{j\in J} c_{j}$.
The definition of the spectral metric is:
$$d_{\gamma}(L,L')=\limsup_{|| H^{L,L'} ||_{C^{0}}\to 0}  \ v(L,L'; H^{L,L'}) ~.~$$
It is non-trivial but well-known that this is finite and that it satisfies the properties of a pseudo-metric.
The objects in $\mathcal{L}ag^{(ex)}(D^{\ast}N)$ are marked Lagrangians, they come with fixed primitives and gradings. The spectral metric does not ``see'' the choices of grading  and primitive in the sense that, with the notation of the paper, $d_{\gamma}(T^{r}\Sigma^{s}L,L')=d_{\gamma}(L,L')$ for all $r\in \Z$, $s\in \R$
 (here $T$ indicates a shift in grading and $\Sigma$ one in ``action''). Thus, this pseudo metric descends to the actual space of exact, closed Lagrangian submanifolds in $D^{\ast}N$ without any additional structure and, remarkably, 
 it is non-degenerate on this space. 
 \end{rem} 
 
 \subsubsection{Local $\epsilon$-TPC-approximating data.}\label{subsubsec:local_approx_d}

Our aim here is to use Theorem \ref{thm:nearby} to show that we can obtain approximating data $(\Phi,\F)$
for $(\lag^{(ex)}(D^{\ast}N), d_{\gamma})$ in the sense of Definition \ref{def:TPC-approx} where the categories $\mathscr{Y}_{\eta}$ are of the form $\msc_{p}$ and thus complete the proof of Theorem \ref{thmmain1} i.

\

There are a few differences between the type of approximability used in the  statement of Theorem \ref{thm:nearby} (namely Definition \ref{d:approx-sys}) and that of 
the Definition \ref{def:TPC-approx} used in Theorem \ref{thmmain1}. 

 First,  Theorem \ref{thm:nearby} makes use of the interleaving pseudo-metric on $\msc_{p}^{\epsilon}$ and not of the shift invariant pseudo metric. However, by Remark \ref{rem:relation-TPCapp},  the inequality claimed in Theorem \ref{thm:nearby}  remains true with respect to the shift invariant interleaving pseudo-metric $\bar{d}^{\msc_{p}^{\epsilon}}_{int}$, as defined in equation (\ref{eq:sdint}). 

A second difference is that Definition \ref{d:approx-sys} contains an inequality 
of the form $ \dint (-, -) < \epsilon + c_{\epsilon}\nu(p)$ while in Definition \ref{def:TPC-approx} the  term $c_{\epsilon}\nu(p)$ does not appear.  To address this point, pick some positive $\epsilon' <\epsilon$ and recall from Remark \ref{rem:dep_on_e} that the system of TPCs, $\widehat{\msc}
(D^{\ast}N) =\{\msc_{p} \}$ in Theorem \ref{thm:nearby} can be defined for this $\epsilon'$
 through  the construction of the Lefschetz fibration  $\bar{h}:E\to \mathbb{C}$, as in \S\ref{subsec:Lef-Dehn} (see also Remark \ref{rem:compare_data}). To avoid confusion we will include  $\epsilon'$ in the notation, and thus write $\msc_{p}^{\epsilon'}$ and $\widehat{\msc}^{\epsilon'}(D^{\ast}N)$. We  denote the action of the Yoneda embedding on objects by 
$$\mathcal{Y}_{\epsilon',p}:\mathcal{L}ag^{(ex)}(D^{\ast}N) \to \Ob(\msc_{p}^{\epsilon'})~.~$$
Theorem \ref{thm:nearby} shows that for this $\epsilon' $,  there is a finite family of fibers $\{F_{x_{1}},\ldots F_{x_{l}}\}$ such that for any 
$\delta$ sufficiently small the set $\mathcal{L}ag^{(ex)}(D^{\ast}N)$ is $\epsilon'$-approximable,
in the sense of Definition \ref{d:approx-sys},  by the family of modules $\mathscr{F}_{x_{i}}$ corresponding to the $F_{x_{i}}$ in each triangulated persistence category $\msc_{p}^{\epsilon'}$ with $\nu(p) \leq \delta$.  In this case, if we put:
$$\Phi_{\epsilon,\eta}= \mathcal{Y}_{\epsilon', p_{\eta}}$$ for $p_{\eta}$ depending on $\eta$ and picked such that $\nu(p_{\eta})<\eta/4$,  and we denote $\F_{\epsilon,\eta}=\{F_{x_{1}},\ldots, F_{x_{l}}\}$, then   the system $$(\Phi, \F)=(\Phi_{\epsilon,\eta}, \F_{\epsilon,\eta})$$ satisfies
the second point in Definition \ref{def:TPC-approx} relative to $\epsilon$ for $\eta$ small enough (such that $\eta/4 < (\epsilon -\epsilon')/c_{\epsilon'}$). 

A third and last difference is that the first point in Definition \ref{def:TPC-approx} is not 
present in Definition \ref{d:approx-sys}.  This has to do with the comparison between the spectral metric on the domain of the maps $\Phi$ and the interleaving pseudo-metric defined on the target of
these maps.
Thus, to show that $(\Phi$, $\F)$ as above are a choice of TPC $\epsilon$-approximating data for $$(\lag^{(ex)}(D^{\ast}N), d_{\gamma})\ , \ $$ in the terminology from \S \ref{subsubsec:approx_TPC}
it remains to show that  $\mathcal{Y}_{\epsilon', p}$ also satisfies the first point in Definition \ref{def:TPC-approx}, namely that $\mathcal{Y}_{\epsilon', p}$ is  a $(A,\eta)$-quasi-isometric embedding  where 
the pseudo-metric on $\Ob (\msc_{p}^{\epsilon'})$ is the interleaving metric, the pseudo-metric
on the domain is $d_{\gamma}$, and  the constant
$A>0$ can be picked independent of all other choices. 
 
 \
 
 We will see that the statement is true for $A=2$. We start by revisiting the definition of the interleaving distance from equation 
 (\ref{eq:d-int}). An equivalent definition is 
 \begin{equation*} \label{eq:d-int2}
  \begin{aligned}
d_{int}(X,Y)=\inf\{r\geq 0 \ | \ \exists\ \phi\in\hom^{r}(X,Y), \psi\in \hom^{r}(Y,X) \\  \     \mathrm{\ \ such\ that}\ \psi\circ \phi =i_{0,2r}(id_{X}), \  \phi \circ \psi = i_{0,2r}(id_{Y})\} \ , 
\end{aligned}
\end{equation*}
 where we recall that
 $X,Y\in \Ob(\C)$, $\C$ is a persistence category (see \S\ref{sb:pc}) 
 and $$i_{\alpha,\beta}: \hom^{\alpha}(X,Y)\to \hom^{\beta}(X,Y)$$ are the persistence structural maps. An alternative definition
of an interleaving type pseudo-metric is the following:
 
 \begin{equation} \label{eq:d-int3}
  \begin{aligned}
D_{int}(X,Y)=\inf\{a+b \ | \ \exists\ \phi\in\hom^{a}(X,Y), a \geq 0, \psi\in \hom^{b}(Y,X), b\geq 0 \\     \mathrm{such\ that}\ \psi\circ \phi =i_{0,a+b}(id_{X}), \  \phi \circ \psi = i_{0,a+b}(id_{Y})\}
\end{aligned}
\end{equation} 

It is immediate to see that:

\begin{equation}\label{eq:ineq_var_inter}
\frac{1}{2} D_{int}(X,Y)\leq d_{int}(X,Y)\leq D_{int}(X,Y)~.~
\end{equation}

It is obviously also possible to stabilize the distance $D_{int}$, just as in 
formula (\ref{eq:sdint}), thus putting
\begin{equation*}\label{eq:sdint2}
\bar{D}_{int} (X,Y)=\inf_{r,s\in \R} D_{int} (\Sigma^{r}X,\Sigma^{s}Y)
\end{equation*}
and we again have $\frac{1}{2} \bar{D}_{int} \leq \bar{d}_{int}\leq    \bar{D}_{int}$.
The corresponding interleaving type pseudo-metric, $\bar{D}^{\msc_{p}^{\epsilon}}$, is defined on each of the  categories $\msc_{p}^{\epsilon}$ and, for all $L,L'\in\mathcal{L}ag^{(ex)}(D^{\ast}N)$ we can consider the limit:
 $$\hat{D}_{int}(L,L')=\limsup_{\nu(p)\to 0} \bar{D}_{int}^{\msc_{p}^{\epsilon}}(L,L')~.~$$
 where the limit is taken in the system of categories $\hat{\msc}_{p}^{\epsilon}(D^{\ast}N)$. This is again a pseudo-metric defined on
 $\lag^{(ex)}(D^{\ast}N)$. We claim that we have:
 \begin{equation}\label{eq:equality_int}
 \hat{D}_{int}(L,L')= d_{\gamma}(L,L') \ \ \forall \ L,L'\in \lag^{(ex)}(D^{\ast}N)
 \end{equation}
 
 This identity is well-known in the subject, the argument can be 
 found in 3.4.1.2.(i) \  in  \cite{BCZ:tpc} (see also Remark 3.26 there).
 
 Taking into account $\nu(p_{\eta})< \eta/4$ and the definition of $\nu(p)$ from 
 (\ref{eq:nu-fuk}) it is immediate to see that (\ref{eq:equality_int}) implies that the maps $\mathcal{Y}_{\epsilon', p_{\eta}}$ are $(2,\eta)$-quasi-isometric embeddings when $\eta$ is sufficiently small and thus finishes the proof of Theorem \ref{thmmain1} i.

 \begin{rem}\label{rem:div_int_metr}
 It is clear from the argument above that the interleaving type pseudo-metric $D_{int}$ from Definition \ref{eq:d-int3} is more directly tied to the spectral metric compared to the pseudo-metric $d_{int}$ from Definition \ref{eq:d-int}. On the other hand, $d_{int}$, when applied to the homotopy category of filtered chain complexes, coincides with the bottleneck distance from persistence theory - as can be seen through the isometry theorem as in, for instance, \S 2.2 \cite{PRSZ:top-pers} - and thus is, in a way, the more ``standard'' notion.
 \end{rem}

 \subsubsection{Ambient $\epsilon$-TPC-approximating data.} \label{subsubsec:amb_approx_d} We will see here that the
 proof of Theorem \ref{thm:nearby} also provides a different sort of  $\epsilon$-approximating data  (in the terminology from \S \ref{subsubsec:approx_TPC}) for the 
 space $(\lag^{(ex)}(D^{\ast}_{r}N), d_{\gamma})$. It is important to note from the outset
 that, by contrast to \S\ref{subsubsec:local_approx_d}, we obtain {\em retract} approximating data in this case - see Corollary \ref{thm:nearby_far}.
  
 For  a fixed  $\epsilon >0$ consider some  positive $\epsilon'<\epsilon$.
 Recall from (\ref{eq:amb_Fuk}) the persistence derived Fukaya categories
 $$\mathscr{D}_{p}(E)=PD(\fuk(\mathcal{L}ag^{(ex)}(E), p))$$
 as well as the Lagrangian spheres $\hat{S}_{x_{i}}\in \mathcal{L}ag^{(ex)}(E)$ that appear
 in Lemma \ref{lem:in_large} that can be constructed for $\epsilon' $ (for some small $r$). 
 The corresponding Yoneda embeddings yield maps:
 $$\mathcal{Y}'_{\epsilon',p}:\mathcal{L}ag^{(ex)}(D^{\ast}_{r}N) \to \Ob(\mathscr{D}_{p}(E))~.~$$
 We define  the family $\F'_{\epsilon}=\{\dots,\hat{S}_{x_{i}},\ldots\}$ where here $\hat{S}_{x_{i}}$ represents
 the Yoneda module of the respective Lagrangian sphere (these are preserved by the comparison 
 functors $\mathcal{H}_{p,q}$, $p\preceq q$). We put $\Phi'_{\epsilon,\eta}=\mathcal{Y}'_{\epsilon',p_{\eta}}$
 where $\nu(p_{\eta})\leq \eta/4$, just as in \S\ref{subsubsec:local_approx_d}. Here the system 
 of perturbations $p$ is defined relative to  $\lag^{(ex)}(E)$ and thus $\nu(-)$ is calculated for Hamiltonians
 defined over $E$. The proof that $\Phi'_{\eta}$ is a $(2,\eta)$-quasi-isometric  embedding is the same 
 as in \S\ref{subsubsec:local_approx_d} because the identity (\ref{eq:equality_int}) remains true even if the interleaving type
 distance $\hat{D}_{int}$ is computed in the larger category $\mathscr{D}_{p}(E)$ instead of $\msc_{p}(r)$ (this is due again
 to the argument in 3.4.1.2 (i) in \cite{BCZ:tpc}).
 
 It remains to show that $\F'_{\epsilon}$ is a retract $\epsilon$-approximating family for $\lag^{(ex)}(D^{\ast}_{r}N)$ in each 
 $\mathscr{D}_{p_{\eta}}(E)$ for $\eta$ sufficiently small.
 
 To show this we proceed as in the proof of Proposition \ref{prop:estimates-in-r}. From the proof of this proposition we 
 know that there is a sequence (\ref{eq:iterated_cones2}) of exact triangles in $\mathscr{D}_{p}(E)$:
 \begin{equation}\label{eq:iterated_cones3}
 \Delta'_{i} \ : \ Z'_{i}\longrightarrow X'_{i}\longrightarrow X'_{i+1} \ , \ 0\leq i < m
\end{equation}
with $Z'_{j}$ of the form $\Sigma^{l(j)}\hat{S}_{x_{s(j)}}$ or $=0$,
 with each $\Delta'_{i}$  exact in
 $(\mathscr{D}_{p}(E))^{0}$, with $X'_{0}$ the Yoneda module of a Lagrangian disjoint from $D^{\ast}_{r}N$, and 
 such that  there is a $6c$-isomorphism $$f: X'_{m}\to L~.~$$
 
 We represent all iterated cones in (\ref{eq:iterated_cones3}) by attachment of cones in the pre-triangulated
 category of filtered $A_{\infty}$-modules (this is possible because the sequence (\ref{eq:iterated_cones3})  is in $(\mathscr{D}_{p}(E))^{0}$ ) and we consider the obvious compositions $u_{j}: X_{0}'\to X'_{j}$. 
 The sequence (\ref{eq:iterated_cones3}) induces a new iterated sequence of cones (in the same category):
   \begin{equation}\label{eq:iterated_cones4}
 \tilde{\Delta}_{i} \ : \ Z'_{i}\longrightarrow \tilde{X}_{i}\longrightarrow \tilde{X}_{i+1} \ , \ 0\leq i < m
\end{equation}
with $\tilde{X}_{i}= \mathrm{Cone}\ (u_{i})$. We first remark that because $X_{0}'$ is a Yoneda module all $\tilde{X}_{i}$'s are objects of $\mathscr{D}_{p}(E)$. 
We now consider the cone attachment $$X'_{0}\stackrel{u_{m}}{\longrightarrow} X'_{m} \longrightarrow \tilde{X}_{m}~.~$$ Because $X'_{0}$ is the Yoneda module of a Lagrangian that does not intersect $D^{\ast}_{r}N$ it follows
that the composition $f\circ u_{m}$ is $0$-chain homotopic to the null map. As a result, the map $f$ induces
another $6c$-isomorphism $\tilde{f}:\tilde{X}_{m}\to L\oplus TX'_{0}$. This means that 
$$d_{\text{r-int}}^{\mathscr{D}_{p}(E)}(L,\tilde{X}_{m})\leq 6c$$ and, moreover,
in $(\mathscr{D}_{p}(E))^{0}$ we have that $\tilde{X}_{0}\cong 0$. 

As a result, similarly to Proposition \ref{prop:estimates-in-r}, we obtain:
\begin{cor} \label{cor:estimate-in-r-E}
With the notation above, for $r$ sufficiently small and for a choice of perturbations $p$ with $\nu(p)$ sufficiently small, any exact, marked Lagrangian $L\subset D^{\ast}_{r}(N)$ satisfies in $\mathscr{D}_{p}(E))$:
$$d_{\text{r-int}}(\ L\ , \Ob\  \langle\{\hat{S}_{x_{0}},\ldots, \hat{S}_{x_{m}}\} \rangle ^{\Delta}\ ) \leq 48 Kr/\delta~.~$$  
\end{cor}

The arguments in the proof of Theorem \ref{thm:nearby} that are found in \S\ref{subsec:general} can then be pursued
in this somewhat different setting to get rid of the dependence on $r, K,\delta$, and they lead to a conclusion that parallels
the similar statement for $(\Phi, \F)$:

\begin{cor} \label{thm:nearby_far} 
In the setting above and for a fixed $\epsilon>0$ we have: 

a. There exists a Lefschetz fibration $\bar{h}:E\to \mathbb{C}$ as before such that the family
$\F'_{\epsilon}$ of the Yoneda modules of the Lagrangian spheres $\{\hat{S}_{x_{0}}, \ldots, \hat{S}_{x_{m}}\}$ is a retract $\epsilon$-approximating family for $(\mathcal{L}ag^{(ex)}(D^{\ast}_{r}N), d_{\gamma})$ in $\mathscr{D}_{p}(E)$ for each $p$ with $\nu(p)$ sufficiently small. Moreover, $(\Phi',\F'_{\epsilon})$ is TPC retract $\epsilon$-approximating for  $(\mathcal{L}ag^{(ex)}(D^{\ast}_{r}N), d_{\gamma})$. 

b. For each $L\in\lag^{(ex)}(D^{\ast}_{r}N)$ there exists
$X_{L}\subset \lag^{(ex)}(E)$ disjoint from $D^{\ast}_{r}N$ such that 
$$\dint (L\oplus X_{L}, \Ob\  \langle\{\hat{S}_{x_{0}},\ldots, \hat{S}_{x_{m}}\} \rangle ^{\Delta}) <\epsilon + c_{\epsilon}\nu(p)$$ in the category $\mathscr{D}_{p}(E)$ and for $\nu(p)$ small enough.
\end{cor}

While the result above is established for small $r$, given that we can construct  the
Lefschetz fibration as appropriate, the same result can be inferred for all values of $r$.
 
 \begin{rem} \label{rem:compare_data} a. The advantages of the approximating data $(\Phi, \F_{\epsilon})$ from \S\ref{subsubsec:local_approx_d} are that, first, it is more intrinsic relative to the cotangent disk bundle $D^{\ast}N$ and, second, it does not involve retracts. To expand on the first point, 
 fix some $\epsilon>0$. The category $\msc_{p}$, and the object level part of the Yoneda embedding $\mathcal{Y}_{p}:\mathcal{L}ag^{(ex)}(D^{\ast}N) \to \Ob(\msc_{p})$ that form the family $\Phi$ are constructed - as in \S\ref{subsec:back} - for some admissible choice of perturbation data  $p$, 
 independently of $\epsilon$ and any other auxiliary construction. The dependence on $\epsilon$ only appears when we fix
 a set of fibers $\F_{\epsilon}$ and require that the associated modules be preserved  by the
 comparison maps $\mathcal{H}_{p,q}$, $p\preceq q$ - this is our assumption ($\ast$). 
 The proof of Theorem \ref{thm:nearby} shows that it is possible
 to find comparison maps preserving the finite family of modules associated to $\F_{\epsilon}$ but it is 
 not clear whether there is a choice of comparison maps that preserves each fiber module $F_{x}$, $\forall x\in N$  (see also Remark \ref{rem:dep_on_e}).  
 
 b. The approximating data $(\Phi',\F'_{\epsilon})$  described in the current subsection 
 is highly dependent on the choice of Lefschetz fibration $\bar{h}:E\to \mathbb{C}$ and also is only retract approximating.  On the other hand, the category
 $$\mathscr{D}_{p}(E)= PD \fuk(\lag^{(ex)}(E),p)$$ has the property that all  its objects can be written as finite
 iterated cones of Yoneda modules of closed Lagrangians, and as a result each $\hom_{\mathscr{D}_{p}(E)}(X,Y)$ is a 
 finite type persistence module (in the sense that the associated bar-code contains only a finite number of bars) for all $X,Y\in \Ob(\mathscr{D}_{p}(E))$. Moreover, the retract $\epsilon$-approximating family $\F'_{\epsilon}$ consists of the Yoneda modules of the Lagrangian spheres $ \hat{S}_{x_{i}}\subset E$.  
 By contrast, this may not be true for the elements of the family $\F_{\epsilon}$ which are defined as modules over  the category  $\fuk (\lag^{(ex)}(D^{\ast}N),p)$ corresponding to fibers. Moreover, for two fibers $F_{x_{1}}, F_{x_{2}}$,  $\hom_{\msc_{p}}(F_{x_{1}}, F_{x_{2}})$ is not identified with some variant of Floer
 homology $HF(F_{x_{1}}, F_{x_{2}})$ and, indeed, it is not necessarily of finite type.  \end{rem}

% !TEX root = approx8.tex

\newpage

\section{Abouzaid's splitting principle in the filtered case and approximability} \label{sec:split-app}
In this section we prove the second and third points in Theorem \ref{thmmain1}. To this aim, we establish a persistence version of the split-generation criterion due to Abouzaid \cite{Ab:geom-crit} (for the closed monotone case see \cite{she:fanofuk}). The key point is that the energy cost of reaching the unit in the quantum homology of the ambient symplectic manifold through the open-closed map can be directly interpreted in terms of  retract  approximability, as defined in \S\ref{subsubsec:approx_TPC}.  We provide two examples of computations, completing the proof of Theorem \ref{thmmain1}. 
Along the way,  we adjust to this persistence setting the various
ingredients in Abouzaid's result, namely the open-closed and closed open maps and their relation to a persistence version of Hochschild homology.

%We briefly recall the definition of the ambient quantum homology $QH(X)$ of $M$. Consider the monotone family $\lagmon$, a finite subfamily $\mathcal{B}$ {\color{purple} Not sure about the notation for the subfamily.}. Fix for the moment a perturbation datum $p\in \mathcal{P}$ and consider the monotone Fukaya category $\fuk(\lagmon;p)$.
%=======================================
%=======================================
%=======================================

\subsection{The results}
Let $(X,\omega)$ be a monotone symplectic manifold. We work in the setting outlined in \S\ref{sb:fuk-mon}. We recall that given a Novikov element $\mathbf{d} \in \Lambda_0$, then $\lagmon$ is the collection of closed, graded and monotone Lagrangians $L = (\overline{L}, a_L, s_L)$ with Maslov-$2$ disk count equal to $\mathbf{d}$. We fix a choice of $\mathbf{d}\in \Lambda_0$ for the rest of this chapter. Recall from \S\ref{sb:sys-fuk} that we associated with $\lagmon$ a homotopy system of $A_\infty$-categories (see \S\ref{sbsb:homotopy-sys} for the definition) $\widehat{\fuk}(\lagmon)$ indexed by the space $\mathcal{P}$ of perturbation data, such that each filtered and strictly unital $A_\infty$-category $\fuk(\lagmon;p)$, $p\in \mathcal{P}$, has $\lagmon$ as set of objects. We will often abbreviate and write the system as $\widehat{\mathcal{A}}:=\widehat{\fuk}(\lagmon)$ and each component as 
$$\mathcal{A}_p:= \fuk(\lagmon;p)~.~$$ We omit the manifold $X$ from the notation as it is also fixed for the remainder of the section.

A subset $\mathcal{B}\subset \lagmon$ induces a full $A_\infty$-subcategory $\mathcal{B}_p$ of $\mathcal{A}_p$ for any perturbation datum $p\in \mathcal{P}$. This  ultimately leads to a subsystem $\widehat{\mathcal{B}}$ of $\widehat{\mathcal{A}}$. %{\color{red} Moreover, $\mathcal{B}_p$ induces a filtered $\mathcal{A}_p$-bimodule $\overline{\mathcal{B}_p}$ as explained above in \S\ref{WRSG} {\huge to move somewhere else}}.

The following geometric criterion for retract approximability,
in the sense of Definition \ref{d:approx-sys-ret} with $\lagmon \subset \text{Obj}(\widehat{\mathcal{A}})$, is the main result 
of the section.

\begin{thm}\label{gio:wabomain}
	Let $\epsilon>0$. If there is a finite subset $\mathcal{F}_\epsilon\subset\lagmon$ such that the colimit persistence open-closed map $$\widehat{OC}\colon PHH_*(\widehat{\mathcal{F}_\epsilon})\to QH^{*+n}(X,\Lambda)_{\mathbf{d}}$$ satisfies $$i_{0,\epsilon}(u_{\mathbf{d}})\in \text{image}\left(PHH^\epsilon(\widehat{\mathcal{F}_\epsilon})\to QH^\epsilon(X,\Lambda)_{\mathbf{d}}\right)$$ where $u_\mathbf{d}\in QH^0(X,\Lambda)_\mathbf{d}$ is the projection of the quantum cohomology unit  $u\in QH(X,\Lambda)$ to $QH(X,\Lambda)_\mathbf{d}$ and $i_{0,\epsilon}$ is part of the structure maps of the quantum cohomology persistence module, then $\lagmon$ is $\frac{\epsilon}{2}$-retract-approximable by $\mathcal{F}_\epsilon$ in $PD(\widehat{\fuk}(\lagmon))$.
\end{thm}
\noindent The open closed map $OC$, the sense in which we take a colimit, and the various other notions appearing in the statement of Theorem \ref{gio:wabomain} and its proof, among them, the persistence Hochschild homology $PHH( - )$, will be made explicit below.
%One should think of it as the open closed map associated with the %Fukaya category with full accuracy (see discussion at the %beginning of \S\ref{s:sys-tpc}).
We now restate the points $(ii)$ and $(iii)$ of Theorem \ref{thmmain1} as corollaries of Theorem \ref{gio:wabomain}.
\begin{cor}\label{S2T2}
The classes $\lag^{\text{(mon,} \mathbf{0} \text{)}}(S^2)$ and $\lagwex(\mathbb{T}^2)$ are TPC retract approximable.
\end{cor}
\noindent The proof of Theorem \ref{gio:wabomain} is contained in \S\ref{gio:proofwabomain}. The proof of Corollary \ref{S2T2} occupies \S\ref{gio:proofS2T2}. The proof of the last two points of Theorem \ref{thmmain1} follows directly from this Corollary and is concluded in \S\ref{subsec:proof_ThmAii}.

%=======================================
%=======================================
%=======================================

\subsection{Proof of Theorem \ref{gio:wabomain}}\label{gio:proofwabomain}
We split the proof of Theorem \ref{gio:wabomain} into several steps. \S\ref{WRSG} contains a persistence version of Abouzaid's algebraic split-generation criterion. The main result, in Proposition \ref{gio:algfilabo}, shows that retract approximability 
is naturally present in that context. \S\ref{gio:ffuk} recalls some of the main technical ingredients in the definition of the filtered Fukaya categories $\fuk(\lagmon;p)$ following \cite{Amb:fil-fuk} in sufficient detail so that, in the following sections \ref{gio:coproduct}, \ref {AQH}, \ref{gio:oc} we are able to set-up in the persistence setting the main geometric constructions needed for the split-generation criterion. Finally, in \S\ref{gio:aboappr} we put the constructions together and prove Theorem \ref{gio:wabomain}.

\subsubsection{Retract approximability}\label{WRSG}
Let $\mathcal{A}$ be a filtered $A_{\infty}$-category over $\Lambda$, such as one of the filtered Fukaya categories $\A_{p}$ from the beginning of \S\ref{sec:split-app}, and
let $\mathcal{B}\subset \text{Obj}(\mathcal{A})$. We regard $\mathcal{B}$ as a full $A_\infty$-subcategory of $\mathcal{A}$. In contrast to our general conventions, we do not assume here that $\mathcal{B}$ is closed under shifts or translations. We define a filtered $(\mathcal{A},\mathcal{A})$-bimodule $\overline{\mathcal{B}}$, which will be called the $\mathcal{B}$-bar bimodule. For two objects, $A$ and $B$,  of $\mathcal{A}$, we set, with the notation in Appendix \ref{a:fil-ai},  $$\overline{\mathcal{B}}(A,B):=\mathcal{Y}^r(A)\otimes_\mathcal{B}\mathcal{Y}^l(B) = \bigoplus_{d\geq 0}\bigoplus_{L_0,\ldots, L_d\in  \text{Ob}(\mathcal{B})}\mathcal{A}(A,L_0,\ldots, L_d,B)~.~$$ 
Here $\mathcal{Y}^l$ and $\mathcal{Y}^r$ denote the left and right filtered Yoneda embeddings, defined in Appendix \ref{gio:fyon-def}, and we implicitly view $\mathcal{Y}^r(A)$ and $\mathcal{Y}^l(B)$ as $A_\infty$-modules over $\mathcal{B}$. The tensor product $\otimes_\mathcal{B}$ is the usual derived tensor product over $\mathcal{B}$. Note that $\mathcal{A}(A,B)$ is not a summand in $\overline{\mathcal{B}}(A,B)$. We will often write a pure tensor $\gamma_1\otimes a_1\otimes \cdots \otimes a_d\otimes \gamma_2 \in \mathcal{A}(A,L_0,\ldots, L_d,B)$ as $\vec{\gamma}$. The bar differential $\mu^{\text{bar}}_{0|1|0}$ (which we will also denote by $d_{\text{bar}}$) is defined via \begin{align*}\mu^{\text{bar}}_{0|1|0}(\vec{\gamma}):=& \sum_{i=0}^d\mu_{1+i}(\gamma_1,a_1,\ldots, a_i)\otimes a_{i+1}\otimes \cdots \otimes a_d\otimes \gamma_2\\ +& \sum_{i\leq j} \gamma_1\otimes a_1\otimes \cdots \otimes a_{i-1}\otimes \mu_{j-i}(a_i,\ldots, a_j)\otimes a_{j+1}\otimes \cdots\otimes a_d\otimes \gamma_2\\ +&\sum_{i=1}^{d+1} \gamma_1 \otimes a_1 \otimes \cdots \otimes a_{i-1}\otimes \mu_{d-i+2}(a_i,\dots, a_d,\gamma_2) \end{align*} while the higher operations $\mu^{\text{bar}}_{l|1|r}$ are defined via\begin{itemize}
	\item for $l\geq 1$ and $r=0$: \begin{align*}\mu^{\text{bar}}_{l|1|0}(x_1,\ldots, x_l, \vec{\gamma}) :=& \sum_{i=0}^{d}\mu_{l+i+1}(x_1,\ldots, x_l, \gamma_1,a_1,\ldots, a_i)\otimes a_{i+1}\otimes \cdots \otimes a_d\otimes \gamma_2 \end{align*}
	\item for $l=0$ and $r\geq 1$: \begin{align*} \mu^{\text{bar}}_{0|1|r}(\vec{\gamma},y_1,\ldots, y_r):= \sum_{i=1}^{d+1} \gamma_1\otimes a_1\otimes \cdots \otimes a_{i-1}\otimes \mu_{d+r-i+2}(a_i,\ldots, a_d,\gamma_2,y_1,\ldots, y_r) \end{align*}
	\item for $l,r\geq 1$: $\mu^{\text{bar}}_{l|1|r}=0$ \end{itemize}
It is straightforward to see that $\overline{\mathcal{B}}$ is filtered, since $\mathcal{A}$ is. 
There is an obvious length filtration \begin{equation}\label{eq:lengthfiltration}
F^N\overline{\mathcal{B}}(A,B)= \bigoplus_{d=0}^N\bigoplus_{L_0,\ldots, L_d\in \text{Ob}(\mathcal{B})}\mathcal{A}(A,L_0,\ldots, L_d,B), \quad N\geq 0,
\end{equation} which induces $A_\infty$-bimodules $F^N{\mathcal{B}}$. The $A_\infty$-structure maps on $\mathcal{A}$ induce a filtered morphism of $(\mathcal{A},\mathcal{A})$-bimodules $$\mu:=\mu^{\overline{\mathcal{B}}} \colon \overline{\mathcal{B}}\to \Delta_\mathcal{A}$$ in an obvious way, where $\Delta_\mathcal{A}$ stands for the diagonal bimodule of $\mathcal{A}$.

Let now $K\in \text{Obj}(\mathcal{A})$ and consider the filtered chain map 
\begin{equation}\label{eq:bimod_map1}
	\mu^{\overline{\mathcal{B}}} (K) :=\mu^{\overline{\mathcal{B}}}_{0|1|0}\colon \overline{\mathcal{B}}(K,K)\to \mathcal{A}(K,K) ~.~
\end{equation} 
Denote by $$[\mu^{\overline{\mathcal{B}}}(K)]\colon H(\overline{\mathcal{B}}(K,K))\to H(\mathcal{A}(K,K))$$ the morphism of persistence modules induced by $\mu^{\overline{\mathcal{B}}}(K).$

In the context above,  given an object $K$ of $\mathcal{A}$, we denote its strict unit by $e_K\in \mathcal{A}^{\leq 0 }(K,K)$ and by $[e_K]\in \hom_{PD(\mathcal{A})}^0(K,K)$ its homology class in the TPC $PD(\mathcal{A})$. 

Finally, before stating the main result, we need one more notion. 
Let $V$ and $W$ be persistence modules and $f\colon V\to W$ a persistence morphism. We will make use of a measurement that estimates the energy gap
separating an element in $W$ from the image of $f$. More precisely, for  $w\in W_{r}$, we put:
\begin{equation}\label{eq:en_gap}
	R(w,f):=\inf\left\{s\geq r \ \lvert \ \exists v\in V_s, \ f_s(v) = i_{r,s}^W(w)\right\}
\end{equation} where $i^W_{r,s}\colon W_r\to W_s$ denotes the persistence structure map for $W$. If $i_{r,s}^W(w)\notin \text{image}(f_s)$ for all $s\geq r$ we set $R(w,f)=\infty$. 

\ 

Below and in what follows we denote by $\mathcal{F}\mathbf{Ch}$ the filtered $dg$-category of filtered chain complexes over $\Lambda$, and by $\hfch$ its persistence homological category in cohomological degree $0$, which is a TPC (see \cite[Section 2.5.2]{BCZ:tpc}). We recall that in our conventions, a filtered chain complex over $\Lambda$ is a chain complex over $\Lambda$, filtered by an increasing sequence of $\Lambda_0$-subcomplexes indexed by the real line (see Appendix \ref{a:sb:fil-ai}).

\begin{prop}\label{gio:algfilabo}
	Let $K$ be an object of $\mathcal{A}$, $\mathcal{B}$ be a finite full $A_\infty$-subcategory of $\mathcal{A}$ and $\alpha\in \mathbb{R}_{\geq 0}$. The following statements are equivalent:
	\begin{enumerate}
		\item\label{(1)} The estimate $R([e_K],[\mu^{\overline{\mathcal{B}}}(K)])\leq \alpha$ holds.%{\color{red}, where $[e_K]\in \hom_{PD(\mathcal{A})}^0(K,K)$ is the class of the strict unit $e_K\in \mathcal{A}^{\leq 0}(K,K)$.}%The unit $[e_K]\in H(\mathcal{A}(K,K))$ is in the image of $[\mu^{\overline{\mathcal{B}}}(K)^{\leq \alpha}]$, 
		%\item The element $\eta_\alpha^K\in \text{hom}_{PD(\mathcal{A})^0}(\Sigma^\alpha K,K)$ lies in the image of the induced map $$\mu^{\overline{\mathcal{B}}}(K)^\alpha: H(\overline{\mathcal{B}}^{\leq \alpha}(K,K))\to \text{hom}_{PD(\mathcal{A})^0}(\Sigma^\alpha K,K),$$
		\item\label{(2)} The map $\mu^{\overline{\mathcal{B}}}(K)\in \text{hom}_{(\hfch)^0}(\overline{\mathcal{B}}(K,K), \mathcal{A}(K,K))$ is an $\alpha$-isomorphism in $\hfch$.%{\color{red}, the homotopy category of filtered chain complexes over $\k$ (see \cite[Example 2.22]{BCZ:tpc}),}%		{\color{red} Or in the category of persistence modules? no, this is not traingulated...}
		\item\label{(3)} The object $K$ is $\frac{\alpha}{2}$-retract-approximable in $PD(\mathcal{A})$ by the family $\text{Obj}(\mathcal{B})$ in the sense of Definition \ref{def:simple_approx}.
	\end{enumerate}
	Moreover, each of the above statements holds for all objects $K$ of $\mathcal{A}$ if and only if the cone of $\mu^{\overline{\mathcal{B}}}(K)$ seen as a morphism of $(\mathcal{A},\mathcal{A})$-bimodules is $\alpha$-acyclic in $H^0(F\text{bimod}_{\mathcal{A}, \mathcal{B}})$.% {\color{orange} the homological category of the filtered $dg$-category of filtered $A_\infty$-bimodules over $\mathcal{A}$}.
\end{prop}

The notion of $r$-isomorphism for $r\geq 0$ is introduced in \cite{BCZ:tpc} and is recalled in \S\ref{sbsb:r-iso}, and $r$-acyclicity is recalled in \S\ref{sbsb:acyclic}.

\begin{rem} 
	Note that by \cite[Lemma 2.85]{BCZ:tpc}, if one of the statements in Proposition \ref{gio:algfilabo} holds, then the interleaving distance in $\hfch$ between $\overline{\mathcal{B}}(K,K)$ and $\mathcal{A}(K,K)$ is $\leq \alpha$.
\end{rem}
%\begin{rem} 	To make computations easier, we will work with twisted complexes over the image of the (left) Yoneda embedding, as in \cite[Appendix A]{Ab:geom-crit}. This is possible as the Yoneda embedding is a filtered $A_\infty$-functor, so that the homotopy category of twisted complexes over Yoneda modules is TPC isomorphic to the image of the extended Yoneda embedding over twisted complexes (see \S\ref{extendedYoneda} and \cite{Se:book-fukaya-categ}). The lengthier version of these computations, working directly over $FTw\mathcal{A}$ is carried out in Appendix \ref{abouznomodules}. {\texttt{I guess that this (computations in FTwA) will not appear in the paper, but maybe its worth it to put somewhere (my thesis?)}} \end{rem}
Before proving Proposition \ref{gio:algfilabo}, we start with two simple auxiliary results that will be used in the proof.
Let $\mathcal{C}$ be a PC, and consider objects $R$ and $X$ of $\mathcal{C}$. The measurement $\dret(R,X)$ has been defined in (\ref{eq:d-rint}). We define  a similar measurement $d'_{\text{r-int}}$ via \begin{align*}
	d'_{\text{r-int}}(R,X):=\frac{1}{2}\inf\Bigl\{r_1+r_2\bigm|&\, r_1,r_2\geq 0, \, \exists \, \varphi: \Sigma^{r_1}R \longrightarrow X,  \, \exists
	\, \psi: \Sigma^{r_2}X
	\longrightarrow R \\ &
	\;\; \text{such that} \; \psi \circ \Sigma^{r_2} \varphi = \eta_{r_1+r_2}^R
	\Bigr\}.
\end{align*} Of course we have $$d'_{\text{r-int}}(R,X) \leq \dret(R,X)\leq2 d'_{\text{r-int}}(R,X).$$

In the proof of Proposition \ref{gio:algfilabo} it is more convenient to use  $d'_{\text{r-int}}$ instead of $d_{\text{r-int}}$.
This is possible because their shift-invariant versions agree:
\begin{lem} For any objects $R$ and $X$ of $\mathcal{C}$ we have
	$$\sdret(R,X) = \bar{d'}_{\text{r-int}}(R,X)$$ where $\bar{d'}_{\text{r-int}}$ is defined as in (\ref{eq:sdint}) but with $\sdret$ replaced by $ d'_{\text{r-int}}(R,X)$.
\end{lem} 
\begin{proof}
	Assume $d'_{\text{r-int}}(R,X)<r$. Let $\varphi: \Sigma^{r_1}R \longrightarrow X$ and $\psi: \Sigma^{r_2}X
	\longrightarrow R$ such that $\psi \circ \Sigma^{r_2} \varphi = \eta_{r_1+r_2}^R$ and $r_1+r_2=2r$. Assume (without loss of generality) that $r_1\geq r_2$ and define $\alpha :=\frac{r_1-r_2}{2}$. Note that $r-\alpha = r_2$ Then $\Sigma^{-\alpha}\varphi\colon \Sigma^{r}R\to \Sigma^{-\alpha}X$ and $\psi\colon \Sigma^{-\alpha+r}X\to R$ show that $\sdret(R,X)<r$. The opposite inequality is obvious.
\end{proof}

%The next lemma is immediate.

%\begin{lem}\label{ezlemma0}
%	Consider the following commutative square of persistence maps
%	\begin{equation}
%		\label{persistencediagram} \xymatrixcolsep{3pc}
%		\xymatrixrowsep{3pc} \xymatrix{ V' \ar[r]^-{h^V}
%			\ar[d]_-{f'} & V
%			\ar[d]^-{f}\\
%			W' \ar[r]^-{h^W} & W}
%	\end{equation}
%	then for any $w\in W_r$ we have $$R(w,f)\leq \inf\left\{R(w',f')\ : \ r'\geq  r, \ w'\in W'_{r'} \text{ such that } h^W_{r'}(w')= i^W_{r',r}(w)\right\}.$$
%\end{lem}
%\begin{proof}
%	Let $v'\in V'_{s'}$ for some $s'\geq {r'}$ such that $f'_{s'}%(v')=i^{W'}_{r',s'}(w')$. Then $$f_{s'}\circ h^V_{s'}(v') = %h^{W}_{s'}\circ f'_{s'}(v') =  h^W_{s'}\circ i^{W'}_{r',s'}(w') %= i^{W}_{r',s'}\circ h^W_{r'}(w') = i^{W}_{r',s'}\circ i^{W}%_{r,s'}(w) = i^{W}_{r,s'}(w)$$ proving the claim.
%\end{proof}

%\ganote{\texttt{The following contains something that i think %we (I) have to look for V2 (?). I feel like it's conceptually %necessary. The starting point is J.Miller'spapepr https://%arxiv.org/pdf/2503.20603 We can discuss this on Wed. }
	%\begin{rem}
	%	The retract interleaving distance is generally not a %symmetric measurement. It is clear that an algebraic %symmetrization does not make sense. An alternative to the usual %symmetrization process is to consider $\dint$ on the Karoubi %envelope $\text{Split}(\mathcal{C})$ of $\mathcal{C}$. %Indeed, $\dret^{\mathcal{C}}(R,X)=\dint^{\text{Split}
		%(\mathcal{C})}(R,X')$ where $X'$ is the direct summand in $X$ %corresponding to $R$. 
	%\end{rem}}
	
	\begin{proof}[Proof of Proposition \ref{gio:algfilabo}]
		
		We start by assuming (\ref{(1)}) and proving (\ref{(3)}). We construct a TPC version of Diagram A.4 in \cite[Appendix A]{Ab:geom-crit}. We refer to this diagram as the (filtered) Abouzaid's algebraic split generation diagram. 
		
		Let $K$ be an object of $\mathcal{A}$ and $\vec{L}=(L_0,\ldots, L_d)$ a tuple of objects of the full subcategory $\mathcal{B}$ and $N\geq 1$ such that the unit $[e_K]$ lies in the image of $\mu^{\overline{\mathcal{B}}} =\mu^{\overline{\mathcal{B}}}(K)$ (we omit $K$ here to ease notation) restricted to $F^N\overline{\mathcal{B}}^{\leq \alpha}(K,K)$. Recall that here $F^N$  indicates the $N$th term in the length filtration of $\overline{\mathcal{B}}$ (as in \ref{eq:lengthfiltration}), while the exponent $\leq \alpha$ stands for the real filtration level.

		Consider the filtered twisted complex $$T_{\vec{L}}K:= T_{L_0}\cdots T_{L_d}K$$ obtained by composing twistings by objects from the tuple $\vec{L}$, as explained in \S\ref{filtwistings} and consider the inclusion $$i_{\vec{L},K}\colon K\to T_{\vec{L}}K$$ which is a composition of morphisms of twisted complexes induced by strict units (see \S\ref{filtwistings}). Let $N\geq 1$. We define $i^N_K$ to be the direct sum over all the maps $i_{\vec{L},K}$ where $\vec{L}$ varies over all tuples of objects of $\mathcal{B}$ of length $\leq N+1$.
		We define the filtered twisted complex $$U^N_K:=\text{Cone}(i^N_K)$$ and its image $$\mathcal{U}^N_K:=\widetilde{\mathcal{Y}}_{U^N_K}$$ under the filtered extended Yoneda embedding $\widetilde{\mathcal{Y}}\colon FTw\mathcal{A}\to F\md_\mathcal{A}$ constructed in \S\ref{extendedYoneda}. By construction, the underlying vector space of the chain complex of $\mathcal{U}^N_K$ applied to an object $Q$ of $\mathcal{A}$ has the form
		$$\bigoplus_{d\leq N}\bigoplus_{\vec{L}=(L_0,\ldots, L_d)\subset \mathcal{B}} \mathcal{A}(Q,L_0)\otimes \mathcal{A}(\vec{L})\otimes \mathcal{A}(L_d,K).$$
		It is an iterated filtered mapping cone built with objects in $\mathcal{B}$.
		
		%The strict unit $e_K\in \mathcal{A}^{\leq 0}(K,K)$ induces a morphism of twisted complexes $$i_{\vec{L},K}\colon K\to T_{\vec{L}}K$$ given by setting it to be $e_K$ in the $(K,K)$ coordinate of the morphism and zero everywhere else (i.e. in the $(K,L_i)$ components). We write the cone of $i_{\vec{L},K}$ as $$U^{\vec{L}}_K:= \text{Cone}(i_{\vec{L},K}).$$ It is easy to see that this twisted complex is a direct sum of tensors of the form $$L'_0\otimes \mathcal{A}(\vec{L'})\otimes \mathcal{A}(L'_{d'},K)$$ for all possible ordered subtuple $\vec{L'}$ of $\vec{L}$. We repeat the above construction for all tuples $\vec{L}$ of objects of $\mathcal{B}$ of length $\leq N$ and construct a filtered twisted complex $U^N_K$ by identifying isomorphic summands. \ocnote{The paragraph above is obscure}. 		Consider its image $$\mathcal{U}^N_K:=\widetilde{\mathcal{Y}}_{U^N_K}$$ under the filtered extended Yoneda embedding $\widetilde{\mathcal{Y}}\colon FTw\mathcal{A}\to F\md_\mathcal{A}$ constructed in \S\ref{extendedYoneda}: it has the form $$\bigoplus_{d\leq N}\bigoplus_{\vec{L}=(L_0,\ldots, L_d)\subset \mathcal{B}} \mathcal{Y}^l_{L_0}\otimes \mathcal{A}(\vec{L})\otimes \mathcal{A}(L_d,K).$$
		
		We have the following result. 
		\begin{lem}\label{Hcommabouzlemma}
			There is a commutative persistence diagram
			\begin{equation} \label{Hcommabouz} \xymatrixcolsep{3pc}
				\xymatrixrowsep{3pc} \xymatrix{ H_*(F^N\overline{\mathcal{B}}(K,K))\ar[r]^-{\overline{\lambda}}
					\ar[d]_-{[\mu^{\overline{\mathcal{B}}}]} & \hom_{PD(\mathcal{A})}(K, U^N_K)
					\ar[d]^-{\xi\circ -}\\
					\hom_{PH(\mathcal{A})}(K,K) \ar[r]^-{\lambda} & \hom_{PD(\mathcal{A})}(K, K)}
			\end{equation}
			where $PH(\mathcal{A})$ denotes the persistence homological category of $\mathcal{A}$  (\S\ref{sbsb:fai}) and $PD(\mathcal{A})$ its persistence derived category constructed using $A_\infty$-modules (\S\ref{sbsb:fmod}).
		\end{lem}
				\begin{proof}[Proof of Lemma \ref{Hcommabouzlemma}]
			
			We define a chain level version of Diagram (\ref{Hcommabouz}).
			
			\begin{equation} \label{commabouz} \xymatrixcolsep{3pc}
				\xymatrixrowsep{3pc} \xymatrix{ F^N\overline{\mathcal{B}}(K,K) \ar[r]^-{\overline{\lambda}}
					\ar[d]_-{\mu^{\overline{\mathcal{B}}}} & \md_\mathcal{A}(\mathcal{Y}^l_K,\mathcal{U}^N_K)
					\ar[d]^-{\mu_2^{\md}(-,\xi)}\\
					\mathcal{A}(K,K) \ar[r]^-{\lambda} & \md_\mathcal{A}(\mathcal{Y}^l_K,\mathcal{Y}^l_K)}
			\end{equation}
			which will be commutative up to a filtered chain homotopy $H\colon F^N\overline{\mathcal{B}}(K,K)\to \md_\mathcal{A}(\mathcal{Y}^l_K,\mathcal{Y}^l_K)$.\\ We now define the maps appearing in Diagram (\ref{commabouz}) and the chain homotopy $H$. The map $\lambda$ is a variation of the Yoneda embedding and is defined in \S\ref{lambdamap}. The map $\mu_2^{\md}$ is the composition in the filtered $dg$-category of filtered $A_\infty$-modules over $\mathcal{A}$. We define $\overline{\lambda}, \xi$, and $H$ below. 
			
			Let $\vec{\gamma}=\gamma_1\otimes a_1\otimes \cdots \otimes a_d\otimes \gamma_2\in F^N\overline{\mathcal{B}}(K,K)$.\begin{enumerate}
				\item[i.] We define $$\overline{\lambda}(\vec{\gamma}) =: \overline{\lambda}_{\vec{\gamma}}\in  \md_\mathcal{A}(\mathcal{Y}^l_K,\mathcal{U}^N_K)$$ to be the morphism of $A_\infty$-modules $$\overline{\lambda}_{\vec{\gamma}} = \sum_{j=0}^d\mathcal{Y}_{j+1}(\gamma_1,a_1,\ldots,a_j)\otimes a_{j+1}\otimes \cdots \otimes a_d\otimes \gamma_2$$ where $\mathcal{Y}_{j+1}$ is the $(j+1)$-component of the (left) Yoneda embedding $\mathcal{Y}$ (see \S\ref{gio:fyon-def}). More explicitly, we can write the component $$(\overline{\lambda}_{\vec{\gamma}})_{l|1}\colon \mathcal{A}(X_0,X_1)\otimes \cdots \otimes \mathcal{A}(X_l,K)\to \mathcal{U}^N_K(X_0)$$ as $$(\overline{\lambda}_{\vec{\gamma}})_{l|1}(x_1,\ldots, x_l,y) = \sum_{j=0}^d\mu_{j+l+2}(x_1,\ldots, x_l,y,\gamma_1,a_1,\ldots,a_j)\otimes a_{j+1}\otimes \cdots \otimes a_d\otimes \gamma_2.$$
				\item[ii.] The map $\xi\in \md_\mathcal{A}(\mathcal{U}^N_K,\mathcal{Y}^l_K)$ is the full contraction map, that is, $$\xi_{l|1}\colon \mathcal{A}(X_0,X_1)\otimes \cdots \otimes \mathcal{A}(X_{l-1},X_l)\otimes \mathcal{U}^N_K(X_l)\to \mathcal{A}(X_0,K)$$ contracts an input coming from a summand of length $d\leq N$ in $\mathcal{U^N_K}$ via $\mu_{l+d+1}$.
				\item[iii.] The map $H\colon F^N\overline{\mathcal{B}}(K,K)\to \md_\mathcal{A}(\mathcal{Y}^l_K,\mathcal{Y}^l_K)$ is also induced by a full contraction, but in the following way: we define  $$H(\vec{\gamma})=: H_{\vec{\gamma}}\in \md_\mathcal{A}(\mathcal{Y}^l_K,\mathcal{Y}^l_K)$$ to be the morphism of $A_\infty$-modules with $l|1$ component given by $$(H_{\vec{\gamma}})_{l|1}(x_1,\ldots, x_l,y) := \mu_{l+d+3}(x_1,\ldots, x_l,y,\gamma_1,a_1\ldots, a_d,\gamma_2).$$
			\end{enumerate}		
			It is straightforward to see that $\xi$ is a filtered morphism of $A_\infty$-modules (that is, it lies in $\md_\mathcal{A}(\mathcal{U}^N_K,\mathcal{Y}^l_K)^{\leq 0}$), and that all the five maps appearing in Diagram (\ref{commabouz}) are filtered chain maps.\\
			We prove that the diagram commutes up to $H$, that is, that $$\lambda\circ \mu^{\overline{\mathcal{B}}}(\vec{\gamma})+\mu_2^\md(\overline{\lambda}_{\vec{\gamma}},\xi) = \mu_1^\md(H_{\vec{\gamma}})+H_{\mu_{0|1|0}^{bar}(\vec{\gamma})}$$ for all $\vec{\gamma}\in F^N\overline{\mathcal{B}}(K,K)$. Let $l\geq 0$ and $\vec{X}=(X_0,\ldots, X_l)$ be a tuple of objects of $\mathcal{A}$. Consider $x_1\otimes \cdots\otimes x_l\in \mathcal{A}(\vec{X})$,  $y\in \mathcal{A}(X_l,K)$ and $\vec{\gamma}\in F^N\overline{\mathcal{B}}(K,K)$ as above. We compute the four terms in the equation above, applied to a composable element $x_1\otimes \cdots \otimes x_l\otimes y$ as above: we have $$\begin{aligned}
					\mu_2^\md\left(\overline{\lambda}_{\vec{\gamma}},\xi\right)_{l|1}(x_1,\ldots, x_l,y) = &\sum_{j=0}^l\sum_{i=0}^d \mu_{j+d-i+2}\bigl(x_1,\ldots, x_j,\mu_{l-j+i+2}(x_{j+1},\ldots, x_l,y, \gamma_1,a_1,\ldots, a_i),\\& \qquad \qquad  a_{i+1},\ldots, a_d, \gamma_2\bigr)
				\end{aligned}$$		
and $$\left(\lambda\circ\mu^{\overline{\mathcal{B}}}(\vec{\gamma})\right)_{l|1}(x_1,\ldots, x_l,y):= \mu_{l+2}(x_1,\ldots, x_l,y,\mu_{d+2}(\gamma_1,a_1,\ldots, a_d,\gamma_2))$$
and \begin{align*}\left(H_{\mu^\text{bar}_{0|1|0}(\vec{\gamma})}\right)_{l|1}(x_1,\ldots, x_l,y)= & \sum_{i=0}^d\mu_{l+d-i+3}(x_1,\ldots, x_l,y,\mu_{1+i}(\gamma_1,a_1,\ldots, a_i), a_{i+1}, \ldots, a_d, \gamma_2)\\ +& \sum_{i\leq j} \mu_{l+i+d-j+3}(x_1,\ldots, x_l,y,\gamma_1, a_1, \ldots , a_{i-1}, \mu_{j-i+1}(a_i,\ldots, a_j),\\ & \qquad a_{j+1}, \ldots, a_d, \gamma_2)\\ +&\sum_{i=1}^{d+1} \mu_{l+i+2}(x_1,\ldots, x_l,y,\gamma_1, a_1 , \ldots , a_{i-1},\mu_{d-i+2}(a_i,\dots, a_d,\gamma_2)) \end{align*}
and \begin{align*}\mu_1^{\md}(H_{\vec{\gamma}})_{l|1}(x_1,\ldots, x_l,y)=&\sum_{i=0}^l \mu_{i+1}(x_1,\ldots, x_{i},(H_{\vec{\gamma}})_{l-i|1}(x_{i+1},\ldots, x_l,y))\\
					 & +\sum_{i=0}^l(H_{\vec{\gamma}})_{i|1}(x_1,\ldots, x_i,\mu_{l-i+1}(x_{i+1},\ldots, x_l,y))\\ &+ 
					\sum_{i=0}^l\sum_{j=i+1}^l (H_{\vec{\gamma}})_{l-j+i|1}(x_1,\ldots,x_{i},\mu_{j-i}(x_{i+1},\ldots, x_j),x_{j+1},\ldots, x_l,y)\\
					=& \sum_{i=0}^l \mu_{i+1}(x_1,\ldots, x_{i},\mu_{l-i+d+3}(x_{i+1},\ldots, x_l,y, \gamma_1,a_1\ldots, a_d,\gamma_2))\\
					& +\sum_{i=0}^l\mu_{i+d+3}(x_1,\ldots, x_i,\mu_{l-i+1}(x_{i+1},\ldots, x_l,y), \gamma_1,a_1\ldots, a_d,\gamma_2)\\ &+ 
					\sum_{i=0}^l\sum_{j=i+1}^l \mu_{l-j+i+d+3}(x_1,\ldots,x_{i},\mu_{j-i}(x_{i+1},\ldots, x_j),x_{j+1},\ldots, x_l,y,\\ & \qquad \qquad \quad \gamma_1,a_1\ldots, a_d,\gamma_2)
				\end{align*}
The sum of the above terms is the $(d+l+2)$-th term in the $A_\infty$-equations for $\mathcal{A}$, hence it vanishes. This proves the claim
		\end{proof}
		
		Consider the identity $id_K\in \hom_{PD(\mathcal{A})}^0(K,K)$. It is straightforward to see that, in the notation of Lemma \ref{Hcommabouzlemma} $$\bar{d'}_{\text{r-int}}(K,U^N_K)\leq \frac{1}{2}R(id_K,\xi\circ -).$$
		The fact that (\ref{(1)}) implies (\ref{(3)}) is then a consequence of the following simple result.
		\begin{lem}\label{ezlemma}
			Consider a commutative square of persistence modules and maps
			\begin{equation}
				\label{persistencediagram2} \xymatrixcolsep{3pc}
				\xymatrixrowsep{3pc} \xymatrix{ V' \ar[r]^-{h^V}
					\ar[d]_-{f'} & V
					\ar[d]^-{f}\\
					W' \ar[r]^-{h^W} & W}
			\end{equation}
			Then for any $w\in W_r$ we have $$R(w,f)\leq \inf\left\{R(w',f')\ \lvert \ r'\geq  r, \ w'\in W'_{r'} \text{ such that } h^W_{r'}(w')= i^W_{r,r'}(w)\right\}.$$
		\end{lem}
		Indeed, since Diagram (\ref{Hcommabouz}) commutes and $\lambda([e_K])=id_K$, we get $$\bar{d'}_{\text{r-int}}(K,U^N_K)\leq \frac{1}{2} R(id_K,\xi\circ -)\leq \frac{1}{2}R\left([e_K],[\mu^{\overline{\mathcal{B}}}(K)]\right)\leq \frac{\alpha}{2}$$ which is point (\ref{(3)}) in the statement of the proposition.\\

		We prove that statement (\ref{(1)}) implies (\ref{(2)}). Consider the map $\mu^{\overline{\mathcal{B}}}\colon \overline{\mathcal{B}}\to \Delta_\mathcal{A}$ as a map of $(\mathcal{A},\mathcal{A})$-bimodules. Consider the $(\mathcal{A},\mathcal{A})$-bimodule $C_\mathcal{B}:=\text{Cone}(\mu^{\overline{\mathcal{B}}})$. Denote the $A_\infty$-bimodule maps of $C_\mathcal{B}$ as $\mu^{C_\mathcal{B}}_{l|1|r}$ for $l,r\geq 0$.
		Let $X_0,X_1$ and $X_2$ be objects of $\mathcal{A}$ and consider $\vec{a}=a_1\otimes\cdots \otimes a_d\in C_\mathcal{B}(X_0,X_1)$ and $\vec{b}=b_1\otimes \cdots \otimes b_n\in C_\mathcal{B}(X_1,X_2)$. We define $$\vec{a}\star \vec{b}:= \sum_{j=0}^{n-1}\sum_{k=0}^{d-1}a_1\otimes \cdots \otimes x_{a}\otimes \mu_{d-k+j}(a_{k+1},\ldots, a_d,b_1,\ldots, b_j)\otimes b_{j+1}\otimes \cdots \otimes b_n\in C_\mathcal{B}(X_0,X_2).$$ It is straightforward to see that $\star$ satisfies the Leibniz rule with respect to $\mu^{C_\mathcal{B}}_{0|1|0}$.

		We consider $C_\mathcal{B}(K,K)$, which is the cone of $\mu^{\overline{\mathcal{B}}}(K)\colon \overline{\mathcal{B}}(K,K)\to \mathcal{A}(K,K)$. Since $R(e_K,[\mu^{\overline{\mathcal{B}}}])\leq \alpha$, there exists a cycle $\vec{h}\in \overline{\mathcal{B}}^{\leq \alpha}(K,K)$ such that $\mu^{\overline{\mathcal{B}}}(\vec{h})$ is a representative of the unit of the object $K$ in $\mathcal{A}^{\leq 0}(K,K)$. More explicitly, we write $$\mu^{\overline{\mathcal{B}}}(\vec{h})= e_K + da_K$$ for some $a_K\in \mathcal{A}^{\leq 0}(K,K)$. We assume without loss of generality that $\vec{h}= h_1\otimes \cdots \otimes h_n$, $n\geq 1$ is a pure tensor.  We define a map $H\colon C_\mathcal{B}(K,K)\to C_\mathcal{B}(K,K)$ and prove it is a contracting homotopy. For a pure element $\vec{x} = x_1\otimes \cdots \otimes x_d\in \text{Cone}(\mu^{\overline{\mathcal{B}}})$, where $d\geq 1$, we define
		$$H(\vec{x}):=\vec{x}\star (\vec{h}+a_K)$$
		Then $$\begin{aligned}
			\mu^{C_\mathcal{B}}_{0|1|0}(H(\vec{x}) + H(\mu^{C_\mathcal{B}}_{0|1|0}\vec{x}) =& \mu^{C_\mathcal{B}}_{0|1|0}(\vec{x}\star (\vec{h}+a_K)) + \mu^{C_\mathcal{B}}_{0|1|0}(\vec{x})\star (\vec{h}+a_K)\\ =& \vec{x}\star\mu^{C_\mathcal{B}}_{0|1|0}(\vec{h}+a_K) = \vec{x}\star e_K = \vec{x}
		\end{aligned}$$ that is, $H$ is indeed a contracting homotopy. It follows that $C_\mathcal{B}(K,K)$ is $\alpha$-acyclic in $\hfch$, or, equivalently, that $\mu^{\mathcal{B}}(K)$ is an $\alpha$-isomorphism, i.e. statement (\ref{(2)}) in Proposition \ref{gio:algfilabo}.\\

		The fact that (\ref{(2)}) implies (\ref{(1)}) is quite straightforward: by Proposition 2.28(i) in \cite{BCZ:tpc}, there exists a right $\alpha$-inverse $$\psi\in \hom_{(\hfch)^0}(\Sigma^\alpha \mathcal{A}(K,K),\overline{\mathcal{B}}(K,K))$$ of $\mu^{\overline{\mathcal{B}}}$, that is a map such that $\mu^{\overline{\mathcal{B}}}\circ \psi = \eta^{\mathcal{A}(K,K)}_\alpha$, which in turn equals the identity seen as a map with shift $\alpha$, so that the conclusion follows.\\

		%% NEW VERSION OF 3 \to 1
		We prove that (\ref{(3)}) implies (\ref{(1)}). Suppose there is an iterated cone $C_K$ over $\mathcal{B}$ in $PD(\mathcal{A})^0$ such that $\sdret(K,C_K)\leq \alpha$. Pick maps $\varphi\in \hom_{PD(\mathcal{A})}^0(\Sigma^{r_1}K, C_K)$ and $\psi\in \hom_{PD(\mathcal{A})^0}(\Sigma^{r_2}C_K,K)$ such that $$r_1+r_2=\alpha, \quad \text{and} \quad \psi\circ\Sigma^{r_2}\varphi = \eta^K_{\alpha}.$$
		We represent $C_K$ by a filtered twisted complex $$\left(\bigoplus_{i=1}^n \Sigma^{\beta_i}X_i[d_i], (q_{ij})_{ij}\right)$$ where $X_i$ is an element of $\mathcal{B}$ for all $i$. Let $f\in FTw^{\leq0}(\Sigma^{r_1}K, C_K)$ be a representative of $\varphi$ and let $g\in FTw^{\leq0}(\Sigma^{r_2}C_K, K)$ be a representative of $\psi$. Then $f$ can be written as a matrix $(f_1,\ldots, f_n)$ where $f_i\in FTW^{\leq r_1-\beta_i}(K,X_i)[d_i]$. Similarly, $g$ can be written as a matrix $(g_1,\ldots, g_n)^T$ where $g_i\in FTW^{\leq r_2+\beta_i}(X_i,K)[-d_i]$. Since our $A_\infty$-category $\mathcal{A}$ is strictly unital, we have that $$e_K=\mu_2^{FTw}(f,g) = \sum_{i=1}^n\sum_{j=1}^n\sum_{i<i_1<\ldots<i_k<j}\mu_{2+k}(f_i, q_{ii_1},\ldots,q_{i_{k}j}, g_j).$$ Note that every tensor $f_i\otimes q_{ii_1}\otimes \ldots\otimes q_{i_{k}j}\otimes g_j$ is an element of $\overline{\mathcal{B}}^{\leq \alpha}(K,K)$. This proves (\ref{(1)}).
		%% OLD VERSION OF 3\to 1 
		%We now prove that (\ref{(3)}) implies (\ref{(1)}). Suppose there is an iterated cone $C_K$ over $\mathcal{B}$ in $PD(\mathcal{A})^0$ such that $\sdret(K,C_K)\leq \alpha$. Pick maps $\varphi\in \hom_{PD(\mathcal{A})}^0(\Sigma^{r_1}K, C_K)$ and $\psi\in \hom_{PD(\mathcal{A})^0}(\Sigma^{r_2}C_K,K)$ such that $$r_1+r_2=\alpha, \quad \text{and} \quad \psi\circ\Sigma^{r_2}\varphi = \eta^K_{\alpha}$$  At the chain level, this composition can be represented using the $\mu_d$-maps from $\mathcal{A}$ using maps from $K$ to the components of $C_K$ and viceversa, with deformations given by the structure maps of $C_K$, which is a matrix with entries in the subcategory $\mathcal{B}$. In particular (after shifting $\varphi$ and $\psi$ by $r_2$), we see that $\mu^{\overline{\mathcal{B}}}$ restricted to $H^0_*(\overline{\mathcal{B}}(\Sigma^{\alpha}K,K))=H^\alpha_*( \overline{\mathcal{B}}(K,K))$ hits $i_{0,\alpha}([e_K])$, i.e. (\ref{(1)}).
		
		This completes the proof of Proposition \ref{gio:algfilabo}.		
	\end{proof}

%==================================
%==================================
%==================================
\subsubsection{Perturbation data leading to filtered $A_\infty$-Fukaya categories}\label{gio:ffuk}
We recall here the choice of perturbation data introduced in  \cite{Amb:fil-fuk}.  
		
The construction of the Fukaya category as it appears in Seidel's book, \cite{Se:book-fukaya-categ}, does not produce, in
general, a filtered category. Indeed, let $p$ be a perturbation datum in the sense of Seidel. Let $\vec{L}:=(L_0,\ldots,L_d)$ be a tuple of Lagrangians in $\lagmon$ such that any two subsequent Lagrangians (in cyclic order) are geometrically different. Let $\vec{\gamma}=(\gamma_1,\ldots, \gamma_d)\in CF(\vec{L};p)$ be a tuple made of generators of the Floer complexes (these are defined with respect to Hamiltonian perturbations prescribed by $p$) and let $\gamma_+\in CF(L_0,L_d;p)$ be another generator. Let $u$ be a perturbed $J$-holomorphic polygon, defined on some punctured disk $S$ with $d$ entries and one exit, that contributes to the $\gamma_+$-coefficient in the expression of $\mu_d(\gamma_1,\ldots,\gamma_d)$. By definition, this $u$ satisfies a perturbed pseudoholomorphic equation with a perturbation term corresponding to a $1$-form $K^p$ (see Section 8 in  \cite{Se:book-fukaya-categ}). Standard computations give us that the energy $E(u)$ of $u$ satisfies $$0\leq E(u)\leq \omega(u) + \sum_{i=1}^d\int_0^1H^{L_{i-1},L_i}_p\circ \gamma_i \ dt - \int_0^1H^{L_{0},L_d}_p\circ \gamma_+ \ dt + \int_{S}R^p\circ u $$ where $R^p \in \Omega^2(S, C^\infty(M))$ is the so-called curvature form of the perturbation datum $K^p$ induced by $p$ on the punctured disk $S$. In conformal coordinates $(s,t)$ on $S$ we can locally write $$K^p = Fds + G dt$$ and the curvature term has then the form $$R^p = \left(\partial_sG - \partial_t F +\omega\left(X^{F}, X^{G}\right)\right)ds\wedge dt$$ where $X^{F}, X^{G}$ are the Hamiltonian vector fields induced by $F$ and $G$. Writing $\mathbb{A}_p$ for the action functional defined with respect to $p$, we can rewrite the action-energy estimate above as $$\mathbb{A}_p(T^{\omega(u)}\gamma_+)\leq \sum_{i=1}^d\mathbb{A}_p(\gamma_i) + C^p(u)$$ where $C^p(u):= \int_{S}R^p\circ u$, is called the \textit{curvature term} of the polygon $u$ associated with the choice of the perturbation data $p$. We formalize the conclusion of this computation as follows.

\

\textbf{Negative curvature requirement:} 
\begin{eqnarray}
\nonumber \mathrm{To\ construct\ filtration\ preserving\ maps\ that\ are\ defined\ by\ means\ of\ counts\ } \\ 
\label{eq:neg_c_req} \mathrm{ \ of\  perturbed}\ J- \mathrm{\ holomorphic\ curves\  it\ is\ enough\ to\ show\ that\ we\ can\ pick} \\ 
\nonumber  \mathrm{ \ perturbation\ data\ such\ that\  the\ curvature\ terms\ admit\  a\ strictly\ negative\ } \\ 
\nonumber  \mathrm{\  constant\ as \  an\ } \  apriori  \mathrm{\ uniform\ upper\  bound. \hspace{1in}}
\end{eqnarray}

\

The requirement for a  \textit{stricly negative} upper bound is due to the fact that, for transversality reasons, we might need to introduce further, arbitrarily small perturbations. 
	
\
	
We now outline, following \cite{Amb:fil-fuk}, the construction of the perturbation data $\mathcal{P}$ 
satisfying (\ref{eq:neg_c_req}).  This construction depends on a parameter $\frac{1}{2}<\delta<1$
that we now fix. 

We describe the choices made for $p\in \mathcal{P}$ and for a tuple $\vec{L}$ of cyclically different geometric Lagrangians as above. 
To start,  \label{gio:FD} consider a couple of different Lagrangians $L_0$, $L_1$ in $\lagmon$.
The Hamiltonian perturbation function $H^{L_0,L_1}_p$ assigned to this couple, and part of the perturbation data
 $p$, is subject to the condition 
 $$\delta\nu(p)<H^{L_0,L_1}_p(t,x)<\nu(p)$$ 
 where $\nu$ is defined in equation (\ref{eq:nu-fuk}).  Consider now a triple $\vec{L}=(L_0,L_1,L_2)$ of Lagrangians in $\lagmon$. Let $S$ be the unique equivalence class of disks with three punctures, equipped with strip-like ends (two entries and one exit, see \cite[Section (8d)]{Se:book-fukaya-categ}, \cite[Section 2.1]{Amb:fil-fuk}). The Hamiltonian perturbation $K_p^S\in \Omega^1(S,C^{\infty}_c(X\times [0,1]))$ on $S$ is required to have a controlled behaviour on  the strip-like ends near punctures, and moreover it essentially vanishes away from the strip-like ends. More precisely: 
	        \begin{itemize}
		\item[-] on the $i$th negative  strip-like end, $i=1,2$, $K_p^S$ restricts to a monotone homotopy from the Hamiltonian Floer datum to the zero map, that is $$K_p^S = (1+\beta_p^{L_{i-1},L_i}(s+1))H^{L_{i-1},L_i}dt$$ with strip-like ends coordinate $(t,s)\in [0,1]\times (-\infty,0]$, where $\beta_p^{L_{i-1},L_i}:\mathbb R \to [0,1]$ is a smooth, increasing and surjective function with derivative supported in $[0,1]$;
		\item[-] on the positive strip-like end the Hamiltonian part $K_p^S$ restricts to a monotone homotopy from the zero map to the Hamiltonian Floer datum, that is $$K_p^S = \beta_p^{L_{0},L_2}(s)H^{L_{0},L_2}dt$$ with strip-like ends coordinate $(t,s)\in [0,1]\times [0,\infty)$, where $\beta_p^{L_{0},L_2}:\mathbb R \to [0,1]$ is a smooth, increasing and surjective function with derivative supported in $[0,1]$;
		\item[-] $K^S_p$ vanishes away from strip-like ends.
	\end{itemize} 
Assume for a moment tranversality for $K_p^S$ as above and consider a $K_p^S$-Floer polygon $u$ joining Floer generators $\gamma_1\in CF(L_0,L_1;p)$ and $\gamma_2\in CF(L_1,L_2;p)$ to $\gamma_+\in CF(L_0,L_2;p)$. Then the curvature term $C^p(S,u)$ of $u$ is supported in the strip-like ends only and its a simple computation to see that $$C^p(S,u)<\nu(p)(1-2\delta)<0$$ as, by assumption, $\frac{1}{2}<\delta<1$ . It follows that the map $\mu_2$ on $\vec{L}$ defined via $K_p^S$ is filtered. Perturbation data for longer tuples of cyclically different Lagrangians is obtained by piecing together lower order choices (see \cite[Section 3.2]{Amb:fil-fuk}). Transversality is achieved by the recipe in \cite{Se:book-fukaya-categ}. Given that $\nu(p)(1-2\delta)$ is negative (and not only non-positive), we can achieve both transversality and filtration preserving maps $\mu_d$ by considering small enough additional perturbations (see \cite[p. 479]{Amb:fil-fuk}). Extending this perturbation system  to arbitrary tuples of Lagrangian (and not only cyclically different ones) is taken care by using clusters type moduli spaces mixing Floer polygons and Morse-pearly trajectories (as, for instance, in the pearl homology construction in \cite{Bi-Co:rigidity}, or in \cite{Cor-La:Cluster-2}, \cite{she:fanofuk}). Full details of this construction appear in \cite{Amb:fil-fuk}. We remark that we work with the quantum model of $CF(L,L)$ in order to achieve (strict) units at vanishing filtration levels (cfr. \cite[Section 3.3]{Amb:fil-fuk}).

\begin{rem} \label{rem:const_ham} In the construction outlined above it is possible to require, in addition, that
 if $L_0\pitchfork L_1$ intersect transversally, then $H^{L_0,L_1}_p=H^{L_1,L_0}_p$ are constant. This is not a necessary assumption, but it makes computations simpler. Notice, however, that we cannot require Hamiltonian Floer data for transversally intersecting Lagrangians to vanish identically.
\end{rem}

%======================================
%======================================
%======================================
	\subsubsection{The coproduct}\label{gio:coproduct}

	%We fix a perturbation datum $p\in \mathcal{P}$ and an object $K\in \lagmon$. We define a persistence map $$\Delta_p^{\mathcal{P}}\colon PHH_*(\mathcal{B}_p,\mathcal{B}_p)\to \overline{B_p}(K,K).$$ 
	
	We recall that given an $A_\infty$-category $\mathcal{A}$ we denote by $\Delta_\mathcal{A}$ its diagonal $A_\infty$-bimodule. %We fix a perturbation datum $p\in \mathcal{P}$, with $\mathcal{P}$ as constructed in \S\ref{gio:ffuk}, 
	We fix a finite family $\mathcal{B}$ of Lagrangians in $\lagmon$ such that any two distinct elements in $\mathcal{B}$ intersect transversally. We recall that $\mathcal{B}$ induces a finite full $A_\infty$-subcategory $\mathcal{B}_p$ of $\mathcal{A}_p=\fuk(\lagmon;p)$ for any choice of perturbation datum $p\in \mathcal{P}$, as explained on page \pageref{sec:split-app}. We also fix an object $K\in \lagmon$. %	, a finite full $A_\infty$-subcategory of $\mathcal{A}$ as in the statement of Proposition \ref{gio:algfilabo}, and recall that $\mathcal{B}+p$ 
	The aim here is to refine the perturbative choices described in the last section and show the next result.
	\begin{prop}\label{coproductmain} For any $p\in \mathcal{P}$ and any object $K\in \lagmon$ there is a non-empty space of perturbation data $\mathcal{P}(p,\mathcal{B},K)$ such that for each $q\in  \mathcal{P}(p,\mathcal{B},K)$ there exists a morphism of $(\mathcal{B}_p,\mathcal{B}_p)$-bimodules $$\Delta=\Delta^{\mathcal{B},K}_q\colon \Delta_{\mathcal{B}_p}\to \mathcal{Y}_{p}^l(K)\otimes \mathcal{Y}_{p}^r(K)$$ which shifts filtration by $\leq 2\nu(p)$ and extends the usual coproduct map. Moreover, in the case wher $\mathcal{B}=\{L\}$ consists of a single Lagrangian, $\Delta^\mathcal{B}_q$ is filtered.%In particular, the induced map $$\Delta\colon CC_*(\mathcal{B},\Delta_{\mathcal{B}_p})\to\overline{\mathcal{B}}(K,K)$$ is a filtered chain map. 
	\end{prop}
	%\begin{prop}\label{coproductmain} 		There is a non-empty subset $\mathcal{P}(\mathcal{B})\subset \mathcal{P}$ such that for any $p\in \mathcal{P}(\mathcal{B})$ and any object $K\in \lagmon\setminus \mathcal{B}$ there is a non-empty class $\mathcal{P}(p,\mathcal{B},K)$ such that for each $q\in  \mathcal{P}(p,\mathcal{B},K)$ there exists a morphism of $(\mathcal{B}_p,\mathcal{B}_p)$-bimodules $$\Delta=\Delta^{\mathcal{B},K}_q\colon \Delta_{\mathcal{B}_p}\to \mathcal{Y}_{p}^l(K)\otimes \mathcal{Y}_{p}^r(K)$$ which is filtered and extends the usual coproduct map. %In particular, the induced map $$\Delta\colon CC_*(\mathcal{B},\Delta_{\mathcal{B}_p})\to\overline{\mathcal{B}}(K,K)$$ is a filtered chain map. \end{prop}
	\noindent We will often drop $q$, $\mathcal{B}$ and $K$ from the notation of the coproduct and simply write it as $\Delta$ when there is no risk of confusion. The proof of Proposition \ref{coproductmain} will be sketched in the remaining of this subsection.
	%\begin{rem}     Let $p\in E$ be a perturbation datum. The convention of setting $H^{L_0,L_1}_p=v(p,\mathcal{B})$ for any pair $(L_0,L_1)$ of transversely intersecting Lagrangians in $\mathcal{L}^*(M,\omega)$, and not to any other constant in $(\delta v(p,\mathcal{B}),v(p,\mathcal{B}))$ is made exactly in order for $\Delta^\mathcal{B}$ to be automatically filtered whenever the elements of $\mathcal{B}$ are pairwise transverse, which is what we are gonna deal with in most cases. \end{rem}
	\begin{rem} a. The perturbations in the class $\mathcal{P}(p,\mathcal{B},K)$ are specific to the definition of
	the coproduct map $\Delta$.% This in particular this family of perturbations is not a subclass of $\mathcal{P}(\mathcal{B})$.

%b. Proposition \ref{coproductmain} can be easily extended to the case where $K$ is contained in $\mathcal{B}$ by a slight modification of the proof below. In any case, this scenario is uninteresting from the point of view of the approximability.
%b. As it will be apparent from the proof, higher order components of $\Delta$ can be made filtered by adjusting the choice of perturbations, while this is not possible with the linear level $\Delta_{0|1|0}$.
b. As it will be apparent from the proof, one could refine the choice of the Floer and perturbation data for elements in the finite family $\mathcal{B}$ and get a map of $A_\infty$-bimodules $\Delta$ which shifts filtration by an arbitrarily small amount. Since this is not particularly useful for the aim of this paper, we do not explain the details here.

c. We remark that the fact that $\Delta$ shifts filtration by $\leq 2\nu(p)$ is equivalent to saying that the map $\eta_{2\nu(p)}\circ \Delta\colon \Delta_{\mathcal{B}_p}\to \Sigma^{-2\nu(p)}(\mathcal{Y}_{p}^l(K)\otimes \mathcal{Y}_{p}^r(K))$ is filtered. Here $\Sigma^{-2\nu(p)}(\mathcal{Y}_{p}^l(K)\otimes \mathcal{Y}_{p}^r(K))$ is the filtered bimodule with $(\Sigma^{-2\nu(p)}(\mathcal{Y}_{p}^l(K)\otimes \mathcal{Y}_{p}^r(K)))^{\leq \alpha} = (\mathcal{Y}_{p}^l(K)\otimes \mathcal{Y}_{p}^r(K))^{\leq \alpha + 2\nu(p)}$ and $\eta_{2\nu(p)}$ is the structural map from \S\ref{sb:pc} (see also \cite[Section 2.2.3]{BCZ:tpc}).

\end{rem}

\begin{figure}
	\begin{subfigure}{.5\textwidth}
		\centering
		\includegraphics[width=.7\linewidth]{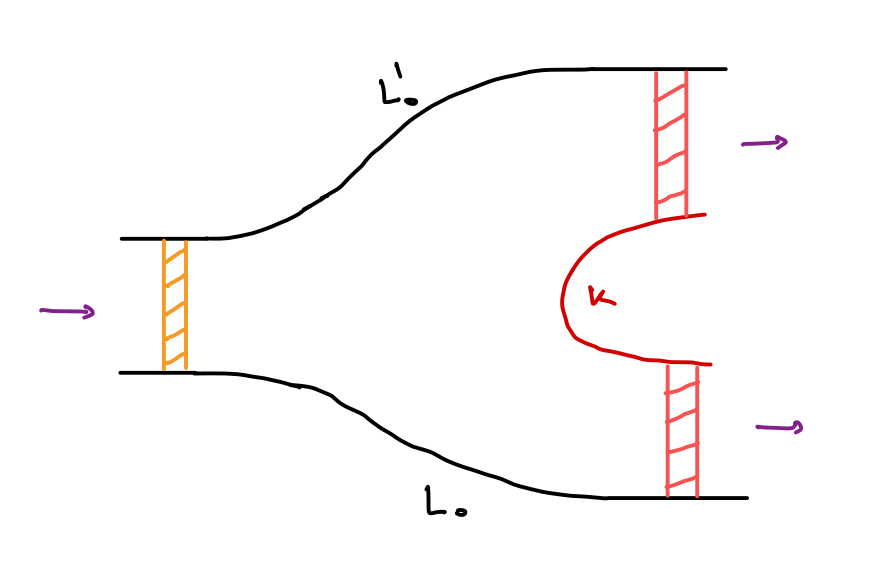}
		\caption{A possible configuration counted for $\Delta_{0|1|0}\colon CF(L_0,L_0')\to CF(L_0,K)\otimes CF(K,L_0').$}
		\label{fig:sfig1}
	\end{subfigure}\hspace{1em}
	\begin{subfigure}{.5\textwidth}
		\centering
		\includegraphics[width=.7\linewidth]{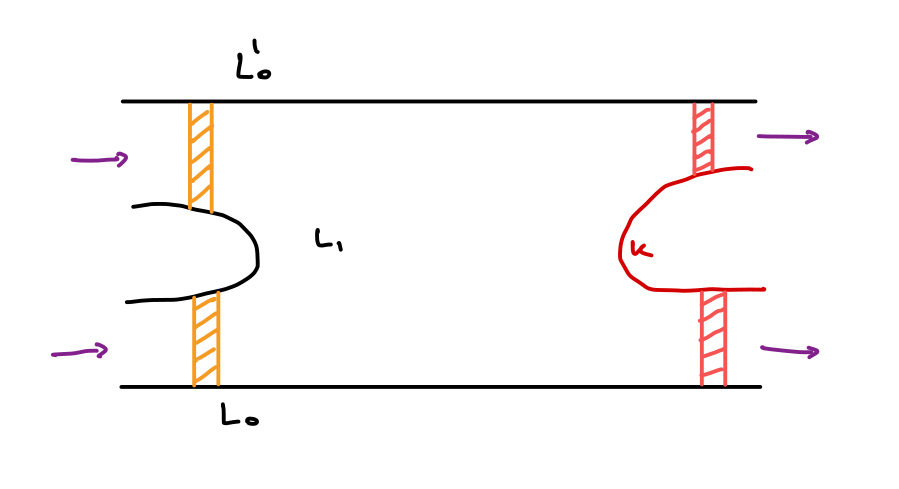}
		\caption{A possible configuration counted for $\Delta_{1|1|0}\colon CF(L_0,L_1)\otimes CF(L_1,L_0')\to CF(L_0,K)\otimes CF(K,L_0').$}
		\label{fig:sfig2}
	\end{subfigure}\hspace{1em}
	\begin{subfigure}{.5\textwidth}
		\centering
		\includegraphics[width=.7\linewidth]{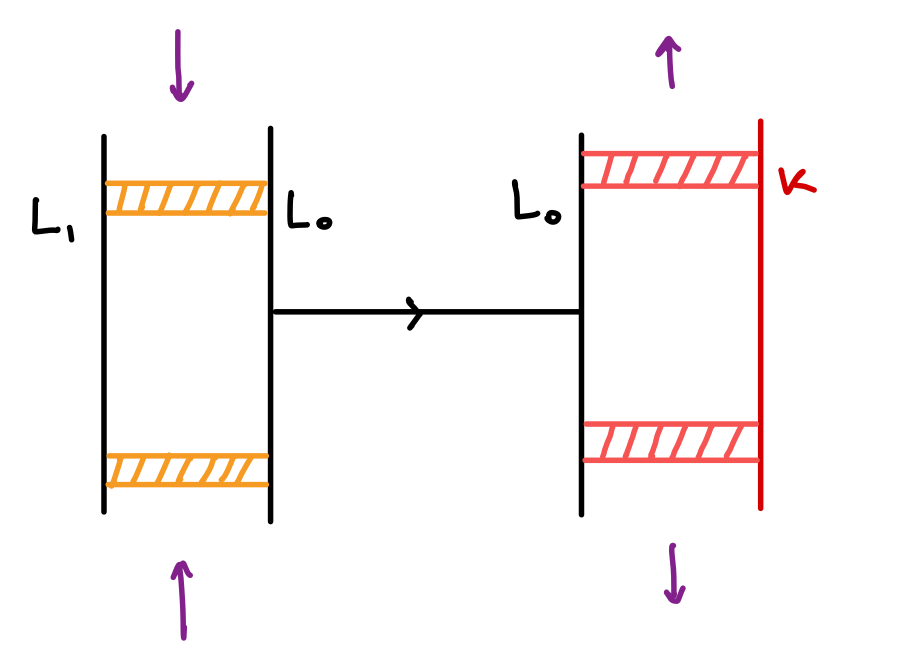}
		\caption{A possible configuration counted for $\Delta_{1|1|0}\colon CF(L_0,L_1)\otimes CF(L_1,L_0)\to CF(L_0,K)\otimes CF(K,L_0).$ The horizontal line segment represents a finite Morse flowline on $L_0$.}
		\label{fig:sfig2}
	\end{subfigure}
	\caption{Examples of configurations contributing to the coproduct $\Delta^{\mathcal{B},K}_q\colon \Delta_{\mathcal{B}_p}\to \mathcal{Y}_{p}^l(K)\otimes \mathcal{Y}_{p}^r(K)$. The colored strips indicate the support of the non-trivial parts of Hamiltonian perturbation data.}
	\label{fig:coproduct}
\end{figure}
\begin{proof}
The construction of $\Delta$ depends on the construction of certain moduli spaces whose definition is 
well understood (see \cite{Ab:geom-crit}, \cite{she:fanofuk}). Therefore, we will only sketch the construction and only emphasize the adjustments required to ensure that $\Delta$ has a controlled shift in action filtration. We sketch in Figure \ref{fig:coproduct} some configurations of clusters contributing to the definition of $\Delta$.

 Let $\epsilon>0$ and fix a regular filtered perturbation datum $p\in \mathcal{P}$ of size $\nu(p)=\epsilon$. $\Delta$ will consists of a family of linear maps $$\Delta_{l|1|r}\colon CF(L_0,\ldots, L_l)\otimes CF(L_l,L_0')\otimes CF(L_0',\ldots, L'_r)\to CF(L_0,K)\otimes CF(K,L'_r)$$ for any $l,r\geq 0$ and any two tuples $\vec{L}=(L_0,\ldots, L_l)$ and $\vec{L'}=(L'_0,\ldots, L'_r)$ of Lagrangians in $\mathcal{B}$, where all the Floer chain complexes are defined with respect to $p$. The source spaces for $\Delta_{l|1|r}$ are parametrized by  moduli spaces of clusters (following the terminology in  \cite{Amb:fil-fuk}) $\mathcal{R}^{l,r,2}_C(\vec{L},\vec{L'},K)$ labelled by $\vec{L}$, $\vec{L'}$ and $K$ defined as follows:
\begin{enumerate}
	\item We start with the moduli space $\mathcal{R}^{l,r}(\vec{L},\vec{L'},K):=\mathcal{R}^{l+r+3}(\vec{L}\cup\vec{L'}\cup K)$ of disks with $l+r+3$ marked points, labelled by our fixed Lagrangians as follows: we choose a marked point and label by $K$ the first arc going clockwise from this point and then keep on labelling clockwise successive arcs first with $\vec{L}$, and then with $\vec{L'}$.
	\item Over $\mathcal{R}^{l,r}(\vec{L},\vec{L'},K)$ we have a bundle $$\pi^{l,r}(\vec{L},\vec{L'},K)\colon \mathcal{S}^{l,r}(\vec{L},\vec{L'},K)\to \mathcal{R}^{l,r}(\vec{L},\vec{L'},K)$$ such that given $r\in \mathcal{R}^{l,r}(\vec{L},\vec{L'},K)$ its preimage is the standard unit disks with labels and prescribed marked points but with the marked points where two arcs corresponding to geometrically different Lagrangians meet replaced by a puncture.
	\item We endow $\mathcal{R}^{l,r}(\vec{L},\vec{L'},K)$ with strip-like ends in the usual manner: near marked points adjacent to the arc labelled by $K$, we choose positive strip-like ends, while on the remaining $l+r+1$ marked points we choose negative ones (this choice will be required to be consistent with gluing and breaking of disks, as usual).
	\item We consider the compactification $\overline{\mathcal{R}^{l,r}(\vec{L},\vec{L'},K)}$ of $\mathcal{R}^{l,r}(\vec{L},\vec{L'},K)$ which has the structure of a manifold with corners
%\footnote{If the strip-like ends are chosen in a compatible way, which we do, even though we will not go into details here.} 
of dimension $l+r+1$, and add collar neighbourhoods to some boundary components as \cite[p. 14]{Amb:fil-fuk} to get the needed moduli space  $\mathcal{R}_C^{l,r}(\vec{L},\vec{L'},K)$
	\item Over this moduli space we define the bundle $$\pi_C^{l,r}(\vec{L},\vec{L'},K)\colon \mathcal{S}_C^{l,r}(\vec{L},\vec{L'},K)\to \mathcal{R}_C^{l,r}(\vec{L},\vec{L'},K)$$ where fibers are cluster of marked disks with punctures with $l+r+1$ entries and two (adjacent) exits.
\end{enumerate}
Similarly to the case of filtered $A_\infty$-categories (sketched in \S\ref{gio:ffuk} and worked out in \cite{Amb:fil-fuk}), there is a space $\mathcal{P}^{\Delta}(p,\mathcal{B},K)$ of universal choices of perturbation data on the family of bundles of clusters $\pi_C^{l,r}(\vec{L},\vec{L'},K)$ (with fixed $K$), and an associated compatibility requirement with the fixed choice of perturbation datum $p\in \mathcal{P}$. 

Given a choice of $q\in \mathcal{P}^{\Delta}(p,\mathcal{B},K)$, of tuples $\vec{L}$ and $\vec{L'}$ as above,
 and of generators $$\gamma_0\otimes \cdots \otimes \gamma_l\otimes \gamma\otimes\gamma'_0\otimes \cdots \otimes \gamma'_r \in CF(L_0,\ldots, L_l)\otimes CF(L_l,L_0')\otimes CF(L_0',\ldots, L'_r)$$ and $$\gamma^+_1\otimes \gamma^+_2\in CF(L_0,K)\otimes CF(K,L'_r)$$ we have the moduli spaces $$\mathcal{M}^{l+r+1}_{\vec{L},\vec{L'},K}(\gamma_0,\ldots, \gamma_l;\gamma;\gamma'_1,\ldots, \gamma_r;\gamma^+_1,\gamma^+_2;q)$$ of Morse-Floer clusters $u$ defined for the perturbation datum $q$, joining such generators. One then defines $\Delta_{l|1|r}$ by counting rigid clusters $u$ in these moduli spaces, weighted by their symplectic area (by multiplication with $T^{\omega(u)}$, with $T$ the Novikov ring variable). Standard compactness and transversality arguments then show that for a generic choice of $q\in \mathcal{P}^{\Delta}(p,\mathcal{B},K)$ these maps fit together into a morphism of $A_\infty$-bimodules $$\Delta\colon \Delta_{\mathcal{B}_p}\to \mathcal{Y}^l_{p}(K)\otimes \mathcal{Y}^r_{p}(K)$$ which, of course, depends on the choice of the perturbation datum $q$. Most importantly, the filtered properties of $\Delta$ strongly depend on the choice of $q$: for an arbitrary choice of $q$ there is no reason to expect that the associated morphism $\Delta$ has a controlled shift in filtration. 
However, the same idea introduced in \cite{Amb:fil-fuk} (and outlined in \S\ref{gio:ffuk}) for the case of the $\mu_d$-maps can be replicated here to control the shift of the map $\Delta$. 
 %the idea is to choose perturbation data for the source space with one input and two exits that are supported on strip %like ends, and on each strip-like end they restrict to smart monotone homotopies from the Hamiltonian Floer data to %the zero map (and viceversa, depending on the sign of the strip-like end), and then construct perturbation data for %source spaces with more inputs inductively using gluing with lower order ones and ones coming from $p$. 
 Notice that the term $\Delta_{0|1|0}\colon CF(L,L)\to CF(L,K)\otimes CF(K,L)$ counts Floer curves with a Morse input (recall that we defined $CF(L,L)$ using the pearly model) and two Floer outputs (as $K\notin \mathcal B$ by assumption), which do not pose any problem from a filtered point of view. % as no "$s$-dependent" perturbation data are needed.  
On the other hand, we encounter some novel problems relative to the constructions described previously when dealing with two different Lagrangians $L_0\neq L_0' $ and the map $ \Delta_{0|1|0}\colon CF(L_0,L_0')\to CF(L_0,K)\otimes CF(K,L_0')$. In this case, the methods from \cite{Amb:fil-fuk} give a map of shift $\leq 2\nu(p)-\delta\nu(p)\leq 2\nu(p)$. Since the perturbation data for higher order components $\Delta_{l|1|r}$ of $\Delta$ are constructed --following \cite{Amb:fil-fuk}-- inductively, this shifts propagates. It is important to remark that given the structure of the moduli spaces defining $\Delta_{l|1|r}$ this shift does not grow proportionally to the order, but remains $\leq 2\nu(p)$. Indeed near corners of $\overline{\mathcal{R}_C^{l,r}(\vec{L},\vec{L'},K)}$, the inductively constructed perturbation data come from gluing of perturbation data associated with configurations with $\leq 3$ marked points, of which at most one is a configuration with one entry and two exit (that is, on which $\Delta_{0|1|0}$ is modeled on) and contributes positively to the curvature (and the contribution is bounded above by $2\nu(p)$ as explained above); all the other configuration involved in the gluing have only one output, and do not contribute to the curvature (see Figure \ref{fig:cop_near_bdry}).
\end{proof}
\begin{figure}
	\includegraphics[scale=0.2]{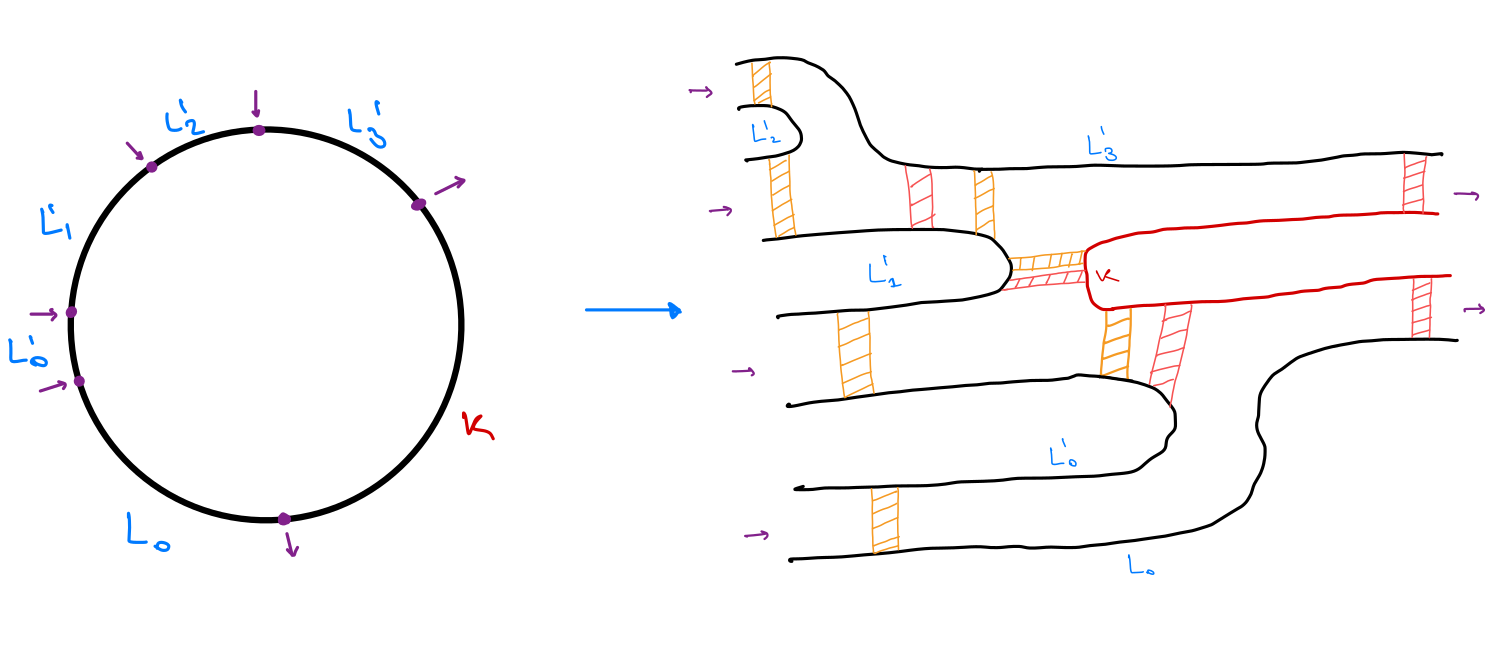}
	\centering
	\caption{A curve contributing to $\Delta_{0|1|3}\colon CF(L_0,L_0',L_1',L_2',L_3')\to CF(L_0,K)\otimes CF(K,L_3')$ whose source space lies near a corner point in $\overline{\mathcal{R}_C^{0,3}(\vec{L},\vec{L'},K)}$. The orange strips represent the region where perturbation data contribute negatively to the curvature, while the red strips where it contributes positively. The total curvature is then bounded from above by $2\nu(p)$.} \label{fig:cop_near_bdry}
\end{figure}

\begin{rem}
	It is clear from the proof of Proposition \ref{coproductmain} that the shift $\leq 2\nu(p)$ of the map $\Delta$ is not sharp for our choices of perturbations (indeed, $2\nu(p)-\delta\nu(p)$ can be used as an upper bound). However, we decided to stick to $2\nu(p)$ for notational ease. Since in Theorem \ref{gio:wabomain} we work with a system of $A_\infty$-categories, this difference is irrelevant.
\end{rem}
%======================================
%======================================
%======================================
\subsubsection{Ambient quantum cohomology}\label{AQH} The material briefly recalled here is well-known
(see \cite{she:fanofuk} for a possible reference). 
Let $(X,\omega)$ be a closed and monotone symplectic manifold. We consider a Morse function $f\colon X\to \mathbb{R}$ with a unique maximum $u\in X$ and a Riemannian Metric $g$ on $X$ such that the pair $(f,g)$ is Morse-Smale.  We define the quantum cochain complex $CQ^*(f,g;\Lambda)$ with $\Lambda$ coefficients associated with the pair $(f,g)$ as its Morse cochain complex over $\Lambda$, that is coindexed by the Morse index. The differential $d$ counts negative gradient flow lines connecting critical points with index difference equal to one. In the following, we
will denote $CQ(f , g ; \Lambda)$ by $CQ(X; \Lambda)$ when the specific choice of Morse-Smale data is not important for our discussion. We filter this complex by setting every generator at zero action level and using the standard filtration on $\Lambda$. The homology of this complex is denoted by $QH(X;\Lambda)$ and  is called the quantum cohomology of $X$. Note that neither the vector space nor the persistence structure of $QH(X;\Lambda)$ depends on the choice of the Morse-Smale pair. We endow $\QH(X,\La)$ with the standard $\mathbb{R}/2\mathbb{Z}$ grading, with the Novikov variable having degree $0$.

The vector space $QH(X;\Lambda)$ is endowed with a product which deforms the usual Morse product, denoted by $*$ and called the quantum product. This product is defined at the chain level by counting (and weighting in the Novikov coefficient by the symplectic area of) $J$-holomorphic spheres with three marked points, which  lie in stable and unstable manifolds of perturbations of the Morse function $f^M$, much like in the definition of the $A_\infty$-maps for tuples of identical Lagrangians in the definition of the filtered Fukaya category. The quantum product with the first Chern class $c_1(X)$ defines an endomorphism of $QH(X;\Lambda)$ and hence we have a decomposition $$QH^*(X;\Lambda) = \bigoplus_{\mathbf{d}} QH^*(X;\Lambda)_\mathbf{d}$$ where the index corresponds to eigenvalues $\mathbf{d}\in \Lambda$ of this endomorphism, and $QH(X,\Lambda)_\mathbf{d}$ is the eigenspace associated with the eigenvalue $\mathbf{d}\in \Lambda$. It is well known that the Fukaya category associated with the class $\lagmon$ is non-trivial if and only if $\mathbf{d}$ is an eigenvalue of the map above. Moreover, all the eigenvalues $\mathbf{d}$ in fact lie in $\Lambda_0$.

%======================================
%======================================
%======================================

\subsubsection{The persistence open-closed and closed-open maps}\label{gio:oc}
In this section we first define a persistence version of the open-closed map $OC$ relating Fukaya categories with ambient quantum cohomology. This  follows standard constructions in the subject  (cfr. \cite{Ab:geom-crit, she:fanofuk}) therefore we only give details sufficient to justify the persistence aspects.

Let $\mathcal{B}\subset \lagmon$ be a finite subset. Let $\epsilon>0$ and fix a regular filtered perturbation datum $p\in \mathcal{P}$ of size $\nu(p)=\epsilon$. 

We now outline the construction of the persistence open-closed map 
\begin{equation}\label{eq:oc_map}
OC=OC^\mathcal{B}_p\colon PHH(\mathcal{B}_p)\to QH(X,\Lambda)
\end{equation} where $\mathcal{B}_p$ is the full subcategory of $\fuk(\lagmon;p)$ induced by $\mathcal{B}$, and $PHH(\mathcal{B}_p)$ stands for the Hochschild homology of $\mathcal{B}_p$ with coefficients in the diagonal bimodule of $\mathcal{B}_p$ (see Appendix \ref{app:PHH}). 

\

Recall that $OC$ shifts degree up by $n \mod 2$. At the chain level, $OC$ consists of a family of linear maps $$OC_d\colon CF(L_0,L_1)\otimes \cdots \otimes CF(L_{d-1},L_d)\otimes CF(L_d,L_0)\to CQ(X,\Lambda) $$ for each tuple $\vec{L}=(L_0,\ldots, L_d)$ of Lagrangians in $\mathcal{B}$, where all the Floer chain complexes are defined with respect to the perturbation data $p$. 
The construction of these maps is based on moduli spaces and corresponding choices of perturbation data that follows the scheme used to define the filtered
$A_{\infty}$ operations $\mu_{d}$ in \S \ref{gio:ffuk} with a few modifications that we now outline. 

\begin{itemize}
\item[-]The perturbed $J$-holomorphic polygons $u$ used to define $OC_{d}$
only have negative strip-like ends near the punctures (by contrast, in the $\mu_{d}$ case, there is one positive strip-like end).
\item[-]  There is an interior marked point $x_{u}$ in the interior of the domain of $u$. 
\item[-] There is a fixed Morse smale pair $(f^X,g^X)$ on $X$ as in the definition $QH(X,\Lambda)$, and an $\omega$-compatible almost complex structure $J$ on $X$ as needed to define the quantum product on $QH(X,\Lambda)$. The equation satisfied by $u$ is $J$-holomorphic (without Hamiltonian perturbations)
in a small neighbourhood of $x_{u}$.
\item[-]  Given $$\gamma_1\otimes\cdots\otimes \gamma_{d+1}\in CF(L_0,L_1)\otimes \cdots \otimes CF(L_d,L_0)$$ and a critical point $x$ of $f^X$, there are moduli spaces $$\mathcal{M}^{d+1,1}_{\vec{L}}(\gamma_1,\ldots, \gamma_{d+1};x;q)$$ defined just like in the definition
of the $\mu_{d}$'s except that the output condition is replaced by the requirement that $u(x_{u})$ is carried by a trajectory of $-\nabla_{g^{X}}f^{X}$ to $x$.  
\item[-] $OC_{d}$ is defined by counting rigid curves as in the moduli space
above. Of course, these moduli spaces are defined with respect to regular perturbation data that is compatible  in the obvious sense with $p$.
\end{itemize}

\begin{figure}
	\includegraphics[scale=0.4]{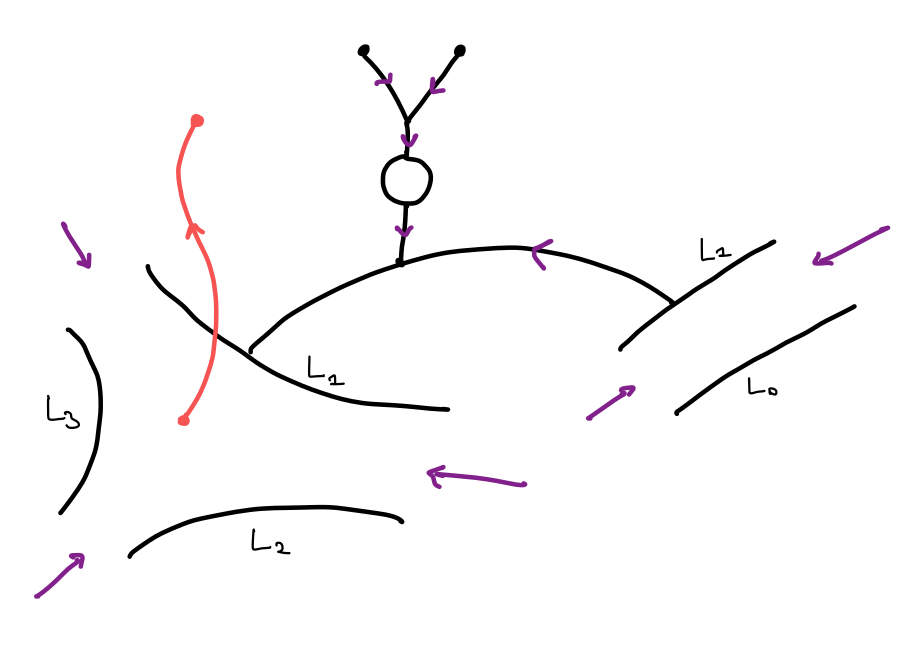}
	\centering
	\caption{A curve counted in the definition of $OC_7$ for the tuple $\vec{L}=(L_0,L_1,L_1,L_1,L_3,L_2,L_1,L_0)$.} \label{fig:OC}
\end{figure}

In this case (because there are no positive-strip-like ends), the arguments showing that the $A_{\infty}$ operations from \S\ref{gio:ffuk} are filtered apply directly.  As a result the $OC$ map described above induces  a morphism of persistence modules as in (\ref{eq:oc_map}). Moreover, we also have the following result whose proof is completely analogous to \cite[Corollary 2.11]{she:fanofuk}.
\begin{lem}
	For any $\mathcal{B}$ and any $p\in \mathcal{P}$, $\text{image}(OC^\mathcal{B}_p)\subset QH(X,\Lambda)_\mathbf{d}$, where $\mathbf{d}\in \Lambda_0$ is the parameter prescribing our class of monotone Lagrangians.
\end{lem}

\  

We end this section by briefly recalling from \cite{Bi-Co:rigidity} the definition of the linear part of the closed-open map. Let $p\in \mathcal{P}$ and consider the $\omega$-compatible almost complex structure $J_p^K$-prescribed by $p$ for $CF(K,K)$. Then $CO_p^K\colon CQ(X,\Lambda)\to CF(K,K)$ is defined by counting rigid $J_p^K$-pearly trajectories starting at a critical point of $f_p^X$ and ending at a critical point of $f^K_p$ (see \cite[Figure 3]{Bi-Co:Yasha-fest} with $x=e_K$). Since there is no Hamiltonian perturbation involved, it is straightforward to see that $CO_p^K$ induces a persistence map 

\begin{equation}\label{eq:co_map}
CO=CO_p^K\colon QH(X,\Lambda)\to HF(K,K).
\end{equation} Recall (see \cite[Proposition 2.9]{she:fanofuk}) that $CO$ restricted to $QH(X,\Lambda)_{\mathbf{d}'}$ vanishes unless $\mathbf{d}'=\mathbf{d}$. We recall that $CO$ is unital in the sense that it sends the projection to $QH(X,\Lambda)_{\mathbf{d}}$ of the unit to the unit in $HF(K,K)$.

\begin{rem}
	It is possible to define a version of the richer, more general, (non-linear) closed-open map, that is a persistence map $$CO^\mathcal{B}_p\colon QH^*(X;\Lambda)\to PHH^*(\mathcal{B}_p),$$ where $PHH^*(\mathcal{B}_p)$ denotes the persistence Hochschild cohomology of the $A_\infty$-category $\mathcal{B}_p$, but this goes beyond the scope of this paper.% The first author plans to include a discussion of the retract-approximability criterion using Hochschild cohomology and a persistence version of the Calabi-Yau morphism (see \cite{Gan:thesis, she:bigfuk}) in his Phd thesis.
\end{rem}

%\begin{prop}	Let $p\in \mathcal{P}$. There are persistence maps $$OC^\mathcal{B}_p:PHH_*(\mathcal{B}_p,\mathcal{B}_p)\to QH^{*+n}(X,\Lambda), \quad CO^\mathcal{B}_p\colon QH^*(X,\Lambda)\to PHH^*(\mathcal{B}_p,\mathcal{B}_p)$$ that compute the classical open-closed and closed-open maps \cite{Ab:geom-crit, she:fanofuk} (>>>>).\end{prop}
%We recall from \cite{she:fanofuk} that, in our compact setting, open-closed and closed-open maps are dual to each others.

%======================================
%======================================
%======================================

\subsubsection{Abouzaid retract-approximability criterion for a fixed choice of perturbations}\label{gio:aboappr}
Consider a finite family $\mathcal{B}\subset\lagmon$ of Lagrangians and an element $\mathcal{K}\in \lagmon\setminus \mathcal{B}$. We recall the class $\mathcal{P}(\mathcal{B})\subset\mathcal{P}$ of perturbations adapted to $\mathcal{B}$ defined in \S\ref{gio:coproduct}. Let $p\in \mathcal{P}(\mathcal{B})$. The coproduct $\Delta_p\colon \Delta_{\mathcal{B}_p}\to \mathcal{Y}_{p}^l(K)\otimes \mathcal{Y}_{p}^r(K)$ constructed as in \S\ref{gio:coproduct} (where we often dropped the subscript $p$ from the notation) induces a  filtered chain map, which we still denote by $\Delta_p$, $$\Delta_p\colon CC_*(\mathcal{B}_p)\to \Sigma^{-2\nu(p)}CC_*(\mathcal{B}_p,\mathcal{Y}^l(K)\otimes \mathcal{Y}^r(K))$$ (see \cite[Equation 2.178]{Gan:thesis}) defined by  $$\Delta_p(\gamma_1\otimes \cdots \otimes \gamma_{d+1}) := \sum_{i=0}^{d+1}\sum_{j=i}^{d+1}(\Delta_p)_{d-j|1|i}(\gamma_{j+1},\ldots, \gamma_{d+1},\gamma_1,\ldots,\gamma_i)\otimes \gamma_{i+1}\otimes \cdots \otimes \gamma_j.$$ Consider the filtered chain isomorphism $\varphi: CC_*(\mathcal{B}_p,\mathcal{Y}^l(K)\otimes \mathcal{Y}^r(K))\cong \overline{\mathcal{B}_p}(K,K)$ given by rotating pure tensors once to the left.  It is defined on a pure tensor $$a\otimes b\otimes \gamma_1\otimes \cdots \otimes \gamma_n\in  CC_*(\mathcal{B}_p,\mathcal{Y}^l(K)\otimes \mathcal{Y}^r(K))$$ as $$\varphi(a\otimes b\otimes \gamma_1\otimes \cdots \otimes \gamma_n):= b\otimes \gamma_1\otimes \cdots \otimes \gamma_n\otimes a.$$  In the following we will consider the composition $$(\Sigma^{-2\nu(p)}\varphi)\circ \Delta_p\colon CC_*(\mathcal{B}_p)\to \Sigma^{-2\nu(p)}\overline{\mathcal{B}_p}(K,K),$$ and, by a slight abuse of notation, still denote this by $\Delta_p$.

%\ocnote{Please write what this ``filterered chain isomorphism $\varphi$'' is as it's hard to follow, also write explicitly what is $\Delta_{p}$ in the next diagram - maybe like I did above}.  

\begin{prop}\label{gio:propabo}
	Let $p\in \mathcal{P}$. Then there is a commutative diagram of persistence modules and maps
	\begin{equation}		\label{persistencediagram} \xymatrixcolsep{3pc}
	\xymatrixrowsep{3pc} \xymatrix{PHH(\mathcal{B}_p) \ar[r]^-{\Delta_p}
		\ar[d]_-{OC^\mathcal{B}_p} & \Sigma^{-2\nu(p)} H(\overline{\mathcal{B}_p}(K,K))
		\ar[d]^-{\Sigma^{-2\nu(p)}\mu^{\overline{\mathcal{B}_p}}}\\
		QH(X,\Lambda) \ar[r]^-{\eta_{2\nu(p)} \circ CO^K_p} & \Sigma^{-2\nu(p)}HF(K,K) }
\end{equation}
	where $\eta_{2\nu(p)}\colon HF(K,K)\to \Sigma^{-2\nu(p)}HF(K,K)$ is the standard map induced by the structural maps of the persistence module $HF(K,K)$ as in \cite[Section 2.2.4]{BCZ:tpc}.
\end{prop}

\noindent Consider the measurement $R(-,-)$ introduced in (\ref{eq:en_gap}) 
and denote by $u_\mathbf{d}$ the projection of the unit in $QH(X,\La)$ to $QH(X,\La)_\mathbf{d}$.
\begin{cor}\label{gio:keycor}
	Let $p\in \mathcal{P}$. 
	If $R(u_\mathbf{d},OC^\mathcal{B}_p)\leq \alpha$, then $\lagmon$ is $\frac{\alpha}{2}+\nu(p)$-retract approximable in $PD(\mathcal{A}_p)$ by the family $\mathcal{B}$.
\end{cor}
The proof of the Corollary will make use of the following simple result.
\begin{lem}\label{ezlemma2}
	Let $f\colon V\to W$ be a map of persistence modules and consider $w\in W_r$. Given $\delta>0$ we have $$R(w,f)\leq R(i_{r,r+\delta}(w),\Sigma^{-\delta}f)+\delta.$$ 
\end{lem}
\begin{proof}
	Consider $\Sigma^{-\delta}f\colon \Sigma^{-\delta}V\to \Sigma^{-\delta}W$ and $i^W_{r,r+\delta}(w)\in (\Sigma^{-\delta}W)_r$. Let $s\geq r$ such that there is $v\in (\Sigma^{-\delta}V)_s$ with $(\Sigma^{-\delta}f)_s(v) = i^{\Sigma^{-\delta}W}_{r,s}(i^W_{r,r+\delta}(w))$. Since $i^{\Sigma^{-\delta}W}_{r,s} = i^{W}_{r+\delta,s+\delta}$ and $(\Sigma^{-\delta}V)_s=V_{s+\delta}$ the above is equivalent to $f_{s+\delta}(v) = i^W_{r,s+\delta}(w)$. In particular, $s+\delta\geq R(w,f)$.
\end{proof}
%\begin{proof}
%	Let $v'\in V'_{s'}$ for some $s'\geq {r'}$ such that $f'_{s'}%(v')=i^{W'}_{r',s'}(w')$. Then $$f_{s'}\circ h^V_{s'}(v') = %h^{W}_{s'}\circ f'_{s'}(v') =  h^W_{s'}\circ i^{W'}_{r',s'}(w') %= i^{W}_{r',s'}\circ h^W_{r'}(w') = i^{W}_{r',s'}\circ i^{W}%_{r,s'}(w) = i^{W}_{r,s'}(w)$$ proving the claim.
%\end{proof}
\begin{proof}[Proof of Corollary \ref{gio:keycor}]	
%We return to the proof of the Corollary and a
Assume that $R(u,OC^\mathcal{B}_p)\leq \alpha$. By Lemma \ref{ezlemma}, Proposition \ref{gio:propabo}, Lemma \ref{ezlemma2} and the fact that $CO$ is unital, we directly get $R([e_K],[\mu^{\overline{\mathcal{B}}}])\leq \alpha+\nu(p)$ for all $K\in \lagmon$. Using Proposition \ref{gio:algfilabo} this implies the claim.	
\end{proof}
\begin{proof}[Proof of Proposition \ref{gio:propabo}.]
The proof is essentially a combination of the proofs of Theorem 8.1.1 in \cite{Bi-Co:Yasha-fest} and Lemma 2.15 in \cite{she:fanofuk} while keeping track of the changes in filtration. 

We outline the construction of a chain homotopy $H\colon CC_*(\mathcal{B}_p)\to \Sigma^{-2\nu(p)}CF(K,K)$
that makes  Diagram (\ref{persistencediagram}) commutative. 
Given a tuple $\vec{L}:=(L_0,\ldots, L_d)$ of elements chosen from the family $\mathcal{B}$, we intend to define $$H\colon CF(L_0,\ldots, L_d,L_0)\to \Sigma^{-2\nu(p)}CF(K,K).$$ This again consists of three steps: picking the correct moduli space (this part is already present in the literature);
adjust the spaces of perturbations conveniently so that the construction is compatible with
filtrations (following the methods from \cite{Amb:fil-fuk}); ensure that the energy bounds are
indeed as required so that the resulting chain homotopy preserves filtrations.

The source spaces for this $H$ are parametrized by a moduli space denoted $\mathcal{R}_C^{d,A}(\vec{L},K)$ (the superscript $A$ indicates that \textit{annuli} appear in the domain) labelled along the boundary by $\vec{L}$ and $K$ and defined as follows:
\begin{enumerate}
	\item We consider the moduli space $\mathcal{R}^{d,A}(\vec{L},K)$ of annular configurations
consisting of an annulus $S^{1}\times [1,\rho]\subset \mathbb{C}$ of unit internal radius and with outside radius $\rho\in (1,\infty)$.  We fix marked points  $z_0=1$, $z_1=-\rho$ as well as $d-1$ other marked points  $z_2,\ldots, z_d$, $|z_{i}|=\rho,\ 1\leq i \leq d$ along the exterior boundary of the annulus, ordered in clockwise direction starting from $z_{1}$.  We  label  the arcs on the exterior circle clockwise by $\vec{L}$ starting from $x_{1}$, and we label the internal circle by $K$.
	\item Over $\mathcal{R}^{d,A}(\vec{L},K)$ we have a bundle $$\pi^{d,A}(\vec{L},K)\colon \mathcal{S}^{d,A}(\vec{L},K)\to \mathcal{R}^{d,A}(\vec{L},K)$$ such that the fiber of a point in
	$\mathcal{R}^{d,A}(\vec{L},K)$  is a boundary-punctured annulus with the interior circle of radius one, the exterior circle of radius some $\rho\in (1,\infty)$ and such that the marked points on the exterior circle that are adjacent to two geometrically distinct Lagrangian labels are replaced with punctures.
	\item We endow $\pi^{d,A}(\vec{L},K)$ with a compatible choice of strip-like ends, all of negative type, associated with each of the punctures at the previous point.
	\item We consider the compactification $\overline{\mathcal{R}_C^{d,A}(\vec{L},K)}$ of $\mathcal{R}_C^{d,A}(\vec{L},K)$. This moduli space of clusters mixes flow lines, disks, and polygons, just as in the definition of the operations $\mu_{d}$ from \S\ref{gio:ffuk}, but also has one annular component.  The boundary of $\overline{\mathcal{R}^{d,A}_C(\vec{L},K)}$ in the case $d=1$ is drawn in Figure \ref{fig:deg1}.
	\item There are incidence conditions for each of the marked points and punctures. The punctures correspond to generators of the respective Floer complexes $CF(L_{i},L_{i+1};p)$ (with $p$ the perturbation data); the marked points on the exterior circle (they are so that the adjacent edges have the same label, $L_{j}$, for some $j$) are mapped to critical points of the fixed Morse
	function on $L_{j}$ that is part of $p$; the marked point $z_{o}$ is mapped to a critical
	point of the fixed Morse function on $K$, which is part of the perturbation data $p$. Perturbation data for all of this is picked such that it is compatible with $p\in \mathcal{P}$ and are constructed as in \cite{Amb:fil-fuk}.
	\end{enumerate}
	\begin{figure}\label{deg1}
		\includegraphics[scale=0.7]{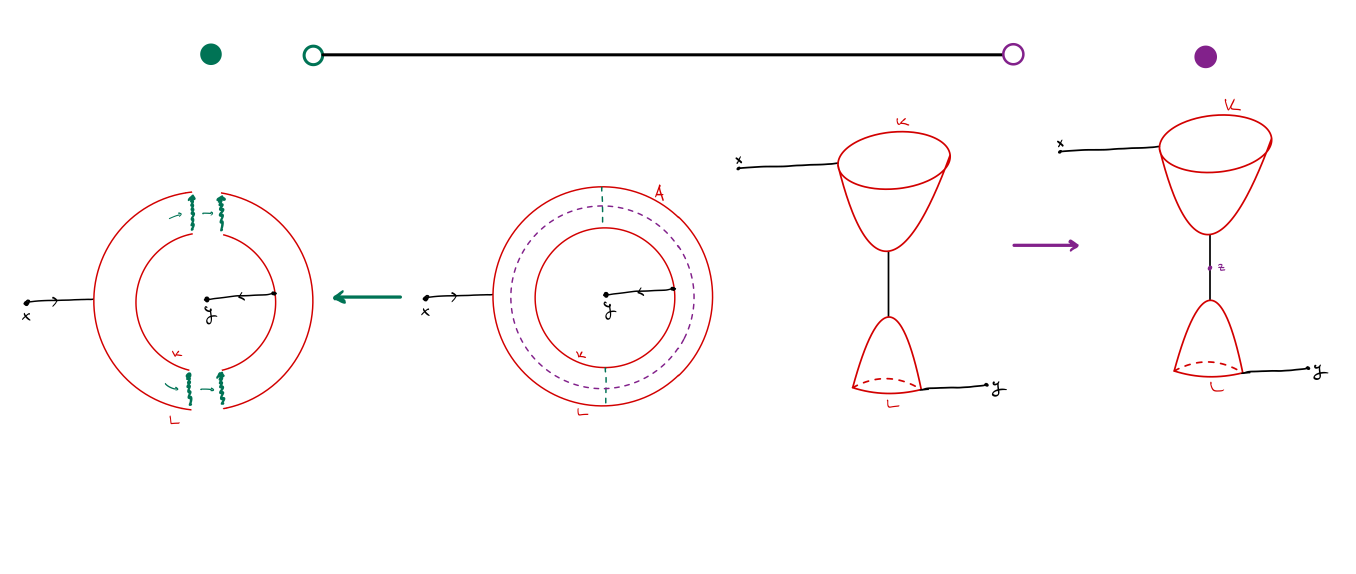}
		\centering
		\caption{A schematic description of the interior of $\overline{\mathcal{R}^{1,A}_C(L,K)}$, isomorphic to a closed interval, and the fibers over it. In the middle, given Lagrangians $L$ and $K$ we see configurations contributing to the $y$-summand (where $y\in \text{Crit}(f_K)$) of $H(x)$ (where $x\in \text{Crit}(f_L)$). On the left (green dot) we see the configuration contributing to $\mu_2\circ \Delta$, while on the right (purple dot) a configuration contributing to $OC\circ CO$ (breaking at a critical point $z$ of $f_X$).} \label{fig:deg1}
	\end{figure}

The construction sketched above appeared before in the literature -  see \cite{she:fanofuk} as well as \cite{Bi-Co:Yasha-fest} - and the structure of the compactification $\overline{\mathcal{R}_C^{d,A}(\vec{L},K)}$ implies that it is a chain homotopy as desired. From the point of view of the current 
paper the only additional point of interest is that no further modifications are needed to ensure
that this homotopy is filtration preserving as a map to $\Sigma^{-2\nu(p)}CF(K,K)$ (that is, it shifts filtration by $\leq 2\nu(p)$ as a map to $CF(K,K)$). 
\end{proof}

\begin{rem}
	a. The statement of Proposition  \ref{gio:propabo} contains an intrinsic limitation in the 
	sense  that for a fixed perturbation $p$ it allows to discuss retract-approximability with accuracy bounded below by a constant times $\nu(p)$. Gaining adequate control on the structures associated with a $p$ that varies so that the accuracy level gets below some fixed $\epsilon$ was the main motivation behind the notion of systems of categories with increasing accuracy in \S\ref{s:sys-tpc}. 
	
	b. We could refine the statement of Proposition \ref{gio:propabo} and avoid the shift $\Sigma^{-2\nu(p)}$ in the bottom right of Diagram (\ref{persistencediagram}). This can be done by choosing more elaborate perturbation data for the maps appearing in the diagram. This would have the result of allowing for a slightly better approximation accuracy given a fixed perturbation datum $p$, but the remark above would continue to apply.
	
	b. Note that when $\mathcal{B}$ consists of a single Lagrangian $L$, then $R(u_\mathbf{d}, OC^\mathcal{B}_p)\leq \alpha$ implies that $\lagmon$ is $\frac{\alpha}{2}$-retract approximated in $PD(\mathcal{A}_p)$ by the family $\mathcal{B}$. That is, there is no $\nu(p)$ term appearing in this case, in contrast to the general statement in Corollary \ref{gio:keycor}. The reason is that, in this case, the map $\Delta$ is filtered (see Proposition \ref{coproductmain}); the proof uses a version of Proposition \ref{gio:aboappr} where no shift functor $\Sigma^{-2\nu(p)}$ appears. The only place where this refinement will be used is Proposition \ref{singlelagsphere}.
\end{rem}

\subsubsection{Colimit of open-closed maps}

In this section we prove that the open-closed map behaves well with respect to continuation functors and define
the colimit open-closed map that appears in the statement of Theorem \ref{gio:wabomain}.

Consider the system $$\widehat{\fuk}(\lagmon)= \left\{\left\{\fuk(\lagmon;p)\right\}_{p\in \mathcal{P}}, \left\{\mathcal{H}_{p,q},\mathcal{H}_{q,p}\right\}_{p\preceq q}\right\}$$ of filtered Fukaya categories associated with $\lagmon$. As explained in \S\ref{sb:sys-fuk}, this is in fact a homotopy system in the sense of \S\ref{sbsb:homotopy-sys}.
As a result of that, via Proposition \ref{p:homotopy-HH}, we get a directed system $$\left\{\left\{PHH(\fuk(\lagmon;p))\right\}_{p\in \mathcal{P}}, \left\{\mathcal{H}_{p,q}^{PHH}\right\}_{p\preceq q}\right\} $$ of persistence modules, where $\mathcal{H}_{p,q}^{PHH}$ is the persistence map induced in persistence Hochschild homology by the (family of) continuation functors $\mathcal{H}_{p,q}$ ($\mathcal{H}_{p,q}^{PHH}$ is uniquely defined by Proposition \ref{p:homotopy-HH}). In particular, the colimit $$PHH(\widehat{\fuk}):=\colim_{p\in \mathcal{{P}}}PHH(\fuk(\lagmon;p))$$ is well-defined and a persistence module, as explained in \S\ref{sb:sys-inv}. Moreover, since the system $\widehat{\fuk}(\lagmon)$ has, in particular, a fixed full base of objects (see \S\ref{sbsb:sys-base}), the following is true: given a subfamily $\mathcal{B}\subset \lagmon$ we get a colimit $$PHH(\widehat{\mathcal{B}}):=\colim_{p\in \mathcal{P}}PHH(\mathcal{B}_p)$$ where $\mathcal{B}_p$ is the full $A_\infty$-subcategory of $\fuk(\lagmon;p)$ induced by $\mathcal{B}$ and $\widehat{\mathcal{B}}$ is the subsystem of $A_\infty$-categories of $\widehat{\fuk}$ induced by $\{\mathcal{B}_p\}_{p\in \mathcal{P}}$.\\

For any $N\geq 1$, we denote by $F^{N}PHH(\mathcal{B}_p)$ the persistence homology of $F^{N}CC(\mathcal{B})$, that is, the $N$-th level in the length filtration of the Hochschild chain complex (see \S\ref{app:PHH}). \\ 
Let $p\preceq q$ and consider the continuation $A_\infty$-functor $\mathcal{H}_{q,p}\colon \mathcal{B}_q\to \mathcal{B}_p$ restricted to $\mathcal{B}_q$. We recall (see \S\ref{sb:sys-fuk}, \cite{Amb:fil-fuk}) that these functors have linear deviation rate $\leq 2(\nu(p)-\nu(q))$. We write $s(q,p):= 2(\nu(p)-\nu(q))$. In particular, $\mathcal{H}_{q,p}$ does \textit{not} induce a map in persistence Hochschild homology, as the shift in filtration at the chain level blows up as $N\to \infty$. To overcome this problem, we use the length filtration on $CC(\mathcal{B}_q)$. Note that, at the chain level, $\mathcal{H}_{q,p}$ induces for any $N\geq 1$ a chain map $F^NCC(\mathcal{B}_q)\to F^NCC(\mathcal{B}_q)$ of shift $\leq Ns(q,p)$. In particular, by precomposing and postcomposing with structural maps $\eta$ associated with the shift functor (see \S\ref{subsec:PC}), $\mathcal{H}_{q,p}$ induces a persistence map $$\mathcal{H}_{q,p}^{F^NPHH}\colon \Sigma^{-2N\nu(q)}F^NPHH(\mathcal{B}_q)\to \Sigma^{-2N\nu(p)}F^NPHH(\mathcal{B}_p).$$
The next lemma lemma follows directly from the properties of the continuation functors.
\begin{lem}\label{gio:homcont}
	Let $p\in \mathcal{P}$. The family $$\left\{\mathcal{H}_{q,p}^{F^NPHH}\colon \Sigma^{-2N\nu(q)}F^NPHH(\mathcal{B}_q)\to \Sigma^{-2N\nu(p)}F^NPHH(\mathcal{B}_p)\right\}_{q\in \mathcal{P}, \ p\preceq q}$$ defines a homomorphism of directed systems of persistence modules.
	%By the properties of continuation functors and the definition of the colimit persistence module $PHH(\widehat{\mathcal{B}})$, t
	%In particular, the map $\widehat{OC}^\mathcal{B}$ is well-defined and the maps $\mathcal{H}_{q,p}^{F^NPHH}$ induce a persistence map  $$H^{F^NPHH}_{\infty,p}\colon \Sigma^{-2N\nu(p)}F^NPHH(\widehat{\mathcal{B}})\to F^NPHH(\mathcal{B}_p)$$ for all $p\in \mathcal{P}$. Moreover, the persistence diagram 
	
\end{lem}

Exploiting again the fact that the length filtration is preserved by maps induced on Hochschild chains by $A_\infty$-functors, we define $$F^NPHH(\widehat{\mathcal{B}}):=\colim_{p\in \mathcal{P}}F^NPHH(\mathcal{B}_p)$$ for any $N\geq 1$. 

As a result of Lemma \ref{gio:homcont} and of the fact that $\nu(q)\to 0$ as $q\to \infty$, we get a unique map $$H^{F^NPHH}_{\infty,p}\colon F^NPHH(\widehat{\mathcal{B}})\to \Sigma^{-2N\nu(p)}F^NPHH(\mathcal{B}_p)$$ for any $p\in \mathcal{P}$, such that the diagram 
\begin{equation}		\label{gio:colimitcont} \xymatrixcolsep{3pc}
	\xymatrixrowsep{3pc} \xymatrix{\Sigma^{-2N\nu(q)}PHH(\mathcal{B}_q) \ar[rr]%^-{\Delta_p}
		\ar[dr]_-{\mathcal{H}^{F^NPHH}_{q,p}} & & F^NPHH(\widehat{\mathcal{B}})
		\ar[dl]^-{\mathcal{H}^{F^NPHH}_{\infty,p}}\\
		& \Sigma^{-2N\nu(p)}F^NPHH(\mathcal{B}_p)& }
\end{equation}
commutes for all $q\in \mathcal{P}$ with $p\preceq q$.

So far, what we proved are intrinsic properties of homotopy systems, with no geometric considerations. We now move to geometry and bring in the open-closed maps.

We consider the quantum cohomology $QH(X,\Lambda)$ as a trivial directed system of persistence modules parametrized by $\mathcal{P}$.
\begin{lem}\label{gio:homoc}
	The family of persistence maps $$\left\{OC^\mathcal{B}_p\colon PHH(\mathcal{B}_p)\to QH(X,\Lambda)\right\}_{p\in \mathcal{P}}$$ defines a homomorphism of directed systems of persistence modules. 
\end{lem}
\begin{proof}
	To show that $\{OC^\mathcal{B}_p\}_p$ is a homomorphism of directed systems, we have to prove that $OC^\mathcal{B}_q\circ \mathcal{H}^{PHH}_{p,q}=OC^\mathcal{B}_p$. To this aim, one defines a chain homotopy between the two maps at the chain level, which are filtered chain maps, and shows that it is filtered. This is very similar to the proof of Proposition \ref{gio:propabo} and is omitted here.
\end{proof}

%In Lemma \ref{gio:occont} below we prove that That is, we prove that $$OC^\mathcal{B}_q\circ \mathcal{H}_{p,q}^{PHH} = OC^\mathcal{B}_p$$ for all $p\preceq q$. %In particular, if $i_\infty\circ OC^\mathcal{B}_p$ hits the unit for some $p\in \mathcal{P}$, then the same holds for all other $q\in \mathcal{P}$.
As a result of Lemma \ref{gio:homoc}, we obtain a unique persistence morphism $$\widehat{OC}^\mathcal{B}:=\colim_{p\in \mathcal{P}}OC^\mathcal{B}_p\colon PHH(\widehat{\mathcal{B}})\to QH(X,\Lambda)$$ such that the diagram 
\begin{equation}		\label{gio:colimitoc} \xymatrixcolsep{3pc}
	\xymatrixrowsep{3pc} \xymatrix{PHH(\mathcal{B}_p) \ar[rr]%^-{\Delta_p}
		\ar[dr]_-{OC^\mathcal{B}_p} & & PHH(\widehat{\mathcal{B}})
		\ar[dl]^-{\widehat{OC}^\mathcal{B}}\\
		& QH(X,\Lambda)& }
\end{equation}
commutes for all $p\in \mathcal{P}$. The morphism $\widehat{OC}^\mathcal{B}$ is called the \textit{colimit open-closed map} for $\mathcal{B}$ and is the persistence map appearing in the statement of Theorem \ref{gio:wabomain}.\\

In fact, we will prove more.

\begin{lem}
	Let $p\in \mathcal{P}$. The persistence diagram \begin{equation}		\label{gio:colimitoc2} \xymatrixcolsep{3pc}
		\xymatrixrowsep{3pc} \xymatrix{\Sigma^{-2N\nu(p)}F^NPHH(\mathcal{B}_p) %^-{\Delta_p}
			\ar[d]_-{\Sigma^{-2N\nu(p)}OC^\mathcal{B}_p}  & &F^NPHH(\widehat{\mathcal{B}}) \ar[ll]^-{\Sigma^{-2N\nu(p)}\mathcal{H}^{F^NPHH}_{\infty,p}}
			\ar[d]^-{\widehat{OC}^\mathcal{B}}\\
			\Sigma^{-2N\nu(p)}QH(X,\Lambda) & & QH(X,\Lambda) \ar[ll] }
	\end{equation} commutes, where the bottom map is induced by the structural maps of the persistence module $QH(X,\Lambda)$.
\end{lem}
\begin{proof}
	Let $p,q$ in $\mathcal{P}$ such that $p\preceq q$. Consider the persistence diagram
	\begin{equation}		\label{gio:colimitoc3} \xymatrixcolsep{3pc}
		\xymatrixrowsep{3pc} \xymatrix{\Sigma^{-2N\nu(p)}F^NPHH(\mathcal{B}_p) %^-{\Delta_p}
			\ar[d]_-{OC^\mathcal{B}_p}  & & F^NPHH(\mathcal{B}_q) \ar[ll]^-{\mathcal{H}^{F^NPHH}_{q,p}}
			\ar[d]^-{OC^\mathcal{B}_q}\\
			\Sigma^{-2N\nu(p)}QH(X,\Lambda) & & QH(X,\Lambda) \ar[ll] }
	\end{equation}
	This diagram commutes, and the proof goes along the same lines as the proof of Lemma \ref{gio:homoc}, with the slight complication that one has to keep track of the shift in filtration. For varying $q$ such that $p \preceq q$, the above diagram induce a commutative diagram of direct systems of persistence modules. Note that the colimit of $\Sigma^{Ns(q,p)}QH(X,\Lambda)$ over this family is equal to $\Sigma^{-2N\nu(p)}QH(X,\Lambda)$. This proves the claim.
\end{proof}

\subsubsection{Proof of Theorem \ref{gio:wabomain}}\label{subsec:proof_split_app}

It is known that the infinity-level open closed map provides an isomorphism $PHH^\infty(\mathcal{B}_p)\cong QH^\infty(X,\Lambda)$ as vector spaces \cite{Gan:thesis}, and that the unit $u\in QH^0(X,\Lambda)$ persists to $QH^\infty(X,\Lambda)$. Any element in $PHH^\infty(\mathcal{B}_p)$ can be represented by a Hochschild chain of length $\leq N_p$ for some $N_p\geq 0$. Moreover, we can choose $N_p$ such that the spectral invariants of $F^{N_p}PHH(\mathcal{B}_p)$ are optimal, i.e. they agree with those of $PHH(\mathcal{B}_p)$ for for the homology classes that persist to infinity in both persistence modules. We claim that $N_p$ is independent of $p\in \mathcal{P}$. This follows from the following fact: by forgetting filtrations, the functors $\mathcal{H}_{q,p}$ induce a unique map $PHH^\infty(\mathcal{B}_q) \to PHH^\infty(\mathcal{B}_p)$ which is an isomorphism  (since, forgetting filtrations, continuation funtors are quasi-equivalence of $A_\infty$-categories) and, hence, preserve the length filtration (note that, of course, all of this holds for the maps induced by $\mathcal{H}_{p,q}$ on $PHH^\infty$). We write $N_\mathcal{B}:=N_p$.

At this point, the proof of Theorem \ref{gio:wabomain} becomes a matter of book-keeping.

	Let $\mathcal{F}^\epsilon\subset \lagmon$ as in the statement. We work with the notation above in the case $\mathcal{B}=\mathcal{F}^\epsilon$. Consider the constant $N_{\mathcal{F}^\epsilon}$ introduced above. Let now $a\in PHH^\epsilon(\widehat{\mathcal{F}^\epsilon})$ be such that $\widehat{OC}(a)=i_{o,\epsilon}(u_\mathbf{d})$, where $u_\mathbf{d}\in QH^0(X,\Lambda)_\mathbf{d}$ is the projection of the cohomological unit in $QH(X,\Lambda)$. From the above it follows that $a$ can be seen as an element in $F^{N_{\mathcal{F}^\epsilon}}PHH(\widehat{\mathcal{F}^\epsilon})$. From this and the commutativity of Diagram (\ref{gio:colimitoc2}) we get $$R(u, OC^{\mathcal{F}^\epsilon}_p)\leq \epsilon+2N_{\mathcal{F}^\epsilon}\nu(p)$$ for all $p\in \mathcal{P}$. Using Proposition \ref{gio:propabo}, we conclude that $\lagmon$ is $\frac{\epsilon}{2}$-retract approximable by $\mathcal{F}^\epsilon$ in the system $\widehat{\fuk}(\lagmon)$ in the sense of Definition \ref{d:approx-sys-ret}. \qed

\begin{rem}
%	a. \TBD The constant $N_{\mathcal{F}^\epsilon}$ appearing in the proof of Theorem \ref{gio:wabomain} depends on the family $\mathcal{F}_\epsilon$, but not really on $\epsilon$: if $\mathcal{F}\subset \lagmon$ is a family such that for all $\epsilon>0$ there is a \textit{finite} subfamily $\mathcal{F}_\epsilon$ satisfying the assumptions of Theorem \ref{gio:wabomain}, then $\mathcal{F}$ approximates $\lagmon$ in $\widehat{\fuk}(\lagmon)$. Indeed, the constant $N_{\mathcal{F}_\epsilon}$ is guaranteed not to explode if the add "useless" (in the sense of $\epsilon$-approximation) Lagrangians to the family $\mathcal{F}_\epsilon$ as $\epsilon\to 0$.
%	\\
%	b. 
	The complex formalism in the proof of Theorem \ref{gio:wabomain} is due to the lack of a concept of colimit of homotopy system of filtered $A_\infty$-categories. Indeed, the open-closed map appearing in the statement of Theorem \ref{gio:wabomain} should be thought as the open-closed map associated with a limit Fukaya category (see the discussion at the beginning of \S\ref{sb:sys-inv}).
\end{rem}

%	In this section we prove that the open-closed map behaves well with respect to continuation functors and prove Theorem \ref{gio:wabomain}.
%	\begin{rem}
%		The complications in the proof of Theorem \ref{gio:wabomain} are due to the lack of a concept of colimit of homotopy system of filtered $A_\infty$-categories. Indeed, the open-closed map appearing in the statement of Theorem \ref{gio:wabomain} should be tought as the open-closed map associated with a limit Fukaya category.
%	\end{rem}
%	\newpage
%	
	
	%=======================================
	%=======================================
	%=======================================
	
	\subsection{Proof of Corollary \ref{S2T2}}\label{gio:proofS2T2}
	% !TEX root = approx8.tex

	We split the proofs of the two geometric situations into subsections. The basic idea is the same: cut the ambient manifold into more and more equally spaced straight slices. We will keep on writing Fukaya categories as $\mathcal{A}_p$ and systems of Fukaya categories as $\widehat{\mathcal{A}}$ as indicated at the beginning of Section \ref{sec:split-app}.
\subsubsection{Approximability of equators on the $2$-sphere} In what follows we will refer to a closed monotone Lagrangian $L\subset S^2$ as an "equator". Note that these Lagrangians are precisely the embedded curves in $S^2$ which separate it into two domeains of equal area.

The first Chern class of $S^2$ vanishes in $H^2(S^2;\mathbb{Z}_2)$, so the eigenspace of the quantum homology $QH(S^2,\Lambda)$ associated to the zero eigenvalue is the whole space. \\ In this subsection, we will only consider (without loss of generality) perturbation data $p\in \mathcal{P}$ satisfying the following two conditions:
\begin{enumerate}
	\item Let $L\in \lag^{\text{(mon,} \mathbf{0} \text{)}}(S^2)$ be an equator on $S^2$ and consider the Morse function $f^L_p\colon L\to \mathbb{R}$ in the $p$-Floer datum of $L$. We assume that $f^L_p$ has a unique maximum $e_L\in L$ and a unique minimum $\text{pt}_L\in L$. We will always assume that both $e_L$ and $\text{pt}_L$ are neither the north nor south poles of $S^2$.
	\item Consider the Morse function $f^{S^2}_p\colon S^2\to \mathbb R$ in the $p$-Floer datum for $S^2$. We assume that $f^{S^2}_p$ has a unique maximum $u_{S^2}\in S^2$ and a unique minimum $\text{pt}_{S^2}\in S^2$ and no other critical points. We further assume that both $u_{S^2}$ and $\text{pt}_{S^2}$ differ from both the north and south poles of $S^2$, 
\end{enumerate}
We will abuse notation and write this "restriction" of $\mathcal{P}$ again as $\mathcal{P}$ and consider the homotopy system of $A_\infty$-categories $\widehat{\mathcal{A}}= \widehat{\fuk}(\lag^{\text{(mon,} \mathbf{0} \text{)}}(S^2))$ as parametrized by this restricted $\mathcal{P}$.
\begin{prop}\label{singlelagsphere}
	Let $L\in \lag^{\text{(mon,} \mathbf{0} \text{)}}(S^2)$ be an equator on $S^2$. Then $\lag^{\text{(mon,} \mathbf{0} \text{)}}(S^2)$ is $\frac{1}{4}$-retract approximable by $L$.
	%for any perturbation datum $p\in E^{sg}$, the open closed-map $OC_p^L\colon HH(L,L)\to QH(S^2)$ satisfies $R(u_{S^2},OC_p^L)\leq \frac{1}{2$}$. {\color{orange} Maybe write it as "Then, for any perturbation datum $p\in E^\mu$ we have $$R(u,OC_p^L)=\frac{1}{2}.$$}
\end{prop}

\begin{proof}
	%We assume that the Morse function $f_L$ associated to the Floer datum of $p$ on $L$ is the height function. Denote by $e_L$ the maximum of $f_L$ and by $\text{pt}_L$ its minumum. Then $CF(L,L)= \Lambda\cdot e_L\oplus \Lambda\cdot \text{pt}_L$. As above, we will work with the height function $f_{S^2}$ on the sphere to define quantum homology, which has a unique maximum $u_{S^2}$ and a unique minumim $\text{pt}_{S^2}$. In particular $CQ(M) = \Lambda\cdot u_{S^2}\oplus \text{pt}_{S^2}$.\\
	%Consider the length filtration $(F^NCC_*(L,L))$ on the Hochschild chain complex as introduced in Section \ref{persistencehochschild}, and the restristion $F^1OC_p^L$ of the chain-level open closed map to $F^1CC_*(L,L)$.
	Let $p\in \mathcal{P}$. It is easy to see
	%that $$OC_p^L(\text{pt}_L\otimes e_L) = OC_p^L(e_L\otimes \text{pt}_L) = OC_p^L(e_L\otimes e_L) = 0 $$
	that at the chain level we have $$OC_p^L(\text{pt}_L\otimes \text{pt}_L) = T^{\frac{1}{2}}u_{S^2}.$$ Moreover, since $d_{CC}$ is cyclic, $\text{pt}_L\otimes \text{pt}_L$ is a cycle in $CC_*(L,L)$. Note that the persistence Hochschild homology $PHH_*(L_p)$ does not depend on $p\in \mathcal{P}$. The claim then follows by Theorem \ref{gio:wabomain}.
	%All of the listed elements in $F^1CC(L,L)$ are cycles as $d_{CC}$ is cyclic. The reason behind the fact that there is no $\text{pt}_{S^2}$ contribution in $OC_p^L(\text{pt}_L\otimes \text{pt}_L)$ has to do with transversality of pearly trajectories; for the same reason that there is no ${\text{pt}}_{S^2}$ contribution in the quantum product ${\text{pt}}_{S^2}*{\text{pt}}_{S^2}= Tu_{S^2}$. {\color{orange} see appendix B for the converse.}
\end{proof}
\begin{rem}
a. We actually proved that for a single equator the constant $\delta>0$ appearing in Definition \ref{d:approx-sys-ret} can be taken to be arbitrarily big. This is due to the fact that for the self-Floer homology of $L$ we do not need any Hamiltonian perturbation.

b. The linear component $ HF(L,L)\to QH(S^2)$ of the open-closed map does not hit the unit in our setting: indeed $OC(e_L)=0$ by a degree argument while $OC(\text{pt}_L) = \text{pt}_{S^2}+T^{\frac{1}{2}}u_{S^2}$. Note however, that if we would have worked with coefficients in the Novikov field over the real numbers, the situation would have been different, as $T^{-1\frac{1}{2}}\text{pt}_{S^2}+u_{S^2}$ is the projection of the unit in the zero eigensummand of quantum homology of $\mathcal{C}P^1$ over this coefficient field.\\ Note that as $\text{pt}_{S^2}+T^{\frac{1}{2}}u_{S^2}$ and $u_{S^2}$ are non-homologous cycles in $CQ(S^2)$, it follows that $[\text{pt}_{L}]$ and $[\text{pt}_{L}\otimes \text{pt}_{L}]$ generate $HH(L)$.

c. As we know that $OC^L\colon HH(L)\to QH(S^2)$ is an isomorphism, it follows from the computation above that $e_L\in PHH^\infty(L,L)$ has to be a boundary. In fact, it equals $T^{-\frac{1}{2}}d_{CC}(a_L\otimes a_L\otimes a_L)$, where $a_L = \text{pt}_L+T^{\frac{1}{4}}e_L$; indeed notice that $a_L\in CF(L,L)$ is a cycle satisfying $$\mu_k(a_L,\ldots, a_L) =T^{\frac{1}{2}}e_L \text{  for any }k\geq 2.$$
Hence, $[e_L]$ has boundary depth $\frac{1}{2}$ in $PHH(L)$. Note that $[e_L\otimes \text{pt}_L]$ also has boundary depth equal to $\frac 12$. Indeed $$d_{CC}(a_L^{\otimes 4}) = T^{\frac{1}{2}}(e_L\otimes \text{pt}_L + \text{pt}_L\otimes e_L)$$ and $\text{pt}_L\otimes e_L = d_{CC}(\text{pt}_L\otimes e_L\otimes e_L)$.% This behaviour is further studied in \TBD.
\end{rem}

We will now show that we can approximate the Fukaya category of $S^2$ with better accuracy by considering larger approximating families. Fix an equator $S^1\subset S^2$ as a reference and consider the set $\mathcal{E}\subset \lag^{\text{(mon,} \mathbf{0} \text{)}}(S^2)$ great circles which pass trough the north $n\in S^2$ and south pole $s\in S^2$. %We will pick the orientation on each $L\in \mathcal{E}$ obtained by clockwise rotations of $S^1$ looking from the north pole. 
Of course, any two Lagrangians in $\mathcal{E}$ intersect transversally. Recall that in our setting, given a perturbation datum $p\in \mathcal{P}$, the Hamiltonian part of the $p$-Floer datum associated to a pair $(L_0,L_1)$ of transversally intersecting Lagrangians is assumed to be constant (see page \pageref{gio:FD}), and this constant is $<\nu(p)$, the size of $p$, by definition. Moreover, we require $H^{L_0,L_1}_p=H^{L_1,L_0}_p$. Hence the Floer complexes $CF(L_0,L_1;p)$ and $CF(L_1,L_0;p)$ are generated by the (finitely many) intersection points in $L_0\cap L_1$, and these lie at the same filtration level. Suppose $L_1\in \mathcal{E}$ is obtained by rotating another equator $L_0\in \mathcal{E}$ by an angle smaller than $\pi$. Of course, $L_0\cap L_1$ consists of the north and south poles only: we work with the convention that in the complex $CF(L_0,L_1;p)$ the north pole $n_{0}\in CF(L_0,L_1;p)$ has grading $1$ and the south pole $s_{0}\in CF(L_0,L_1;p)$ has grading $0$, while in the complex $CF(L_1,L_0;p)$ this is viceversa (and the poles are denoted by $n_{1}$ and $s_{1}$ in this complex).

The relevant notion of approximability for the next result is the retract version of Definition \ref{def:simple_approx}.
\begin{prop}\label{OCS2}
	Let $p\in \mathcal{P}$ and $N\geq 2$. Consider a finite family of great circles $\mathcal{E}(N)=\{L_1,\ldots, L_N\}\subset \mathcal{E}$ on $S^2$, all passing through the north and south poles and such that $L_i$ lies at angle difference $\frac{\pi}{2N}$ from $L_{i+1}$. Then $\lag^{\text{(mon,} \mathbf{0} \text{)}}(S^2)$ is retract-approximable in $\mathcal{A}_p$ by the family $\mathcal{E}(N)$ with accuracy $\frac{1}{4N}+2\nu(p)$.
	%{\tiny $$OC^{\mathcal{E}(N)}_p\colon HH^{\leq \frac{1}{2N} + 2\varepsilon}(\mathcal{E}(N))\to QH(S^2)$$ hits the unit. Moreover, given any subfamily $\tilde{\mathcal{B}}$ of cardinality $N$ and any $\alpha< \frac{1}{2N}+2\varepsilon$, the map $$OC^{\tilde{\mathcal{B}}}_p\colon HH^{\leq \alpha}(\tilde{\mathcal{B}})\to QH(S^2)$$ does not hit the unit. {\color{orange} see orange comment in above prop}}
\end{prop}
\begin{figure}
	\includegraphics[scale=0.3]{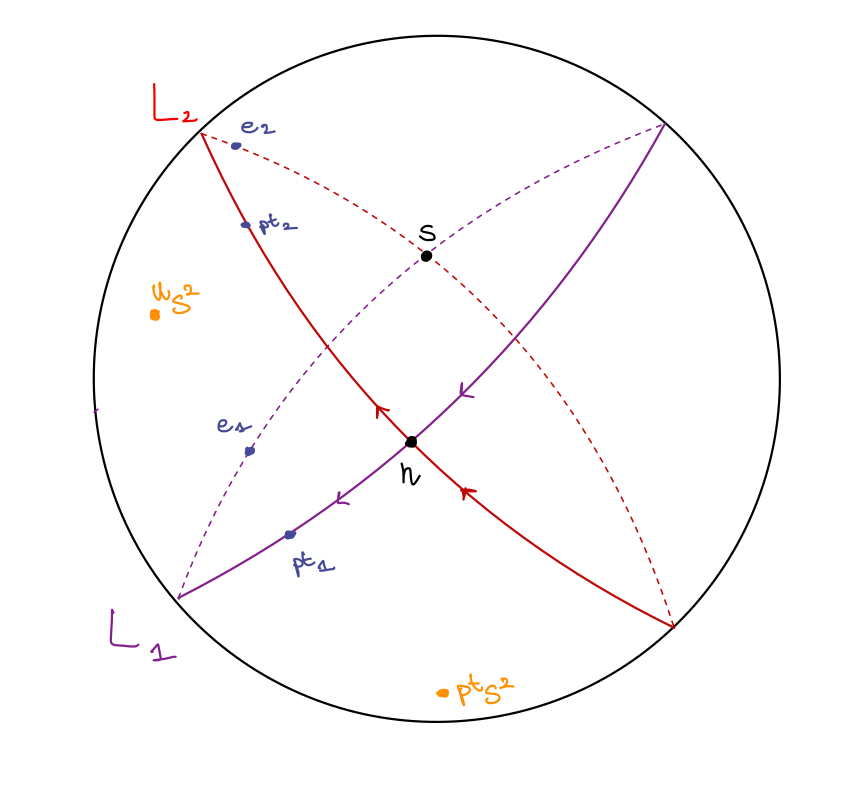}
	\centering
	\caption{The geometric situation described in Proposition \ref{OCS2} in the case $N=2$.} \label{fig:S2approx}
\end{figure}
\begin{rem}
\begin{proof}
	%Recall that the $p$-Floer complex of any two transversally intersecting Lagrangians is generated by their intersection points, all set at action $\nu(p)$ by definition of our class of Floer data (see Definition \ref{FD}).
	We will number our equators so that looking from the north pole, they are ordered clockwise starting from $L_1$. We drop $p$ from the notation and write $C_i:=CF(L_{i},L_{i+1})$ and $C_i':= CF(L_{i+1},L_i)$ for any $i=1,\ldots, N$\footnote{With the convention $L_{N+1}=L_1$.}. For any $i$, we denote by $n_i, s_i \in C_i$  and $n_i',s_i'\in C_i'$ the north and south poles respectively. To simplify the computations in the remainings of the proof, we impose the following additional restrictions on the functions $f^{L_i}_p$ in the $p$-Floer datum of each equator $L_i$: we require that flowing from the north pole to the south pole, one encounters first $\text{pt}_i$ and then $e_i$, before getting to the south pole. We also require that the critical points $u_{S^2}$ and $\text{pt}_{S^2}$ of the Morse function $f^{S^2}_p$ in the $p$-Floer datum of the ambient manifold $f_p^{S^2}\colon S^2\to \mathbb{R}$ both lie in the area between $L_1$ and $L_2$, but not on $L_1$ or $L_2$.
	
	We claim that $$\vec{\gamma}:=\sum_{i=1}^N n_i\otimes s_i'\in \bigoplus_{i=1}^N C_i\otimes C_i'\subset CC_*(\mathcal{E}(N))$$ is a Hochschild cycle, and moreover $$OC_p^{\mathcal{E}(N)}\left(\sum_{i=1}^N n_i\otimes s_i'\right)=T^{\frac{1}{2N}}u_M.$$ This would prove the statement as $$\mathbb{A}_{CC}\left(\sum_{i=1}^N n_i\otimes s_i'\right) = \mathbb{A}_{CC}(n_i\otimes s_i') = \mathbb{A}(n_i)+\mathbb{A}(s_i')\leq 2\nu(p)$$ by definition of the filtration on the Hochschild complex.
	
	We first prove that $\vec{\gamma}$ is a Hochschild cycle. Let $i\in\{1,\ldots, N\}$. We compute $$\mu_2(n_i,s_i') = T^{\frac{1}{2N}}e_i \text{  and  } \mu_2(s_i',n_i) = T^{\frac{1}{2N}}e_{i+1}$$ by just counting slices of the sphere as in Figure \ref{fig:S2approx}. By definition of the Hochschild differential we then have $$d_{CC}(n_i\otimes s_i')=T^{\frac{1}{2N}}(e_i+e_{i+1}),$$ hence $$d_{CC}(\vec{\gamma}) = T^{\frac{1}{2N}}\sum_{i=1}^N \left(e_i+e_{i+1}\right) = 0$$ as we work with $\mathbb{Z}_2$ coefficients. This shows that $\vec{\gamma}$ is an Hochschild cycle.
	
	The computation of $OC_p^{\mathcal{B}}(\vec{\gamma})$ is easy in this setting. Because of the choice of $f^{S^2}_p$ we have $$OC_p^\mathcal{B}(\vec{\gamma}) = OC_p^\mathcal{B}(n_1\otimes s_1') = T^{\frac{1}{2N}}u_{S^2}$$ by degree reasons.
\end{proof}

	Another way to prove the above proposition is to show that the element $\vec{\gamma}$ is homologous, in $CC(\mathcal{E}(N))$, to $T^{-\frac{N-1}{2N}}\text{pt}_{L_i}\otimes \text{pt}_{L_i}\in CF(L_i,L_i)\otimes CF(L_i,L_i)$ for any $i$.
\end{rem}

%=======================================
%=======================================
%=======================================

\subsubsection{Approximability of non-contractible circles on the $2$-torus}
Consider the standard symplectic $2$-torus $\mathbb{T}^2=\mathbb{R}^2/\mathbb{Z}^2$ with unit volume. We are interested in the class of non-contractible Lagrangians $\lagwex(\mathbb{T}^2)$. We will use coordinates $(x,y)$ on $\mathbb{T}^2$ and denote by $(n,m)$ the lattice coordinates on $\mathbb{Z}^2$. We recall that, for us, the weakly exact case is a particular instance of the monotone setting. %The main result of this section is the following.
%\begin{thm}\label{retractapproxt2}	The set of non-contractible Lagrangians $\lagwex(\mathbb{T}^2)$ of the $2$-torus is Fukaya retract-approximable.\end{thm}

We introduce the following notation: $$L_y:=\{x=0\}\subset \mathbb{T}^2\text{    and    }L^\theta_x:=\{y=\theta\}\subset \mathbb{T}^2 \text{  for any $\theta\in S^1$} $$ and denote the unique intersection point between $L_y$ and $L_x^\theta$ by $a^\theta\in L^\theta_x\pitchfork L_y.$ We denote $L_x:=L_x^0$ and $a:=a^0$. In this section, we will work with the family $$\mathcal{N}:=\{L_y\}\cup\{L_x^\theta\bigm| \theta\in S^1\}\subset \lagwex(\mathbb{T}^2).$$ %of non-contractible circles on $\mathbb{T}^2$ containing $L_y$ and $L_x^\theta$ for all $\theta \in S^1$. 
We will write $a^\theta$ seen as the generator of Floer complexes as $$a^\theta_{xy}\in CF(L^\theta_x,L_y) \text{   and   } a^\theta_{yx}\in CF(L_y,L^\theta_x)$$ for any choice of perturbation datum $p\in \mathcal{P}$, which we omit from the notation. In this section, we will only consider (without loss of generality) perturbation data $p\in \mathcal{P}$ satisfying the following:
\begin{enumerate}
	\item Consider the Morse function $f^{\mathbb{T}^2}_p\colon \mathbb{T}^2\to \mathbb R$ in the $p$-Floer datum for $\mathbb{T}^2$. We assume that $f^{\mathbb{T}^2}_p$ has exactly four critical points: a unique maximum $u_{\mathbb{T}^2}\in \mathbb{T}^2$, two saddle points $s_1,s_2\in \mathbb{T}^2$ and a unique minimum $\text{pt}_{\mathbb{T}^2}\in \mathbb{T}^2$. We further assume that all these critical points do not lie on the Lagrangian $L_y$ introduced above.
	\item Let $L\in \lagwex(\mathbb{T}^2)$ and consider the Morse function $f^L_p\colon L\to \mathbb{R}$ in the $p$-Floer datum of $L$. We assume that $f^L_p$ has a unique maximum $e_L\in L$ and a unique minimum $\text{pt}_L\in L$. We will always assume that both $e_L$ and $\text{pt}_L$ do not lie on the Lagrangian $L_y$ introduced above. In the special cases where $L$ is one of the Lagrangians $L_y$ or $L_x^\theta$ in $\mathcal{N}$, we denote the critical points as $e_y$, $pt_y$ or $e_x^\theta$, $pt_y^\theta$.
\end{enumerate}
\begin{rem}
	The role of $x$ and $y$ are intercheangeable. Moreover, we could work with the countable subfamily $\mathcal{N}_c$ containing $L_y$ and $L_x^\theta$ for all $\theta \in \mathbb{Q}/\mathbb{Z}$.
\end{rem}

As for the case of the $2$-sphere in the subsection above, we first prove a warm up low-accuracy retract-approximation result, before moving to a more general case. In the following proposition we will work with the Lagrangian $L_x$, but the same results hold for all $L_x^\theta$.
\begin{prop}
	Let $p\in \mathcal{P}$. Then $\lagwex(\mathbb{T}^2)$ is retract approximable in $\mathcal{A}_p$ by $\mathcal{B}_{xy}:=\{L_x,L_y\}$ with accuracy $\frac{1}{2}+3\nu(p)$.
\end{prop}
\begin{proof}
	Consider the Hochschild chain $\vec{a}=a_{xy}\otimes a_{yx}\otimes a_{xy}\otimes a_{yx}\in CC_*(\mathcal{B}_{xy})$. We show it is a cycle. For this, we compute all possible cyclic partial contractions.
	\begin{enumerate}
		\item First, both $\mu_2(a_{xy},a_{yx})\in CF(L_x,L_x)$ and $\mu_2(a_{yx},a_{xy})\in CF(L_y,L_y)$ vanish, as the only allowed contributions are Morse trajectories to $\text{pt}_x$ and $\text{pt}_y$ respectively, which come in pairs.
		\item The terms $\mu_3(a_{xy},a_{yx},a_{xy})\in CF(L_x,L_y)$ and $\mu_3(a_{yx},a_{xy},a_{yx})$ vanish as well. For instance, $$\mu_3(a_{xy},a_{yx},a_{xy}) = 2\sum_{n,m>0}T^{nm}a_{xy}$$ as we can see the "two opposite" $a_{xy}$ in the fundamental square of $\mathbb{T}^2$ as exits in two different ways (for each lattice-polygon $(n,m)$).
		\item The terms $\mu_4(a_{xy},a_{yx},a_{xy},a_{yx})\in CF(L_x,L_x)$ and $\mu_4(a_{yx},a_{xy},a_{yx},a_{x,y})\in CF(L_y,L_y)$ both vanish. For instance $$\mu_4(a_{xy},a_{yx},a_{xy},a_{yx})=2\sum_{n,m>0}nT^{nm}e_x$$ since in a lattice polygon of width $n$ we can have $e_x$ as an exit in $2n$ different ways.
	\end{enumerate}
	It follows that $\vec{a}\in CC_*(B_{xy})$ is a cycle. We compute $$OC^{\mathcal{B}_{xy}}(\vec{a}) = \sum_{n,m>0}nmT^{nm}u_{\mathbb{T}^2}$$ since in every $(n,m)$ lattice polygon we can see $u$ as an exit in $nm$ different position (one for every fundamental square embedded in the polygon). Since we are working over $\mathbb{Z}_2$ it follows that $$OC^{\mathcal{B}_{xy}}(\vec{a}) = \sum_{n>0}n^2T^{n^2}u_{\mathbb{T}^2} =  \sum_{n\geq0}T^{(2n+1)^2}u_{\mathbb{T}^2}.$$ Note that the smallest Novikov exponent in this expression is $1$. Since intersection points lie at filtration level $\leq \nu(p)$ by definition, the claim follows.
\end{proof}
\begin{rem}
	More formally, given the Novikov element $$f(T)= \sum_{n\geq0}T^{(2n+1)^2}$$ we have $$f^{-1}(T) = T^{-1}\sum_{n\geq 0}g_nT^n$$ where $g_0=1$ and $$g_n = \sum_{e\in E\setminus\{0\}, \ e\leq n}g_{n-e}$$ where $E= \{(2n+1)^2-1\ : \ n\geq 0\}$. Indeed, writing $\frac{f(T)}{T}= \sum_{n\geq 0}f_nT^n$ where $f_n =1 $ if and only if $n\in E$, we have that $f_0\cdot g_0=1$ while the coefficient of $T^n$ for $n\geq 1$ in $f\cdot f^{-1}$ is 
	\begin{align*}
		\sum_{j=1}^nf_jg_{n-j} = \sum_{j\in E, \ j\leq n}g_{n-j} = g_n+  \sum_{j\in E\setminus\{0\}, \ j\leq n}g_{n-j}=0
	\end{align*}
	It follows directly from our definition of the action filtration on $\Lambda$, that $f^{-1}\in \Lambda$ lies at filtration level $\leq 1$.\\
	Using the formalism of theta-functions we can write $$\sum_{n\geq 0}T^{(2n+1)^2}= \sum_{n\in \mathbb Z}T^{(4n+1)^2}=\theta_{\frac{1}{4},0}(16,0).$$ But also notice that $$\sum_{n\geq 0}n^2T^{n^2}=T\theta'_{0,0}(1,0)$$
\end{rem}
\begin{prop}\label{prop235}
	Let $p\in \mathcal{P}$ and and $N\geq 1$. Consider the family $\mathcal{N}(N)= \{L_y,L_1,\ldots, L_N\}\subset \mathcal{N}$, where $$L_j=L_x^{\frac{j-1}{N}}$$ for $j=1,\ldots, N$. $\lagwex(\mathbb{T}^2)$ is retract approximable in $\mathcal{A}_p$ by $\mathcal{N}(N)$ with accuracy $\frac{1}{2N}+ 3\nu(p)$.%Then $\mathcal{N}(N)$ retract-approximates $\mathcal{L}^{we}(\mathbb{T}^2)$ with accuracy $\frac{1}{N}+ 4\nu(p)$ in $\mathcal{C}\mathcal{F}uk(\mathbb{T}^2,p)$.
\end{prop}
\begin{figure}
	\centering
	\begin{tikzpicture}[x=1cm,y=1cm, line cap=round, line join=round]
		
		% --- Colors
		\definecolor{myPurple}{RGB}{204,0,0}
		\definecolor{myOrange}{RGB}{0,100,200}
		\definecolor{myGreen}{RGB}{90,170,70}
		
		% --- Geometry (SQUARE)
		\def\xL{0}
		\def\xR{7}    % square: width = height = 7
		
		\def\yTop{7.0}
		\def\yAA{6.2}
		\def\yA{4.8}
		\def\yB{4}
		\def\yC{2.8}
		\def\yD{2}
		\def\yE{0.8}
		\def\yBot{0.0}
		
		% --- Orange vertical sides
		\draw[myOrange, line width=3pt] (\xL,\yBot) -- (\xL,\yTop);
		\draw[myOrange, line width=3pt] (\xR,\yBot) -- (\xR,\yTop);
		
		% --- Helper: horizontal line + endpoints
		\newcommand{\hlinewithdots}[1]{%
			\draw[myPurple, line width=3pt] (\xL,#1) -- (\xR,#1);
			\fill[black] (\xL,#1) circle (1.6pt);
			\fill[black] (\xR,#1) circle (1.6pt);
		}
		
		% --- Horizontal lines
		\hlinewithdots{\yTop}
		\hlinewithdots{\yAA}
		\hlinewithdots{\yA}
		\hlinewithdots{\yB}
		\hlinewithdots{\yC}
		\hlinewithdots{\yD}
		\hlinewithdots{\yE}
		\hlinewithdots{\yBot}
		
		% --- Interior points
		\fill[myPurple] (2.0,\yTop) circle (1.4pt);
		\fill[myPurple] (4.8,\yTop) circle (1.4pt);
		
		\fill[myPurple] (2.1,\yBot) circle (1.4pt);
		\fill[myPurple] (4.9,\yBot) circle (1.4pt);
		
		% --- Vertical dots
		\foreach \yy in {5.6,5.4,5.2} { \fill[myPurple] (3.5,\yy) circle (1.1pt); }
		\foreach \yy in {3.6,3.4,3.2} { \fill[myPurple] (3.5,\yy) circle (1.1pt); }
		\foreach \yy in {1.6,1.4,1.2} { \fill[myPurple] (3.5,\yy) circle (1.1pt); }
		
		% --- Left labels
		\node[myOrange, anchor=south west, scale=1.2] at (-0.3,\yTop) {$L_y$};
		
		\fill[myOrange] (\xL, {(\yA+\yB)/2}) circle (2.5pt);
		\node[myOrange, anchor=east, scale=1.1] at (-0.2, {(\yA+\yB)/2}) {$e_y$};
		
		\node[black, anchor=east, scale=1.2] at (-0.25,\yBot) {$a^1$};
		\node[black, anchor=east, scale=1.2] at (-0.25,\yE) {$a^2$};
		\node[black, anchor=east, scale=1.2] at (-0.25,\yD) {$a^i$};
		\node[black, anchor=east, scale=1.3] at (-0.35,{(\yE+\yD)/2}) {$\vdots$};
		\node[black, anchor=east, scale=1.3] at (-0.35,{(\yB+\yC)/2}) {$\vdots$};
		
		% --- Right labels
		\node[myPurple, anchor=west, scale=1.2] at (\xR+0.25,\yTop) {$L_1$};
		\node[myPurple, anchor=west, scale=1.2] at (\xR+0.25,\yAA) {$L_{N}$};
		\node[myPurple, anchor=west, scale=1.2] at (\xR+0.25,\yA) {$L_{\ell+1}$};
		\node[myPurple, anchor=west, scale=1.2] at (\xR+0.25,\yB) {$L_{\ell}$};
		\node[myPurple, anchor=west, scale=1.2] at (\xR+0.25,\yC) {$L_{i+1}$};
		\node[myPurple, anchor=west, scale=1.2] at (\xR+0.25,\yD) {$L_i$};
		\node[myPurple, anchor=west, scale=1.2] at (\xR+0.25,\yE) {$L_2$};
		\node[myPurple, anchor=west, scale=1.2] at (\xR+0.25,\yBot) {$L_1$};
		
		% --- Green point
		\fill[myGreen] (4.1,2.3) circle (1.6pt);
		\node[myGreen, anchor=west, scale=1.2] at (4.2,2.3) {$u_{\mathbb{T}^2}$};
		
		% --- Bottom labels
		\node[myPurple, anchor=north, scale=1.2] at (2.1,-0.25) {$\text{pt}_1$};
		\node[myPurple, anchor=north, scale=1.2] at (4.9,-0.25) {$e_1$};

		% --- Purple dots on top & bottom lines (interior)
		\fill[myPurple] (2,\yTop) circle (2.1pt);
		\fill[myPurple] (4.85,\yTop) circle (2.1pt);
		
		\fill[myPurple] (2,\yBot) circle (2.1pt);
		\fill[myPurple] (4.85,\yBot) circle (2.1pt);
		
	\end{tikzpicture}
	\caption{The fundamental domain of $\mathbb{T}^2$ with the approximating family of Lagrangians $\mathcal{N}(N)$ discussed in Proposition \ref{prop235}.} \label{fig:T2approx}
\end{figure}
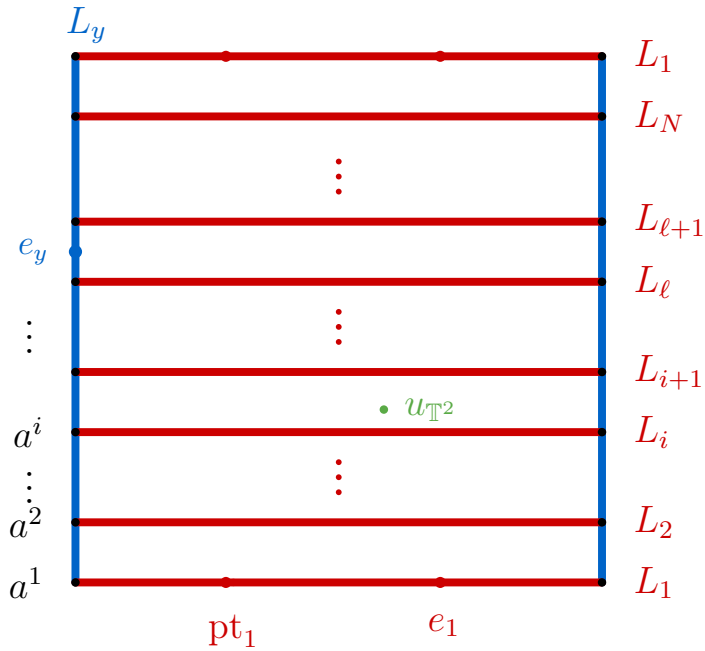

\begin{proof}
	We assume that the critical point $u_{\mathbb{T}^2}$ doesn't lie on any Lagrangian in $\mathcal{N}(N)$. We denote by $e_j\in CF(L_j,L_j)$ and $\text{pt}_j\in CF(L_j,L_j)$ the maximum and the minimum of the Morse function $f_p^{L_j}$ for any $j=1,\ldots, N$. Moreover, for any $j$, we denote by $a^j_{xy}\in CF(L_j,L_y)$ and $a^j_{yx}\in CF(L_y,L_j)$ the intersection point $a$ seen as a generator of the Floer complexes. We denote by $i\in \{1,\ldots, N\}$ the unique index such that $u_{\mathbb{T}^2}$ lies in the area between $L_i$ and $L_{i+1}$. Similarly, we denote by $l$ the unique index such that $e_y\in CF(L_y,L_y)$ lies in the $L_y$-segment between $L_l$ and $L_{l+1}$. Figure \ref{fig:T2approx} provides a schematic depiction of the geometric setup.
	
	For any $j=1,\ldots, N$, consider the Hochschild chain\footnote{Here $N+1=1$.} $$\gamma^j:=a^j_{xy}\otimes a^{j+1}_{yx}\otimes a^{j+1}_{xy}\otimes a^{j}_{yx}\in CC_*(\mathcal{N}(N)).$$ Given a real number $A\in \mathbb{R}$, let\footnote{Here $h$ stands for \textit{horizontal} and $v$ for \textit{vertical.}} $$q^h_{A}(T):= \sum_{n,m>0}n(T^{n(m-1+A)}+ T^{n(m-A)}) \text{   and   } \tilde{q}^h_{\frac{1}{N}}(T):`= \sum_{n,m>0}T^{n(m-1+A)}+ T^{n(m-A)}$$ as well as 
	$$q^v_A(T):= \sum_{n,m>0}m(T^{n(m-A)}+ T^{n(m+A)}) \text{   and   } \tilde{q}^v_A(T):= \sum_{n,m>0}T^{n(m-A)}+ T^{n(m+A)} $$ be elements of the Novikov field $\Lambda$. We compute the Hochschild differential of $\gamma^j$ by computing all possible contractions. 
	\begin{enumerate}
		\item We have $$\mu_4(a^{j}_{xy},a^{j+1}_{yx}, a^{j+1}_{xy}, a^{j}_{yx}) = q^h_{\frac{1}{N}}(T)e_j\text{   and   } \mu_4(a^{j+1}_{xy},a^{j}_{yx},a^{j}_{xy},a^{j+1}_{yx})= q^h_{\frac{1}{N}}(T)e_{j+1},$$ while $$\mu_4(a^{j}_{yx},a^{j}_{xy},a^{j+1}_{yx}, a^{j+1}_{xy}) = \mu_4(a^{j+1}_{yx},a^{j+1}_{xy},a^{j}_{yx},a^{j}_{xy})= \begin{cases}
			q^v_{\frac{1}{N}}(T) e_y, \text{  if }j\neq l\\
			q^v_{1-\frac{1}{N}}(T)e_y, \text{  if }j=l.
		\end{cases} $$ %$$\mu_4(a^{j}_{yx},a^{j}_{xy},a^{j+1}_{yx}, a^{j+1}_{xy}) = \mu_4(a^{j+1}_{yx},a^{j+1}_{xy},a^{j}_{yx},a^{j}_{xy})= \begin{cases}                \sum_{n,m>0}m(T^{n(m-\frac{1}{N})}+ T^{n(m+\frac{1}{N})})e_y, \text{  if }j\neq l\\                \sum_{n,m>0}m(T^{n(m+\frac{1}{N}-1)}+ T^{n(m+1-\frac{1}{N})})e_y, \text{  if }j=l            \end{cases} $$ 
		\item We have $$\mu_3(a^{j}_{xy},a^{j+1}_{yx}, a^{j+1}_{xy}) = \tilde{q}^h_{\frac{1}{N}}(T)a^j_{xy} \text{   and   }\mu_3(a^{j+1}_{yx},a^{j+1}_{xy},a^{j}_{yx}) = \tilde{q}^h_{\frac{1}{N}}(T)a^{j}_{yx}$$
		while $$\mu_3(a^{j}_{yx},a^{j}_{xy},a^{j+1}_{yx}) = \tilde{q}^h_{\frac{1}{N}}(T)a^{j+1}_{yx} \text{   and   }\mu_3(a^{j+1}_{xy},a^j_{yx},a^{j}_{xy}) = \tilde{q}^h_{\frac{1}{N}}(T)a^{j+1}_{xy}$$
		
	\end{enumerate}
	We now compute the relevant $\mu_3$'s.
	It follows that \begin{align*}
		d_{CC}(\gamma^j) & = q^h_{\frac{1}{N}}(T)(e_j+e_{j+1}) + 2q^v_{c_j}(T)e_y + \tilde{q}^h_{\frac{1}{N}}(T)(2a^j_{xy}\otimes a^j_{yx} + a^{j+1}_{xy}\otimes a^{j+1}_{yx} + a^{j+1}_{yx}\otimes a^{j+1}_{xy}) \\ & = q^h_{\frac{1}{N}}(T)(e_j+e_{j+1}) + \tilde{q}^h_{\frac{1}{N}}(T)(a^{j+1}_{xy}\otimes a^{j+1}_{yx} + a^{j+1}_{yx}\otimes a^{j+1}_{xy})
	\end{align*}
	where $c_j = \frac{1}{N}$ if $j\neq l$ and $c_j=1-\frac{1}{N}$ otherwise.
	Hence $$d_{CC}\left(\sum_{j=1}^N \gamma^j\right) =  \tilde{q}^h_{\frac{1}{N}}(T)\sum_{j=1}^N(a^{j}_{xy}\otimes a^{j}_{yx} + a^{j}_{yx}\otimes a^{j}_{xy}) $$ as by summing over $j$ the summands of the form $e_j$ cancel out in pairs. In particular, $\sum_{j=1}^N\gamma^j$ is not a cycle. However, we claim that $$\vec{\gamma} = \sum_{j=1}^N\gamma^j + \tilde{q}^h_{\frac{1}{N}}(T)\ e_j\otimes a^j_{xy}\otimes a^j_{yx}\in CC_*(\mathcal{N}(N))$$ is a cycle. Indeed: $$d_{CC}(e_j\otimes a^j_{xy}\otimes a^j_{yx}) = \mu_2(e_j,a^j_{xy})\otimes a^j_{yx} + \mu_2(a^j_{yx},e_j)\otimes a^j_{xy} =a^{j}_{xy}\otimes a^{j}_{yx} + a^{j}_{yx}\otimes a^{j}_{xy},$$ since $e_j$ is a strict unit and the product of the point with itself vanishes.\\
	
	We now compute $OC^{\mathcal{N}(N)}(\vec{\gamma})\in CQ(\mathbb{T}^2)$. Recall that $i\in \{1,\ldots, N\}$ is defined as the only index such that the critical point $u=u_{\mathbb{T}^2}$ of the Morse function $f^{\mathbb{T}^2}_p$ used to define $QH_*(M)$ lies in the area between $L_i$ and $L_{i+1}$. Let $j\in \{1,\ldots, N\}$, then \begin{align*}
		OC(\gamma^j) &= \begin{cases}
			\sum_{n,m>0}nm(T^{n(m-\frac{1}{N})}+T^{n(m+\frac{1}{N})})u,  \text{ if }j\neq i\\
			\sum_{n,m>0}nm(T^{n(m-1+\frac{1}{N})}+T^{n(m+1-\frac{1}{N})})u,  \text{ if }j=i
		\end{cases}
	\end{align*}
	On the other hand, for all $j\in \{1,\ldots, N\}$ it is easy to see that we have $$OC(e_j\otimes a^j_{xy}\otimes a^j_{yx}) =0.$$ If $N$ is odd, then we have $$OC(\vec{\gamma}) = OC(\gamma^i) = \sum_{k\geq 0}\sum_{\substack{d=\pm 1 \ (2N)\\ d|2k+1}}T^{\frac{2k+1}{N}}$$ while if $N$ is even we have, for some $j\neq i$: $$OC(\vec{\gamma}) = OC(\gamma^i)+OC(\gamma^j) = \sum_{k\geq 0}\sum_{\substack{d=\pm 1, \pm1+N \ (2N)\\ d|2k+1}}T^{\frac{2k+1}{N}}$$
	%If $N$ is odd, then $$OC(\vec{\gamma}) = OC(\gamma^i)$$ since all other contributions cancel out over $\Z_2$. If $N$ is even, we have $$OC(\vec{\gamma}) = $$
	In all cases the lowest order term is $T^{\frac{1}{N}}u_{\mathbb{T}^2}$. The claim follows.
	
\end{proof}

%\begin{proof}[Proof of Theorem \ref{retractapproxt2}]	Let $\Delta>0$ and consider the family $\mathcal{N}=\{L_y\}\cup \{L_x^\theta \ : \ \theta \in S^1\}$ introduced above. We choose a subset $\mathcal{N}(N)$ of cardinality $N+1$ as in Proposition \ref{prop235}. Then for any filtered perturbation datum $p$ as in the proof of proposition \ref{prop235} of size $\nu(p)<\frac{1}{4}(\delta-\frac{1}{N})$, the family $\mathcal{L}^{we}(\mathbb{T}^2)$ of non-contractible Lagrangians on $\mathbb{T}^2$ is retract approximated by $\mathcal{N}(N)$ with accuracy $\delta$ in the TPC $\mathcal{C}\mathcal{F}uk(\mathbb{T}^2,p)$.\end{proof}

In Proposition \ref{prop235} we worked with only one Lagrangian in the $y$ direction. We can generalize the result to families containing $N+M$ circles, $N$ in the $x$ direction and $M$ in the $y$ direction, to get retract-approximation of order $\frac{1}{NM}$. Below we prove this for $N=M$.

\begin{prop}
	Let $p\in \mathcal{P}$ and $N\geq 1$. Consider the family $$\mathcal{N}(N)= \{L^1_y,\ldots, L_y^N,L_x^1,\ldots, L_x^N\}\subset \mathcal{N},$$ where $$L_y^j:= L_y^{\frac{j-1}{N}}\text{    and    }L_x^j=L_x^{\frac{j-1}{N}}$$ for $j=1,\ldots, N$. Then $\mathcal{N}(N)$ retract-approximates $\mathcal{L}^{we}(\mathbb{T}^2)$ with accuracy $\frac{1}{2N^2}+ 3\nu(p)$ in $\mathcal{A}_p$.
\end{prop}
\begin{proof}
	Essentially, the proof is the same as Proposition \ref{prop235}. We denote by $a^{j,k}$ the unique intersection point in $L_x^j\cap L_y^k$ and write it as $a^{j,k}_{xy}$ when seen as a generator of $CF(L_x^j,L_y^k)$ and as $a^{j,k}_{y,x}$ when seen as a generator of $CF(L_y^k,L_x^j)$. Let $i_x$ and $i_y$ the only indices such that the critical point $u=u_{\mathbb{T}^2}$ of the Morse function defining $QH(\mathbb{T}^2)$ lies in the square bounded by the Lagrangians $L_x^{i_x}$, $L_x^{i_x+1}$, $L_y^{i_y}$ and $L_y^{i_y+1}$. The key fact is that 
	\begin{align*}
		OC(a^{i_x,i_y}_{xy}\otimes a^{i_y,i_x+1}_{yx}\otimes a^{i_x+1,i_y+1}_{xy}\otimes a^{i_y+1,i_x}_{yx})&=\sum_{n,m>0}nm(T^{(n-1+\frac{1}{N})(m-1+\frac{1}{N}}+T^{(n+1-\frac{1}{N})(m+1-\frac{1}{N}})u \\& =\sum_{n>0}n^2(T^{(n-1+\frac{1}{N})^2}+T^{(n+1-\frac{1}{N})^2})u\\ & = \sum_{n\geq 0}T^{(2n+\frac{1}{N})^2}+T^{(2n+2-\frac{1}{N})^2}u\\ & = \sum_{n\in \mathbb{Z}}T^{(2n+\frac{1}{N})^2 }u= \theta_{\frac{1}{2N},0}(4,0)u
	\end{align*} In particular, the smallest exponent in this serie is $\frac{1}{N^2}$. One finishes the proof by showing that summing all possible cyclic combinations of tensors of intersection point plus some deformations involving the units in the Lagrangian Floer complex as in the proof of Proposition \ref{prop235} one gets a Hochschild cycle.
\end{proof}

%=======================================
%=======================================
%=======================================

\subsubsection{End of proof of Corollary \ref{S2T2}}

The proof of Corollary \ref{S2T2} is now a matter of packing the results of the last two sections.
\begin{proof}[Proof of Corollary \ref{S2T2}]
	We begin with the case of the sphere. 
	Let $\varepsilon_0>0$ and consider the family $\mathcal{E}\subset \lag^{\text{(mon,} \mathbf{0} \text{)}}(S^2)$ of straight equators passing through north and south pole introduced above. Let $N_0\geq 1$ such that $\frac{1}{4N_0}<\varepsilon_0$ and choose a subset $\mathcal{E}(N_0)\subset \mathcal{E}$ of cardinality $N_0$ as in Proposition \ref{OCS2}. Let $\delta_0:=\frac{1}{2}\left(\varepsilon_0-\frac{1}{4N_0}\right)$. Then for any perturbation datum $p$ as in the proof of Proposition \ref{OCS2} of size $\nu(p)<\delta_0$, the family $\lag^{\text{(mon,} \mathbf{0} \text{)}}(S^2)$ is $\varepsilon_0$-retract approximated by $\mathcal{E}(N_0)$ in $PD(\fuk(\lag^{\text{(mon,} \mathbf{0} \text{)}}(S^2);p))$. It follows that $\lag^{\text{(mon,} \mathbf{0} \text{)}}(S^2)$ is retract-approximable by $\mathcal{E}$ in the sense of Definition \ref{d:approx-sys-ret}. 
	\\ Let $\varepsilon_1>0$ and consider the family $\mathcal{N}\subset\lagwex(\mathbb{T}^2)$ introduced above. Let $N_1\geq 1$ such that $\frac{1}{2N_1}<\varepsilon_1$ and  choose a family $\mathcal{N}(N_1)$ of cardinality $N_1+1$ and containing $L_y$ as in Proposition \ref{prop235}. Let $\delta_1:= \frac{1}{3}(\varepsilon_1-\frac{1}{2N_1})$. Then for any perturbation datum as in te proof of Proposition \ref{prop235} of size $\nu(p)<\delta_1$, the family $\lagwex(\mathbb{T}^2)$ is $\varepsilon_1$-retract-approximable by $\mathcal{N}(N_1)$ in $PD(\fuk(\lagwex(\mathbb{T}^2);p))$. It follows that $\lagwex(\mathbb{T}^2)$ is retract-approximable by $\mathcal{N}$ in the sense of Definition \ref{d:approx-sys-ret}. \end{proof}

\subsection{Proof of Theorem \ref{thmmain1} ii and iii.}\label{subsec:proof_ThmAii}
Corollary \ref{S2T2} claims retract approximability in the sense of Definition \ref{d:approx-sys-ret} for the classes of Lagrangian submanifolds, $\lag^{\text{(mon,} \mathbf{0} \text{)}}(S^2)$ and $\lagwex(\mathbb{T}^2)$, that appear at the points ii and iii of Theorem \ref{thmmain1}. To conclude the proof of this theorem we need to remark that retract approximability in the sense of  Definition \ref{d:approx-sys-ret} implies retract approximability in the sense of Definition \ref{def:TPC-approx}. This argument has already been discussed in the proof of the point i of Theorem \ref{thmmain1}, in \S\ref{subsubsec:local_approx_d} but we will revisit it here.

First, notice that the approximating data $(\Phi, \F)$ is clear from the proof of Corollary \ref{S2T2}. For instance, for $\lag^{(mon,\bf{0})}(S^{2})$ the maps
$$\Phi_{\eta} : \lag^{(mon,\bf{0})}(S^{2})\longrightarrow \Ob (PD(\fuk(\lag^{\text{(mon,} \mathbf{0} \text{)}}(S^2);p)))$$ are Yoneda embeddings (with $\nu(p) < \eta$, small enough)
and $\F_{\epsilon}$ is a family of sufficiently many great circles going through the north
and south poles on $S^{2}$, as in Proposition \ref{OCS2}.

The only point that remains to be clarified is the relation between the spectral metric on the domain of $\Phi$ and the interleaving metric on the target of the maps $\Phi$.
This relationships is analogous to the one described in \S\ref{subsubsec:local_approx_d} but there 
are some adjustments. 
First, we need to provide a definition of the spectral metric $d_{\gamma}$ that applies to monotone Lagrangians as well as to weakly-exact Lagrangians, and not only to exact Lagrangians
as in \S\ref{subsubsec:local_approx_d}. This is well-known in the literature but, for completeness, we recall this definition here. Consider two Lagrangians that are Hamiltonian istotopic $L$, $L'$ (in one of the classes considered in Theorem \ref{thmmain1})  and a Hamiltonian $H:[0,1]\times M\to \R$ such that the Floer homology $HF(L,L';H)$ is defined. Because $L$ and $L'$ are Hamiltonian isotopic the PSS map  
$QH(L)\to HF(L,L';H)$ is well defined and we let $\gamma ([L];H)\in \R$ be the spectral
invariant in (the persistence module) $HF(L,L';H)$ of the PSS-image of the fundamental class $[L]\in QH(L)$.  We then put, see for instance \cite{Kis-Sh}:

$$d_{\gamma}(L,L')=\limsup_{||H||_{H}\to 0}\  \gamma([L]; H)+ \gamma([L];\overline{H})$$
where $\overline{H}$ is the inverse Hamiltonian $\overline{H}(t,x)=-H(1-t,x)$.
For exact Lagrangians this definition is equivalent to the one in \S\ref{subsubsec:local_approx_d}.
Moreover, the arguments in \S\ref{subsubsec:local_approx_d} -  and in particular, the identity (\ref{eq:equality_int})  and the arguments in \cite{BCZ:tpc} justifying it -  remain true in this context and they show that the maps $\Phi$  restricted to a Hamiltonian isotopy class  are quasi-isometric embeddings. Therefore, this completes the argument
in the case of $\lag^{\text{(mon,} \mathbf{0} \text{)}}(S^2)$ given that all the Lagrangians in this space are Hamiltonian isotopic to the standard equator.

Finally, we consider the space $\lagwex(\mathbb{T}^2)$. The spectral distance, given as above, 
is not defined for two $L$ and $L'$ that are not Hamiltonian isotopic. However, the interleaving type distance $\hat{D}_{int}(-,-)$ from \S\ref{subsubsec:local_approx_d} is defined  on all of
$\lagwex(\mathbb{T}^2)$ and, as explained before, it coincides with the spectral distance on each Hamiltonian isotopy class. Thus
by putting
$$d_{\gamma} (L,L'):= \hat{D}_{int} (L,L'), \forall \ \ L , L'\in \lagwex(\mathbb{T}^2)$$
the proof of Theorem \ref{thmmain1} is complete.

	\newpage
% !TEX root = approx8.tex

\section{Corollaries of approximability and weighted
  complexity}\label{sec:complex}
The aim of this section is to prove Corollaries \ref{cor:no-t-b},
\ref{cor:quasi-rig}, and \ref{cor:link}, and to discuss some notions
of complexity that are then used to deduce Corollary
\ref{cor:no-t-b}. This type of complexity has interest in itself and
we start the section with a discussion of this topic, that is purely
algebraic in nature, in \S\ref{subsec:complex_alg}.  We continue with
the proof of Corollary \ref{cor:no-t-b} in \S\ref{subsec:non-v_Cor}.
In \S\ref{subsec:qrig-cor} and \S\ref{subsec:Gw-cor}, respectively, we
prove the other two corollaries.

\subsection{Complexity in TPCs}\label{subsec:complex_alg}

Most of this subsection is purely algebraic and can be read
independently of any considerations related to symplectic topology,
Fukaya categories and so forth.  In \S\ref{subsec:wconel} we introduce
the notions of complexity we are interested in, the most important one
being that of {\em weighted cone-length}. We also state the main
result of the subsection, Proposition \ref{prop:lower-bd}, providing a
lower bound for weighted cone-length in terms of a count of bars in a
certain barcode.  The proof of this proposition occupies the rest of
the subsection. In \S\ref{subsec:retr} we establish some algebraic
properties that lead to the inequality in Corollary
\ref{cor:nice-ineq} that reduces the statement of the proposition to
some properties of weighted cone-length in the homotopy category of
filtered chain complexes. These properties are established in
\S\ref{subsec:pers}.  In \S\ref{subsec:proof_Prop} the various
arguments are put together to show Proposition \ref{prop:lower-bd}. In
\S\ref{subsec:entrop} we introduce a notion of weighted categorical
entropy which is a weighted analogue of a similar notion introduced by
Dimitrov-Haiden-Katzarkov-Kontsevich \cite{DHKK:dyn_cat} and, as a
corollary of Proposition \ref{prop:lower-bd}, we estimate this
weighted entropy from below by the barcode entropy as defined by
\c{C}ineli-Ginzburg-Gurel \cite{CGG_bar_ent}. Finally, in
\S\ref{subsubsec:complex_Lag} we return to Lagrangians in cotangent
bundles and make some first comments concerning their complexity.

\

The setting of the subsection is that of a triangulated persistence category $\C$ endowed with the interleaving pseudo-metric 
$\dint$ as in equation (\ref{eq:d-int}). 

\subsubsection{Weighted generating rank and
  cone-length} \label{subsec:wconel} In the setting above consider
$X\subset \Ob(\C)$ and define the $\epsilon$-{\em generating rank}
of $X$ by:
\begin{equation}\label{eq:no-of-gen}
  g_{\C}(X; \epsilon) =
  \inf \bigl\{  \# (\F_{\epsilon})\
  |\  \F_{\epsilon}\subset \Ob (\C)   ,
  \   \forall x\in X, \
  d_{int}\bigl(x, \Ob(\langle\F_{\epsilon}\rangle ^{\Delta})\bigr)
  \leq \epsilon \bigr\}.
\end{equation} 
Here $\# (\mathcal{S})$ indicates the cardinality of the set
$\mathcal{S}$, and $\langle \mathcal{S} \rangle ^{\Delta}$ is defined
in~\S\ref{sb:tpc}.

Finiteness of the $\epsilon$-generating rank of
$X\subset \Ob(\C)$ is equivalent to $\epsilon$-approximability of
$X$ in $\C$, as seen by inspecting Definition \ref{def:simple_approx}.

\

A more important complexity measurement, also related to
approximability, is a weighted version of the classical notion of
cone-length of a topological space in homotopy theory \cite{Gan:cat},
\cite{Cor:cone-LS}, \cite{Cor:cone-SCat}.  We fix some preliminary
notation. Let $\C'$ be a triangulated category and let
$\F\subset \Ob(\C')$.  Consider a sequence
$\eta=(\Delta_{0},\ldots, \Delta_{m})$ of exact triangles in
$\C'$ of the form:
\begin{equation}\label{eq:triang-dec}\Delta_{i} \ \ : \ \ \
  F_{i}\longrightarrow A_{i}\longrightarrow A_{i+1}\longrightarrow
  TF_{i} \ \ , \ 0\leq i \leq m
\end{equation}
with $F_{i}\in \F$ and with $A_{m+1}=A$. We symbolically denote such a
sequence by $$\eta: A_{0} \stackrel{\F}{\rightsquigarrow} A$$ and we
let $\#(\eta)=m+1$ be the number of triangles in the sequence.  We
call the ordered family $(F_{0},\ldots, F_{m})$ the linearization of
$\eta$ and denote it by $\ell(\eta)$. By analogy with classical
topology, we will refer to each exact triangle as in
(\ref{eq:triang-dec}) as a {\em cone attachment over} $F_{i}$.  Assume
now that $\C$ is a TPC. In this case we will use the
notation $$\eta:A_{0}\stackrel{\F^{\Sigma, T}}{\rightsquigarrow} A$$
to refer to a sequence similar to (\ref{eq:triang-dec}) where
each triangle $\Delta_{i}$ is now {\em strict exact in}
$\C$. More specifically this means that each triangle
  $\Delta_i$ is now replaced by a strict exact triangle (in the sense
  of~\cite[Definition~2.42]{BCZ:tpc})
\begin{equation*} \label{eq:triang-dec-weighted} \Delta'_{i} \ \ : \ \
  \ F_{i}\longrightarrow A_{i}\longrightarrow A_{i+1}\longrightarrow
  \Sigma^{-w_i} TF_{i},
\end{equation*}
for some weight $w_i \geq 0$, $i = 0, \ldots, m$.

We emphasize that in this case the $F_{i}$'s are taken to be elements
in the set
$\mathcal{\F}^{\Sigma, T}=\{ T^{i}\Sigma^{\alpha} F \ | \ F\in \F ,
i\in \Z, \alpha\in \R\}$. The weight $w(\eta)$ of the decomposition
$\eta$, is defined as the sum of the weights of each of the triangles
$\Delta'_i$ in $\eta$, $w(\eta) = w_0 + \cdots + w_m$ (see
\cite{BCZ:tpc} for the basics on TPCs). If $\eta$ is of weight $0$, as
it will be the case in most of what follows, then each of the strict
exact triangles in $\eta$ is an exact triangle in the usual
triangulated category $\C^{0}$.
 
\begin{dfn}\label{def:w-cl}
  Fix $\C$ a TPC, as above, a family $\F\subset \Ob(\C)$, and also
  $\epsilon\in [0,\infty)$.  For any two objects $A, B$ of $\C$ the
  {\em $\epsilon$-weight cone-length of $A$ relative to $B$ with
    linearization in $\mathcal{F}$} is given by:
  \begin{equation}\label{eq:Number}
    N_{\C}(A, B;\F, \epsilon):=\inf\left\{  \#(\eta)\ \  | \ \ \eta: B\stackrel{\F^{\Sigma, T}}{\rightsquigarrow} A',  \  \ w(\eta)=0, \ \ \dint (A, A')\leq \epsilon \right\}.
  \end{equation}
\end{dfn}
In case the category $\C$ is clear from the context we write instead
$N(A,B;\F,\epsilon)$. Moreover, we write $N(A;\F,\epsilon)$ when
$B=0$, $$N(A;\F,\epsilon) := N(A, 0;\F,\epsilon)~.~$$ In short,
$N(A,B;\F,\epsilon)$ tells us how many iterated exact triangles in
the (usual) triangulated category $\C^{0}$ are needed to get
$\epsilon$-close to $A$ in the interleaving pseudo-distance,
starting from $B$, by attaching cones over objects of the form
$T^{i}\Sigma^{\alpha}F$, $F\in \F$, $\alpha\in\R$, $i\in \Z$.
 
\
 
The relation between weighted cone-length and approximability is that
if $\F$ is finite and $N_{\C}(A;\F,\epsilon) < \infty$,
$\forall \ A\in\Ob(\C)$, then $\F$ is an $\epsilon$-approximating
family in the sense of Definition \ref{def:simple_approx}, and
$g_{\C}(X;\epsilon) \leq \# (\F)$.
 
\begin{rem}\label{rem:cone-l} a. The definition of
  $N(A,B;\F,\epsilon)$ above is, in some sense, the simplest
  possible as it does not require any TPC machinery. Indeed, $\eta$ in
  (\ref{eq:Number}) consists of exact triangles in the usual sense, in
  $\C^{0}$, and $\dint(-)$ is the interleaving pseudo-metric which is
  already defined in a persistence category, see
  (\ref{eq:d-int}). Another measurement, more natural from the TPC
  perspective, can be defined by
  $$N'(A, B;\F,\epsilon)=\inf\{\ \#(\eta)\ \ | \ \
  \eta:B\stackrel{\F^{\Sigma, T}}{\rightsquigarrow} A\ \ , \ \
  w(\eta)\leq \epsilon \}~.~$$ This does not bring any essential
  new information because basic TPC algebra - see proof of Lemma 2.87
  in \cite{BCZ:tpc} - implies:
  $$N(A;\F,2\epsilon)\leq N'(A;\F,\epsilon)\leq  N(A;\F,\frac{\epsilon}{4})~.~$$

  b. Weighted cone-length satisfies some simple splitting inequalities
  associated with refinement of decompositions. These have
  generally a simpler expression for $N'(-,-)$ rather than for
  $N(-,-)$. For instance, we have
 $$N'(A,A'';\F,\epsilon'+\epsilon'')
 \leq N'(A, A';\F,\epsilon')+N'(A', A'';\F,\epsilon'')$$ which
 leads to the definition of fragmentation metrics on $\Ob (\C)$ (see
 \cite{BCZ:tpc}).

 c. The definition of cone-length above is very flexible.  By changing
 the family of objects over which cones are attached one obtains
 several other variants that are also of interest. One such choice,
 leading to smaller values for the resulting cone-length, is to
 replace $\F^{\Sigma, T}$ by
 $$ \F^{\otimes}=\{ F\otimes V \ |
 \ F\in \F, \ \ \ V \ \mathrm{\ is\ a\ finite\ dimensional, \ graded,\
   filtered \ vector\ space} \}$$ with the convention that
 $F\otimes \k[a] = \Sigma^{v(a)}T^{|a|}F$ where $\k[a]$ is the
 $1$-dimensional vector space generated by $a$ with
 degree $|a|$ and filtration level $v(a)$. Notice that
 $\F^{\Sigma, T}\subset \F^{\otimes}$. The resulting notion of cone
 length is such that, at each stage, one may attach a cone over a
 finite sum of copies of shifts and translates of elements of $\F$.
 We will denote this notion of cone-length by
 $N(A, B;\F^{\otimes},\epsilon)$.

 d. The numbers $N(A,B;\F,\epsilon)$ are decreasing in
 $\epsilon$ and, by taking $\epsilon \to \infty$, we obtain at
 the limit variants of the same notions that are often more familiar
 from standard topology and homological algebra.
\end{rem}
 
%\subsubsection{Lower bounds for weighted cone-length}\label{subsec:lowbd}
Under certain constraints, there exists a useful lower bound for
$N_{\C}(L;\F,\epsilon)$ which is well defined whenever $\F$ is
finite.  The inspiration for much of the algebraic considerations
related to the next result is found in
Dimitrov-Haiden-Katzarkov-Kontsevich \cite{DHKK:dyn_cat}.

\begin{prop} \label{prop:lower-bd}Let $\C$ be a TPC which is the homological category of a pre-triangulated category
  of filtered $A_{\infty}$-modules over a strictly unital filtered
  $A_{\infty}$-category and such that $\hom_{\C}(X,Y)$ is of finite
  type for all $X,Y\in \Ob(\C)$. Assume that $\F\subset \Ob (\C)$ is
  finite and consists of Yoneda modules. Under these assumptions,
  there exists a constant $k(\F)$ depending on the family $\F$ (and on
  $\C$) such that for every $L\in \Ob (\C)$ we have:
  $$N(L;\F,\epsilon)\geq k(\F) \sum_{F\in \F}\#
  \left( \mathcal{B}^{2\epsilon}_{\hom_{\C}(F,L)}\right)$$ where
  $\mathcal{B}^{\delta}_{V}$ is the barcode consisting of the bars of
  length greater than $\delta$ in the persistence module $V$.
\end{prop}
Recall that a persistence module is of finite type if its barcode
contains finitely many bars. The units of a filtered
$A_{\infty}$-category are always assumed to be in filtration $0$ (see
\S\ref{gio:fyon-def} for our conventions).

The conditions on $\C$ in the statement are satisfied for some of the
TPCs of Fukaya type that are among the main examples of interest in
the paper as well as for the homotopy category
$H^0(\mathcal{F}\ch^{fg})$ of finitely generated filtered
chain complexes over the field $\k$ (our ground field). We denote by
$\hfch$ the corresponding category without the finite generation
condition. Both are TPCs as shown in \cite{BCZ:tpc}.

The proof of the proposition appears in \S\ref{subsec:proof_Prop} and
is based on the results in the next two subsections that have
some interest in themselves.

\subsubsection{Weighted retracts}\label{subsec:retr} We fix here a
triangulated persistence category $\C$ and we refer again to
\cite{BCZ:tpc} for the basic definitions and notation relevant to
TPCs.

\

Similarly to (\ref{eq:Number}) we define the $\epsilon$-weight {\em
  retract} cone-length of $A$ relative to $B$ by:
\begin{equation}\label{eq:Number2}
  N^{r}_{\C}(A, B;\mathcal{F}, \epsilon)=\inf\left\{  \#(\eta)\ \
    | \ \ \eta: B\stackrel{\mathcal{F}^{\Sigma,T}}{\rightsquigarrow} A',
    \  \ w(\eta)=0, \ \ \dret  (A, A') \leq \epsilon \right\}
 \end{equation}
 (we recall that $\dret$ is defined in (\ref{eq:d-rint})).  In what
 follows it is useful to use the following terminology: given two
 objects $A,\bar{A}$ in $\C$ we will say that $A$ is {\em an
   $\epsilon$-retract of} $\bar{A}$ if
 $\dret(A, \bar{A}) < \epsilon$.  It is obvious that
 $$N^{r}_{\C}(A, B;\mathcal{F}, \epsilon)
 \leq N_{\C}(A, B;\mathcal{F}, \epsilon)~.~$$ The point of the
 definition is that the retract cone-length is, in practice, easier to
 estimate as we will see below. As before, if $B=0$ we omit it from
 the notation and we skip $\C$ if it is clear from context. The proof
 of Proposition \ref{prop:lower-bd} will make use \pbred{of} some
 algebraic properties of $N^{r}( -; -, \epsilon)$ and, indeed, the
 argument shows a stronger inequality than the one in the statement of
 the Proposition, namely:
 
 \begin{equation}\label{eq:ineq-r-bar}
   N^{r}(L;\F,\epsilon)\geq k(\F) \sum_{F\in \F}\#
   \left( \mathcal{B}^{2\epsilon}_{\hom_{\C}(F,L)}\right)
\end{equation}

\begin{rem}
  a. There is yet another variant of weighted retract cone-length
  which is defined by:
  $$\tilde{N}^{r}(A, B;\mathcal{F}, \epsilon)=
  \inf\left\{ N'(\bar{A}, B;\mathcal{F}, \epsilon) \ | \ \dret
    (A,\bar{A})=0\right \}.$$ This variant is particularly relevant if
  $\C^{0}$ is split complete in which case the condition
  $$\dret (A,\bar{A})=0$$ means that $A$
  is isomorphic to a direct summand of $\bar{A}$.

  b. It is easy to see using Remark \ref{rem:cone-l} a. that:
  $$N^{r}(A, B;\mathcal{F}, 2\epsilon)\leq
  \tilde{N}^{r}(A, B;\mathcal{F}, \epsilon)~.~$$

  c. The definition in (\ref{eq:Number2}) and the point a.~of the
  remark, are reminiscent of the relation between the
  Lusternik-Schnirelmann (LS) category and cone-length in topology,
  Theorem 1.1 in \cite{Cor:cone-LS}: the LS category of a topological
  space $X$ is the smallest integer $n$ such that $X$ is a homotopy
  factor of an $n$-cone.

\end{rem}

For the triangulated persistence category $\C$ and
$\mathcal{F}\subset \Ob (\C)$, assume that $\mathcal{F}$ is finite.
Let $G_{\F}=\oplus_{F\in\F}F$.  We will identify the family
$\{G_{\F}\}$ with the single object $G_{\F}$. One useful property of
the weighted retract cone-length is: 
\begin{lem}\label{lem:comp-sw}
  With the notation above we have
  $$N^{r}(A, B;G_{\F}, \epsilon) \leq N^{r}(A, B;\mathcal{F},
  \epsilon) \leq N^{r}(A, B;G_{\F}, \epsilon) \cdot
  \#(\mathcal{F}).$$
\end{lem}

The proof of the lemma is a simple exercise in manipulation of exact
triangles in triangulated categories, in this case in $\C_{0}$.  The
interest of the statement is that the cone-decomposition giving
$N^{r}(A,B; G_{\F},\epsilon)$ is a sequence of exact triangles as
in (\ref{eq:triang-dec}) but such that each $F_{i}$ is replaced by a
possible shift and/or translate of the single object
$G_{\mathcal{F}}$. As a result, it makes sense to consider the
measurement $N_{\C}^{r}(A,B; G,\epsilon)$ for any three objects
$A,B,G$ in $\C$.

\

Both variants of weighted retract cone-length, $N^{r}$ as well as
$\tilde{N}^{r}$, satisfy the additive splitting inequality in Remark
\ref{rem:cone-l} b but also a different, multiplicative one which
requires some assumptions on $\C$. Without weights, this type of
inequality was first noticed and used by
Dimitrov-Haiden-Katzarkov-Kontsevich \cite{DHKK:dyn_cat}. The
  weighted version will be instrumental here in proving
Proposition~\ref{prop:lower-bd}.

\begin{lem}\label{lem:hard-ineq} Assume that $\C$ is the persistence
 homological category associated with a
  pre-triangulated, filtered $A_{\infty}$-category. For any
  $A,G,G'\in \Ob (\C)$ we have:
  $$ N^{r}_{\C}(A;G,\epsilon)\leq N^{r}_{\C}(A;G',\epsilon')
  N^{r}_{\C}(G'; G,\epsilon'')$$ whenever
  $\epsilon \geq \epsilon ' + N^{r}_{\C}(A;G',\epsilon')
  \epsilon''$.

  The same formula remains true for $\tilde{N}^{r}(-;-,-)$ under a
  different assumption on $\C$, namely that $\C^{0}$ is split
  complete.
\end{lem}

\begin{rem} a. The relation tying
  $\epsilon, \epsilon', \epsilon''$ in the lemma is
  complicated but in this paper we will only need to use the case when
  $\epsilon''=0$ and $\epsilon=\epsilon'$.

  b. The weighted cone length $N_{\C}(-;-,\epsilon)$ also satisfies
  a multiplicative inequality of the type in Lemma \ref{lem:hard-ineq}
  but the precise formula is more complicated to state and thus we
  skip it here as it is not necessary later in the paper.

  c. The main examples of interest in this paper of categories $\C$ as
  in the first part of the Lemma are the homological category
  of filtered modules over a filtered $A_{\infty}$-category as well
  as, the simplest example, the homotopy category of filtered chain
  complexes.
\end{rem}

\begin{proof}[Proof of Lemma \ref{lem:hard-ineq}] The statement for
  $N^{r}_{\C}(-)$ is an immediate consequence of Lemma 6.12 in
  \cite{Bi-Co-Sh:LagrSh}.  We recall the statement of this lemma
  reformulated for filtered modules. Assume that $\mathcal{M}$ is an
  iterated cone of filtered $A_{\infty}$-modules (over $0$-shift
  maps):
  \begin{eqnarray*}
    \mathcal{M}=\tcn(\K_{s} \to \tcn(\K_{s-1}\to \ldots
    \to\tcn(\mathcal{N}\to \tcn(\K_{i-1} \to \ldots \to \\ \to
    \tcn (\K_{2}\to \K_{1})\ldots ))~.~
  \end{eqnarray*}
  Lemma 6.12 from \cite{Bi-Co-Sh:LagrSh} states that if $\mathcal{N}$
  is an $r$-retract of another filtered module $\mathcal{N}'$, then
  $\mathcal{M}$ is an $r$-retract of a filtered module $\mathcal{M}'$
  which is an iterated cone of the same form as $\mathcal{M}$ but with
  $\mathcal{N}'$ replacing $\mathcal{N}$ in the decomposition above.
  The proof of this result is based on the fact that in a category
  such as those in the statement of \ref{lem:hard-ineq} a homotopy
  commutative square has the property that the map induced between the
  cones over the two horizontal maps in the square can be expressed
  explicitly by a formula involving the two vertical maps and the
  commuting homotopy.

  This result is applied in our context as follows. We start with a
  decomposition in $\C^{0}$
  $\eta: 0\stackrel{(G')^{\Sigma,T}}{\cobto} \bar{A}$ and we assume
  that $A$ is an $\epsilon'$-retract of $\bar{A}$. We also assume
  that there is a second decomposition, also in $\C^{0}$,
  $\eta':0\stackrel{G^{\Sigma,T}}{\cobto}\bar{G}'$ and we assume that
  $G'$ is an $\epsilon''$-retract of $\bar{G}'$. By applying the
  algebraic result mentioned above to each of the terms in the
  linearization $\ell(\eta)$ of $\eta$ we transform $\eta$ into a new
  decomposition
  $\bar{\eta}:0\stackrel{(\bar{G}')^{\Sigma,T}}{\cobto} \bar{A}'$,
  still in $\C^{0}$, such that $A$ is now an
  $(\epsilon' + \#\ell(\eta) \epsilon'')$-retract of $\bar{A'}$.
  Finally, by refining each term in the linearization of $\bar{\eta}$
  by using the decomposition $\eta'$ we obtain the final decomposition
  $\bar{\eta}':0\stackrel{G^{\Sigma}}{\cobto} \bar{A}'$ which shows
  the claim (see Proposition 2.55 in \cite{BCZ:tpc} for the precise
  description of the refinement process).

  The argument for $\tilde{N}^{r}(-;--)$, under the assumption that
  $\C^{0}$ is split-complete, is similar but simpler as it does not
  require Lemma 6.12 from \cite{Bi-Co-Sh:LagrSh}, and we leave it as
  an exercise.
\end{proof}

One more property of weighted retract cone-length is needed.
\begin{lem}\label{lem:coef-ch} Assume that $\C$ satisfies the
  hypotheses in the statement of Proposition \ref{prop:lower-bd}. Let
  any $A, B\in\Ob(\C)$ and assume that $X\in \Ob(\C)$ is a Yoneda
  module. For any $\epsilon\geq 0$ we have:
  $$N^{r}_{\C}(A;B,\epsilon) \geq
  N^{r}_{\hfch}(\hom_{\C}(X,A);\hom_{\C}(X,B), \epsilon)~.~$$
\end{lem}

\begin{proof} Denote by $\D$ the pre-triangulated
  $A_{\infty}$-category of modules such that $\C=H^0\D$. Given
  that the underlying $A_{\infty}$-category is strictly unital, and
  that $X$ is a Yoneda module, we see that the functor
  $\hom_{\D}(X, -)$ transforms exact triangles in $\D$ into exact
  triangles in $[H^0(\mathcal{F}\ch)]^0$. This follows from
  the identification $\mathcal{M}(X)\cong \hom_{\D}(X,\mathcal{M})$,
  for each module $\mathcal{M}$, that is part of the properties of the
  Yoneda embedding functor - see \S\ref{gio:fyon-def} for the
  persistence setting - and the explicit formula giving the cone of a
  module morphism $\phi:\mathcal{M}\to \mathcal{N}$, \S3e in
  \cite{Se:book-fukaya-categ}. By applying $\hom_{\D}(X,-)$ this
  formula reduces to the usual formula giving the cone of the induced
  chain morphism
  $\phi_{X}:\hom(X,\mathcal{M})\to\hom(X,\mathcal{N})$. All these
  constructions respect filtrations and all the cones are taken over
  shift $0$ morphisms. As a result, exact triangles in $\C^{0}$ are
  represented by cone-decompositions in $\D$, these are sent by
  $\hom_{\D}(X,-)$ to corresponding cone-decomposition in
  $\mathcal{F}\ch^0$ that, in turn, give exact triangles in
  $[\hfch]^{0}$. The inequality in the statement is an immediate
  consequence of this observation.
\end{proof}

\

We proceed under the assumptions in the statement of Proposition
\ref{prop:lower-bd} and we also assume, as there, that
$\mathcal{F}\subset \Ob (\C)$ is finite and consists of Yoneda
modules. In particular, we can consider $G_{\F}=\oplus_{F\in\F}F$. In
this case, Lemma \ref{lem:coef-ch} implies that:
\begin{equation} \label{eq:NrC-GF}
  N^{r}_{\C}(A; G_{\F},\epsilon)\geq
  N^{r}_{\hfch}(\hom_{\C}(G_{\F}, A);
  \hom_{\C}(G_{\F},G_{\F}),\epsilon).
\end{equation}

Before we go on we temporarily introduce some simplified
  notation that will be used from now on till the end
  of~\S\ref{subsec:proof_Prop}. Namely, we will use $\k$ to denote two
  different (albeit closely related) things. Recall that throughout
  the paper $\k$ stands for the ground field for our (unfiltered)
  categories and chain complexes, but we will also denote by the same
  $\k$ the filtered chain complex with a single generator in degree
  $0$ and of filtration level $0$. When we want to specify that
  generator, say $x$, we will denote this chain complex also by
  $\k \langle x \rangle$.

Getting back to our considerations and continuing
  from~\eqref{eq:NrC-GF} we deduce from Lemma \ref{lem:hard-ineq}
  that
\begin{eqnarray*}
  N^{r}_{\hfch}(\hom_{\C}(G_{\F}, A);
  \hom_{\C}(G_{\F},G_{\F}), \epsilon) \cdot
  N^{r}_{\hfch}(\hom_{\C}(G_{\F}, G_{\F});\k,0)\geq \\
  \geq 
  N^{r}_{\hfch}(\hom_{\C}(G_{\F}, A); \k,\epsilon)
\end{eqnarray*}
%\pbred{\pbst{Here $\k$ represents the filtered chain complex with a
  %  single generator in degree $0$ and of filtration $0$. }}
 Combining
this with Lemma~\ref{lem:comp-sw} we deduce the next inequality.

\begin{cor}\label{cor:nice-ineq} Under the assumptions of Proposition
  \ref{prop:lower-bd} and for
  $$k(\F):=1/ N^{r}_{\hfch}(\hom_{\C}(G_{\F}, G_{\F});\k,0)$$ we have
  for each object $A$ of $\C$ and all $\epsilon\geq 0$:
  $$ N^{r}_{\C}(A;\F,\epsilon) \geq  k(\F)
  N^{r}_{\hfch}(\hom_{\C}(G_{\F}, A); \k,\epsilon)~.~$$
\end{cor}

\subsubsection{Weighted cone-length for filtered chain
  complexes}\label{subsec:pers}
The aim of this subsection is to examine the various notions of
weighted cone-length that we introduced before in the very basic but
important example when the triangulated persistence category $\C$ is
the homotopy category of filtered, finite dimensional chain complexes,
$\hfch^{fg}$, over the base field $\k$.

%\pbred{\pbst{As before, we denote by $\k$ the filtered chain complex
   % with a single generator, which is in degree $0$ and of filtration
%    $0$.}}

\begin{ex}\label{ex:element} The following statements hold in the
  category $\hfch^{fg}$.

  a) Each object $V$ in $\hfch^{fg}$ can be written as a finite direct
  sum, unique up to permutation, of translations of elementary
  filtered chain complexes of two types
  $$E_{2}(a, b)=\k(a, b: da=0, db=a), \ \mathrm{and}\ \ E_{1}(c)=\k(c
  : dc=0).$$ For each of the generators $x$ of these chain complexes
  we denote by $v(x) \in \mathbb{R}$ its filtration level. For
  $E_2(a,b)$ we assume $v(b)\geq v(a)$ with degrees $|a|=0$,
  $|b|=1$. For $E_1(c)$ we assume $v(c)=0$ and $|c|=0$. Note
    that this notation differs form other conventions common in the
    literature. For example, in~\cite[Section~2.5.2]{BCZ:tpc}
    $E_2(a,b)$ denotes a chain complex with two generators as above
    but $a$, $b$ stand for the filtration levels of these two
    generators rather than for the generators themselves (and
    similarly for $E_1(c)$). In our notation the barcode of the
    persistence homology $H(E_2(a,b))$ is $[v(a), v(b))$ while
    in~\cite{BCZ:tpc} this barcode is $[a,b)$.

  b) The translation $T$ acts by shifting degrees in the obvious way:
  $|T^{k}x|=|x|+k$ and the shift functor changes the filtration levels
  by $v(\Sigma^{\alpha} x)=v(x)+\alpha$.

  c) Recall that $\k\langle x \rangle$ also stands for the
  $1$-dimensional filtered chain complex over $\k$ with one
    generator $x$, with $|x|=0$, $v(x)=0$. We then have:
  $$E_{2}(a,b)=\tcn (\ \Sigma^{v(b)}\k \langle b \rangle \ \to \
  \Sigma^{v(a)}\k\langle a \rangle\ )$$ where the cone is taken over
  the map $b\to a$. As for $E_1(c)$, we have:
  $$E_{1}(c)\cong \Sigma^{v(c)}\k\langle c \rangle ~.~$$
  As a result, for $\C=\hfch^{fg}$ and $X=\mathcal{O}b(\hfch^{fg})$ we
  have $g(X;0)=1$ (see (\ref{eq:no-of-gen})) with
  $\mathcal{F}_{0}=\k\langle x\rangle $. Thus $\hfch^{fg}$ is
  $0$-approximable with generating rank equal to $1$.
\end{ex}

The next result is also elementary but less immediate. We fix some
additional notation. For a persistence module $K$, we denote by
$\mathcal{B}_{K}$ its barcode, by $\mathcal{B}^{\delta}_{K}$ the
barcode consisting of only the bars of length $> \delta$ in $K$, and
by $K^{\infty}$ the $\infty$-limit of the persistence module $K$ (it
is a vector space of dimension equal to the number of semi-infinite
bars in $K$). 
%\pbred{\pbst{Moreover, we will denote from now on the
 %  chain complex $\k\langle x\rangle $ simply by $\k$ in case the
  %  labeling of the generator is irrelevant.}}
\begin{lem}\label{lem:noc-chains}
  Let $V$ be a filtered, finite dimensional chain complex.  Then:
  \begin{equation}\label{eq:identi} N(V,0;\k,\epsilon)=
    N^{r}(V,0;\k,\epsilon)= 2\#
    \left(\mathcal{B}^{2\epsilon}_{H(V)}\right) - \dim_{\k}
    (H(V)^{\infty}) ~.~
  \end{equation}
\end{lem}

\begin{proof}
  For a filtered chain complex $V$ in our class and $\delta\geq 0$,
  let $V_{\delta}$ be the chain complex obtained by writing $V$ as a
  direct sum of elementary terms, as at point a) in Example
  \ref{ex:element}, and eliminating from the sum all the terms
  $E_{2}(a,b)$ with $v(b)-v(a)\leq \delta$. Up to a filtered
    chain homotopy equivalence of $V$ there is an obvious inclusion
  of chain complexes $i_{V,\delta}: V_{\delta}\to V$ as well as a
  projection $j_{V,\delta}:V\to V_{\delta}$.  It is easy to see that:
\begin{itemize}
\item[i.] $V$ and $V_{2\delta}$ are $\delta$ interleaved. This happens
  because the map
  $\eta^{V}_{2\delta}:\Sigma^{\delta}V\to \Sigma_{-\delta}V$ is easily
  seen to be chain homotopic to the composition $i_{V}\circ j_{V}$
  (because $\eta^{E_{2}(a,b)}_{2\delta}$ is null-homotopic if
  $v(b)-v(a)\leq 2\delta$).
\item[ii.] The dimension of $V_{2\delta}$ is given by
  $2\# (\ \mathcal{B}^{2\delta}_{H(V)}\ ) - \dim_{\k} (H(V)^{\infty})$.
  This is because a pair of generators $a,b$ in a term of type
  $E_{2}(a,b)$ (in the direct sum writing of $V_{2\delta}$) are
  associated with a finite bar of the form $[v(a),v(b))$ and each
  generator of type $c$ in a term $E_{1}(c)$ corresponds to one
  semi-infinite bar, and survives to a generator in $H(V)^{\infty}$.
\item[iii.] There exists an obvious cone decomposition
  $\eta:0\stackrel{\langle \k\rangle^{\Sigma,T}}{\cobto} V_{2\delta}$
  with a number of cones equal to the dimension of $V_{2\delta}$.
\end{itemize}
As a result of these three points we conclude that
$N(V,0;\k,\epsilon)\leq 2\# \left(
  \mathcal{B}^{2\epsilon}_{H(V)}\right) - \dim_{\k} (H(V)^{\infty})$.

To conclude the proof of (\ref{eq:identi}) we need to show that, if
$V$ is an $\epsilon$-retract of $W$, then $W$ is at least a
$\dim_{\k}(V_{2\epsilon})$-cone for a decomposition
$\eta':0\cobto W$ with a linearization with terms of the form
$\Sigma^{\alpha}T^{i}\k$ (these are the elements of
$\langle\k\rangle^{\Sigma,T}$).

To show this we first notice that for any $W$ in our class of filtered
chain complexes we have that any cone decomposition $\eta:0\cobto W$
in $(\hfch^{fg})^{0}$ (this is the $0$-level category of the
TPC $\hfch^{fg}$) with linearization terms in
$\langle\k\rangle^{\Sigma,T}$ has at least $\dim_{\k}(W_{0})$ terms,
where $W_{0}$ is obtained from $W$ by omitting all terms
  $E_{2}(a,b)$ such that $v(a)=v(b)$. Indeed, any such decomposition
is the image of a decomposition $\bar{\eta}:0\cobto W'$ in the
underlying category of filtered chain complexes,
$\mathcal{F}{\bf Ch}$, and there is a filtered quasi-isomorphism (with
$0$-shift) $\phi:W\to W'$.  The number of cones in $\bar{\eta}$ is at
least $\dim_{\k}(W')$ for dimension reasons. On the other hand,
$\phi|_{W_{0}}$ is injective (this can be seen by noting that $\phi$
induces an identification $\mathcal{B}_{W}\cong \mathcal{B}_{W'}$ and
$\mathcal{B}_{W}=\mathcal{B}_{W_{0}}$).  Therefore, it remains to show
that if $V$ is an $\epsilon$-retract of $W$, then
$\dim_{\k}(W_{0})\geq \dim_{\k}(V_{2\epsilon})$. If $V$ is an
$\epsilon$-retract of $W$ it is immediate that it also is an
$\epsilon$-retract of $W_{0}$. It is a simple exercise to show that
in this case $V_{2\epsilon}$ injects into $W_{0}$, which concludes
the proof of the lemma.
\end{proof}

\begin{rem}\label{rem:another-ineq}
  An argument similar to the proof of Lemma \ref{lem:noc-chains}
  shows: 
    \begin{equation}\label{eq:identi2} N(0,V;\k,\epsilon)= 2\#
      \left(
      \mathcal{B}^{2\epsilon}_{H(V)}\right) - \dim_{\k} (H(V)^{\infty})~.~
  \end{equation}
 As a result of this identity and (\ref{eq:identi}), and of the
triangular inequality in Remark \ref{rem:cone-l} b) we obtain rough
upper bounds for $N(V,V'; \k,\epsilon)$,
$N^{r}(V,V';\k,\epsilon)$ in terms of the barcodes of
$H(V)$ and of $H(V')$.
\end{rem}

\subsubsection{Proof of Proposition \ref{prop:lower-bd}}
\label{subsec:proof_Prop}
Using the preparatory material from the last two sub-subsections 
we can finalize the proof of Proposition \ref{prop:lower-bd}.

We start with two other simple algebraic remarks concerning filtered
chain complexes. For this let $V$ be a finitely generated, filtered
chain complex.  First, for any $\delta\geq 0$, we obviously have
the inequality
  $$ \#\left( \mathcal{B}^{\delta}_{H(V)}\right)
  - \dim_{\k}(H(V)^{\infty})\geq 0$$ which we rewrite as
$2 \ \#(\mathcal{B}^{\delta}_{H(V)})
  -\dim_{\k}(H(V)^{\infty})\geq \#(\mathcal{B}^{\delta}_{H(V)})$.
Secondly, if $V'$ is another chain complex in our class we have

  $$\# \left( \mathcal{B}^{\delta}_{H(V\oplus V')} \right)=\#
  \left(\mathcal{B}^{\delta}_{H(V)}\right) +\# \left(
    \mathcal{B}^{\delta}_{H(V')}\right )~.~$$  Combining these
relations with Corollary \ref{cor:nice-ineq} and Lemma
\ref{lem:noc-chains} and recalling $G_{\F}=\oplus_{F\in\F}F$ we
deduce: 
\begin{eqnarray*}
  N_{\C}(L;\F,\epsilon) \geq  N^{r}_{\C}(L;\F,\epsilon)
  \geq k(\F) N^{r}_{\hfch}(\hom_{\C}(G_{\F}, L); \k,\epsilon) \geq \\ \geq 
  k(\F)\cdot  \#
  \left( \mathcal{B}^{2\epsilon}_{\hom_{\C}(G_{\F},L)}
  \right)=k(\F)\sum_{F\in\F} \#
  \left(\mathcal{B}^{2\epsilon}_{\hom_{\mathcal{C}}(F,L)}\right)
\end{eqnarray*}

which shows the inequality (\ref{eq:ineq-r-bar}) and also
concludes the proof of the proposition. \qed
%\end{proof}

\subsubsection{Weighted categorical entropy.}\label{subsec:entrop} We introduce here a weighted notion of categorical entropy that is the weighted analogue of a notion introduced by Dimitrov-Haiden-Katzarkov-Kontsevich in \cite{DHKK:dyn_cat}. 

Consider a triangulated persistence category $\C$ and fix $X\subset \Ob (\C)$ as well as a family $\mathcal{F}\subset \mathcal{O}b(\C)$, as in \S\ref{subsec:wconel}.   

\begin{dfn}\label{def:entropy} Given a (set theoretic) map $\Psi:  X\to X$ and one object $A\in X$, the $\epsilon$-weighted categorical entropy of $\Psi$ at $A$, relative to $\F\subset \C$, is defined by
$$h_{\F}(\Phi; A,\epsilon)=\limsup_{n\to \infty} \frac{\log(N_{\C}(\Psi^{n}A;\F,\epsilon))}{n}$$
\end{dfn}
Notice that this definition is only of interest when $\F$ is an $\epsilon$-approximating family for $X$ in $\C$ (as in Definition \ref{def:simple_approx}) which we will assume from now on, otherwise there is no way to ensure that $N_{\C}(\Psi^{n}A;\F, \epsilon)$ is finite. In other words, TPC approximability is needed for this definition to make sense.

\begin{rem}\label{rem:variants_ent}
a. Of course, a similar notion can be defined using $N_{\C}^{r}(-;--)$ leading to a variant denoted by $h_{\mathcal{F}}^{r}(\Psi; A,\epsilon)$.

b. The numbers $h_{\F}(\Psi;A,\epsilon)$, $h_{\mathcal{F}}^{r}(\Psi; A,\epsilon)$  are decreasing functions in $\epsilon$. It is useful to allow for $\epsilon=\infty$. In that case, the weight constraints disappear and we obtain purely (triangulated) categorical notions. In particular,  when $\F$ is finite, the categorical entropy introduced in \cite{DHKK:dyn_cat}  
is given by  $h_{G_{\F}}(\Psi;A,\infty)$ with $G_{\F}=\oplus_{F\in\F}F$. 

c. By contrast to the un-weighted case, the weighted entropy definition depends on the choice of the family $\F$. 

d. There are also notions of {\em slow} entropy that are defined by formulae of a similar type, in analogy with the standard corresponding dynamical systems definitions. 
\end{rem}

We will now see that Proposition \ref{prop:lower-bd} implies that weighted categorical entropy  admits as lower bound a version of the barcode entropy introduced in \cite{CGG_bar_ent}.

\

To make this relationship precise we define the notion of barcode entropy that we will use here in the setting
of the TPC, $\C$.  We have as before the map $\Psi:X\to X$ and we pick two objects $A \in X$, $B\in \Ob(\C)$. We now consider the 
$\epsilon$-barcode entropy of $\Psi$ at $A$, relative to $B$:
\begin{equation}\label{eq:bar-code_ent}
\hbar(\Psi; A, B;\epsilon)=\limsup_{n\to \infty}\frac{\log\left( \# (\mathcal{B}^{\epsilon}_{\hom_{\C}(B,\Psi^{n}A)})\right)}{n}
\end{equation}

\begin{cor}\label{cor:entr_rel} Assume that the assumptions in Proposition \ref{prop:lower-bd} are satisfied. In particular,  $\mathcal{F}$ is finite and all $\hom_{\C}(-,-)$ are finite type persistence modules.
We have the inequalities:
$$
h_{\F}(\Psi; A,\epsilon)\geq h^{r}_{\F}(\Psi; A,\epsilon) \geq  h_{G_{\F}}^{r}(\Psi; A,\epsilon)\geq \hbar(\Psi; A, G_{\F};2\epsilon)
$$
where $G_{\F}=\oplus_{F\in\F} F$.
\end{cor}

This follows directly from Proposition \ref{prop:lower-bd}, see also (\ref{eq:ineq-r-bar}).

\subsubsection{Weighted complexity of Lagrangians} \label{subsubsec:complex_Lag}
We return here to the geometric context from \S\ref{subsec:back} and discuss some implications of the nearby approximability Theorem \ref{thm:nearby} from the point of view of the complexity measurements
introduced earlier in this section. 

We start with some definitions of complexity that are natural from the point of view of TPC-approximability - 
see Definition \ref{def:TPC-approx} as well as Remarks \ref{rem:def-TPC-approx} and \ref{rem:relation-TPCapp}. 

\begin{dfn} Assume $(X,d)$ is TPC-approximable and that, for a fixed $\epsilon >0$, the quasi-isometric embeddings $$\Phi = \{\Phi_{\eta}\} \ , \ \Phi_{\eta}: (X,d)\to (\Ob (\mathscr{Y}_{\eta}), \bar{d}_{\mathrm{int}}^{\ \eta})~.~$$
 that are part of the $\epsilon$-TPC approximating data are fixed. 
 \begin{itemize}
\item[i.] The $\epsilon$-generating rank of $(X,d)$ relative to $\Phi$ is defined by 
$$g_{\Phi}(X;\epsilon)= \limsup_{\eta\to 0}\   g_{\mathscr{Y}_{\eta}}(\Phi_{\eta}(X);\epsilon)  $$ See  (\ref{eq:no-of-gen}) for the definition of $g_{\mathscr{Y}_{\eta}}$.
\item[ii.] At this point  fix also a corresponding family $\mathcal{F}_{\epsilon,\eta}\subset \Ob (\mathscr{Y}_{\eta})$ as in Definition \ref{def:TPC-approx}.  For some $\epsilon'\geq \epsilon$, the $\epsilon'$-cone-length of an element $a\in X$ relative to the family $\F_{\epsilon}= \{\F_{\epsilon,\eta}\}_{0< \eta<\epsilon }$  is given by:
$$N_{\Phi}(a; \F_{\epsilon},\epsilon')= \limsup_{\eta\to 0} \  N_{\mathscr{Y}_{\eta}}(\Phi_{\eta}(a); \F_{\epsilon,\eta}, \epsilon') ~.~$$
\end{itemize}
\end{dfn}

Notice that for $\epsilon_{1}\geq \epsilon_{2}\geq \epsilon$ we have $ N_{\Phi}(a; \F_{\epsilon},\epsilon_{1})\leq N_{\Phi}(a; \F_{\epsilon},\epsilon_{2})$.
These notions have properties very similar to those of the corresponding measurements from \S\ref{subsec:wconel}. In particular, given a map $\Psi : X\to X$ and assuming that $(X,d)$ is TPC-approximable we define the $\epsilon'$-weighted categorical entropy of $\Psi$ at $a\in X$, relative
to the approximating data ($\Phi$, $\F_{\epsilon}=\{\F_{\epsilon,\eta}\})$, for $\epsilon'\geq \epsilon$ by

\begin{equation}\label{eq:cat_entropy2}
h_{\Phi,\F_{\epsilon}}(\Psi;a,\epsilon')=\limsup_{n\to\infty}\frac{\log  N_{\Phi}(\Psi^{n}a;\F_{\epsilon},\epsilon')}{n}~.~
\end{equation}

We also have similar measurements for the corresponding retract notions that will be denoted by  
$N^{r}_{\Phi}(a; \F_{\epsilon},\epsilon')$, $h^{r}_{\Phi,\F_{\epsilon}}(\Psi;a,\epsilon')$
respectively. 
\begin{rem}\label{rem:tb_complex} 
 Assume that the metric space $(X,d)$ is totally bounded and that it is TPC-approximable (respectively, retract-approximable) with a system
  of $(A,\eta)$-quasi-isometric embeddings $\Phi=\{\Phi_{\eta}\}$ and a corresponding family $\{\F_{\epsilon,\eta}\}$ as above. 
Then, for each fixed $\epsilon$, there exists $K >0$ depending on $(X,d)$, $\Phi$ and $\{\F_{\epsilon,\eta}\}$ 
such that $$N_{\Phi}(a;\F_{\epsilon},2\epsilon)\leq K$$ (respectively, $N^{r}_{\Phi}(a;\F_{\epsilon},2\epsilon)\leq K$) 
for all $a\in X$. This is a simple exercise: the bound $K$ is given by considering a finite $\epsilon_{1}$-net of $(X,d)$, $x_{1},\ldots, x_{k}\subset X$
such that $\epsilon_{1}/A < \epsilon$ and  taking $$K=\max_{i} N_{\mathscr{Y}_{\eta'}} (\Phi_{\eta'}(x_{i});\mathcal{F}_{\epsilon,\eta'}, \epsilon_{1}) \ \ (\mathrm{respectively}\ K=\max_{i} N^{r}_{\mathscr{Y}_{\eta'}} (\Phi_{\eta'}(x_{i});\mathcal{F}_{\epsilon,\eta'}, \epsilon) \ )$$  for some $\eta' < \epsilon$. In particular, if $(X,d)$ is totally bounded, 
for all $x\in X$ and any map $\Psi:X\to X$
the entropies $h_{\Phi,\F_{\epsilon}}(\Psi; a, 2\epsilon)$ and $h^{r}_{\Phi,\F_{\epsilon}}(\Psi; a, 2\epsilon)$ vanish. Thus, to show that $(X,d)$ is not totally bounded, it is enough to find one $\epsilon>0$ and a sequence $a_{k}\in X$ such that 
$N_{\phi}(a_{k};\F_{\epsilon}, 2\epsilon)\to \infty$ in $k$.
\end{rem}

We now turn to geometry.  We have shown in \S\ref{sec:nearby} that the metric space $$\mathcal{L}(N)=(\lag^{(ex)}(D^{\ast}N), d_{\gamma})$$ is approximable in the sense of Definition \ref{def:TPC-approx}. 
In \S\ref{subsubsec:local_approx_d} and \S\ref{subsubsec:amb_approx_d} we produced two types of TPC  $\epsilon$- approximating data for this space: $(\Phi, \F_{\epsilon})$ in the first case and $(\Phi', \F'_{\epsilon})$ in the second, which is only {\em retract} approximating. We referred to the first type as {\em local} and to the second as {\em ambient}.

\

To improve readability we recall the basic features of these two choices of approximating data.
The family $\Phi=\{\Phi_{\eta}\}$ of quasi-isometric embeddings 
consists of Yoneda embeddings - see \S\ref{subsec:nearby-TPC}:
$$\Phi_{\eta}= \mathcal{Y}_{\epsilon', p}: \lag^{(ex)}(D^{\ast}N)\to \Ob (\msc_{p})~.~$$ 
Here  $\msc_{p}$ are persistence homology categories of filtered $A_{\infty}$-modules over the filtered 
Fukaya category $\fuk(\lag^{(ex)}(D^{\ast}N), p)$ constructed for the choice of perturbations $p$, with $\nu(p) < \eta$ and $\epsilon'<\epsilon$. 
The family $\F_{\epsilon}$ consists of modules corresponding to a finite family 
$\{ F_{0},\ldots,  F_{x_{m}}\}$ of fibers of $D^{\ast}N$ where $\{ x_{0},x_{1},\ldots, x_{m}\}$ are the critical points of an auxiliary Morse function 
$\varphi_{N,K,\delta}:N\to [0, K]$.
The family $\Phi'$ is very similar except that the place of $\msc_{p}$ is taken  by the persistence derived category 
$$\mathscr{D}_{p}(E)=PD(\fuk(\mathcal{L}ag^{(ex)}(E), p))$$
that is constructed using the auxiliary Lefschetz fibration $\bar{h}:E\to \mathbb{C}$ (see \S\ref{subsec:Lef-Dehn})
 that itself depends on $\varphi_{N,K,\delta}$. The family $\F'_{\epsilon}$ consists of the Yoneda modules of 
Lagrangian spheres $\hat{S}_{x_{i}}$, one for each critical point of $\varphi_{N,K,\delta}$ and such that
$\hat{S}_{x_{i}}\cap D^{\ast}N=F_{x_{i}}$.

In essence, one can view the approximability results in $\msc_{p}$ as obtained by pull-back from those
in $\mathscr{D}_{p}(E)$. Notice though that $(\Phi',\F'_{\epsilon})$  is only  retract $\epsilon$-approximating,  as indicated in Corollary  \ref{thm:nearby_far}.

\begin{rem}\label{rem:coherence}
a. Recall that the comparison functors $\mathcal{H}_{p,q}:\msc_{p}\to \msc_{q}$ are TPC-functors for $p \preceq q$.  
They also preserve the generating families $\F_{\epsilon}$  in the sense that they send
the module  corresponding to a fibre $F_{x_{i}}$ in the $p$-category $\msc_{p}$ to a module $0$-isomorphic to the  module associated to the same $F_{x_{i}}$ in $\msc_{q}$.  As a result, with $\epsilon$ fixed, our approximating families
$\F_{\epsilon}$ are no longer dependent of $\eta$. Moreover, for each $L\in \lag^{(ex)}(D^{\ast}N)$
we have $$N_{\msc_{p}}(L;\F_{\epsilon}, \epsilon)\geq  N_{\msc_{q}}(L;\F_{\epsilon}, \epsilon)$$
as soon as $p\preceq q$. The same remark also applies to the system of categories $\{\mathscr{D}_{p}(E)\}_{p}$ and
the family $\F'_{\epsilon}$, and leads to the same inequality, with $\mathscr{D}_{-}(E)$ in the place of $\msc_{-}$ and as well with $N^{r}(-;-\epsilon)$ replacing $N(-;-\epsilon)$.

b. In the setting above we also have the inequality:
$$N_{\msc_{q}}(L;\F_{\epsilon}, \epsilon)\geq  N_{\msc_{p}}\left(L;\F_{\epsilon}, \epsilon + c |\eta(p)-\eta(q)|\right)$$ also under the assumption
$p\preceq q$ and for $c$ a constant independent of $L$. 
\end{rem}

\

Given that the number of elements in $\F_{\epsilon}$ and in $\F'_{\epsilon}$ is
equal to the number of critical points of $\varphi_{N,K,\delta}$  we deduce:
$$g_{\Phi}(\mathcal{L}(N);\epsilon)\leq  \# \ \Crit (\varphi_{N,K, \delta})$$
as well as the same inequality for $\Phi'$.
When $\epsilon \to 0$ the number of elements in $\F_{\epsilon}$ and $\F'_{\epsilon}$ 
goes to $\infty$ at a rate that can be estimated through the construction in  \S\ref{subsec:Morse}.

\

We now discuss in the same context and with the same notation the $\epsilon$-weight cone-length
$N_{\Phi}(L;\F_{\epsilon}, \epsilon)$ for $L\in \lag^{(ex)}(D^{\ast}N)$.

An upper bound for this cone-length can be obtained by counting all the exact triangles $\Delta_{i,j}$ from (\ref{eq:tr-ref}) that  give the decomposition of $L$. As a result, $m(i)\leq 2 \#(\mathcal{B}_ {HF(\hat{S}_{x_{m-i}}, L_{i})})$ and 
 thus:
 \begin{equation}\label{eq:upperBars}N_{\Phi}(L;\F_{\epsilon},\epsilon)\leq\sum_{i=0}^{m}2 \#(\mathcal{B}_ {HF(\hat{S}_{x_{m-i}}, L_{i})}) ~.~
 \end{equation}
The Lagrangians $L_{i}\subset E$ are related through Dehn  twists  - as discussed in \S\ref{subsubsec:back_to_alg}: $L_{m+1}=L$,  $L_{m-i}=\tau_{\hat{S}_{i}}(L_{m-i+1})$.
As a result, we can obtain iteratively an upper bound for $\#(\mathcal{B}_{HF(\hat{S}_{x_{m-i}}, L_{i})})$, and, moreover it is easy to see that there exists a constant $k'(\F_{\epsilon})$, independent of $L$, as well as $\epsilon''>0$, $\epsilon'' \ll \epsilon$, depending again on $\F_{\epsilon}$ and not on $L$, such that
\begin{equation}\label{eq:upper_bound}
N_{\Phi}(L;\F_{\epsilon}, \epsilon)\leq k'(\F_{\epsilon})\ \max_{S\in \F_{\epsilon}}\ \left( \#(\mathcal{B}^{\epsilon''}_{HF(S,L)}) \right)
\end{equation}
These estimates are, of course, first established for $N_{\mathscr{Y}_{\eta}}(\Phi_{\eta}(a); \F_{\epsilon,\eta}, \epsilon)$ in  $\msc_{p}^{\epsilon}$ for $\nu(p)$ sufficiently small. By making $\nu(p)\to 0$ we see that they apply  also to $N_{\Phi}(L;\F_{\epsilon}, \epsilon)$, as stated above. 
The estimates  (\ref{eq:upperBars})  and (\ref{eq:upper_bound}) remain obviously true also for $\Phi',\F'_{\epsilon}$ by using $N^{r}_{\Phi'}( - )$ in that case.  

\

In the opposite direction, from Proposition \ref{prop:lower-bd}  together with the stronger form in (\ref{eq:ineq-r-bar}) we deduce a lower bound for weighted cone-length and entropy relative to the data $(\Phi',\F'_{\epsilon})$:

\begin{cor} \label{cor:ineq_bar2} For each $L\in \lag^{(ex)}(D^{\ast}N)$ and $\epsilon'\geq \epsilon$ we 
have:
$$N^{r}_{\Phi'}(L;\F'_{\epsilon},\epsilon')\geq k(\F'_{\epsilon}) \sum_{S\in \F'_{\epsilon}} \#(\mathcal{B}^{2\epsilon'}_{HF(S,L)})$$
for a constant $k(\F'_{\epsilon})$  only depending on the family $\F'_{\epsilon}$ and we also have:
$$h^{r}_{\Phi',\F'_{\epsilon}}(\Psi; L ,\epsilon')\geq  \hbar(\Psi; L, G_{\F'_{\epsilon}}; 2\epsilon')$$ 
for $G_{\F'_{\epsilon}}=\oplus_{S\in \F'_{\epsilon}} S$.
\end{cor}
Corollary \ref{cor:ineq_bar2} is stated for the ambient $\epsilon$-TPC -approximating data $(\Phi',\F'_{\epsilon})$. This is essential for two reasons: first, the category  $\mathscr{D}_{p}(E)$ has the property that $\hom_{\mathscr{D}_{p}(E)}(A,B)$ is a persistence module of finite type $\forall A, B\in \Ob (\mathscr{D}_{p}(E))$, as noted in Remark  \ref{rem:compare_data}, and, secondly, the family $\F'_{\epsilon}$ consists of Yoneda modules. Thus, we may apply 
Proposition \ref{prop:lower-bd} for each $\mathscr{D}_{p}(E)$ 
and by then letting $\nu(p)\to 0$ the statement follows. In contrast,  it is not clear whether a similar statement is valid
for the local approximating data $(\Phi, \F_{\epsilon})$. Indeed, even the constant $k(-)$, introduced in Corollary \ref{cor:nice-ineq}, is not {\em apriori} well-defined for $\F_{\epsilon}$. 

\begin{rem}\label{rem:entrop_bars}
In the setting above,  by combining  Corollary \ref{cor:ineq_bar2}
and (\ref{eq:upper_bound}) we deduce

\begin{equation}\label{eq:entropy2}\hbar(\Psi; L, G_{\F'_{\epsilon}}; \epsilon'') \geq h^{r}_{\Phi',\F'_{\epsilon}}(\Psi; L ,\epsilon) \geq 
\hbar(\Psi; L, G_{\F'_{\epsilon}}; 2\epsilon) \ \  \mathrm{for\ }\ 0<\epsilon'' \ll \epsilon
\end{equation}
which shows that, in this situation ($M=D^{\ast}N$ etc),  weighted categorical entropy is very closely related to the barcode entropy from \cite{CGG_bar_ent}.  In \cite{DHKK:dyn_cat} the authors obtain an upper bound for (the unweighted) categorical entropy by an algebraic approach that is  more general than the one used here so that, in principle, one expects an inequality such as (\ref{eq:entropy2}) to remain true in
wider generality.
\end{rem}

\subsection{Non-vanishing weighted categorical entropy and Corollary \ref{cor:no-t-b}}
\label{subsec:non-v_Cor} The aim of this subsection is to produce a class of examples with non-vanishing $\epsilon$-weighted categorical entropy in the sense of (\ref{eq:cat_entropy2}) but with vanishing  ``classical'' categorical entropy. We formulate this more precisely in the next statement. 

\begin{prop}\label{cor:hyp_ent} Let $(N,\mathtt{g})$  be a closed hyperbolic manifold. Consider the metric space 
$(\lag^{(ex)}(D^{\ast}N), d_{\gamma})$ and, for some small $\epsilon$, fix the ambient $\epsilon$-TPC - approximating data
$(\Phi',\F'_{\epsilon})$ as recalled in \S\ref{subsubsec:complex_Lag}.
There exists  a Lagrangian $L\in\lag^{(ex)}(D^{\ast}N)$ Hamiltonian isotopic to the zero section and a Hamiltonian diffeomorphism $\Psi: D^{\ast}N\to D^{\ast}N$ with support inside $D^{\ast}N$, such that
$$h^{r}_{\Phi',\F'_{\epsilon}}(\Psi; L, 2\epsilon)>h^{r}_{\Phi',\F'_{\epsilon}}(\Psi;L,\infty)=0$$
where the relevant $\epsilon$-weighted entropy is defined in (\ref{eq:cat_entropy2}).
\end{prop}

The result 
means that filtered categorical algebra is sensitive to the dynamics in a cotangent bundle in ways in which the non-filtered version is not.  Moreover, in view of Remark \ref{rem:tb_complex}, this proposition implies one part of Corollary \ref{cor:no-t-b} from the introduction, namely that the metric space $(\lag^{(ex)}(D^{\ast}N), d_{\gamma})$ is not totally bounded. The second part of the Corollary that refers to $M=S^{2}$ and with 
$\lag(S^{2})$ consisting of equators is addressed in \S\ref{subsec:tb_S2}.

\

The vanishing of the non-weighted categorical entropy, $h^{r}_{\Phi',\F'_{\epsilon}}(\Psi;L,\infty)$, in the statement is immediate  because  all $\Psi^{n}L$ are Hamiltonian isotopic to the zero section and thus, without weight constraints, they admit cone-decompositions with the same number of terms. 

\

Till now our arguments did not depend on the ground field $\k$, however from this point on we will
assume in this subsection that $\k=\Z_{2}$.

%Notice that any Hamiltonian diffeomorphism  $\Psi$ with support inside $D^{\ast}N$ trivially extends to $E$ and it also induces a map $\lag^{(ex)}(D^{\ast}N)\to \lag^{(ex)}(D^{\ast}N)$, still denoted by $\Psi$.
%To prove the non-vanishing of the weighted categorical  entropy in the statement,  it is enough to show that, for a convenient choice of $L$ and $\Psi$,  the number $\# (\mathcal{B}^{\epsilon}_{HF(\hat{S}_{x}, \Psi^{n}L)})$ increases fast enough in $n$ for some $\hat{S}_{x}\in \mathcal{F}'_{\epsilon}$.  
%Indeed, recall  that  $HF(\hat{S}_{x},\Psi^{n}L)=\hom_{\mathscr{D}_{p}(E)}(\hat{S}_{x},\Psi^{n}L)$ and that  $G_{\F'_{\epsilon}}=\oplus _{S\in\F'_{\epsilon}}S$. Thus,  exponential type growth of  $\# (\mathcal{B}^{\epsilon}_{\hom_{\mathscr{D}_{p}(E)}(\hat{S}_{x},\Psi^{n}L)})$, for $\nu(p)$ small enough, and Corollary \ref{cor:entr_rel}
%imply the claim. 

\

The proof of Proposition \ref{cor:hyp_ent} makes essential use of Corollary \ref{cor:entr_rel} and  
is contained in the next five sub-subsections.  The proposition is stated for the unit disk bundle but in the proof
we will perform a series of constructions involving several nested disk bundles that, through appropriate rescaling, we will assume are all included in $E$ with $\bar{h}:E\to\mathbb{C}$, the Lefschetz fibration from \S\ref{sec:nearby}. The largest $r$ needed is $r=10$. Most of the argument takes place inside this disk bundle and is described in the  first four subsections. We return to $E$ in \ref{subsubsec:back_to_E} to conclude. 

\subsubsection{The choice of  $\Psi$ and $L$}\label{subsubsec:choices}
We will consider an auxiliary Hamiltonian flow $\phi$ generated by a smooth
Hamiltonian $G : T^{\ast}(N)\to \R$ which is defined by 
\begin{equation}\label{eq:Ham_G}
G(q,p)= \eta ( ||p|| )
\end{equation}
where $|| - ||$ is the norm with respect to the metric induced by $\mathtt{g}$ and $\eta:[0,\infty)\to [0,\infty)$ is smooth with the properties that: $\eta(x)= \sigma_{G}x- k_{G}$ for $2\leq x\leq 7$ with $0<k_{G}<2$ and $1\leq\sigma_{G}< \frac{3}{2}$ fixed constants chosen in a way that  will be further discussed later below;  $\eta$ is non-decreasing; $\eta$ is constant equal to $0$ for $0\leq x\leq 1$
and constant equal to some value $=\mathrm{Var}(G) <10$  for $x\geq 8$; $\eta''(x)>0$ for $x\in (1,2)$
and $\eta''(x)<0$ for $x\in (7,8)$ see Figure \ref{fig:shape_G}.  
\begin{figure}
\includegraphics[scale=0.82]{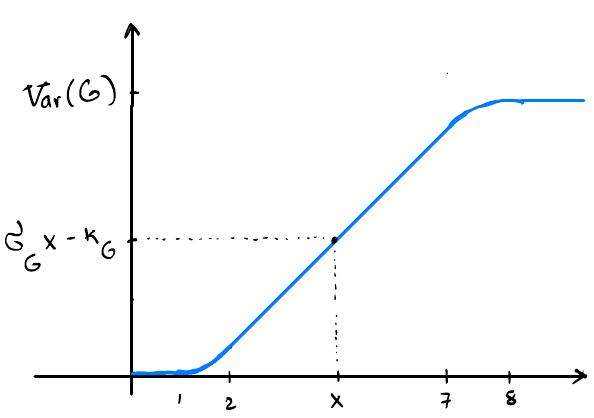}
  \centering
  \caption{The graph of the function $\eta$.} \label{fig:shape_G}
\end{figure}

We notice that $\eta'(x)\leq \frac{3}{2}$, $\forall x\in [0,\infty)$. 
The Hamiltonian diffeomorphism $\Gamma=\phi_{1}^{G}$ generated by $G$
is supported inside $D^{\ast}_{8}(N)$. The Hamiltonian diffeomorphism in the 
statement is the inverse of $\Gamma$:
$$\Psi=\Gamma^{-1}~.~$$

We now describe the Lagrangian $L$. This Lagrangian submanifold is obtained in two steps.
The first is to consider a function $\varphi:N\to [0,\frac{1}{2}]$ as the one constructed in Proposition \ref{prop:Morse} but for $\delta<10$  and define a Lagrangian submanifold $L'$ by
 $$L'=\mathrm{graph}\ (d\varphi)~.~$$
 We denote by $x_{i}$, $1\leq i \leq m$, the critical points of $\varphi$ and take this function $\varphi$ in such a way that $L'\subset \mathrm{Int}(D^{\ast}_{10}(N))$
 and such that $L'\cap D^{\ast}_{9}(N)$ is a disjoint union:
 
 $$L'\cap D^{\ast}_{9}(N) =\coprod_{i=1}^{m} Z_{i}$$ with each
 $Z_{i}\subset W_{i}$ where $W_{i}$ is a small tubular neighbourhoud of $F_{x_{i}}$ (this is the fiber of $T^{\ast}N$ over the critical point $x_{i}$ of $\varphi$). Finally, $L$ is obtained by modifying $L'$ by a Hamiltonian isotopy in such a way that (see  Figure \ref{fig:Lag}):
 $$L\subset \mathrm{Int}(D^{\ast}_{10}(N)) \ , \ L\cap D_{9}^{\ast}(N)=\coprod_{i=1}^{m} F_{x_{i}}\cap D_{9}^{\ast}(N)~.~$$
 \begin{figure}
\includegraphics[scale=0.97]{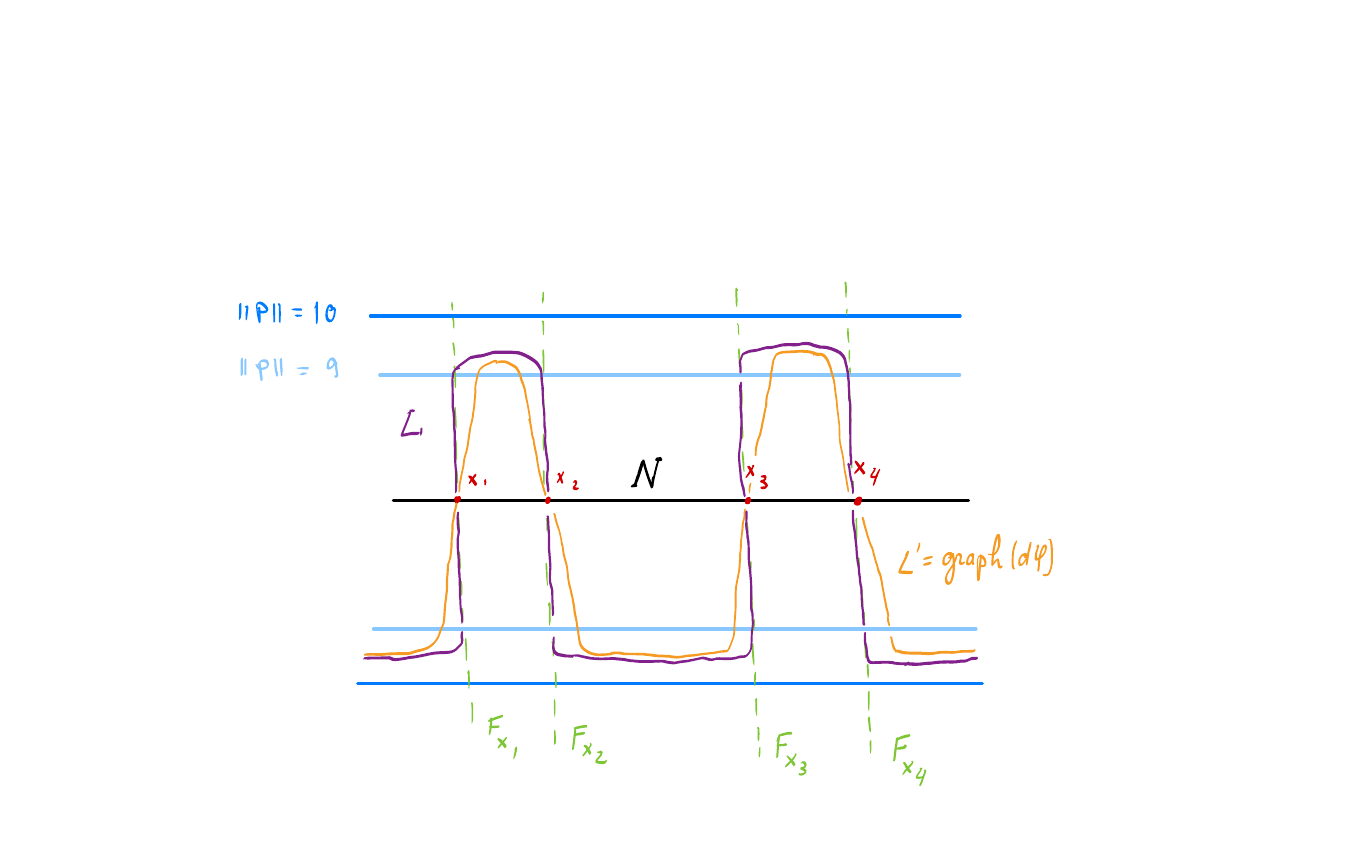}
  \centering
  \caption{The shape of the Lagrangian $L$.} \label{fig:Lag}
\end{figure}

\begin{rem} The constants $8,9,10$ might need to be replaced by $8 < k < k'$ because in the construction in the proof of Proposition \ref{prop:Morse} we do not directly obtain an upper bound for $||d\varphi||$ while keeping $\mathtt{g}$ fixed (see also Remark \ref{rem:conf_metric}). However, the choice  $k=9$, $k'=10$ is simply a choice of convenience and has no real incidence on the proof. 
 \end{rem}
 
 We will fix some additional notation for the Lagrangian $L$. We let 
 \begin{equation}\label{eq:pieces}
 L_{i}=F_{x_{i}}\cap D^{\ast}_{9}(N)
 \end{equation}
 so that $L\cap D^{\ast}_{9}(N)=\coprod_{i} L_{i}$. We also consider a primitive
 $f_{L}:L\to \R$ of $\lambda|_{L}$ which, up to a shift, is a small perturbation of the primitive $\varphi$ of $\lambda|_{L'}$. In view of this,  we may assume that $f_{L}$ has values inside the interval $[-1,1]$. Of course, $f_{L}$ is constant on each $L_{i}$, with a constant $k_{i}$ on each such component, such that each of these constants belongs to $[-1,1]$ and we may assume
 also that $k_{1}=0$.

\subsubsection{Floer homology} In this subsection we establish some properties of certain Floer homology groups that will be essential for the proof. All the Floer homologies considered here are defined in $T^{\ast}(N)$ for an almost complex structure that is convex outside of $D^{\ast}_{8}(N)$.  The maximum principle implies that the Floer trajectories contributing to the relevant differentials are contained in $D^{\ast}_{10}(N)$. The ground field is $\k=\Z_{2}$.

We consider two exact Lagrangian manifolds $K$, $K'$ with primitives
$f_{K}$ and $f_{K'}$, respectively and with fixed choices of gradings, as in the rest of the paper.
In our arguments these Lagrangians $K$, $K'$
are either compact,  included in the interior of  $D^{\ast}_{10}(N)$, or they coincide with fibers $F_{x}$, for $x\in N$. We also consider a Hamiltonian $H: T^{\ast}(N)\to \R$  and an almost complex structure $J=\{J_{t}\}_{0\leq t\leq 1}$ on $T^{\ast}(N)$ that is autonomous and coincides with a convex almost complex structure $J^{\infty}$ 
outside (the interior of) $D^{\ast}_{8}(N)$. We assume that the Hamiltonian $H$ is constant in a neighbourhood of $\partial D^{\ast}_{10}(N)$ and 
that $K$ and $\phi^{-H}_{1}(K')$ intersect transversely. 
We denote by $\mathcal{P}(K,K'; H)$ the Hamiltonian orbits 
$\gamma:[0,1]\to D_{10}^{\ast}(N)$ associated with the Hamiltonian flow $\phi^{H}_{t}$ of $X^{H}$ with $\gamma(0)\in K$, $\gamma(1)\in K'$. There is an action functional on the path-space 
 $$\mathcal{P}(K,K')=\{ \gamma :[0,1]\to D_{10}^{\ast}(N) \ |\ \  \gamma \ \mathrm{smooth,}\  \ \gamma (0)\in K, \ \gamma (1)\in K' \ \}$$
\begin{equation}\label{eq:action1}
\mathcal{A}_{K,K';H}(\gamma)=\int_{0}^{1}H(\gamma(t))dt - \int_{0}^{1}\lambda(\dot{\gamma}(t))dt+f_{K'}(\gamma(1))-f_{K}(\gamma(0))\end{equation}
whose critical points are the elements in $\mathcal{P}(K,K'; H)$. 
As a vector space, the Floer complex $CF(K,K'; H,J)$, when it is defined (thus assuming regularity)
is spanned by $\mathcal{P}(K,K'; H)$.  The differential of the Floer complex counts solutions  $u: \R\times [0,1] \to D_{10}^{\ast}(N)$ of Floer's equation:
\begin{equation}\label{eq:Floer1}
\frac{\partial u}{\partial s} + J_{t} \frac{\partial u}{\partial t} +\nabla_{\rho_{t}} H(u)=0
\end{equation}
where the gradient is taken with respect to the Riemannian metric $\rho_{t}=\omega (- , J_{t})$
and $u(\R\times \{0\})\subset K$, $u(\R\times \{1\})\subset K'$.
We denote by $CF^{a}(K,K'; H,J)$ the subcomplex of $CF(K,K';H,J)$ which is spanned by the Hamiltonian chords $x$ such that $\mathcal{A}_{K,K';H}(x)\leq a$.

 In our application we will have $H=nG$ (with $G$ from (\ref{eq:Ham_G}) )for successively larger values of $n\in \N$. In particular, $H$ is constant outside $D^{\ast}_{8}(N)$.
 This type of Hamiltonian of the form $H=h(||p||)$, with $h$ smooth, has been extensively studied in the literature (see \cite{Bi-Po-Sa:Propagation}\cite{Mac-Sch:top-entr}).  The basic situation of interest for us is when both $K$ and $K'$ are fibers. For convenience, we introduce a special shortened  notation 
$$ C(x,y; n) := CF(F_{x},F_{y}; nG, J)$$
where $F_{x}$ and $F_{y}$ are two different fibers of $T^{\ast}N$.
On both fibers $F_{x}$ and $F_{y}$ we take as primitives the identically zero functions. Regularity can be achieved in this case because, in all our examples, all Floer trajectories will be included inside $D^{\ast}_{8}(N)$ and we have the freedom to pick $\{J_{t}\}$ as desired there while keeping it equal to $J^{\infty}$ outside of $D^{\ast}_{8}(N)$. We emphasize that in the Floer complexes we take into account all the homotopy classes of paths. 
Denote by $$\mathcal{G}_{x,y}=\{\  \bar{\gamma} \ \ |\ \  \bar{\gamma}: [0,\ell]\to N \ \mathrm{is\ a\ geodesic \ } ,\ 
||\dot{\bar{\gamma}}(t)||=1\ \ \ \bar{\gamma}(0)=x, \ \bar{\gamma}(\ell)=y \  \} $$  the geodesic arcs in $N$, parametrized with unit speed length, that go from $x$ to $y$. We denote the length of such a  geodesic arc $\bar{\gamma}$ by $\ell(\bar{\gamma})$. The associated length spectrum is:

$$\mathcal{L}_{x,y}=\{\ \ell(\bar{\gamma}) \ | \ \bar{\gamma}\in \mathcal{G}_{x,y} \} \subset [0,\infty)~.~ $$

Because $(N, \mathtt{g})$ is hyperbolic each homotopy class of paths (with fixed ends) contains a unique geodesic arc in $\mathcal{G}_{x,y}$  and $\mathcal{L}_{x,y}$ is discrete. In particular, $\mathcal{L}_{x,y}$ is countable. It follows, that
the set of quotients $Q\mathcal{L}_{x,y}=\{ \frac{\ell}{n}\ | \  \ell \in \mathcal{L}_{x,y}\ , \  n\in \N^{\ast}\}$ is also countable. We now assume:
\begin{equation}\label{eq:assumption_1}
\sigma_{G}\not \in  Q\mathcal{L}_{x,y}~.~
\end{equation}
This is obviously a generic assumption because the set $Q\mathcal{L}_{x,y}$ is of
zero Lebesgue measure.
\begin{lem}\label{lem:fibers} Under the assumption above the complexes $CF(x,y;n)$ are well defined and, for any $\delta>0$, the number of bars of length $\geq \delta$ in the persistence module $HC(x,y;n)$ increases with $n$ at least as fast as  $\frac{e^{hn}}{hn}$  where $h=\frac{5}{8}h_{top}$ with $h_{top}$ the topological entropy of the geodesic flow.
\end{lem}
\begin{rem}\label{rem:thanks_VG} i. We thank Viktor Ginzburg who suggested to the third author the method to estimate the numbers of bars in $HC(x,y;n)$ as in the proof below.

ii. The quotient $\frac{5}{8}$ in the statement is an artefact of the choice of $\eta$,
it can be brought as close as needed to $1$ by diminishing appropriately the length of the transition regions $(1,2)$ and $(7,8)$ in the definition of $\eta$.
\end{rem}

\begin{proof}  We will use here some of the calculations in \cite{Bi-Po-Sa:Propagation}, see Equation (16) and  Lemma 5.3.2 there, adapted to Lagrangian Floer homology. These calculations imply:
\begin{itemize} 
\item[i.]  The complex $C(x,y;n)$ is well defined and finite dimensional if $n\sigma_{G}\not\in\mathcal{L}_{x,y}$ which is the case due to  (\ref{eq:assumption_1}). The only contributions to the differential of $C(x,y;n)$ is provided by solutions to (\ref{eq:Floer1})
(with $H=nG$) that are contained in $D^{\ast}_{8}(N)$. 
\item[ii.] The generators $\gamma$ of $C(x,y;n)$ are in bijection with geodesic segments $\bar{\gamma}$ in $N$ that join the point $x$ to the point $y$ in the sense that if $\gamma(t)=(q(t),p(t))$, then $q(t)$ is a geodesic from $x$ to $y$ parametrized such that $||\dot{q}(t)||=\ell$; $\bar{\gamma}$ is the same geodesic, re-parametrized with unit speed, and of length $\ell(\bar{\gamma})=\ell$. 
\item[iii.] For each such generator, $\gamma(t)=(q(t),p(t))$, we have that 
$$\frac{p(t)}{r}=\frac{\dot{q}(t)}{||\dot{q}(t)||}\ \ \ \mathrm{with}\ \  r>0\  \ \mathrm{such\ that}\ \ \  n\eta'(r)=\ell~.~$$ In particular, $||p(t)||$ is constant and is equal to $r$ and $\ell\leq n\sigma_{G}$.
\item[iv.] The action of the generator $\gamma$ as above is:
$$\mathcal{A}_{n}(\gamma):=\mathcal{A}_{F_{x},F_{y}; nG}(\gamma)=n\eta(r)- r\ell$$
\end{itemize}
As a result of these remarks, each geodesic arc $\bar{\gamma}$ from $x$ to $y$ 
of length $\ell(\bar{\gamma}) < n\sigma_{G}$ gives rise to a generator of 
$C(x,y;n)$ given by 
\begin{equation}\label{eq:sol}
\gamma(t)=(q(t),\frac{r\dot{q}(t)}{\ell(\bar{\gamma})})
\mathrm{\ \ \ for\ each}\  r\ \mathrm{such\ that\ } \eta'(r)=\frac{\ell(\bar{\gamma})}{n} 
\end{equation}
 where, as above,  $q(t)$ is the geodesic $\bar{\gamma}$ reparametrized such that $||\dot{q}(t)||=\ell(\bar{\gamma})$ (we identify here $T^{\ast}N$ and $TN$ using $\mathtt{g}$).
In view of the properties of the function $\eta$, for a given value $\ell < n\sigma_{G}$, there are exactly two such values $r_{n}(\ell), r'_{n}(\ell)$  with one value, say, $r_{n}(\ell)\in (1,2)$ and the other $r'_{n}(\ell)\in (7,8)$. 

Fix a geodesic arc $\bar{\gamma}$. Denote by $\gamma_{n}$ and $\gamma'_{n}$ the two corresponding Hamiltonian chords associated  to $\bar{\gamma}$ and to the values $r_{n}=r_{n}(\ell(\bar{\gamma}))$ and $r'_{n}=r'_{n}(\ell(\bar{\gamma}))$, respectively, through formula 
(\ref{eq:sol}). They appear as generators of $C(x,y;n)$ as soon as $\sigma_{G}>\frac{\ell(\bar{\gamma})}{n}$, in particular for all $n > \ell (\bar{\gamma})$
(because $\sigma_{G}\geq 1$). 

We now consider the action values of the two relevant generators. We have:
\begin{eqnarray*} \mathcal{A}_{n}(\gamma'_{n})-\mathcal{A}_{n}(\gamma_{n}) =\ n(\eta(r'_{n})-\eta(r_{n})) -\ell(\bar{\gamma}) (r'_{n}-r_{n}) \  \geq \\ \geq  n (\eta(7)-\eta(2)) -7\ell(\bar{\gamma})\geq 5n - 7\ell(\bar{\gamma}) 
\end{eqnarray*} with the last inequality resulting from $\eta(7)-\eta(2)= 5\sigma_{G}$ and $\sigma_{G}\geq 1$. 

Fix $\delta>0$ as in the statement. We deduce: 
\begin{equation}
\label{eq:length_bound} 
\mathcal{A}_{n}(\gamma'_{n})-\mathcal{A}_{n}(\gamma_{n}) \geq \delta \ \ 
\mathrm{if}\ \ 
n\  \geq \  \frac{7}{5} \ \ell(\bar{\gamma})+\frac{\delta}{5} ~.~
\end{equation}
In summary, for $n\geq \ell(\bar{\gamma})$,  $\gamma_{n}$
and $\gamma'_{n}$ appear as generators in $CF(x,y;n)$ and when 
$n\geq \frac{7}{5}\ell(\bar{\gamma})+\frac{\delta}{5}$ their actions are separated by no less than $\delta$.  

The standard Floer homology  $HC^{\infty}(x,y; n)$ vanishes
(forgetting the persistence structure is indicated by the superscript $\infty$) and the Floer differential in $C(x,y; n)$ is very simple.
Indeed, Floer trajectories can only join Hamiltonian chords that project to  paths in $N$ in the same homotopy class. Given that each geodesic in $\mathcal{G}_{x,y}$
is unique in its homotopy class, it follows that when $n>\frac{7}{5}\ell(\bar{\gamma})+\frac{\delta}{5}$, we have $d(\gamma'_{n})=\gamma_{n}$ (recall that we work over $\Z_{2}$ and that $HC^{\infty}(x,y;n)$ vanishes) and the 
couple $(\gamma_{n},\gamma'_{n})$ defines a bar of length at least $\delta$ in the
persistence module $HC(x,y;n)$. Thus, for $n$ big enough, the number of bars of length at least $\delta$ in $HC(x,y;n)$ is at least the number of geodesic arcs in $\mathcal{G}_{x,y}$ of length $\leq \frac{5}{8}n < \frac{5}{7} n(1 -\frac{\delta}{5n})$.
By classical results, such as in %\cite{Buser},
 \cite{Parry_Pollicott:Zeta} Chapter 9, 
the number of geodesics in $\mathcal{G}_{x,y}$ of length less than $T$ increases
as fast as $$\frac{e^{Th_{top}}}{Th_{top}}~.~$$ By replacing $T$ with $\frac{5}{8} n$
we obtain the result.
 \end{proof}

We next turn our attention to the filtered complex  
 $$C(x ,L; n):=CF(F_{x}, L; nG,J)$$
 where $L$
is the Lagrangian fixed in \S\ref{subsubsec:choices}.  Recall that $L$ has the property $L\cap D^{\ast}_{9}(N) =\cup_{i} L_{i}$ with $L_{i}=F_{x_{i}}\cap D^{\ast}_{9}(N)$ (see (\ref{eq:pieces})).
 
We assume that $x\not= x_{i}$ for all the critical points $x_{i}$ of $\varphi$ and we add two additional assumptions relative to $F_{x}$,$L$, and $G$:
\begin{equation}\label{eq:sigma_2}
\sigma_{G}\not\in Q\mathcal{L}_{x,x_{i}}, \ \ \forall x_{i} \end{equation}
As mentioned after (\ref{eq:assumption_1}), this is a generic assumption on the choice
of $\sigma_{G}$. Note that:
\begin{equation}\label{eq:inters}
L\cap F_{x}= \{z_{x}\} \subset D^{\ast}_{10}(N)\backslash D^{\ast}_{9}(N)~.~
\end{equation}
To make sense of this recall from \S\ref{subsubsec:choices} that $L$ is a small deformation of $L'$ which is the graph of $d\varphi$. It follows that, under the assumption that $x\not= x_{i}$, $\forall i$, we have that $F_{x}$ intersects $L$ in a single point - denoted by $z_{x}$ - which lies outside of $D^{\ast}_{9}(N)$. 
 
Given that $G$ is constant outside of $D^{\ast}_{8}(N)$ we deduce that the  underlying
 vector space of $C(x,L; n)$ can be written as:
 \begin{equation} C(x,L;n)=\oplus_{i}\Sigma^{k_{i}} C(x,x_{i};n)\oplus \Z_2 \langle z_{x} \rangle ~.~
 \end{equation} 
 Moreover, for the same reason, the complexes $C(x,x_{i},n)$ are finite dimensional over $\Z_{2}$. The shift $\Sigma^{k_{i}}$ appears here because the value of the primitive 
 $f_{L}$ on the component $L_{i}$ is not $0$, but rather $k_{i}\in [-1,1]$. The action of the point $z_{x}$ is $f_{L}(z_{x})$, see (\ref{eq:action1}). Denote by $d^{x_{i};n}$ the differential of the complex $C(x,x_{i};n)$ and let 
 $D^{n}$ be the differential of the complex $C(x,L;n)$.  Finally, recall that for a filtered chain complex $C$ we denote by $C^{a}$ the filtration $\leq a$ subcomplex. 
 \begin{lem} \label{lem:defor}With the notation above there exists a constant $\xi>0$ that depends
 on $(N,\mathtt{g})$, on $L$, on $x$, and on the almost complex structure $J$ (but not on $n$) such 
 that $$ D^{n} = \oplus_{i} d^{x_{i};n} + \bar{D}^{n}$$ with the property that
 $$\bar{D}^{n}(C^{\alpha}(x,L;n))\ \subset \ C^{\alpha-\xi}(x,L;n) \ , \ \forall \ \alpha  
 ~.~$$
 \end{lem}
 
 \begin{proof} The Floer strips that contribute to the differential $D^{n}$ are of two types: 
 those that stay inside $D^{\ast}_{8}(N)$ - in this case they appear in one of the differentials
 $d^{x_{i};n}$, and strips $u$, called below of {\em second type}, such that $\mathrm{Im}(u)\not\subset D^{\ast}_{8}(N)$. Recall that the Lagrangian $L$ coincides with a union of fibers
 inside $D^{\ast}_{9}(N)$. Using the maximum principle,  this implies that 
 any strip $u$ of the second  type has to reach the non-fiber-like region of $L$. In other words, there exists a point $(s_{0},0)\in \R\times\{0\}$ such that $u(s_{0},0)\not\in D^{\ast}_{9}(N)$. 
 The aim of the proof is to show that there is a constant $\delta$, as in the statement, such 
 that the energy
 $$E(u)=\int ||\frac{\partial u}{\partial s}||^{2}\ dsdt $$
 satisfies $E(u)\geq \delta$ for any strip $u$ of the second type.
 
The argument is based on standard monotonicity  results such as in Sikorav \cite{Sik:est}. 
The first step is to fix the constant $\xi$. Let $K= D^{\ast}_{9}N\backslash \mathrm{Int}(D^{\ast}_{8.5}(N))$. The results in  \cite{Sik:est} - see Proposition 4.7.2 there -  imply that there exist positive constants $r_{0}$ and $C$ that depend only on $J$, $(N,\mathtt{g})$, and $L$ with the following property.
For each point $p\in L\cap K$, if $v:\Sigma\to K$  is a
$J$-holomorphic curve with $\Sigma$ compact with boundary, such that:
\begin{itemize}
\item[i.] There is  some point $x\in \partial\Sigma$ with $v(x)=p$,
\item[ii.] $v(\Sigma)\subset B_{r}(p)\subset K$ with $r\leq r_{0}$ (where $B_{r}(a)$ is the ball in $T^{\ast}(N)$ of radius $r$ with center $a$),
\item[iii.] $v(\partial \Sigma)\subset \partial B_{r}(p)\cup L$,
\end{itemize}  
 then the symplectic area $A(v)$ of $v$ satisfies:
 $$A(v)\geq Cr^{2}~.~$$
 
Now pick some $r'<r_{0}$ such that for each point in $z\in L\cap (D^{\ast}_{8.8}(N)\backslash \mathrm{Int}(D^{\ast}_{8.6}(N)))$ we have 
\begin{equation}\label{eq:dep_x}
B_{r'}(z)\subset K\backslash F_{x}.
\end{equation}
 We put:
 $$\xi= C(r')^{2}/2~.~$$

 With this notation we return to a Floer strip of the second type $u:\R\times [0,1]\to D^{\ast}_{10}(N)$. We want to show  $E(u)\geq \xi$. We know that there is some $s_{0}$ with the property that the point $u(s_{0},0)$ is in $L\backslash D^{\ast}_{9}(N)$. Moreover, at least one of the asymptotic limits of $u$ is a Hamiltonian chord of $nG$ (for some $n$) and is therefore inside $D^{\ast}_{8}(N)$. It follows that there is also  some point
 $s_{1}\in \R$ such that $p=u(s_{1},0)\in  L\cap (D^{\ast}_{8.8}(N)\backslash \mathrm{Int}(D^{\ast}_{8.6}(N)))$.

 Denote by $u_{r}$ the restriction of $u$ to the set $S_{r}=u^{-1}(B_{r}(p))\subset \R\times [0,1]$ for $ 0 < r \leq r'$. By making use of Sard's theorem, for a generic choice of $r$ the set $S_{r}$ is a compact surface with boundary. We pick such a regular $r_{1}> \frac{r'}{\sqrt{2}}$. We notice that $u_{1}:S_{r_{1}}\to B_{1}(p)$ has the property that $u_{1}(\partial S_{r_{1}})\subset \partial B_{r_{1}}(p)\cup L$. We also have, of course, $u_{1}(S_{r_{1}})\subset B_{r_{1}}(p)\subset K$. Recall that $G$ is constant outside of $D^{\ast}_{8}(N)$ and thus $\nabla (nG)$ vanishes on $K$.
 As a result $u_{1}$ is $J$-holomorphic and thus we obtain:
 $$E(u)\geq A(u_{1})\geq Cr_{1}^{2}\geq \xi$$
 which concludes the proof. 
   \end{proof}
  
  \begin{rem}\label{rem:dep_onx}
  The constant $\xi$ in the lemma depends very lightly on the point $x$. Indeed, if we  assume that $F_{x}$ does not intersect a fixed neighbourhood
 $U$ of $L\cap D^{\ast}_{9}(N)$, then we may replace condition
 (\ref{eq:dep_x}) by $B_{r'}(z) \subset K\cap U$ and under this assumption
 $\xi$ is independent of $x$ and only depends on $L$, $J$, $(N,\mathtt{g})$, and $U$.
  
  \end{rem}

  \subsubsection{An auxiliary algebraic  result}
We will next use the following simple  statement. 
  
  \begin{lem}\label{lem:bar_counts2}
  Suppose that $C$ is a filtered, finite dimensional chain complex over $\Z_2$ that is $\Z_2$-graded. 
  Fix $\delta>0$.
  If the differential $D$ of $C$ can be written as
  $$D=d+D'$$ with $d^{2}=0$ and $D'$ such that $D'(C^{\alpha})\subset C^{\alpha-\delta}$
  for all $\alpha\in\R$, then for any $\epsilon<\delta$, we 
  have $$\#(\ \mathcal{B}^{\epsilon}_{H(C))}\ )\geq \#(\ \mathcal{B}^{\epsilon}_{H(C,d)} \ )~.~$$
  \end{lem}
  
  \begin{proof}
  The result is most likely known by experts but we give a proof for completeness. We denote by 
  $C_{0}$ the complex $(C,d)$. It is easy to see that it is enough to show the statement when $C_{0}$ is acyclic which we will now assume. We write 
  $C_{0}$ as a direct sum of terms $E_{2}(a,b)$   with filtrations given by a valuation $v$ with values in $\R$ (see \S\ref{subsec:pers}). There are two types of terms of this form,
  those such that $v(a)\leq v(b)\leq v(a)+\epsilon$ - they will be denoted
  by $E_{2}^{s}(a,b)$, and those with $v(b)> v(a)+\epsilon$, denoted by 
  $E_{2}^{l}(a,b)$. The number  $n_{l}$ of terms of type $E_{2}^{l}$ equals the number of bars in $\mathcal{B}_{HC_{0}}$ of length greater than $\epsilon$. We now consider one $E_{2}^{s}$ term, $E_{2}^{s}(a_{0},b_{0})$,
such that $v(b_{0})-v(a_{0})\leq \epsilon$ is minimal among all terms $E_{2}^{s}$. We let $C_{1}$ be the subspace of $C$ generated by the generators of $C$ different from $\{a_{0},b_{0}\}$ so that as $d$-chain complexes we have
$C_{0}=E^{s}_{2}(a_{0},b_{0})\oplus C_{1}$.    We consider
the element $$a'_{0}= a_{0} + D'(b_{0})~.~$$ It has the property $a'_{0}=D(b_{0})$. We notice that $v(a'_{0})=v(a_{0})$ because $\epsilon < \delta$ and thus $v(D'(b_{0})) < v(a_{0})$. Therefore, we also have $D'(b_{0})\in C_{1}$.
Let $$ E_{2}(a'_{0}, b_{0})=\Z_2 \langle \  a'_{0},\  b_{0} \ : \ D(b_{0})= a'_{0}, D(a'_{0})=0 \ \rangle ~.~$$ 
Then $E_{2}(a'_{0}, b_{0})$ is a $D$-subcomplex of $C$. We have the projection $p_{1}: C\to C_{1}$ defined on generators as the identity on $C_{1}$ and that sends $b_{0}$ to $0$, and $a_{0}$ to $D'(b_{0})$. This projection has as kernel $E_{2}(a'_{0}, b_{0})$  and is 
filtration preserving in the sense that it sends $C^{\alpha}$ to $C^{\alpha}_{1}$, for all $\alpha\in\R$.
The differential $D$ induces a differential $D_{1}$ on $C_{1}$ which is the unique linear map making  $p_{1}$ a chain map. Moreover, $D_{1}$ can be written:
$$D_{1}= d + D'_{1}$$
where $D'_{1}=p_{1} \circ D'$. It is easy to see that $D'_{1}$ also drops the filtration by
at least $\delta$, just as $D'$ in the statement. Assume for the moment that the map $p_{1}$ admits a section $j_{1}: (C_{1}, D_{1}) \to (C, D)$ which is a filtration preserving chain map. In that case, we can reapply iteratively the same procedure to $(C_{1},D_{1})$
to successively eliminate all the $E_{2}^{s}(a,b)$'s. We are left at the end with a complex $(C_{k}, D_{k})$ whose dimension is $2 n_{l}$  and whose differential $D_{k}=d+D'_{k}$ drops filtration by more than $\epsilon$ ($d$ drops differential by more than $\epsilon$ on $C_{k}$, and $D'_{k}$ by at least  $\delta > \epsilon$). As a result, all the bars 
in $HC_{k}$ are longer than $\epsilon$ and there are at least $n_{l}$ of them.
By our iterative construction, $(C_{k},D_{k})$ is a filtration preserving retract of $(C,D)$ and therefore the number of bars of length more than $\epsilon$ in $HC$ is at least $n_{l}$, as claimed. 

Thus, to finish the proof we are left to construct the section $j_{1}$. 
We will construct a linear map $j_{1}:C_{1}\to C$ that is filtered and has the property that $p_{1}\circ j_{1}=id$ and, additionally, $\mathrm{Im}(j_{1})$ is
a subcomplex of $C$. This implies that $j_{1}$ is also a chain map.
As a vector space, we have the splitting  $C=E_{2}(a_{0}',b_{0})\oplus C_{1}$.
Let $x\in C_{1}$. There are two possibilities. Assume that, relative to this splitting, $D(x)=w\in C_{1}$. In that case, we put $j_{1}(x)=x$. On the other hand, assume that $D(x)= a_{0}'+w$ with $w\in C_{1}$. In this case we put  $j_{1}(x)= x+ b_{0}$. We notice $v(x)\geq v(b_{0})$. Moreover, $D(x+b_{0})\in C_{1}$ and, with this definition, $j_{1}$ is different from the identity only in degree $|a_{0}|+1$.  We also have that $j_{1}$ is filtration preserving in the stronger sense that $v(j_{1}(y))=v(y)$ for all $y\in C_{1}$. It is clear that $p_{1}\circ j_{1}=id$.  We now want to remark that $\bar{C}_{1}=\mathrm{Im}(j_{1})$ is closed with respect to $D$. First, the construction
above means that for each $x\in \bar{C}_{1}$,  $D(x)$ can not 
have the form $a'_{0}+w$, $w\in \bar{C}_{1}$. Indeed,  all $x\in \bar{C}_{1}$ in degree $|a_{0}|+1$ satisfy $D(x)\in C_{1}$ and $C_{1}$ and $\bar{C}_{1}$ coincide in degree $|a_{0}|$. Now assume that for some $x\in \bar{C}_{1}$, $D(x)=b_{0} + \zeta$, $\zeta\in \bar{C}_{1}$.  Then  $0=D^{2}(x)= a'_{0} + D(\zeta)$ which leads to a contradiction. Thus $\bar{C}_{1}$ is a $D$-subcomplex in $C$ and this concludes the proof.
 \end{proof}
 
 \begin{rem} We have used at certain points in the proof the fact that we work over $\Z_2$ but the result remains true with a similar proof over any field $\k$.
 
 \end{rem}
 
\subsubsection{The main step in the proof of Proposition \ref{cor:hyp_ent}} The aim of this subsection is to  show:

\begin{lem}\label{lem:parrila}  For $\epsilon$ sufficiently small and for  any $\epsilon$-approximating family $\mathcal{F}_{\epsilon}$ for $\lag^{(ex)}(D^{\ast}_{10}N)$ in $\msc_{p}(10)$, with $\nu(p)$ sufficiently small - as in Theorem \ref{thm:nearby}, thus with $\F_{\epsilon}$ consisting of fibers - there is a Hamiltonian isotopy $\Psi$ with support in $D^{\ast}_{10}(N)$, a Lagrangian $L$, both as in \S \ref{subsubsec:choices}, and one fiber, $F_{x}\in \F_{\epsilon}$ such that the number of bars:
$$\# (\mathcal{B}^{\epsilon}_{HF(F_{x},\Psi^{n}L)})$$ grows in $n$ at least as fast as  $\frac{e^{hn}}{hn}$, where $h$ is the constant in Lemma \ref{lem:fibers}. 
\end{lem}

%Notice that we use here the  local $\epsilon$-approximating data $(\Phi,\F_{\epsilon})$, %as in \S\ref{subsubsec:local_approx_d} and in Theorem \ref{thm:nearby}, and not the %ambient data $(\Phi',\F'_{\epsilon})$ in the statement of Proposition \ref{cor:hyp_ent}.
%
\begin{proof}
We start the argument by choosing a Lagrangian $L$ as in \S\ref{subsubsec:choices}.  We also fix a small neighbourhood $U$ of $L\cap D_{9}^{\ast}(N)$. We pick $\delta$ to be the constant given by Lemma \ref{lem:defor} under the assumption that $F_{x}\cap U=\emptyset$.
As noted in Remark \ref{rem:dep_onx} this constant is then independent 
of $x$. We pick $\epsilon < \delta$ and let $\F_{\epsilon}$ be an $\epsilon$-approximating family for $\msc_{p}(10)$. If $\epsilon$ is small enough there is 
at least one element $F_{x}$ of $\F_{\epsilon}$ that does not intersect $U$ (this follows from the fact that when $\epsilon\to 0$ the union $\cup_{F\in \F_{\epsilon}}F$ becomes dense in $D^{\ast}_{10}(N)$).
At this point we can pick the Hamiltonian $G$. We pick $G$ as in \S\ref{subsubsec:choices} subject to the assumption (\ref{eq:sigma_2}) which is generic.  This means that the complex $CF(F_{x}, L; nG, J)$ is defined for 
each $n$, it is finite dimensional over $\Z_{2}$, and that it satisfies the statement 
in Lemma \ref{lem:defor}.
This means that 
$$CF(F_{x},L;nG,J)=\oplus \Sigma^{k_{i}}CF(F_{x},F_{x_{i}};nG,J)\oplus \Z_2 \langle z_{x}\rangle $$
as filtered vector spaces, with $F_{x_{i}}$ as in (\ref{eq:pieces}) and 
that the differential $D^{n}$ of $CF(F_{x},L;nG,J)$ satisfies:
$$D^{n}=\oplus_{i}d^{x_{i};n}+\bar{D}^{n}$$ with 
$\bar{D}^{n}$ dropping filtration by at least $\xi$ and with $d^{x_{i};n}$ the differential in $CF(F_{x},F_{x_{i}};nG,J)$. 
In this situation we can  apply Lemma \ref{lem:bar_counts2}
to deduce :

$$\# (\mathcal{B}^{\epsilon}_{HF(F_{x},L;nG,J) })\geq  \sum_{i}\# (\mathcal{B}^{\epsilon}_{HF(F_{x},F_{x_{i}};nG,J)})~.~$$
We use  Lemma \ref{lem:fibers} to conclude that
$\# (\mathcal{B}^{\epsilon}_{HF(F_{x},L;nG,J) })$ increases in $n$ as fast 
as $\frac{e^{hn}}{nh}$.

The argument concludes by noting that by the naturality of Floer's equation, there is a filtered chain isomorphism:
$$CF(F_{x}, \Psi^{n} L; 0, J_{n})\to CF(F_{x},L;nG, J)~.~$$ 
Here $\Psi = \Theta^{-1}$ where $\Theta=\phi_{1}^{G}$ is induced by the Hamiltonian flow of $G$. The almost complex structure
$J_{n}$ is such that $\Theta^{n}_{\ast}J_{n}=J$. Because $G$ is constant outside 
$D^{\ast}_{8}(N)$, the almost complex structures $J_{n}$ coincide with $J$ there and thus the complexes $CF(F_{x}, \Psi^{n} L; 0, J_{n})$ are well defined
because $|| - ||^{2}$  is  $J_{n}$- convex outside $D^{\ast}_{8}(N)$. The persistence
module $HF(F_{x}, \Psi^{n} L; 0, J^{*})$ is in fact independent of the almost complex structure $J^{\ast}$, as long as it coincides with $J$ outside $D^{\ast}_{8}(N)$.

\end{proof}

\subsubsection{End of the proof of Proposition \ref{cor:hyp_ent}}\label{subsubsec:back_to_E} 
The constructions in the subsections above can be appropriately rescaled so  that they take place 
in the unit disk bundle $D^{\ast}N$ instead of $D^{\ast}_{10}(N)$ and, further, 
this disk bundle is included in the total space of  the Lefschetz  fibration $\bar{h}:E\to \mathbb{C}$ that was instrumental in the proof of Theorem \ref{thm:nearby}.
 Moreover, we may assume as in \S\ref{subsec:exact-seq} that the TPC $\epsilon$-approximating
 family $\F'_{\epsilon}$ consists of Lagrangian spheres $\hat{S}_{x_{i}}$ that each intersects 
 $D^{\ast}N$ in a fiber $F_{x_{i}}$ of $D^{\ast}N$. The collection of these fibers is precisely 
 the $\epsilon$-TPC-approximating family $\F_{\epsilon}$, see \S\ref{subsubsec:complex_Lag} for a brief
 recap of the context.
 
 \
 
We now recall the conclusion of  Lemma \ref{lem:parrila}: the number of bars $\# (\mathcal{B}^{\epsilon}_{HF(F_{x},\Psi^{n}L)})$ grows in $n$ at least as fast as  $\frac{e^{hn}}{hn}$. It is easy to see that the constructions leading to this result (that all take place in $D^{\ast}(N)$) can be adjusted so  that we have the same rate of growth for  $\# (\mathcal{B}^{4\epsilon}_{HF(F_{x},\Psi^{n}L)})$ which is what we will assume from now on.
Recall that $F_{x}$ is one of the fibers in the approximating family $\F_{\epsilon}$ and $\Psi^{n}L\subset  \mathrm{Int}(D^{\ast}N)$. 
Therefore, $F_{x}\cap \Psi^{n}L= \hat{S}_{x}\cap \Psi^{n}L \subset \mathrm{Int}(D^{\ast}N)$.  Due to the convexity of $\partial D^{\ast}N$ we deduce that the Floer complexes $CF(F_{x}, \Psi^{n}L;0, J_{n})$ and $CF(\hat{S_{x}}, \Psi^{n}L;0,J_{n})$ coincide. 
As a result 
$$\# (\mathcal{B}^{4\epsilon}_{HF(\hat{S}_{x},\Psi^{n}L)})$$  also grows in $n$ at least 
as fast as  $\frac{e^{hn}}{hn}$.

Both $\hat{S}_{x}$ and $\Psi^{n}L$ are Yoneda modules so that
$HF(\hat{S}_{x},\Psi^{n}L)=\hom_{\mathscr{D}_{p}(E)}(\hat{S}_{x},\Psi^{n}L)$. 
As a result $\hbar  (\Psi; L, \hat{S_{x}}, 4\epsilon)$, as defined in  (\ref{eq:bar-code_ent}) does not vanish. Given that $G_{\F'_{\epsilon}}=\oplus_{S\in \F'_{\epsilon}}S$
we also have $\hbar (\Psi; L, G_{\F'_{\epsilon}},4\epsilon)>0$.
At this point we can apply Corollary \ref{cor:ineq_bar2}
 and we deduce 
 $$ h^{r}_{\Phi',\F'_{\epsilon}}(\Psi; L,2\epsilon)>0$$
 which is the desired statement. \qed
 
% !TEX root = approx8.tex

\subsection{Complexity of equators on $S{^2}$}\label{subsec:tb_S2}
The aim of this subsection is to complete the proof of Corollary \ref{cor:no-t-b} by showing that the space of equators on $S^{2}$
endowed with the spectral metric is not totally bounded.  We will use the notation in \S\ref{sec:split-app}. In particular,
the set of equators on the $2$-sphere is denoted by
$\lag^{\text{(mon,} \mathbf{0} \text{)}}(S^2)$. Here 
the notation reflects the fact that equators are monotone, and $\mathbf{0}$ indicates that the $\Z_{2}$-number of Maslov -$2$  $J$-holomorphic disks with boundaries on an equator $L$, and passing through a fixed point $x\in L$, vanishes. Recall  from Proposition \ref{OCS2} that this class of Lagrangians
is retract approximable  with an $\frac{1}{4N}+2\nu(p)$-approximating family of the form
$L_{1},\ldots, L_{N}$ where $L_{i}$ is a great circle passing through the north and south poles of $S^{2}$ and is  at angle $\frac{\pi}{2N}$ from $L_{i-1}$. We fix  $L_{1}$ such that it passes through  the point $(1,0,0)$. We assume that the metric
on $S^{2}$ is such that its area is equal to $1$.

%\ocnote{Need to add a Remark in Section 4, after the Proposition \ref{OCS2} to note that the result implies approximability of the respective class of Lagrangians with {\em the spectral metric}; same thing for the torus}.

Notice that in this case, by contrast to the case of $M=D^{\ast}N$, the approximating data $(\Phi, \F)$ is simpler in the sense
that the categories $\mathscr{Y}_{\epsilon,\eta}$ do not depend
on $\epsilon$, the only dependence is on $p$ (we may assume $\eta=\nu(p)$) and for each $N$, the family $\mathcal{E}(N)=\{L_{1},\ldots, L_{N}\}\subset \mathcal{E}$ where $\mathcal{E}$ is the set of all great circles passing through the two poles.   
\begin{prop} \label{prop:complex_S2}
There exists a Hamiltonian diffeomorphism $\Psi\colon S^{2}\to S^{2}$,  $L\in\mathcal{E}$,   $N\in \N$, and  $\epsilon' \geq  2 (\frac{1}{4N}+2\nu(p))$, such that:
\begin{equation*}
N^{r}(\Psi ^{k}L, \mathcal{E}(N), \epsilon') \stackrel{k\to \infty}{\longrightarrow} \infty~.~
\end{equation*}
\end{prop}
In view of this proposition and of Remark \ref{rem:tb_complex} it follows that the space
$(\lag^{\text{(mon,} \mathbf{0} \text{)}}(S^2), d_{\gamma})$
is not totally bounded.

\begin{proof}[Proof of Proposition \ref{prop:complex_S2}]
We start by identifying $\Psi$: this is the Dehn twist along the 
horizontal equator on the sphere with support inside the annulus
$$U=\{(x,y,z)\in S^{2} \ |  -a  < z < a  \}$$ with $a>0$ picked in such
a way that the area of $U$ is smaller than $\frac{1}{4}$.
Next we pick $L$ to be the great circle on $S^{2}$ that passes through the north and south poles and also through the 
point $(0,1,0)$. Note that $L$ intersects transversely $L_{1}$, the first element in the family $\mathcal{E}(N)$.

The idea of the proof is simple: we would like to use an analogue of the first  inequality in Corollary \ref{cor:ineq_bar2}, for a fixed $\epsilon'$ and some $N$ (with $\mathcal{E}(N)$ in the place of $\F'_{\epsilon}$),  and then show that the number of bars of length larger than $2\epsilon'$ in 
$HF(L_{1},\Psi^{k}L)$ goes  to infinity with $k$. The key issue to be resolved is that we work here in the monotone context and over the Novikov field $\La$, as in \S\ref{sb:fuk-mon}. As a result, the algebraic
results in Proposition \ref{prop:lower-bd}  and (\ref{eq:ineq-r-bar}) need to be adjusted to this context. This can be achieved by using the results of Usher-Zhang \cite{Usher-Zhang:perh} as we will briefly
review next.

Assume that $C$ is a finite dimensional, filtered chain-complex defined over the Novikov field  $\La$.  The homology $H(C)$ is a persistence module in the classical sense over the field $\Z_{2}$.  However, counting bars in the usual way does not make sense because if one bar of type $[a,b)$ appears, all translates of type $T^{\alpha}[a,b)$ also appear for all $\alpha\in\R$.  The tools required to deal with this problem are introduced in \cite{Usher-Zhang:perh} together with a specific terminology.  First, the methods in 
\cite{Usher-Zhang:perh} apply to rings of Novikov type but more general than our choice here, that are denoted there by $\La^{\mathcal{K},\Gamma}$ where $\mathcal{K}$ is the ground field ( $\Z_{2}$ in our case), and $\Gamma$ is an additive subgroup of $\R$. In our case, $\Gamma=\R$. Moreover, \cite{Usher-Zhang:perh} applies to so-called {\em non-Archimedean normed vector spaces that are orthogonalizable} which means vector spaces $C$ over $\La$ endowed with a filtration function $\ell :C\to \R\cup \{-\infty\}$ with the properties typical for action functionals together with a property of admitting special ``orthogonal'' bases. Such complexes are called of Floer-type  in  Definition 4.1 in \cite{Usher-Zhang:perh} and, in particular, all Floer complexes in our paper fit into this class.  All Floer-type complexes   will always be finite dimensional over $\La$ in our case. To such a Floer-type chain complex $C$ \cite{Usher-Zhang:perh} assigns
a bar-code $\bar{\mathcal{B}}_{HC}$, called there the {\em concise} bar-code of $C$, that consists
of intervals of non-negative length, possibly infinite and that all have $0$ as
lower bound (this is, of course, different from the case considered earlier in this section; the fact that the lower bound of the bars is always $0$ reflects  the issue mentioned above having to do with the action of $\La$). It is shown that such a bar-code determines $C$ up to filtered chain-homotopy equivalence (which explains why we use $H(C)$ in the notation). The construction of this barcode parallels the standard construction over a non filtered field: it goes through a writing of the differential of the complex $C$, $d:C\to C$,
 in  a special basis over $\La$ such that the matrix of the differential is upper triangular and only contains $0$'s and $1$'s, with at most a single $1$ on each row and column.  We denote by $\bar{\mathcal{B}}^{\delta}_{H(C)}$ the 
barcode formed by eliminating from $\bar{\mathcal{B}}_{H(C)}$ all
the bars of length $\leq\delta$, just as in  \S\ref{subsec:pers}. The construction of the barcode $\bar{\mathcal{B}}_{H(C)}$ easily shows that the analogue of 
Lemma \ref{lem:noc-chains} remains true in this context with the place of 
$V$ being taken by a Floer-type finite dimensional complex over $\La$ and 
with $\bar{\mathcal{B}}$ in the place of $\mathcal{B}$. As result 
 Proposition \ref{prop:lower-bd}  and (\ref{eq:ineq-r-bar}) remain also true
 with the  modification that the base $A_{\infty}$-category is a Fukaya
 filtered category  over $\La$ and that $\bar{\mathcal{B}}$ takes the place of $\mathcal{B}$. 
 
 In brief, to prove the proposition we need to show that 
 \begin{equation}\label{eq:estimate_bars3}
 \#\  \left(\bar{\mathcal{B}}^{2\epsilon'}_{HF(L_{1},\Psi^{k}L)}\right)\stackrel{k\to\infty}{\longrightarrow} \infty ~.~
 \end{equation} 

At this point we pick the value  $\epsilon'= \frac{1}{32}$. We will consider the complex $CF(L_{1}, \Psi^{k}L)$ defined over $\La$. The generators of this complex  are the intersection points $L_{1}\cap \Psi^{k}L$. The number of these intersection points is $2k+2$: two of them
are the north and south poles and there are two others for each successive application of $\Psi$.  Floer trajectories are holomorphic and by using the open mapping theorem it is easy to see that if a Floer trajectory starts at a point $\in L_{1}\cap \Psi^{k}L \cap U$ then it  can not remain entirely inside $U$ and thus it fills a connected region of $S^{2}\backslash ( U \cup L_{1} \cup L)$. As a result each such Floer trajectory has area at least $\frac{3}{32}$. At the same time the Floer homology with $\Z_{2}$ coefficients $HF(L_{1}, \Psi^{k}L;\Z_{2})$ is well defined (this is obtained by making $T=1$) and is isomorphic
to $HF(L_{1}, L;\Z_{2})$ which is of dimension $2$ over $\Z_{2}$ and is generated by the north and the south poles. It follows that $\bar{\mathcal{B}}_{HF(L_{1},\Psi^{k}L)}$ contains at least $k$ bars and each of them is of length more than $\frac{2}{32}=2\epsilon'$. This concludes the proof.
\end{proof}

\begin{rem} \label{rem:equators1}
It is easy to reformulate Proposition \ref{prop:complex_S2} in terms of the non-vanishing of an entropy type measurement that one may call $\epsilon$-weighted  {\em slow} (categorical) entropy which is defined just as in formula (\ref{eq:cat_entropy2}) except that $\log (n)$ is used at the denominator, in the place of $n$. In these terms, what we have shown in the proof above is that the slow $\epsilon$-weighted retract entropy of $\Psi$ with respect to $L$, both chosen as in the proof, is  at least $1$ as soon as $\epsilon \leq \frac{1}{32}$. Because  $\Psi$ is a Hamiltonian diffeomorphism, it is easy to see that for $\epsilon=\infty$ the slow entropy vanishes. Relations between entropy and Floer theoretic calculations have been initiated in  \cite{FS:vol-growth}. See  \cite{DawFlBar} \cite{DawA3} for some other results involving bar-code entropy estimates.

\end{rem}

% !TEX root = approx8.tex

\subsection{Quasi-rigidity and Corollary \ref{cor:quasi-rig}}\label{subsec:qrig-cor}

We start by restating Corollary \ref{cor:quasi-rig} in a more precise way.

\begin{cor}\label{cor:quasi-rig2} Fix  a closed  Riemannian manifold $(N,\mathtt{g})$, its unit cotangent
bundle $D^{\ast}N$ as well as an $\epsilon > 0$.  Fix also a family $\F =\{F_{0},\ldots, F_{l}\}$ of fibres of $D^{\ast}N$ that  $\epsilon$-approximates $\lag^{(ex)}(D^{\ast}N)$,  as constructed in Theorem \ref{thm:nearby}
(see Definition \ref{d:approx-sys}).  Consider an exact symplectomorphism  $\phi :D^{\ast}N\to D^{\ast}N$ with support in the
interior of $D^{\ast}N$ and let
 $$\chi (\phi; \F) =\max_{F\in\F}\ \{ \max h_{\phi(F)} - \min h_{\phi(F)}\}~.~$$
 For any $L\in \mathcal{L}ag^{(ex)}(D^{\ast}N)$, we have $$d_{\gamma}(L,\phi(L)) \leq 4(\epsilon + N(L;\F, \epsilon) \chi(\phi;\F))~.~$$
\end{cor}

We recall that for an exact Lagrangian $L\subset D^{\ast}N$ we denote by $h_{L}:L\to \R$ a choice of a primitive on $L$ for the exact $1$-form $\lambda$. This notation is used for $L\in\lag^{(ex)}(D^{\ast} N)$, for fibers $F$ of $D^{\ast}N$, as well as for Lagrangians such as $\phi(F)$ which coincide with the fiber $F$ outside the support of $\phi$. Notice also that $Var (h_{L})=\max h_{L} - \min h_{L}$ is independent of the choice of primitive $h_{L}$.

\begin{rem}\label{rem:quasi-isom}
a. We refer to the property in the Corollary \ref{cor:quasi-rig2} as  {\em quasi-rigidity} because its meaning is that  the size of the action of $\phi$ on the whole of $\lag^{(ex)}(D^{\ast}N)$ is constrained by that on the finite family of fibers $\F$. In particular, notice that if $\phi(F)=F$ for all the fibers $F\in \F$, then, for such an $F$, $Var(h_{\phi(F)})=0$ and  thus $\chi(\phi;\F)=0$. Therefore, in this case
we have $$d_{\gamma}(L,\phi(L)) \leq 4\epsilon \ \ , \forall \ L\in\lag^{(ex)}(D^{\ast}N)~.~$$

b. In both Corollaries \ref{cor:link2} and \ref{cor:quasi-rig2} there is a factor two 
on the right side of the respective stated inequalities that could be avoided by using in the definition of TPC-approximability the pseudo-metric $D_{\mathrm{int}}$ from  (\ref{eq:d-int3}) instead of $\dint$. 

c. It is likely that in the statement of the corollary it is possible to replace $N(L;\F,\epsilon)$ by
$\# (\F )$, the number of elements in the $\epsilon$-approximating family $\F$.
\end{rem}

\begin{proof}[Proof of Corollary \ref{cor:quasi-rig2}] The argument makes use of the 
Lefschetz fibration $\bar{h}:E\to\mathbb{C}$, the Lagrangian spheres  $\hat{S}_{x_{i}}\subset E$ and  of the setting  leading to the definitions of the local and ambient approximability data  in \S\ref{subsubsec:local_approx_d} and \S\ref{subsubsec:amb_approx_d}, respectively. The construction of $\bar{h}$ and its properties are in \S\ref{subsec:Giroux} and \S\ref{subsec:Lef-Dehn}.  

Given $L\in\lag^{(ex)}(D^{\ast}N)$ there exists an iterated cone $C_{L}$ in the category
$\mathscr{D}_{p}(D^{\ast}N)$ (see \S\ref{subsubsec:amb_approx_d}) as given in equation (\ref{eq:iterated_cones4}) :
\begin{equation}\label{eq:iterated_cones7}
 \Delta_{i} \ : \ Z_{i}\longrightarrow X_{i}\longrightarrow X_{i+1} \ , \ 0\leq i < m \ , 
\end{equation}
with $Z_{j}$ of the form $\Sigma^{l(j)}\hat{S}_{x_{s(j)}}$ or $=0$,
with each $\Delta_{i}$  exact in
 $(\mathscr{D}_{p}(E))^{0}$. Here $X_{0}=0$, 
$C_{L}=X_{m}$,  and there is a Lagrangian $X_{L}$  disjoint from $D^{\ast}N$ such that 
$\dint (L\oplus X_{L},C_{L}) <\epsilon + c_{\epsilon}\nu(p)$ - see  Corollary \ref{thm:nearby_far}. We will also assume that this decomposition is of minimal length, thus $m=N(L;\F,\epsilon)$.

The exact symplectomorphism $\phi$ from the statement naturally extends to $E$ by the identity in the exterior of $D^{\ast}N$. We denote this extension still by $\phi$. Notice that $\phi(X_{L})=X_{L}$.

With these conventions, recall from \S\ref{sbsb:symplecto-fun}, see (\ref{eq:phi-fun-2}) and (\ref{eq:phi-fun-3}), that the symplectic diffeomorphism
$\phi$ induces a TPC functor $\bar{\phi}:=PD(\phi) :\mathscr{D}_{p}(D^{\ast}N)\to \mathscr{D}_{\phi_{\ast}p}(D^{\ast}N)$ (we recall that the space of perturbations $\mathcal{P}$ is closed under the action of $\phi$).  

Denote $L'=L\oplus X_{L}$.  Notice that $\dint(L,\phi(L))= \dint(L\oplus X_{L},\phi (L)\oplus X_{L})$ and the interleaving distance $\dint(X,\phi(L))$ is the same in $D^{\ast}N$ and in $E$ (for a choice of perturbation data $p$ on $E$ that extends the choice on $D^{\ast}N$). Let $p, \phi_{\ast}p \preceq q$.
In $\mathscr{D}_{q}(D^{\ast}N)$ we have the inequality: 
$$\dint(\phi(L'),L')\leq \dint(\phi(L'), \mathcal{H}_{\phi_{\ast}p, q}(\bar{\phi}(C_{L})) )+\dint(\mathcal{H}_{\phi_{\ast}p, q}(\bar{\phi}(C_{L})) ,\mathcal{H}_{p,q}(C_{L}))
+\dint( \mathcal{H}_{p,q}(C_{L}),  L')~.~$$
At the same time we also have
$$ \dint(\phi(L'), \bar{\phi}(C_{L}) ) \leq \dint(L',C_{L}) < \epsilon +c_{\epsilon}\nu(p) $$ 
in $\mathscr{D}_{\phi_{\ast}p} (D^{\ast}N)$
because $\bar{\phi}$ is a persistence functor and therefore non-dilating with respect to $\dint$ (see Lemma \ref{l:d-func} i).
The functors $\mathcal{H}_{p,q}$, $\mathcal{H}_{\phi_{\ast}p,q}$  too are non-dilating  with respect to the respective interleaving pseudo-metrics.
 Thus,  in $\mathscr{D}_{q}(D^{\ast}N)$ we have 
$$ \dint(\phi(L'), \mathcal{H}_{\phi_{\ast}p,q}(\bar{\phi}(C_{L}) )) < \epsilon+c_{\epsilon}\nu(p),\ \dint( \mathcal{H}_{p,q}(C_{L}), L') < \epsilon +c_{\epsilon}\nu(p)~.~$$ 

Starting from this point we  first prove the statement  under the additional 
assumption: \begin{equation}\label{eq:fix_F}
\phi(F)=F \ , \ \forall \ F\in \F ~.~\end{equation}
Under this assumption we first notice that the statement follows if we can show that:
\begin{equation}\label{eq:equality_tr_dist}
\dint (\mathcal{H}_{\phi_{\ast}p, q}(\bar{\phi}(C_{L}), \mathcal{H}_{p,q}(C_{L}))=0
 \end{equation}
 for some $q$ such that $p ,\phi_{\ast}p \preceq q$  by  taking into account the relationships among
$\dint$, $D_{\mathrm{int}}$, and $d_{\gamma}$ as  in equations (\ref{eq:d-int3}), (\ref{eq:ineq_var_inter}) and (\ref{eq:equality_int}) and taking $p$ such that $\nu(p)\to 0$.

We now proceed to justify (\ref{eq:equality_tr_dist}) under  the additional assumption (\ref{eq:fix_F}).
The iterated cone $C_{L}$ belongs to the category 
$$\mathscr{D}_{p}(E)=PD(\fuk(\mathcal{L}ag^{(ex)}(E), p))~.~$$
Recall that there is another model for this category in terms of twisted complexes as 
described in \S\ref{sbsb:ftw} that we will denote by $\mathscr{D}^{Tw}_{p}(E)$.
This TPC  is the persistence homological category of the filtered twisted complexes over
the filtered $A_{\infty}$-category with objects the exact, closed Lagrangians in $E$. With
the notation in \S\ref{sbsb:ftw}:
$$\mathscr{D}^{Tw}_{p}(E)=PH\left(FTw\left(\fuk(\mathcal{L}ag^{(ex)}(E), p) \right)\right).$$
The Yoneda embedding extends to a $0$-equivalence of TPCs:
$$\mathcal{Y}_{p}: \mathscr{D}^{Tw}_{p}(E) \to \mathscr{D}_{p}(E)$$ which is compatible (up to $0$-isomorphism) with the comparison functors $\mathcal{H}_{p,q}$, when $p\preceq q$.
In particullar, the iterated cone $C_{L}$ is $0$-isomorphic to a filtered twisted complex 
$$\bar{C}_{L;p}=(\oplus_{i}Z'_{i}, A_{\bar{C}_{L,p}})$$ 
over the shifted Lagrangian spheres $Z'_{i}=\Sigma^{s_{i}}\hat{S}_{x_{j(i)}}$ that 
appear in the decomposition  (\ref{eq:iterated_cones7}).  Such a twisted complex
is characterized by an upper triangular $m\times m$ matrix, $A_{\bar{C}_{L;p}}$, with elements in $CF (\hat{S}_{x_{j}}, \hat{S}_{x_{s}}; p)$ (up to obvious shifts), as described in \cite{Se:book-fukaya-categ} (see also  the filtered $\lambda$-lemma  in \cite{Bi-Co-Sh:LagrSh} and \S\ref{lambdamap}). 
Here $CF(\hat{S}_{x_{i}}, \hat{S}_{x_{j}}; p)$ is the filtered Floer chain complex of $\hat{S}_{x_{i}}$ and $\hat{S}_{x_{j}}$  computed in $E$ with respect to the perturbation data $p$. To emphasize the dependence on the perturbation data we included $p$ in the notation. 

It is easy to track the effect of $\bar{\phi}$ by using twisted complexes through the push-forward of twisted complexes: this transports the Lagrangians involved by $\phi$, it pushes-forward the perturbation data
$p  \to \phi_{\ast}p$ and sends the matrix $A_{\bar{C}_{L;p}}$ to a matrix
$\phi_{\ast}(A_{\bar{C}_{L;p}})$, element to element, by the obvious transformation:
\begin{equation}\label{eq:ident1}
CF(\hat{S}_{x_{j}}, \hat{S}_{x_{s}}; p)\stackrel{\phi_{\ast}}{\longrightarrow} CF (\phi(\hat{S}_{x_{j}}), \phi(\hat{S}_{x_{s}}); \phi_{\ast} p)= CF(\hat{S}_{x_{j}}, \hat{S}_{x_{s}}; \phi_{\ast}p)
\end{equation}
(where we use our assumption that $\phi$ keeps each $\hat{S}_{x_{i}}$ invariant).
We denote the resulting twisted complex by $\phi_{\ast}(\bar{C}_{L,p})$. It has
the form:
$$\phi_{\ast}(\bar{C}_{L,p})=(\oplus_{i}Z'_{i}, A'_{\phi_{\ast}p})\in FTw\left(\fuk(\mathcal{L}ag^{(ex)}(E), \phi_{\ast}p ) \right) $$
where $A'_{\phi_{\ast}p}=\phi_{\ast}(A_{\bar{C}_{L,p}})$. 

At this point we need to be more precise about the various choices of Floer data for the 
pairs $(\hat{S}_{x_{i}},\hat{S}_{x_{j}})$. Due to the properties of the systems of categories
involved here, as described in \S\ref{s:sys-tpc} and in \S\ref{s:pdfuk}, any such choice that is part of our perturbation data family $\mathcal{P}$  is acceptable for this argument. Similarly, we need to discuss the two relevant continuation functors $\mathcal{H}_{p,q}$ and $\mathcal{H}_{\phi_{\ast}p,q}$. It is useful to revisit Figure \ref{fig:new_fibr}  a part of which is reproduced in 
Figure \ref{fig:new_fibr2} with minor adjustments. 

\begin{figure} [h]
  \includegraphics[scale=0.75]{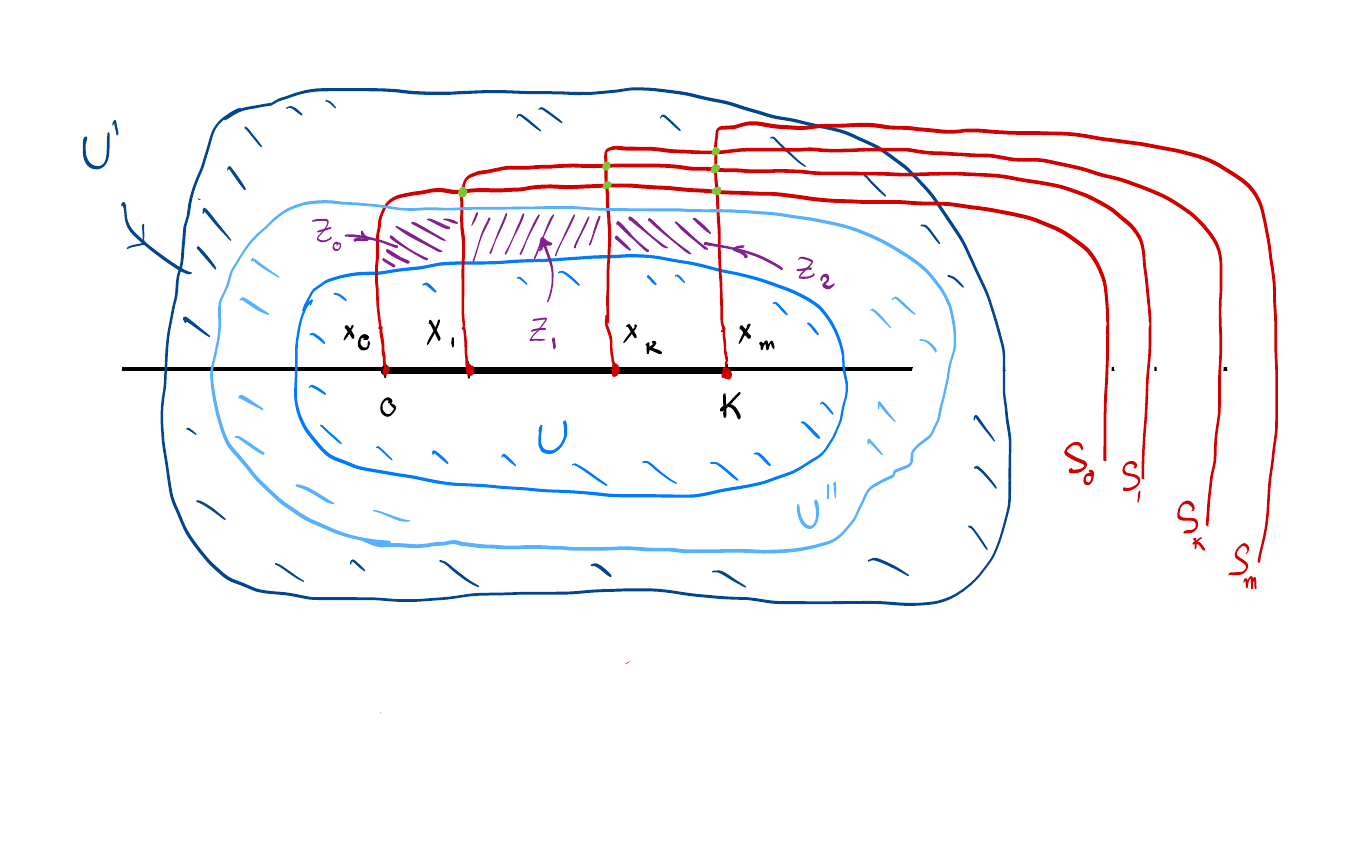}
  \centering
  \caption{The Lagrangian spheres and their intersections.} \label{fig:new_fibr2}
\end{figure}
Recall that $U$ in the picture contains the projection onto $\mathbb{C}$ of 
the disk bundle $D^{\ast}N\subset E$ and that the spheres $\hat{S}_{x_{i}}$ can be assumed 
 to coincide with the spheres $S_{x_{i}}$ in the picture outside of $D^{\ast}N$ (this requires slightly modifying some of the $U$, $U'$ in ways irrelevant for the rest of the
 argument). We may also assume that each two spheres $\hat{S}_{x_{i}}$, $\hat{S}_{x_{j}}$
 interesect transversely.  The Floer data for all the pairs $(\hat{S}_{x_{i}},\hat{S}_{x_{j}})$ is chosen such that the Hamiltonian $H^{\hat{S}_{x_{i}},\hat{S}_{x_{j}}}$ is constant outside
 $D^{\ast}N$. Thus the generators of the respective Floer complexes are the intersection points
 between the relevant spheres. The Morse functions on each of the spheres $\hat{S}_{x_{i}}$ are picked to have exactly two critical points, a maximum and a  minimum, both 
 outside $\bar{h}^{-1}(U')$. The almost complex structure $J^{i,j}=J^{\hat{S}_{x_{i}},\hat{S}_{x_{j}}}$ is picked such that $\bar{h}$ is $(J^{i,j}, i)$-holomorphic over $W=\bar{h}^{-1}(U''\backslash U)$. Here, the set $U''$ represents a neighbourhood of $D^{\ast}N$ that is smaller than $U'$, we could take it as the projection of $D^{\ast}_{2}N$, and such that all the intersection points of the Lagrangian spheres $\hat{S}_{x_{i}}$ are outside $U''$.  Further choices of almost complex structures on $E$ used to define 
 higher operations as well as the continuation maps $\mathcal{H}_{p,q}$ and $\mathcal{H}_{\phi_{\ast}p,q}$ are from the same class, in the sense that they make the projection $\bar{h}$ holomorphic. 

With these choices a simple application of the open mapping theorem shows that there are no Floer strips (or even polygons) that have as inputs intersection
points of paris of distinct $\hat{S}_{x_{i}}$, have output at another such intersection point, have boundaries on (some of) these spheres,  and that intersect the region $U''\backslash U$. For instance, if one such curve would intersect the region $Z_{1}$ in the picture it is easy to see that it also would have to extend to the regions $Z_{0}$ and $Z_{2}$ and, by pursuing this argument, one reaches a contradiction.

Recall that the support of $\phi$ is inside $D^{\ast}N\subset \bar{h}^{-1}(U)$. This means
that in fact $\bar{C}_{L,p}$ coincides with $\phi_{\ast}(\bar{C}_{L,p})$ - the two matrixes $A'_{\phi_{\ast}p}$ and $A_{\bar{C}_{L,p}}$ coincide. Moreover, 
again because the Floer type strips or polygons can not cross the region $U''\backslash U$,
is follows that, with obvious choices of perturbation data,  $\mathcal{H}_{p,q}$ and $\mathcal{H}_{\phi_{\ast}p,q}$ act in the same way on $\bar{C}_{L,p}$ and $\phi_{\ast}(\bar{C}_{L,p})$. 
Therefore,
$$\dint (\mathcal{H}_{p,q}(\bar{C}_{L,p}), \mathcal{H}_{\phi_{\ast}p, q}(\phi_{\ast}(\bar{C}_{L,p})))=0$$
which implies  (\ref{eq:equality_tr_dist}) and concludes the proof under the additional assumption (\ref{eq:fix_F}).

We now drop this assumption (\ref{eq:fix_F}) and consider the general case. Our aim is to show

\begin{equation}\label{eq:equality_tr_dist2}
\dint (\mathcal{H}_{\phi_{\ast}p, q}(\bar{\phi}(C_{L}), \mathcal{H}_{p,q}(C_{L}))\leq 
2m \chi(\phi;\F)~.~
 \end{equation}
which is sufficient to end the proof of the Corollary.
Denote $\bar{S}_{i}=\phi(\hat{S}_{x_{i}})$. These are exact Lagrangians spheres with the 
property that $\bar{S}_{i}$ coincides with $\hat{S}_{x_{i}}$ outside $D^{\ast}N$ and are so that the $\bar{S}_{i}$'s are pairwise disjoint inside $D^{\ast}N$. 
As a result, the same argument used under the assumption (\ref{eq:fix_F}) also
shows in the current more general context that the matrix $A'_{\phi_{\ast}p}$ that defines the twisted complex 
$$\phi_{\ast}(\bar{C}_{L,p})=(\oplus_{i}\phi(Z'_{i}), A'_{\phi_{\ast}p})\in FTw\left(\fuk(\mathcal{L}ag^{(ex)}(E), \phi_{\ast}p ) \right ) $$
is identified as a linear map - without taking into account filtrations - 
to $A_{\bar{C}_{L,p}}$ under the transformation $\hat{S}_{x_{i}}\to \bar{S}_{i}$ (that is akin to a basis change). However, the elements of the $m \times m$ matrix $A'_{\phi_{\ast}p}$ have potentially different filtration levels compared to  $A_{\bar{C}_{L},p}$. To understand this change 
notice that the filtered complex $CF(\bar{S}_{i}, \bar{S}_{j})$ is isomorphic to
$\Sigma^{k_{i,j}}CF (\hat{S}_{x_{i}}, \hat{S}_{x_{j}})$ where $|k_{i,j}|\leq Var(h_{\phi(F_{x_{i}})}) +Var(h_{\phi(F_{x_{j}})}) \leq 2\chi(\phi;\F)$. These shifts cumulate and given that we have $m$ terms in the twisted complexes under consideration we deduce (\ref{eq:equality_tr_dist2}) which concludes the proof.
\end{proof}

\subsection{Lagrangian Gromov width and Corollary \ref{cor:link}}\label{subsec:Gw-cor}

We start by recalling the definition of the Gromov width (\cite{Bar-Cor:Serre}, \cite{Bar-Cor:NATO}) that is relevant here. For this consider a symplectic manifold
$(M,\omega)$, a Lagrangian submanifold $L\subset M$ and a closed subset $K\subset M$. The {\em Gromov width of 
$L$ relative to $K$} is defined by:
\begin{eqnarray*}\mathscr{W}(L; K)=\sup \Bigl\{ \pi \frac{r^{2}}{2} \ | \ \exists\  j: (B_{r}, \omega_{0})\to (M,\omega) , \  j \ \text{is an embedding}, \\ j^{\ast}\omega=\omega_{0}, \ \ j^{-1}(L)= B_{r}\cap \R^{n}, \ j(B_{r})\cap K=\emptyset \Bigr\} 
\end{eqnarray*}
where $(B_{r},\omega_{0})\subset (\mathbb{C}^{n},\omega_{0})$ is the standard symplectic $r$-ball in $\mathbb{C}^{n}$, centred at the origin. 

The aim of this subsection is to show Corollary \ref{cor:link}.  
We reformulate here this corollary in a more precise way.

\begin{cor}\label{cor:link2} 
Assume that $(\mathcal{L}ag (M), d_{\gamma})$ is TPC retract $\epsilon$-approximable  in the sense of Definition \ref{d:approx-sys-ret}  in the system $\widehat{\msc}$ of  Fukaya type TPCs, by the  familiy $\F_{\epsilon}$ such that each element of $\F_{\epsilon}$ is a Yoneda module. 
For each $L\in \lag(M)$ we have
$$\mathscr{W}(L;  \cup_{F\in \F_{\epsilon}}F)\leq 2\epsilon ~.~$$\end{cor}

\begin{rem}\label{rem:width-fibers}
a. Corollary \ref{cor:link2} does not directly apply to the case of $M=D^{\ast}N$ because the $\epsilon$-generating 
families $\F_{\epsilon}$ consist  of modules $\mathscr{F}_{x_{i}}$ that correspond to fibers $F_{x_{i}}$ 
but that are not Yoneda modules in the categories $\msc_{p}$ (these are part of the local approximating 
data $(\Phi, \F_{\epsilon})$ - see \S\ref{subsubsec:local_approx_d}).
However, the Corollary does apply to the ambient approximating data $(\Phi',\F'_{\epsilon})$ from \S\ref{subsubsec:amb_approx_d} because, in this case, the elements of $\F'_{\epsilon}$ are represented
by Lagrangian spheres $\hat{S}_{x_{i}}\subset E$. Therefore, recalling $F_{x_{i}}=\hat{S}_{x_{i}}\cap D^{\ast}N$, we deduce that the conclusion of the Corollary remains valid in the case of $M=D^{\ast}N, \ \F_{\epsilon}=\{ F_{x_{1}},\ldots, F_{x_{l}}\}$.

b. The inequality in the statement immediately implies that, when $\epsilon\to 0$, the family $\F_{\epsilon}$ becomes dense in $M$. Indeed, assume that a family $\mathcal{F}_{\epsilon}$ avoids some  ball 
$B\subset M$. Take $L$ that passes through the center $x$ of this ball and also pick an embedding 
$j$, as in the definition of width, with $j(0)=x$ and such that $j(B_{r})\subset B$. We deduce in this case that 
$\epsilon \geq \pi\frac{{r}^{2}}{2}$ and therefore $\epsilon$ is bounded away from zero under the hypothesis
that $\F_{\epsilon}$ does not intersect $B$.

c. In this remarks consider the case $M=D^{\ast}N$. A fiber is determined 
by its base point, and the disk bundle $D^{\ast}N$ is determined by the Riemannian metric $\mathtt{g}$. In view of this,  the Corollary indicates that there are some symplectic invariants associated with pairs $(\mathcal{S},\mathtt{g})$ where $\mathcal{S}$ is a configuration of $k$ distinct points in the base. To see this, for any finite 
family $\mathcal{S}\in N$ denote by $\mathcal{F}_{\mathcal{S}}$ the corresponding family of fibers inside $D^{\ast}N$ (the metric defining the disk bundle being $\mathtt{g}$). Now define the {\em approximability precision} of $\mathcal{S}$ by:
$$\epsilon_{\mathcal{S}}= \inf \{\  \epsilon > 0 \ | \  \mathrm{there\ exists}\ \epsilon-\mathrm{approximating\ data} \ (\Phi, \mathcal{F}_{\mathcal{S}}) \ \mathrm{as\ in}\  \S\ref{subsubsec:local_approx_d} \ \}~.~$$
This is non-zero, by point b, and it is finite. Indeed,  for $\mathcal{S}$ consisting of  a single point, it follows from the arguments in \cite{Bi-Co:mspectral} (which are partially a precursor to the proof of Theorem \ref{thm:nearby}) that $\epsilon_{\mathcal{S}}$ is already finite. This quantity is, obviously, independent of the point in question and is an upper bound for all $\epsilon_{\mathcal{S}}$, for arbitrary finite $\mathcal{S}$. However, for $\mathcal{S}$ with more than a single element this number generally depends on the placement of the points in $N$. For instance, by taking $\mathcal{S}\subset D$ with $D\subset N$ a small disk, $\epsilon_{\mathcal{S}}$ certainly will stay away from $0$ (again by point b),
independently of the number of elements in $\mathcal{S}$. 
But, if the same number of fibers are more uniformly distributed along $N$, then, by increasing the number of elements of $\mathcal{S}$ one can make choices such that $\epsilon_{\mathcal{S}}\to 0$ because, by Theorem \ref{thm:nearby},  $\epsilon$-approximating data exists for all $\epsilon$. We will not pursue this topic here but it obviously warrants more
study.
% Define for each natural number $k$ the following quantity:
%$$\mathrm{app}(k)=\inf \{ \ \epsilon_{\F} \ | \ \F \ \mathrm{is\ a\ finite\ family\ containing}\ k \ \mathrm{ fibers\ of}\ D^{\ast}N \ \}~.~$$
%Thus, this infimizes the approximability precision achievable with $k$ generators. Finally, a family of $k$ fibers $\F$ is
%symplectically equidistributed if $$\epsilon_{\F}=\mathrm{app}(k)~.~$$
%
\end{rem}

\begin{proof}[Proof of Corollary \ref{cor:link2}] The proof follows directly from Corollary 6.13 i.  in \cite{Bi-Co-Sh:LagrSh}. We only need to relate the notation and conventions in that paper to the context here. The first remark is that \cite{Bi-Co-Sh:LagrSh} works in the setting of weakly-filterered $A_{\infty}$ categories. The setting in the current paper
is that of filtered $A_{\infty}$ categories and thus the results and arguments in \cite{Bi-Co-Sh:LagrSh} apply here too.  
Denote the  TPCs in the system $\widehat{\msc}$ by $\msc_{p}(M)$. Take $L\in\lag(M)$. 
Recall from Remark \ref{rem:relation-TPCapp} that there is no difference whether we use $\sdint$ or 
$\dint$ in the definition of retract approximability. Moreover, from the description of $\sdint$ 
in Lemma 4.2.3, we deduce that there exists an element $a\in \Ob(\msc_{p}(M))$, a cone decomposition
$\eta: 0\cobto a$ with linearization in $\F_{\epsilon}$, and a couple of maps $\alpha\in \hom^{r_{1}}(L, a)$, $\beta\in \hom^{r_{2}}(a,L)$ such that $\beta\circ \alpha = i_{0,r_{1}+r_{2}}(id_{L})$ with $r_{1}+r_{2} < 2\epsilon +c_{\epsilon}\nu(p)$ where $p\in\mathcal{P}$ is
a choice of approriate perturbation data. By inspecting the definition of the quantity $\rho$ in (6.16) in \cite{Bi-Co-Sh:LagrSh} we see 
that $\rho(\alpha) < 2\epsilon + c_{\epsilon}\nu(p)$ (in brief, this $\rho(\alpha)$ infimizes $r_{1}+r_{2}$ such that there is a morphism $\beta$ as before). Moreover, this means that, for $\nu(p)$ small enough, we have, 
with the notation in \cite{Bi-Co-Sh:LagrSh}, $w_{p}( (\alpha, a,\eta)) < 2\epsilon +c_{\epsilon}\nu(p)$ (this $w_{p}$ in (6.19)
\cite{Bi-Co-Sh:LagrSh} is a measurement assigned to triples $(\alpha, a, \eta)$ that infimizes $\rho (\alpha)$ over all choices of $\alpha : L\to a$; it also depends on the perturbation data $p$).  From 
Corollary 6.13. i \cite{Bi-Co-Sh:LagrSh} we 
deduce $2\epsilon \geq \frac{1}{2}\delta (L; \cup_{F\in \F_{\epsilon}}F)$. The quantity $\delta (-;-)$
defined on p. 85 in \cite{Bi-Co-Sh:LagrSh} is the double of the Gromov width as defined here, at the beginning of 
this section. This concludes the proof of the corollary.   
\end{proof}

\begin{rem}\label{rem:direct_cor}
 a. As mentioned in the introduction,  there are some purely geometric implications of  Corollaries \ref{cor:quasi-rig2} and \ref{cor:link2} that are independent of approximability as well as of
 the machinery used for their proofs.  The first statement can be  read as saying that for each $\epsilon>0$ there exists a finite  family of Lagrangians in the class providing approximations (fibers for $T^{\ast}N$, equators for $S^{2}$ etc)   such that the Gromov width of any $L\in \lag^{(ex)}(M)$ relative to this family is less than $\epsilon$. Similarly, focusing on the case of $M=D^{\ast}N$, for any $\epsilon>0$ there exists a finite family of fibers $\F$ such that any exact symplectic diffeormorphism  satisfying $\phi(F)=F$ for all $F\in\F$  only moves each $L\in \lag^{(ex)}(D^{\ast}N)$ less than $\epsilon$ in the spectral metric. 
 
 b. Some of these geometric consequences can also be obtained by more direct means. Indeed, 
 it was noticed by Egor Shelukhin that for any $\epsilon>0$ one can prove by methods unrelated to Fukaya category techniques  that there exists a finite family $\F$ of fibers of $D^{\ast}N$ such that any  Hamiltonian diffeomorphism $\phi$ induced by a Hamiltonian isotopy with support inside $D^{\ast}N$ and disjoint from the fibers in $\F$, has spectral 
 norm smaller than $\epsilon$.  The proof uses the same type of  Morse functions $\varphi$ as constructed in Proposition \ref{prop:Morse} and a result of Usher \cite{Ush10}. This result also implies
 that the Gromov width of the complement of the fibers inside $D^{\ast}M$ can be bounded above by $\epsilon$. This Gromov width  result was also noticed by Mark Gudiev by means of a direct displacement argument using the Morse functions $\varphi$ as in Proposition \ref{prop:Morse}.
\end{rem}

% !TEX root = approx8.tex

\subsubsection{Gromov width, cone length and sizes of approximating
  families} \label{sbsb:gwidth-2}

Let $(M, \omega)$ be a symplectic manifold of one of the types
described in~\S\ref{s:pdfuk} and let $\lag$ be a subset of the
collection of Lagrangians in $M$, as in~\S\ref{sb:sys-fuk}. Recall
that we have the system of filtered $A_{\infty}$-categories
$\widehat{\fuk}(\lag)$ and its filtered derived system of TPC's
$PD(\widehat{\fuk}(\lag))$.

Let $\mathcal{S}$ be a collection of subsets of $M$ and
$\mathcal{L} \subset \lag$ a collection of Lagrangians from
$\lag$. For $w > 0$ define the $w$-density index of $\mathcal{S}$
relative $\mathcal{L}$ to be:
\begin{equation} \label{eq:density-rel} I(\mathcal{S}, \mathcal{L}; w)
  := \inf \bigl\{k \geq 1 \mid \exists \, S_1, \ldots, S_k \in
  \mathcal{S} \text{ such that } \mathscr{W}(L; S_1 \cup \cdots \cup
  S_k) \leq w, \; \forall L \in \mathcal{L} \bigr\}.
\end{equation}
Here and elsewhere in the paper we use the convention that
$\inf \emptyset = \infty$, $\sup \emptyset = -\infty$. Therefore if
$L \subset S_1 \cup \cdots \cup S_k$ for some $L \in \lag$,
$S_1, \ldots, S_k \in \mathcal{S}$, then we have
$\mathscr{W}(L; S_1 \cup \cdots \cup S_k) = -\infty$ hence
$\mathscr{W}(L; S_1 \cup \cdots \cup S_k) \leq w$ is automatically
satisfied in~\eqref{eq:density-rel}.

For a subset $\mathcal{K} \subset \lag$ we denote by
$$\widetilde{\mathcal{K}} \subset
\widetilde{\Ob}(PD(\widehat{\fuk}(\lag)))$$ the collection of
equivalence classes (see~\S\ref{s:approximability}) of the objects
$\mathcal{Y}_p(K) \in \Ob(PD(\fuk(\lag; p)))$, $K \in \mathcal{K}$,
$p \in \mathcal{P}$, where
$\mathcal{Y}_p: \fuk(\lag; p) \longrightarrow F\md_{\fuk(\lag; p)}$ is
the Yoneda embedding (see~\S\ref{gio:fyon-def}
and~\S\ref{sbsb:yoneda-sys}).

The following is an immediate consequence of
Corollary~\ref{cor:link2}.
\begin{cor} \label{c:approx-index} Let
  $\mathcal{L}, \mathcal{F} \subset \lag$ be two collections of
  Lagrangians. Let $w_0 > 0$ and $\mathcal{F}^0 \subset \mathcal{F}$
  be a finite subset such that $\widetilde{\mathcal{L}}$ is
  $w_0$-retract-approximable by $\widetilde{\mathcal{F}^0}$ in
  $PD(\widehat{\fuk}(\lag))$. Then
  \begin{equation} \label{eq:size-index} |\mathcal{F}^0| \geq
    I(\mathcal{F}, \mathcal{L}; w_0).
  \end{equation}
  Moreover, there exists $L \in \mathcal{L}$ such that
  $N^r(L; \widetilde{\mathcal{F}^0}, w_0) \geq I(\mathcal{F},
  \mathcal{L}; w_0)$.
\end{cor}

\begin{rems} \label{r:approx-index}
  \begin{enumerate}
  \item Point (a) of Remark~\ref{rem:width-fibers} applies here as well.
    Namely, the conclusion of Corollary~\ref{c:approx-index} continues
    to hold for $M = D^*N$ with $\mathcal{L} = \lag$ being the
    collection of all closed exact Lagrangians in $D^*N$ and with
    $\mathcal{F}$ being any collection of cotangent fibers.
  \item If $\mathcal{S}' \subset \mathcal{S}''$,
    $\mathcal{L}' \supset \mathcal{L}''$, $w' \leq w''$, then
    \begin{equation} \label{eq:ineq-index} I(\mathcal{S}', \mathcal{L}';
      w') \geq I(\mathcal{S}'', \mathcal{L}'';w'').
    \end{equation}
    In particular, for every $\mathcal{S}$ and $\mathcal{L}$ we have
    $I(\mathcal{S}, \mathcal{L}; w) \geq I(\mathcal{S}, \emptyset;
    w)$. Note that the latter invariant depends on the collection
    $\mathcal{S}$ itself and does not involve any Lagrangian
    submanifolds.
  \item The inequality~\eqref{eq:size-index} may become useful if we
    fix a collection $\mathcal{F}$ that (retract) approximates
    $\mathcal{L}$ and want to bound from below the size of any
    possible finite family $\mathcal{F}^0 \subset \mathcal{F}$ that
    retract-$w_0$-approximates $\mathcal{L}$. Note that there is
    no particular assumption on the collection of Lagrangian $\mathcal{F}$
    in Corollary~\ref{c:approx-index}, besides containing
    $\mathcal{F}^0$ (and being a subset of $\lag$ of course). In
    particular we do not even need assume that $\mathcal{F}$ retract
    approximates $\mathcal{L}$. But in view of~\eqref{eq:ineq-index}
    the most effective bounds from below in~\eqref{eq:size-index} on
    $|\mathcal{F}^0|$ are obtained when $\mathcal{F}$ is minimal as
    possible and in applications it makes sense to assume that
    $\mathcal{F}$ retract approximates $\mathcal{L}$.
  \end{enumerate}
\end{rems}

Here is a simple example showing how Corollary~\ref{c:approx-index}
can be used in practice. Consider $M=S^2$ endowed with its standard
symplectic structure $\omega$ normalized such that
$\text{Area}(S^2, \omega) = 1$. Let $\lag$ be the collection of all
closed monotone Lagrangian in $S^2$ (in this case, all the embedded
closed curves which separate $S^2$ into two domains of area
$\tfrac{1}{2}$). Let $\mathcal{F} \subset \lag$ be the collection of
all great circles passing through the north and south poles of $S^2$.
By Proposition~\ref{OCS2}, $\widetilde{\mathcal{F}}$
retract-approximates $\widetilde{\lag}$ in $PD(\widehat{\fuk}(\lag))$.

We claim that
$I(\mathcal{F}, \emptyset; w) \geq \lceil \tfrac{1}{4w}
\rceil$. Indeed, if $F_1, \dots, F_k \in \mathcal{F}$ then at least
one of the connected components of
$S^2 \setminus (F_1 \cup \cdots \cup F_k)$ must have area
$\geq \tfrac{1}{2k}$ hence
$\mathcal{W}(\emptyset; F_1 \cup \cdots \cup F_k) \geq \tfrac{1}{4k}$.
It follows that
$I(\mathcal{F}, \emptyset; w) \geq \lceil \tfrac{1}{4w} \rceil$. (In
fact, it is not hard to see that we have equality here.)

It follows from Corollary~\ref{c:approx-index} that for every
$\mathcal{L} \subset \lag$, every $w_0>0$, and every finite collection
$\mathcal{F}^0 \subset \mathcal{F}$ which $w_0$-retract-approximates
$\mathcal{L}$ we have:
$$|\mathcal{F}^0| \geq I(\mathcal{F}, \mathcal{L}; w_0) \geq
I(\mathcal{F}, \emptyset; w_0) \geq \Bigl\lceil \frac{1}{4w_0} \Bigr
\rceil.$$
 
% !TEX root = approx8.tex

\appendix

\section{Filtered $A_{\infty}$-categories} \label{a:fil-ai}
\subsection{Filtered $A_\infty$-categories, $A_\infty$-functors, $A_\infty$-natural transformations and shift $A_\infty$-functors}\label{a:sb:fil-ai}
 Filtered
$A_{\infty}$-categories have already appeared in the literature, see
e.g.~\cite[Section~3.2]{BCZ:tpc}
and~\cite[Chapter~2]{Bi-Co-Sh:LagrSh}. We will therefore go over this
subject very briefly and cover below mainly aspects of the theory that
have not been addressed in the literature.

Fix a field $\k$ and $\k_0 \subset \k$ a subring. For most of the
applications in this paper we will either have
$\k_0 = \k = \mathbb{Z}_2$, or $\k = \Lambda$ (the universal Novikov
field with coefficients in $\mathbb{Z}_2$) and
$\k_0 = \Lambda_0 \subset \Lambda$ the positive Novikov ring
(see~\eqref{eq:Nov-ring} and~\eqref{eq:Nov-ring-0}
in~\S\ref{sb:fuk-mon}).

In the followings, we say that $(C,d)$ is a filtered chain complex over $\k$ if $(C,d)$ is chain complex over $\k$ and comes endowed with an increasing filtration parametrized by the real line of $\k_0$-modules that are preserved by $d$.

Let $\mathcal{A}$ be an $A_{\infty}$-category over $\k$. For
$X_0, X_1 \in \Ob(\mathcal{A})$ we abbreviate:
$$\mathcal{A}(X_0,X_1):= \hom_{\mathcal{A}}(X_0,X_1),$$ and
more generally, for a tuple of objects
$\vec{X} := (X_0, \ldots, X_d)$, $d \geq 1$, we write
$$\mathcal{A}(\vec{X}) := \mathcal{A}(X_0, X_1) \otimes_{\k} \cdots
\otimes_{\k} \mathcal{A}(X_{d-1}, X_d)$$ and denote by
$\mu_d: \mathcal{A}(\vec{X}) \longrightarrow \mathcal{A}(X_0, X_d)$
the $A_{\infty}$-operations.

A filtered $A_{\infty}$-category (over $(\k,\k_0)$) is an
$A_{\infty}$-category $\mathcal{A}$ over $\k$ such that the spaces of
morphisms $\mathcal{A}(X, Y) = \hom_{\mathcal{A}}(X, Y)$ between every
two objects $X, Y$ are filtered chain complexes over $\k$
%  are filtered by an increasing filtration of $\k_0$-modules parametrized by $\mathbb{R}$, and such that all the
%$A_{\infty}$-operations $\mu_d$, $d \geq 1$, preserve the filtration
%levels.}The precise definition is as follows. 
For every
$\alpha \in \mathbb{R}$ we denote by
$\mathcal{A}^{\alpha}(X,Y) \subset \mathcal{A}(X,Y)$ the
$\k_0$-submodule which is the filtration level parametrized by
$\alpha$. More generally, for $\alpha \in \mathbb{R}$ and a tuple
$\vec{X} = (X_0, \ldots, X_d)$ of objects of $\mathcal{A}$ we write
$$\mathcal{A}^\alpha(\vec{X}) \subset \mathcal{A}(\vec{X})$$ for the
$\k_0$-submodule generated by the images of all the $\k_0$-modules
$$\mathcal{A}^{\alpha_1}(X_0, X_1) \otimes_{\k_0} \cdots \otimes_{\k_0}
\mathcal{A}^{\alpha_d}(X_{d-1}, X_d), \text{ with } \alpha_1 + \cdots
+ \alpha_d \leq \alpha,$$ under the obvious ($\k_0$-linear) map
$$\mathcal{A}(X_0, X_1) \otimes_{\k_0} \cdots \otimes_{\k_0}
\mathcal{A}(X_{d-1}, X_d) \longrightarrow \mathcal{A}(\vec{X}).$$ We
then require that for every $d \geq 1$, every $\vec{X}$ and every
$\alpha \in \mathbb{R}$, we have:
$$\mu_d (\mathcal{A}^{\alpha}(\vec{X}))
\subset \mathcal{A}^{\alpha}(X_0, X_d).$$

% and for every tuple of objects
% $\vec{X} = (X_0, \ldots, X_d)$ of $\mathcal{A}$ we write
% $\mathcal{A}^{\alpha}(\vec{X})$ for the $\k_0$-submodule of
% $\mathcal{A}(\vec{X})$ generated by all the submodules
% $$\mathcal{A}^{\alpha_1}(X_0, X_1) \otimes_{\k_0} \cdots \otimes_{\k_0}
% \mathcal{A}^{\alpha_d}(X_{d-1}, X_d) \subset \mathcal{A}(\vec{X}), \;
% \text{with} \; \alpha_1 + \cdots + \alpha_d \leq \alpha.$$ Then we
% require that for every $d \geq 1$ and every $\vec{X}$ we have
% $$\mu_d (\mathcal{A}^{\alpha}(\vec{X}))
% \subset \mathcal{A}^{\alpha}(X_0, X_d), \; \forall \alpha \in
% \mathbb{R}.$$ Another useful piece of notation is the following.
% Consider a tuple $\vec{\alpha} = (\alpha_1, \ldots, \alpha_d)$,
% $d \geq 1$, $\alpha_i \in \mathbb{R}$ of real numbers and denote
% $|\vec{\alpha}| = \alpha_1 + \cdots + \alpha_d$. For a tuple $\vec{X}$
% of objects in $\mathcal{A}$ define
% $$\mathcal{A}^{\vec{\alpha}}(\vec{X}) :=
% \mathcal{A}^{\alpha_1}(X_0, X_1) \otimes_{\k_0} \cdots \otimes_{\k_0}
% \mathcal{A}^{\alpha_d}(X_{d-1}, X_d) \subset \mathcal{A}(\vec{X}).$$

A filtered $A_{\infty}$-category $\mathcal{A}$ is called strictly
unital (in the filtered sense) if it is strictly unital as an
$A_{\infty}$-category without filtrations and moreover for every
$X \in \Ob(\mathcal{A})$ the strict unit $e_X \in \mathcal{A}(X,X)$
lies in filtration level $0$, i.e.~$e_X \in \mathcal{A}^0(X,X)$.
Unless otherwise stated, all the filtered $A_{\infty}$-categories
$\mathcal{A}$ in this paper are implicitly assumed to be strictly
unital.

An example of filtered $A_\infty$-category is the $dg$-category $\mathcal{F}\mathbf{Ch}$ of filtered chain complexes over $\k$.

The concept of $A_{\infty}$-functors has a filtered counterpart.  An
$A_{\infty}$-functor
$\mathcal{F}: \mathcal{A} \longrightarrow \mathcal{B}$ between two
$A_{\infty}$-categories is said to be filtered if for every $d \geq 1$
and every tuple of objects $\vec{X}$ its $d$-th order term
$\mathcal{F}_d$ preserves filtrations in the sense that
$$\mathcal{F}_d(\mathcal{A}^{\alpha}(\vec{X})) \subset
\mathcal{B}^{\alpha}(\mathcal{F}X_0, \mathcal{F}X_d), \; \forall
\alpha \in \mathbb{R}.$$ All the filtered $A_{\infty}$-functors in
this paper will be implicitly assumed to be strictly unital.

One can define pre-natural transformations between filtered
$A_{\infty}$-functors as follows.  Let
$\mathcal{F}, \mathcal{G}: \mathcal{A} \longrightarrow \mathcal{B}$ be
two filtered $A_{\infty}$-functors. A pre-natural transformation
$T: \mathcal{F} \longrightarrow \mathcal{G}$ of shift $s$ is a
pre-natural transformation (in the usual $A_{\infty}$-sense) such that
its $d$-th order term $T_d$ shifts filtrations by not more than $s$, namely we have
a family of elements
$T^0 \in \mathcal{B}^{s}(\mathcal{F} X, \mathcal{G}X)$ for every
$X \in \Ob(\mathcal{A})$ and moreover, for every $d \geq 0$ and every
tuple of objects $\vec{X} = (X_0, \ldots, X_d)$ we have:
$$T_d (\mathcal{A}^{\alpha}(\vec{X})) \subset
\mathcal{B}^{\alpha + s}(\mathcal{F}X_0, \mathcal{G}X_d), \; \forall d
\geq 1, \, \alpha \in \mathbb{R}.$$

We will always endow our $A_{\infty}$-categories $\mathcal{A}$ with a
shift functor. This is not a single functor but in fact a system of
filtered $A_{\infty}$-functors
$\{\Sigma^r : \mathcal{A} \longrightarrow \mathcal{A}\}_{r \in
  \mathbb{R}}$, $r \in \mathbb{R}$, all having trivial higher order
terms ($\Sigma^r_d = 0, \, \forall d \geq 2$), and together with
filtered natural transformations
$\eta_{r,s}: \Sigma^r \longrightarrow \Sigma^s$ of shift $s-r$, for
every $r,s \in \mathbb{R}$, and such that:
\begin{enumerate}
\item $\Sigma^0 = \id$. \label{i:shift-a}
\item $\Sigma^r \circ \Sigma^s = \Sigma^{r+s}$ for all $r,s \in \mathbb{R}$.
\item The higher order terms $(\eta_{r,s})_d, d \geq 1$ of
  $\eta_{r,s}$ vanish.
\item For every $X \in \Ob(\mathcal{A})$ the $0$-term of $\eta_{r,s}$
  assigned to $X$ is a cycle
  $(\eta^X_{r,s})_0 \in \mathcal{A}^{s-r}(\Sigma^rX, \Sigma^sX)$.
  Moreover, we have $(\eta^X_{0,0})_0 = e_X \in \mathcal{A}^0(X,X)$.
\item For every $X \in \Ob(\mathcal{A})$ and every $t \in \mathbb{R}$
  we have
  $\mu_2 \bigl( (\eta^X_{r,s})_0, (\eta^X_{s,t})_0 \bigr) =
  (\eta^X_{r, t})_0$.
\item For every $X \in \Ob(\mathcal{A})$ and every $t \in \mathbb{R}$
  we have
  $\Sigma_1^t \bigl( (\eta^X_{r,s})_0 \bigr) = (\eta^X_{r+t, s+t})_0$.
  \label{i:shift-b}
\end{enumerate}
It follows that each $\eta_{r,s}$ is an isomorphism and its inverse is
$\eta_{s,r}$. By composing from the left and from the right with
$\eta^X_{0,r}$ and $\eta^Y_{s,0}$ respectively we obtain isomorphisms
$$\mathcal{A}^{\alpha}(\Sigma^rX, \Sigma^sY) \cong
\mathcal{A}^{\alpha+r-s}(X,Y), \; \forall \, \alpha \in \mathbb{R}$$
that are compatible with the filtration
$\mathcal{A}^{\beta'}(- , -) \subset \mathcal{A}^{\beta''}(-,-)$,
$\beta' \leq \beta''$.

\begin{rem}
  The conditions~\eqref{i:shift-a}-\eqref{i:shift-b} may look
  excessively strong in the $A_{\infty}$-context and can probably be
  relaxed to give a weaker yet meaningful definition of shift
  functors. A more appropriate name for $\Sigma$ as above might be a
  {\em strict} shift functor, but we will not use this term in this
  paper.
\end{rem}

From now on all filtered $A_{\infty}$-functors will be assumed to be
compatible with the shift functor in the sense that
$\mathcal{F} \circ \Sigma^r = \Sigma^r \circ \mathcal{F}$ for every
$r \in \mathbb{R}$ and
$(\eta^{\mathcal{F} X}_{r,s})_0 = \mathcal{F}_1(\eta^{X}_{r,s})_0$ for
every $r, s \in \mathbb{R}$.

\subsection{Filtered $A_{\infty}$-modules and
  bimodules} \label{a:fmod} Let $\mathcal{A}$ be a filtered
$A_{\infty}$-category. A filtered left $\mathcal{A}$-module
$\mathcal{M}$ is an $\mathcal{A}$-module with the following additional
structures and properties. For every $X \in \Ob(\mathcal{A})$ the
chain complex $(\mathcal{M}(X), \mu^{\mathcal{M}}_1)$ (of vector
spaces over $\k$) is filtered by an increasing filtration of
subcomplexes of $\k_0$-modules which is parametrized by
$\mathbb{R}$. For $\alpha \in \mathbb{R}$ we denote by
$\mathcal{M}^{\alpha}(X) \subset \mathcal{M}(X)$ the $\alpha$-level of
this filtration. For $\alpha \in \mathbb{R}$ and a tuple of objects
$\vec{X} = (X_0, \ldots, X_k)$ denote by
\begin{equation} \label{eq:fil-AM}
  (\mathcal{A}(\vec{X}) \otimes \mathcal{M}(X_k))^{\alpha} \subset
  \mathcal{A}(\vec{X}) \otimes_{\k} \mathcal{M}(X_k)
\end{equation}
the $\k_0$-submodule generated by the images of all the
$\k_0$-submodules
$$\mathcal{A}^{\alpha'_1}(X_0, X_1) \otimes_{\k_0} \cdots
\otimes_{\k_0}\mathcal{A}^{\alpha'_k}(X_{k-1}, X_k) \otimes_{\k_0}
\mathcal{M}^{\alpha''}(X_k), \text{ with } \alpha'_1 + \cdots+
\alpha'_k + \alpha'' \leq \alpha,$$ under the obvious map
$$\mathcal{A}(X_0, X_1) \otimes_{\k_0} \cdots
\otimes_{\k_0}\mathcal{A}(X_{k-1}, X_k) \otimes_{\k_0}
\mathcal{M}(X_k) \longrightarrow \mathcal{A}(\vec{X}) \otimes_{\k}
\mathcal{M}(X_k).$$ We require that the higher operations
$\mu^{\mathcal{M}}_{k|1}$, $k \geq 1$, preserves filtrations. Namely, for every tuple of objects
$\vec{X} = (X_0, \ldots, X_k)$ in $\mathcal{A}$ and every
$\alpha \in \mathbb{R}$ we have:
$$\mu^{\mathcal{M}}_{k|1} \bigl( (\mathcal{A}(\vec{X}) \otimes
\mathcal{M}(X_k))^{\alpha}) \subset \mathcal{M}^{\alpha}(X_0).$$ For
uniformity we will denote the differential $\mu_1^{\mathcal{M}}$ also
by $\mu_{0|1}^{\mathcal{M}}$.\\
Alternatively, one can view filtered $\mathcal{A}$-modules as filtered $A_\infty$-functors $\mathcal{A}\to \F\textbf{Ch}^{\text{op}}$.

Pre-homomorphisms between filtered modules are defined as follows. Let
$\mathcal{M}$, $\mathcal{N}$ be filtered $A_{\infty}$-modules. A
module pre-homomorphism $t: \mathcal{M} \longrightarrow \mathcal{N}$
of shift $\leq s \in \mathbb{R}$ is a pre-homomorphism such that for
every $k \geq 0$, every tuple $\vec{X}$, and $\alpha \in \mathbb{R}$ 
we have
$$t_{k|1}( (\mathcal{A}(\vec{X}) \otimes
\mathcal{M}(X_k))^{\alpha}) \subset \mathcal{N}^{\alpha + s}(X_0).$$

Filtered $A_{\infty}$-modules and their pre-homomorphisms form a
filtered $dg$-category $F\md_{\mathcal{A}}$, where
$\hom^{s}_{F\md_{\mathcal{A}}}(\mathcal{M}, \mathcal{N})$ consists of
all pre-homomorphisms $t: \mathcal{M} \longrightarrow \mathcal{N}$ of
shift $\leq s$. Note that $F\md_{\mathcal{A}}$ comes with a natural shift functor $\Sigma$.
$\Sigma$. Its action on objects
$\mathcal{M} \in \Ob(F\md_{\mathcal{A}})$ is given by:
$$(\Sigma^r \mathcal{M})^{\alpha}(X) := \mathcal{M}^{\alpha-r}(X),
\; \forall r \in \mathbb{R}, \alpha \in \mathbb{R}, X \in
\Ob(\mathcal{A}).$$

Unless otherwise stated, all the modules in this paper will be assumed
to be strictly unital. We will also assume all the filtered modules to
be compatible with the shift functor in the obvious sense, i.e. $\mathcal{M}(\Sigma^rX) = (\Sigma^r\mathcal{M})(X)$ for all objects $X\in \Ob(\mathcal{a})$ and all $r\in \mathbb{R}$.

Right $\mathcal{A}$-modules can be defined in a completely analogous
way to left modules. To distinguish the two, whenever confusion may
arise we will denote the category of filtered left $\mathcal{A}$-modules by
$F\md^l_{\mathcal{A}}$ and the right modules by
$F\md^r_{\mathcal{A}}$.

Given two filtered $A_{\infty}$-categories $\mathcal{A}, \mathcal{B}$
we also have the filtered $(\mathcal{A}, \mathcal{B})$ modules. They
are defined by similar means to the above. Filtered bimodules too will
be assumed in this paper to be strictly unital and compatible with the
shift functor. They form a filtered $dg$-category which we
denote by $F\text{bimod}_{\mathcal{A}, \mathcal{B}}$.

\subsection{The filtered Yoneda embedding}
\subsubsection{Definition}\label{gio:fyon-def}
Let $\mathcal{A}$ be a filtered $A_\infty$-category with structure maps $(\mu_d)_{d\geq 1}$. Recall that we always assume that our $A_\infty$-categories are strictly unital. We upgrade the Yoneda embedding (see \cite[Section (1l)]{Se:book-fukaya-categ}) to a filtered version, that is, we define a filtered $A_\infty$-functor $$\mathcal{Y}\colon \mathcal{A}\to F\md_\mathcal{A}.$$ We will work with left $A_\infty$-modules, but the same can be analogously developed for right ones. Whenever confusion may arise we will denote by $\mathcal{Y}^l\colon \mathcal{A}\to F\md^l_\mathcal{A}$ the left Yoneda embedding and by $\mathcal{Y}^r\colon \mathcal{A}\to F\md_\mathcal{A}^r$. Sometimes we will need to specicify the $A_\infty$-category in the notation and we will write the Yoneda embedding as $\mathcal{Y}_\mathcal{A}$.\\

Let $L$ be an object of $\mathcal{A}$. We define $\mathcal{Y}(L)$ via $$\mathcal{Y}(L)(L'):= \mathcal{A}(L',L)$$ for any object $L'$ of $\mathcal{A}$ and endow it with the $A_\infty$-structure given by $\mu^{\mathcal{Y}(L)}_{l|1}:=\mu_{l+1}$ for $l\geq 0$. It is straightforward to see that $\mathcal{Y}(L)$ indeed defines a filtered $A_\infty$-module.\\
We now define higher operations of the functor $\mathcal{Y}$. To ease the notation, we will often write $\mathcal{Y}_L$ instead of $\mathcal{Y}(L)$ in the rest of this subsection. Consider two objects $L_0$ and $L_1$ of $\mathcal{A}$ and an element $c\in \mathcal{A}^{\leq \alpha}(L_0,L_1)$ for some $\alpha \in \mathbb{R}$. We define its image $\mathcal{Y}_1(c)\in F\md_\mathcal{A}(\mathcal{Y}_{L_0},\mathcal{Y}_{L_1})$ as the pre-morphism of modules with components $\mathcal{Y}_1(c)_{l|1}\colon \mathcal{A}(\vec{X})\otimes \mathcal{Y}_{L_0}(X_l)\to \mathcal{Y}_{L_1}(X_0)$ given by $$\mathcal{Y}_1(c)_{l|1}(x_1,\ldots, x_l,y):= \mu_{l+1}(x_1,\ldots,x_l, y,c).$$ We define the higher components $\mathcal{Y}_d$ of the Yoneda embedding in a similar manner using contraction maps, that is, given a tuple $\vec{L}=(L_0,\ldots, L_d)$ and an element $\vec{c}\in \mathcal{A}(\vec{L})$, we define the pre-morphism of modules $\mathcal{Y}_d(\vec{c})\in F\md_\mathcal{A}(\mathcal{Y}_{L_0},\mathcal{Y}_{L_d})$ via $$\mathcal{Y}_d(\vec{c})_{l|1}(x_1,\ldots, x_l,y) := \mu_{l+d+1}(x_1,\ldots, x_l,y,\vec{c}).$$ Notice that $\mathcal{Y}$ is a filtered $A_\infty$-functor, since the maps $\mu_d$, $d\geq 1$, are filtered by assumption.

%=======================
%=======================
%=======================

\subsubsection{The $\lambda$-map}\label{lambdamap}
Consider an object $L$ of $\mathcal{A}$ and a filtered $A_\infty$-module $\mathcal{M}$. We have the following map, usually simply called the $\lambda$-map: $$\lambda\colon \mathcal{M}(L)\to F\md_\mathcal{A}(\mathcal{Y}(L),\mathcal{M})$$ defined for $c\in \mathcal{M}(L)$ via $$\lambda(c)_{l|1}(x_1,\ldots, x_l,y):= \mu^\mathcal{M}_{l+1|1}(x_1,\ldots, x_l,y,c).$$ It is clear that $\lambda$ is well-defined (i.e. that it lands in the category of \textit{filtered} modules), simply because $\mathcal{A}$ is a filtered $A_\infty$-category. Similarly to \cite[Lemma 2.12]{Se:book-fukaya-categ} and \cite[Proposition 2.5.1]{Bi-Co-Sh:LagrSh} we have the following result. This result is usually proved using spectral sequence, but strict unitality of $\mathcal{A}$ allows us for a more explicit proof, which we present here.
\begin{prop}
	The map $\lambda$ is a filtered quasi-isomorphism, with filtered quasi-inverse via filtered chain-homotopies. In particular, it induces an isomorphism in persistence homology.
\end{prop}
\begin{proof}
	The fact that $\lambda$ is filtered is obvious, as $\mathcal{A}$ is filtered. We will prove the rest of the statement by explicitly constructing a quasi-inverse and the chain homotopies. First, we define a candidate quasi-inverse $$\theta\colon F\md_\mathcal{A}(\mathcal{Y}(L),\mathcal{M})\to \mathcal{M}(L)$$ as follows: given a pre-module morphism $\varphi\in F\md_\mathcal{A}(\mathcal{Y}(L),\mathcal{M}) $, we set $$\theta(\varphi):= \varphi_{0|1}(e_L)$$ where we recall that $e_L\in \mathcal{A}(L,L)$ stands for the (given) choice of strict unit for $L$. Clearly $\theta$ is a chain map. It is easy to see that given $c\in \mathcal{M}(L)$, $\theta(\lambda(c))=c$, exactly because $e_L$ is a \textit{strict} unit, so no chain homotopy is needed in this direction. Moreover, since $e_L\in \mathcal{A}^{\leq 0}(L,L)$ by our definition of filtered $A_\infty$-category, we get that $\theta$ is a filtered chain map, i.e. it sends pre-morphisms with shift $\alpha$ to $\mathcal{M}^{\leq \alpha}(L)$ for any $\alpha\in \mathbb R$.\\ We now define a candidate chain homotopy for the other composition: we define $$H\colon F\md_\mathcal{A}(\mathcal{Y}(L),\mathcal{M})\to F\md_\mathcal{A}(\mathcal{Y}(L),\mathcal{M})$$ by sending $\varphi \in F\md_\mathcal{A}(\mathcal{Y}(L),\mathcal{M})$ to the pre-morphism of modules $H(\varphi)$ with components $$H(\varphi)_{l|1}(x_1,\ldots, x_l,y):= \varphi_{l+1|1}(x_1,\ldots, x_l,y,e_L)$$ for any $l\geq 0$. Again, since strict units lie at vanishing filtration level, the map $H$ preserves $F\md_\mathcal{A}(\mathcal{Y}(L),\mathcal{M})^{\leq \alpha}$ for any $\alpha\in \mathbb R$. We prove that $$\lambda\circ \theta + id_{F\md_\mathcal{A}(\mathcal{Y}(L),\mathcal{M})} = \mu_1^{\md}\circ H + H \circ \mu_1^{\md},$$ that is, that $\theta$ is a right quasi-inverse of $\lambda$ via the chain-homotopy $H$. Let $\varphi\in F\md_\mathcal{A}(\mathcal{Y}(L),\mathcal{M})$ and $l\geq 0$, we compute:
	\begin{enumerate}
		\item first, the term \begin{align*}
			\mu_1^{\md}(H(\varphi))_{l|1}(x_1,\ldots, x_l,y)  =&  \sum_{i=0}^l \mu^\mathcal{M}_{i|1}(x_1,\ldots, x_i, \varphi_{l-i+1|1}(x_{i+1},\ldots, x_l,y,e_L)) \\ & + \sum_{i=0}^l \varphi_{i+1|1}(x_1,\ldots, x_i, \mu_{l-i+1}(x_{i+1},\ldots, x_l,y),e_L) \\
			& + \sum_{j=0}^{l-1}\sum_{i=1}^{l-j} \varphi_{l-i+2|1}(x_1,\ldots, x_j, \mu_{i}(x_{j+1},\ldots, x_{j+i}),x_{j+i+1},\ldots, x_l,y,e_L)
		\end{align*}
		\item then the term \begin{align*}
			H(\mu_1^{\md}\varphi)_{l|1}(x_1,\ldots, x_l,y)  =& \ \   \mu_1^{\md}\varphi_{l+1|1}(x_1,\ldots, x_l,y,e_L) \\  =&   \sum_{i=0}^l \mu^\mathcal{M}_{i|1}(x_1,\ldots, x_i, \varphi_{l-i+1|1}(x_{i+1},\ldots, x_l,y,e_L) + \mu_{l+1|1}^\mathcal{M}(x_1,\ldots, x_l,y,\varphi_1(e_L))\\ & +\sum_{i=0}^l \varphi_{i|1}(x_1,\ldots, x_i, \mu_{l-i+2}(x_{i+1},\ldots, x_l,y,e_L))
			\\ & + \sum_{j=0}^{l-1}\sum_{i=1}^{l-j} \varphi_{l-i+2|1}(x_1,\ldots, x_j, \mu_{}(x_{j+1},\ldots, x_{j+i}),x_{j+i+1},\ldots, x_l,y,e_L) \\ & + \sum_{j=0}^l \varphi_{j+1|1}(x_1,\ldots, x_j, \mu_{l-j+1}(x_{j+1},\ldots, x_l,y),e_L)
		\end{align*}
		and note that the second big sum after the last equality equals $\varphi_{l|1}(x_1,\ldots, x_l,y)$ since $e_L$ is a strict unit.
		\item finally, we compute \begin{align}
			(\lambda\circ \theta + id)(\varphi)_{l|1}(x_1,\ldots, x_l,y) = \mu_{l+1|1}^\mathcal{M}(x_1,\ldots, x_l,y,\varphi_1(e_L)) + \varphi_{l|1}(x_1,\ldots, x_l,y).
		\end{align}
		
	\end{enumerate}
	We see that summing all the contributions above we get the expected equality at the $l|1$ level. This proves the proposition.
\end{proof}
The following result is an easy corollary of the above. We denote by $\mathcal{Y}(\mathcal{A})=\text{image}(\mathcal{Y})$ the subcategory of $F\md_\mathcal{A}$ consisting of Yoneda modules.
\begin{cor}
	The induced persistence functor $H(\mathcal{Y})\colon H(\mathcal{A})\to H(\text{image}(\mathcal{Y}))$ is an equivalence of persistence categories.
\end{cor}

\subsection{Persistence Hochschild homology}\label{app:PHH}
Let $\mathcal{A}$ be an $A_{\infty}$-category and $\mathcal{M}$ an
$(\mathcal{A}, \mathcal{A})$-bimodule. Denote by
$CC(\mathcal{A}, \mathcal{M})$ the Hochschild chain complex of
$\mathcal{A}$ with coefficients in $\mathcal{M}$ (a.k.a.~the cyclic
bar complex). Recall that this chain complex decomposes into a direct
sum
$CC(\mathcal{A}, \mathcal{M}) = \oplus_{d \geq 0} CC(\mathcal{A},
\mathcal{M}; d)$, with
$$CC(\mathcal{A}, \mathcal{M}; d) = \bigoplus_{\vec{X}}
\mathcal{M}(X_d,X_0) \otimes_{\k} \mathcal{A}(\vec{X}) ,$$ where the
direct sum runs over all tuples $\vec{X} = (X_0, \ldots, X_d)$ of
objects in $\mathcal{A}$. For $d=0$ we set $\mathcal{A}(\vec{X}) = \k$
so that
$CC(\mathcal{A}, \mathcal{M}; 0) := \bigoplus_{X \in \Ob(\mathcal{A})}
\mathcal{M}(X,X)$. We endow $CC(\mathcal{A}, \mathcal{M})$ with
the Hochschild differential $d_{CC}$ (see e.g.~\cite{Gan:thesis,
  Sher:form-hodge, Ri-Sm:WF-OC}). Its homology is denoted by
$HH(\mathcal{A}, \mathcal{M})$ and is called the Hochschild homology
of $\mathcal{A}$ with coefficients in $\mathcal{M}$.

In case $\mathcal{M} = \Delta_{\mathcal{A}}$ is the diagonal
$(\mathcal{A}, \mathcal{A})$-bimodule we write
$CC(\mathcal{A}, \mathcal{A})$ and $HH(\mathcal{A}, \mathcal{A})$ for
the Hochschild chain complex and homology respectively, with
coefficients in $\Delta_{\mathcal{A}}$. Sometimes we even abbreviate these to $CC(\mathcal{A})$ and $HH(\mathcal{A})$.

Assume now that $\mathcal{A}$ is a filtered $A_{\infty}$-category and
$\mathcal{M}$ is a filtered $(\mathcal{A}, \mathcal{A})$-bimodule.
For every $\alpha \in \mathbb{R}$, $d \geq 0$, and
$\vec{X} = (X_0, \ldots, X_d)$ denote by
$$( \mathcal{M}(X_d, X_0) \otimes \mathcal{A}(\vec{X}) )^{\alpha}
\subset \mathcal{M}(X_d,X_0) \otimes_{\k} \mathcal{A}(\vec{X})$$ the
$\k_0$-submodule generated by the images of all the submodules
$$\mathcal{M}^{\alpha'}(X_d, X_0) \otimes_{\k_0}
\mathcal{A}^{\alpha''_1}(X_0, X_1) \otimes_{\k_0} \cdots
\otimes_{\k_0}\mathcal{A}^{\alpha''_d}(X_{d-1}, X_d) , \text{ with }
\alpha' + \alpha''_1 + \cdots + \alpha''_d \leq \alpha,$$ under the
obvious map
$$\mathcal{M}(X_d, X_0) \otimes_{\k_0} \mathcal{A}(X_0, X_1) \otimes_{\k_0} \cdots
\otimes_{\k_0}\mathcal{A}(X_{d-1}, X_d) \longrightarrow
\mathcal{M}(X_d, X_0) \otimes_{\k} \mathcal{A}(\vec{X}).$$ Denote
$$CC^{\alpha}(\mathcal{A}, \mathcal{M}; d) := \bigoplus_{\vec{X}}
( \mathcal{M}(X_d,X_0) \otimes \mathcal{A}(\vec{X}) )^{\alpha}$$ which
is a $\k_0$-submodule of $CC(\mathcal{A}, \mathcal{M}; d)$. Finally,
denote by
$$CC^{\alpha}(\mathcal{A}, \mathcal{M}) := \bigoplus_{d \geq 0}
CC^{\alpha}(\mathcal{A}, \mathcal{M}; d).$$ This gives us an
increasing filtration of $CC(\mathcal{A}, \mathcal{M})$ by
$\k_0$-submodules
$CC^{\alpha}(\mathcal{A}, \mathcal{M}) \subset CC(\mathcal{A},
\mathcal{M})$, $\alpha \in \mathbb{R}$. Since $\mathcal{A}$,
$\mathcal{M}$ are filtered the Hochschild differential preserves this
filtration, i.e.
$$d_{CC}(CC^{\alpha}(\mathcal{A}, \mathcal{M})) \subset
CC^{\alpha}(\mathcal{A}, \mathcal{M})$$ so
$(CC(\mathcal{A}, \mathcal{M}), d_{CC})$ becomes a filtered chain
complex. We denote by $PHH(\mathcal{A}, \mathcal{M})$ its persistence
homology.

The Hochschild complex $CC(\mathcal{A},\mathcal{M})$ admits also a discrete filtration, traditionally called the \textit{length} filtration, $(F^NCC(\mathcal{A},\mathcal{M}))_{N\geq 0}$ defined for $N\geq 0$ as $$F^NCC(\mathcal{A},\mathcal{M}):=\bigoplus_{d\leq N}CC(\mathcal{A},\mathcal{M};d).$$ It is easy to see that $d_{CC}$ preserves this filtration too, and hence induces a differential on each $F^NCC(\mathcal{A},\mathcal{M})$. We denote by $F^HH(\mathcal{A},\mathcal{M})$ the resulting homology. The real filtration defined above is obviously compatible with the length filtrations. We thus get a sequence $F^NPHH(\mathcal{A},\mathcal{M})$, $N\geq 0$, of persistence modules.

Note that if $\mathcal{A}$ is filtered then the diagonal bimodule
$\Delta_{\mathcal{A}}$ is also filtered and we denote the
corresponding filtered chain complex and its persistence homology by
$CC(\mathcal{A}, \mathcal{A})$ and $PHH(\mathcal{A}, \mathcal{A})$
respectively.

We briefly describe the grading formalism for Hochschild homology used in our geometric application in \S\ref{sec:split-app}. Let $(X,\omega)$ be a monotone symplectic manifold of real dimension $2n$ and $\mathcal{A}$ be a Fukaya category build from Lagrangians in $(X,\omega)$. Floer complexes are defined over the Novikov field $\k=\La$ and are $\mathbb{R}/2\mathbb{Z}$ graded, with the Novikov variable having degree $0$. In particular, the following has to be understood $\mod 2$. The $d$th order $A_\infty$-map $\mu_d$ of $\mathcal{A}$ has cohomological degree $2-d$, hence homological degree $(d-2)+(d-1)n$. This leads to the following grading on the Hochschild complex $CC_*(\mathcal{A},\mathcal{A})$ of $\mathcal{A}$: given $\vec{x}:=x_1\otimes \cdots \otimes x_k\in \mathcal{A}(X_0,\ldots,X_k,X_0)$ we set $$\deg(\vec{x}):= \sum_{i=1}^k \deg(x_i)+k-1-n.$$ It is easy to see that $d_{CC}$ has degree $-1$.
\subsection{Shifted categories} \label{a:shifted} Let $\mathcal{A}$ be
a filtered $A_{\infty}$-category with a shift functor $\Sigma$, and
let $r \geq 0$. Define the normalized $r$-shift $S^r \mathcal{A}$ of
$\mathcal{A}$ to be the filtered $A_{\infty}$-category with the same
objects as $\mathcal{A}$ and morphism spaces $(S^r\mathcal{A})(L,L')$
between two objects $L$ and $L'$:
$$(S^r \mathcal{A})^{\alpha}(L,L'):=
\begin{cases} \mathcal{A}^{\alpha}(L,L'), \text{ if } L=\Sigma^lL'
  \text{ for some } l\in \mathbb R, \\ \mathcal{A}^{\alpha - r}
  (L,L'), \text{ otherwise.}
\end{cases}
$$
The $A_{\infty}$-operations $\mu^{\mathcal{A}}_d$ on $\mathcal{A}$
naturally induce maps $\mu^{S^r \mathcal{A}}_d$ on $S^r
\mathcal{A}$. The shift and translation functors $\Sigma$ and $T$
carry over to $S^r \mathcal{A}$ in the obvious way.

There is an obvious filtered $A_\infty$-functor
$\eta_r^{\mathcal{A}} : S^r \mathcal{A} \longrightarrow \mathcal{A}$
which is the identity on objects, is induced by the identity on
morphisms, and has trivial higher operations.

\subsection{Functors with Linear Deviation} \label{a:func-LD}

For readability, we repeat here the definition of
$A_{\infty}$-functors with linear deviation (LD functors, for short).

Let $\mathcal{A}$ be a filtered $A_{\infty}$-category.  Given a tuple
$\vec{X} = (X_0, \ldots, X_d)$ of objects from $\mathcal{A}$ we define
its reduced tuple $\vec{X}_R := (X_{i_0}, \ldots, X_{i_{d_R}})$ by
omitting from $\vec{X}$ subsequent (in the cyclic order) objects that
are equal up to a shift. The objects forming $\vec{X}_R$ are well
defined only up to shifts, but the length of $\vec{X}_R$ (or rather, its length-$1$)
$0 \leq d_R \leq d$ is well defined. We call it the reduced length of
$\vec{X}$ and denote it by $d_R$ or $d_R(\vec{X})$ whenever we want to
emphasize its dependence on $\vec{X}$. Note that in case every two
consecutive objects in $\vec{X}$ are different (up to shifts) then
$d_R(\vec{X}) = d$. At the other extremity, if all the objects in
$\vec{X}$ are equal up to shifts, then $d_R(\vec{X}) = 0$.

Let $\mathcal{A}$, $\mathcal{B}$ be two filtered
$A_{\infty}$-categories. Let $s\geq 0$. An $A_{\infty}$-functor
$\mathcal{F}: \mathcal{A} \longrightarrow \mathcal{B}$ is said to have
linear deviation rate $s$ if for every $d \geq 1$, every tuple
of objects $\vec{X} = (X_0, \ldots, X_d)$ from $\Ob(\mathcal{A})$ and
every $\alpha \in \mathbb{R}$ we have:
$$\mathcal{F}_d \bigl(\mathcal{A}^{\alpha}(X_0, \ldots, X_d)\bigr)
\subset \mathcal{B}^{\alpha + d_R(\vec{X}) s}(\mathcal{F}X_0,
\mathcal{F}X_d).$$ We will refer to such functors as LD-functors (LD
stands for Linear Deviation).

Let
$\mathcal{F}, \mathcal{G}: \mathcal{A} \longrightarrow \mathcal{B}$ be
two LD-functors, both with deviation rate $\leq s$.  A pre-natural
transformation $T: \mathcal{F} \longrightarrow \mathcal{G}$ is said to
have linear deviation rate $s$ and shift $\leq r$ if for every
$d \geq 0$ it's $d$-th order component $T_d$ shifts filtration levels
by $\leq r + ds$.

LD-functors $\mathcal{F}: \mathcal{A} \longrightarrow \mathcal{B}$
with a given deviation rate $s \geq 0$, and their pre-natural
transformations form a filtered $A_{\infty}$-category
$\text{fun}^{\text{LD};s}(\mathcal{A}, \mathcal{B})$. To emphasize the
given deviation rate we will sometimes denote the objects of these
category by $(\mathcal{F}, s)$. Note that for every $s' \leq s''$ we
have an obvious (faithful but not full) inclusion
$\text{fun}^{\text{LD};s'}(\mathcal{A}, \mathcal{B}) \subset
\text{fun}^{\text{LD};s''}(\mathcal{A}, \mathcal{B})$ of filtered
$A_{\infty}$-categories.

\subsubsection{Homotopy between LD-functors} \label{a:homotopy} We
also have the notion of homotopy between LD-functors, which is an
adaptation of the notion
from~\cite[Section~(1h)]{Se:book-fukaya-categ} to the filtered and LD
setting. Let
$\mathcal{F}, \mathcal{G}: \mathcal{A} \longrightarrow \mathcal{B}$ be
two LD-functors with deviation $\leq s$ and assume that both functors
{\em act in the same way on objects}.  Consider the pre-natural
transformation $D = \{D_d\}_{d \geq 0}$ between $\mathcal{F}$ and
$\mathcal{G}$ defined by $D_0 := 0$ and
$D_d := \mathcal{F}_d - \mathcal{G}_d$. We say that $\mathcal{F}$ and
$\mathcal{G}$ are $r$-homotopic, as LD-functors with deviation
$\leq s$ if there exists a pre-natural transformation
$T \in hom^r_{\text{fun}^{\text{LD};s}(\mathcal{A}, \mathcal{B})}\bigl(
(\mathcal{F},s), (\mathcal{G}, s) \bigr)$ with vanishing $0$-term,
$T^0=0$, and with deviation rate $\leq s$ and shift $\leq r$, such
that $D = \mu_1(T)$.

Recall that filtered functors are a special case of LD-functors, and
homotopies between them (as defined above; i.e.~with deviation $s=0$)
are interesting too. In this paper we will need homotopies both
between filtered functors as well as between LD-functors.

Next we briefly discuss the relation between homotopy of functors and
persistence Hochschild homology. Recall from \S\ref{app:PHH}, that if $\mathcal{A}$ is a filtered
$A_{\infty}$-category, then its Hochschild chain complex
$CC(\mathcal{A}, \mathcal{A})$ inherits a filtration from
$\mathcal{A}$ and consequently the homology of the latter becomes a
persistence module which we call the persistence Hochschild homology
$PHH(\mathcal{A}, \mathcal{A})$ of $\mathcal{A}$. Note that its
$\infty$-limit $PHH^{\infty}(\mathcal{A}, \mathcal{A})$ coincides with
the usual Hochschild homology $HH(\mathcal{A}, \mathcal{A})$.
\begin{prop} \label{p:homotopy-HH} Let
  $\mathcal{F}, \mathcal{G}: \mathcal{A} \longrightarrow \mathcal{B}$
  be two filtered functors. If $\mathcal{F}, \mathcal{G}$ are
  $0$-homotopic then the chain maps induced by $\mathcal{F}$,
  $\mathcal{G}$ on the filtered Hochschild chain complexes
  $\mathcal{F}_{\scriptscriptstyle CC},
  \mathcal{G}_{\scriptscriptstyle CC}:CC(\mathcal{A}, \mathcal{A})
  \longrightarrow CC(\mathcal{B}, \mathcal{B})$ are $0$-chain
  homotopic. In particular the induced maps on the persistence
  Hochschild homologies
  $\mathcal{F}_{\scriptscriptstyle PHH},
  \mathcal{G}_{\scriptscriptstyle PHH}: PHH(\mathcal{A}, \mathcal{A})
  \longrightarrow PHH(\mathcal{B}, \mathcal{B})$ coincide as maps of
  persistence modules.
\end{prop}

\subsubsection{From LD-functors to filtered functors} \label{a:ld-fil}

Let $\mathcal{F} : \mathcal{A} \longrightarrow \mathcal{B}$ be an LD
$A_{\infty}$-functor with deviation rate $\leq r$. Using \S\ref{a:shifted}, define a new
$A_{\infty}$-functor
$$\eta^r \mathcal{F}: S^r \mathcal{A} \longrightarrow \mathcal{B},
\quad \eta^r \mathcal{F} :=\mathcal{F} \circ \eta^{\mathcal{A}}_r,$$
which we call the normalized $r$-shift of $\mathcal{F}$. It is
straightforward to see that $\eta^r \mathcal{F}$ is a {\em filtered}
$A_{\infty}$-functor.

\subsection{Filtered bimodules associated with
  functors} \label{a:func-fil-bimod} Given a filtered
$A_{\infty}$-functor
$\mathcal{F}: \mathcal{A} \longrightarrow \mathcal{B}$, define the
$(\mathcal{B},\mathcal{A})$-bimodule
$$\underline{\mathcal{F}}:=
(\text{id}_\mathcal{B} \otimes \mathcal{F})^*\Delta_{\mathcal{B}},$$
where $\Delta_{\mathcal{B}}$ stands for the diagonal bimodule of
$\mathcal{B}$. In other words, given $Y \in \Ob(\mathcal{B})$,
$X \in \Ob(\mathcal{A})$, we have
$$\underline{\mathcal{F}}(Y,X):= \mathcal{B}(Y,\mathcal{F}(X)).$$
On bimodule-composable elements, the structure maps
$\mu^{\underline{\mathcal{F}}}_{l|1|r}$, for $l,r\geq 0$, of
$\underline{\mathcal{F}}$ are given by
\begin{equation*}
  \begin{aligned}
    & \mu^{\underline{\mathcal{F}}}_{l|1|r}(b_1,\ldots, b_l,y,a_1,\ldots, a_r)
    := \\
  & \sum_{j=1}^r \sum_{\substack{s_1 + \cdots + s_j = r \\ 1 \leq s_i}}
\mu^{\mathcal{B}}_{l+j+1}(b_1,\ldots, b_l, y,
\mathcal{F}_{s_1}(a_1,\ldots, a_{s_1}),\ldots,
    \mathcal{F}_{s_j}(a_{r-s_j},\ldots, a_r)).
  \end{aligned}
\end{equation*}
Note that since $\mathcal{F}$ is filtered, the bimodule
$\underline{\mathcal{F}}$ is filtered too.

The assignment $\mathcal{F} \longmapsto \underline{\mathcal{F}}$
extends to a filtered $A_{\infty}$-functor
$\mathcal{U}: F\text{fun}(\mathcal{A}, \mathcal{B}) \longrightarrow
F\text{bimod}_{\mathcal{B}, \mathcal{A}}$. Its 1st order term
$\mathcal{U}_1$ has the following description. Let
$\mathcal{F}, \mathcal{G}; \mathcal{A} \longrightarrow \mathcal{B}$ be
two filtered $A_{\infty}$-functors and
$T: \mathcal{F} \longrightarrow \mathcal{G}$ a pre-natural
transformation between them, of some given shift $\alpha$
(i.e.~$T \in \hom^{\alpha}_{F\text{fun}(\mathcal{A},
  \mathcal{B})}(\mathcal{F}, \mathcal{G})$). Define the
$(\mathcal{B}, \mathcal{A})$-bimodule pre-homomorphism
\begin{equation*}
  \begin{aligned}
    & \mathcal{U}_1(T)\colon \underline{\mathcal{F}} \longrightarrow \underline{\mathcal{G}},
    \\
    & \mathcal{U}_1(T)_{l|1|r}(b_1,\ldots, b_l, y,a_1,\ldots, a_r):= \\
    & \sum_{j=1}^r\sum_{i=1}^j\sum_{\substack
      {s_1 + \cdots + s_j = r  \\ 1 \leq s_k, \, 
    \forall k \neq i}}
    \mu^{\mathcal{B}}_{l+1+j}\bigl(b_1,\ldots, b_l,y, (\mathcal{F}_{s_1},\ldots,
    \mathcal{F}_{s_{i-1}},T_{s_i},\mathcal{G}_{s_{i+1}}, \ldots,
    \mathcal{G}_{s_j})(a_1,\ldots, a_r)\bigr).
  \end{aligned}
\end{equation*}
It is straightforward to verify that $\mathcal{U}_1(T)$ is bimodule
pre-homomorphism of the same shift $\alpha$. Given composable natural
transformations $T_1,\ldots, T_d$ between $A_\infty$-functors from
$\mathcal{A}$ to $\mathcal{B}$, the pre-homomorphism of
$(\mathcal{B},\mathcal{A})$-bimodules $\mathcal{U}_d(T_1,\ldots, T_d)$
is defined by a similar formula.

\begin{lem}[Theorem~8.2.1.2 in \cite{Lfv:thesis}]
  The component $\mathcal{U}_1$ of the functor $\mathcal{U}$ induces a
  quasi-isomorphism.
\end{lem}

\subsection{Pull back and push forward of filtered
  $A_{\infty}$-modules} \label{ap:pshf}

Let $\mathcal{A}$, $\mathcal{B}$ be filtered $A_{\infty}$-categories
and $\mathcal{F} : \mathcal{A} \longrightarrow \mathcal{B}$ a filtered
$A_{\infty}$-functor. Let $\mathcal{M}$ be a filtered
$\mathcal{B}$-module. Then the pull back $\mathcal{F}^*\mathcal{M}$
can be endowed with the structure of a filtered $\mathcal{A}$-module
in a straightforward way. However, if instead of assuming that
$\mathcal{F}$ is filtered we only assume it to be an LD-functor then
the pull back $\mathcal{F}^*\mathcal{M}$ is in general not filtered,
but rather an $LD$-module (a concept which will not be used in this
paper).

At the same time, there is a way to define push-forward of filtered
modules by LD-functors and still get filtered modules. We present this
construction next.

Let $\mathcal{A}, \mathcal{B}$ be filtered $A_{\infty}$-categories and
$\mathcal{F}: \mathcal{A} \longrightarrow \mathcal{B}$ an LD-functor
with deviation rate $\leq r$. Let $\mathcal{M}$ be a filtered
$\mathcal{A}$-module. The push-forward $(\mathcal{F},r)_* \mathcal{M}$
of $\mathcal{M}$ by $\mathcal{F}$ is defined to be the filtered
$\mathcal{B}$-module:
\begin{equation} \label{eq:psh-m} (\mathcal{F},r)_* \mathcal{M} :=
  \underline{\eta^r \mathcal{F}} \otimes_{S^r \mathcal{A}}
  (\eta_r^{\mathcal{A}})^* \mathcal{M}.
\end{equation}
Note that the value of the parameter $r$ affects the filtration
structure on the push-forward module, and therefore we include it in
the notation. In other words, if $r < r'$ then
$(\mathcal{F},r')_* \mathcal{M}$ has a different filtration structure
than $(\mathcal{F},r)_* \mathcal{M}$. Of course, if one forgets the
filtrations, then the push-forward of $\mathcal{M}$ depends only on
$\mathcal{F}$, and gives the same result for all $r$'s.

Given a pre-homomorphism
$f : \mathcal{M}_0 \longrightarrow \mathcal{M}_1$ beetween filtered
$\mathcal{A}$-modules, define its pushforward
$(\mathcal{F},r)_* f : (\mathcal{F},r)_* \mathcal{M}_0 \longrightarrow
(\mathcal{F},r)_* \mathcal{M}_1$ as follows. The higher order terms
$((\mathcal{F}, r)_*f)_{l|1}$, $l \geq 1$, are defined to be $0$, and
the linear order term is:
$$((\mathcal{F}, r)_*f)_{0|1}
(b\otimes a_1\otimes \cdots \otimes a_d\otimes m)= \sum_{i=1}^d
b\otimes a_1\otimes \cdots \otimes a_i\otimes
f_{d-i|1}(a_{i+1},\ldots, a_d,m).$$

The construction above extends to a filtered $A_{\infty}$-functor
$$PF: \text{fun}^{\text{LD};r}(\mathcal{A}, \mathcal{B})
\longrightarrow F\text{fun}(F\md_{\mathcal{A}}, F\md_{\mathcal{B}}),$$
whose action on the object $(\mathcal{F}, r)$ is the filtered
functor
$(\mathcal{F},r)_* :F\md_{\mathcal{A}} \longrightarrow
F\md_{\mathcal{B}}$. To define the action of $PF$ on morphisms, let
$\mathcal{F}, \mathcal{G} \in \Ob
(\text{fun}^{\text{LD};r}(\mathcal{A}, \mathcal{B}))$, and let
$T\colon \mathcal{F}\to \mathcal{G}$ be a pre-natural transformation
(in the category $\text{fun}^{\text{LD};r}(\mathcal{A}, \mathcal{B})$
of LD-functors with deviation rate $\leq r$) of shift $\rho \geq
0$. The image $PF_1(T)$ of $T$ under the push-forward functor $PF$ is
defined for an object $\mathcal{M} \in \Ob(F\md_{\mathcal{A}})$ as the
pre-module homomorphism
$PF_1(T) \in \hom_{F\md_{\mathcal{B}}}(\mathcal{F}_*\mathcal{M},
\mathcal{G}_*\mathcal{M})$ given by:
\begin{align*}
  PF(T)_{l|1}(b_1,\ldots, b_l, y\otimes \vec{a}\otimes m) & :=
  \sum_{j=0}^d \mathcal{U}_1(T)_{l|1|j}(b_1,\ldots, b_l,y,a_1,\ldots,
  b_j)\otimes a_{j+1}\otimes \cdots \otimes a_d\otimes m \\ =
  \sum_{j=0}^d \sum_{l=1}^j\sum_{i=1}^l\sum_{\substack{s_1 + \cdots +
      s_l = j \\ 1 \leq s_k, \, \forall k \neq i}}
  \mu^{\mathcal{B}}_{l+1+j}\bigl( & b_1,\ldots, b_l,y,(\mathcal{F}_{s_1},\ldots,
  \mathcal{F}_{s_{i-1}},T_{s_i},\mathcal{G}_{s_{i+1}}, \ldots,
  \mathcal{G}_{s_l})(a_1,\ldots, a_j)\bigr)  \\
  & \otimes a_{j+1}\otimes
  \cdots \otimes a_d\otimes m.
\end{align*}
The image of a composable tuple $T_1,\ldots, T_d$ of
$A_\infty$-pre-natural transformations under $PF_d$ is defined in a
very similar manner. It is straightforward to verify that $PF$ is a
filtered (strictly unital) $A_{\infty}$-functor.

The next proposition shows that, up to shifts, the push forward of the
Yoneda module of an objects is the same as the Yoneda module of the
image of this object by the given functor. We will denote here by
$\mathcal{Y}$ the filtered Yoneda embedding, both for $\mathcal{A}$
and for $\mathcal{B}$.
\begin{prop} \label{p:pushf-yon} Let
  $\mathcal{F}: \mathcal{A} \longrightarrow \mathcal{B}$ be an
  LD-functor with deviation rate $\leq r$. For every
  $X \in \Ob(\mathcal{A})$ the push forward
  $(\mathcal{F},r)_* \mathcal{Y}(X)$ of the Yoneda module
  $\mathcal{Y}(X)$ is $0$-quasi-isomorphic to
  $\Sigma^r \mathcal{Y}(\mathcal{F} X)$.
\end{prop}

To simplify the notation, when the deviation rate of a functor is
clear from the context, we will sometimes omit the $r$ from the pair
$(\mathcal{F}, r)$ and simply write $\mathcal{F}_* \mathcal{M}$.

\subsection{Filtered twisted complexes}\label{ap:ftc} In this section, we introduce filtered twisted complexes. The constructions developed here are upgrades of \cite[Chapter 3]{Se:book-fukaya-categ} from the unfiltered to the filtered setting, and of \cite[Section 2.5.1]{BCZ:tpc} from the $dg$ to the $A_\infty$ case.
\subsubsection{The filtered $A_\infty$-categories of filtered twisted complexes} Let $\mathcal{A}$ be a filtered $A_\infty$-category, without a choice of shift functor.
We will write $\k_0^r[d]$ for interval persistence module over $\k_0$ with a unique generator in degree $d\in \mathbb Z$ placed at filtration level $r\in \mathbb R$ and given an object $L$ of $\mathcal{A}$ write $\Sigma^rL[d]$ for the formal tensor product $\k_0^r[d]\otimes L$. 
\begin{dfn}
	We define the (standard) shift completion of $\mathcal{A}$ as the category $\mathcal{A}^\Sigma$ with objects $$\text{Ob}(\mathcal{A}^\Sigma) := \left\{ \Sigma^rL[d]\ : \ L\in \text{Ob}(\mathcal{A}), \ d\in \mathbb{Z}, \ r\in \mathbb R \right\},$$ morphism spaces $$\mathcal{A}^\Sigma(\Sigma^rL_0[d_0],\Sigma^sL_1[d_1]) := \Sigma^{-(r-s)}\mathcal{A}(L_0,L_1)[d_1-d_0]$$ that is, $\mathcal{A}(\Sigma^rL_0[d_0],\Sigma^sL_1[d_1]) $ is the chain complex $\mathcal{A}(L_0,L_1)$ shifted up in degree by $d_1-d_0$ and in filtration by $r-s$, and with the same $A_\infty$-operations of $\mathcal{A}$.
\end{dfn}
\begin{rem}
	\begin{enumerate}
		\item It is easy to see that $\mathcal{A}^\Sigma$ is a filtered $A_\infty$-category. 
		\item On the $A_\infty$-category $\mathcal{A}^\Sigma$ there is an immediate choice of shift $A_\infty$-functor (in the sense of \cite[Section 3.2.1]{BCZ:tpc}), which on objects is $\Sigma^r(L) := \Sigma^rL$. %Moreover, this shift functor is compatible in the sense of Definition \ref{compatibilityofshift}.
		\item In some texts the prefix $\Sigma$ denotes the additive enlargement of an $A_\infty$-category (e.g. \cite[Section (3k)]{Se:book-fukaya-categ}), which is not the case here.
	\end{enumerate}
\end{rem}
\begin{dfn}
	We define the filtered additive enlargement $\mathcal{A}^\oplus$ of $\mathcal{A}$ as the additive enlargement of $\mathcal{A}^\Sigma$ in the sense of \cite[Section (3k)]{Se:book-fukaya-categ}, that is, as the $A_\infty$-category with objects given by formal sums $$\overline{L} = \bigoplus_{i=1}^n\Sigma^{r_i}L_i[d_i],$$ where $n\in \mathbb N$ and $\Sigma^{r_i}L_i[d_i]$ is an object of $\mathcal{A}^\Sigma$ for any $i$; and morphisms spaces between objects $\overline{L}=\bigoplus_{i=1}^n\Sigma^{r_i}L_i[d_i]$ and $\overline{L}'=\bigoplus_{i=1}^m\Sigma^{r_i'}L'_i[d'_i]$ given by $n\times m$ matrices $$f = (f_{ij}), \text{  where } f_{ij}\in \mathcal{A}^\Sigma\left(\Sigma^{r_i}L_i[d_i],\Sigma^{r_j'}L_j'[d_j']\right)$$ with usual matrix grading. Given an object $\overline{L}=\bigoplus_{i=1}^n\Sigma^{r_i}L_i[d_i]$ we will write $|\overline{L}|=n$ and call it the length of $\overline{L}$. The $A_\infty$-maps $\mu_d^\oplus$, $d\geq 1$, are defined by extending $\mu_d$ matrixwise, that is, given $\overline{L}_0,\ldots, \overline{L}_d$ objects of $\mathcal{A}^\oplus$ and matrices $f^1,\ldots,f^d$ with $f^i\in \mathcal{A}^\oplus(\overline{L}_{i-1},\overline{L}_i)$ we set for $i=1,\ldots, |\overline{L}_0|$ and $j=1,\ldots, |\overline{L}_d|$: $$\mu_d^\oplus\left(f^1,\ldots, f^d\right)_{ij}:= \sum_{i_1=1}^{|\overline{L}_1|}\cdots \sum_{i_{d-1}=1}^{|\overline{L}_{d-1}|}\mu_d\left(f^1_{ii_1},\ldots, f^d_{i_{d-1}j}\right).$$
	Moreover, we filter morphism spaces in $\mathcal{A}^\oplus$ by setting the filtration level of a matrix to be the maximum of the filtration level of its entries, that is: $$(\mathcal{A}^\oplus)^{\leq \alpha}(\overline{L},\overline{L}') := \left\{f= (f_{ij})\in \mathcal{A}^\oplus(\overline{L},\overline{L}'):\ f_{ij}\in (\mathcal{A}^\Sigma)^{\leq \alpha}\left(\Sigma^{r_i}L_i[d_i], \Sigma^{r_j'}L_j'[d_j']\right )\text{  for any }i,j\right\}.$$
\end{dfn} 

\begin{rem}
	\begin{enumerate}
	%	\item In the notation of \cite{Se:book-fukaya-categ} we would write $\mathcal{A}^\oplus$ as $\Sigma \mathcal{A}^\Sigma$.
		\item We will not write all the sums in the definition of $\mu_d^\oplus$, but rather use the shorthand notation $\sum_{i_1,\ldots, i_{d-1}}$.
		\item It is straightforward to see that $\mathcal{A}^\oplus$ is a filtered $A_\infty$-category.
	\end{enumerate}
\end{rem}

\begin{dfn}\label{defFTW}
	A filtered (one-sided) twisted complex of $\mathcal{A}^\Sigma$ is a pair $$(\overline{L}, q=q_{\overline{L}}):= \left(\bigoplus_{i=1}^n \Sigma^{r_i}L_i[d_i], (q_{ij})\right)$$ for some $n\in \mathbb N$ such that:
	\begin{enumerate}
		\item $\overline{L}$ is an object of $\mathcal{A}^\oplus$, that is, for any $i=1,\ldots, n$, $ \Sigma^{r_i}L_i[d_i]$ is an object of $\mathcal{A}^\Sigma$,
		\item $q$ is a morphism in $\mathcal{A}^\oplus (\overline{L}, \overline{L})$ lying at degree $1$ and vanishing filtration level.
		\item the matrix $q$ is strictly upper triangular, that is $q_{ij}=0$ for any $i\geq j$,
		\item\label{cond4} the matrix $q$ satisfies $$\sum_{d\geq 1}\mu_d^\oplus (q,\ldots, q)=0,$$ 
		%which we cafor any $i,j=1,\ldots, n$ we have $$\sum_{d\geq 1}\sum_{i_1,\ldots ,i_{d-1}=1}^n\mu_d(q_{ii_1}, q_{i_1i_2},\ldots , q_{i_{d-1}j}) = 0,$$
		to which we will refer as the Maurer-Cartan identity of $q$.
	\end{enumerate}
	The matrix $q=q_{\overline{L}}$ is called the differential of $\overline{L}$ and will be often dropped from the notation.
\end{dfn}
\begin{rem} As $q$ is strictly upper triangular, any entry $\mu_d^\oplus(q,\ldots, q)_{i,j}$, $i,j=1,\ldots, |\overline{L}|$, can be written as $$\sum_{i<i_1<\cdots < i_{d-1}<j} \mu_d(q_{ii_1}, q_{i_1i_2},\ldots , q_{i_{d-1}j}).$$ Moreover, the sum in the Maurer-Cartan identity of $q$ is a finite sum: indeed, for any $d\geq |\overline{L}|$ we have $\mu_d^\oplus(q,\ldots, q)=0$ since the longest possible term is $$\mu_{|L|-1}(q_{12},q_{23},\ldots, q_{|\overline{L}|-1|\overline{L}|})$$ in the matrix $\mu^\oplus_{|L|-1}(q,\ldots, q)$, which has only one non-zero entry in position $(1,|\overline{L}|)$.
\end{rem}
\noindent Given a twisted complex over a filtered $A_\infty$-category in the sense of \cite[Section (3l) and Remark 3.26]{Se:book-fukaya-categ}, it might be the case that the differential is not filtration-preserving (i.e. Condition \ref{cond4} above does not hold). The analogue of Lemma 2.96 in \cite{BCZ:tpc} holds: there are shifts $r_i$ turning this twisted complex in a filtered one.
\begin{dfn}
	We define the filtered pre-triangulated completion $FTw(\mathcal{A})$ of $\mathcal{A}$ as follows:
	\begin{enumerate}
		\item The objects of $FTw(\mathcal{A})$ are filtered one-sided twisted complexes over $\mathcal{A}$,
		\item Given two filtered twisted complexes $\overline{L}$ and $\overline{L}'$, the morphism space $FTw(\mathcal{A})(\overline{L},\overline{L}')$ is defined as $\mathcal{A}^\oplus(\overline{L},\overline{L}')$, with the same filtration, %are matrices $f=(f_{ij})$ such that $$f_{ij} \in \mathcal{A}^\Sigma\left(\Sigma^{r_i}L_i[d_i], \Sigma^{r_j'}L_j'[d_j']\right),$$ 
		\item The $A_\infty$-operations are deformed by the differentiasl in the following way: given any $d\geq 1$ and twisted complexes $$\left(\overline{L_i} = \bigoplus_{k=1}^{n_i}\Sigma^{r_{i,k}}L_{i,k}[d_{i,k}],q_i\right)$$ for any $i=0,\ldots, d$ and morphisms $f^i\in FTw(\mathcal{A})(\overline{L_{i-1}},\overline{L_i})$ for any $i=1,\ldots, d$ we define $\mu_d^{Tw}(f^1,\ldots, f^d)$ as the $|\overline{L}_0|\times |\overline{L}_d|$ matrix  $$\mu_d^{Tw}(f^1,\ldots, f^d) := \sum_{k_0,\ldots, k_d\geq 0}\mu_{d+k_0+\cdots + k_d}^{\oplus}\left(q_0^{\otimes k_0}, f^1,q_1^{\otimes k_1},\ldots, q_{d-1}^{\otimes k_{d-1}},f^d, q_d^{\otimes k_d}\right)$$
		%for $i=1,\ldots, n_0$ and $j=1,\ldots, n_d$ the element $$\mu_d^{FTw(\mathcal{A})}(f^1,\ldots, f^d)_{ij} := \sum_{k_0,\ldots, k_d\geq 0} \mu_{d+k_0+\cdots + k_d}^{\oplus}\left(q_0^{\otimes k_0}, f^1,q_1^{\otimes k_1},\ldots, q_{d-1}^{\otimes k_{d-1}},f^d, q_d^{\otimes k_d}\right)$$ where $\mu_d^{\oplus}$ is the matrixwise extension of $\mu_d^\mathcal{A}$ to direct sums as in \cite[Remark 3.26]{Se:book-fukaya-categ}. {\color{orange} is this understandable?}
		%\item Given two filtered twisted complexes $\overline{L}$ and $\overline{L}'$ as above, the filtration on $FTw(\mathcal{A})(\overline{L},\overline{L}')$ is defined by taking the maximum in filtration of the entries of a morphism, i.e. $$FTw(\mathcal{A})^{\leq \alpha}(\overline{L},\overline{L}') := \left\{f= (f_{ij})\in FTw(\mathcal{A})(\overline{L},\overline{L}'):\ f_{ij}\in (\mathcal{A}^\Sigma)^{\leq \alpha}\left(\Sigma^{r_i}L_i[d_i], \Sigma^{r_j'}L_j'[d_j']\right )\text{  for any }i,j\right\}.$$
		%{\color{purple} here we are at the chain level, so we are dealing with inclusions indeed, and the above def makes sense (?)}
	\end{enumerate}
	Moreover, given $N\geq 1$ we define $FTw^N\mathcal{A}$ to be the full $A_\infty$-subcategory of $FTw\mathcal{A}$ with filtered twisted complexes $\overline{L}$ of length $|\overline{L}|\leq N$ as objects.
\end{dfn}
\begin{rem}\label{finitenesssum} 
a. The fact that $FTw(\mathcal{A})$ and $FTw^N(\mathcal{A})$ are a $A_\infty$-category is a consequence of differentials satisfying the Maurer-Cartan identity above.

b. There is an obvious filtered full and faithful $A_\infty$-functors $\mathcal{I}^N_\mathcal{A}\colon \mathcal{A}\to FTw^N(\mathcal{A})$ for all $N\geq 1$, as well as $\mathcal{I}_\mathcal{A}\colon \mathcal{A}\to FTw(\mathcal{A})$.

c. The shift functor $\Sigma$ on $\mathcal{A}^\Sigma$ induces a shift functor, still denoted by $\Sigma$, on $FTw(\mathcal{A})$. It is defined on an object $\overline{L}= \bigoplus_{i=1}^n \Sigma^{r_i}L_i[d_i]$ by $$\Sigma^r\overline{L} = \bigoplus_{i=1}^n \Sigma^{r_i+r}L_i[d_i]$$ and viewing each entry $q_{ij}$ of the differential $q$ of $\overline{L}$ as a morphism $$q_{ij}\in FTw(\mathcal{A})(\Sigma^{r_i+r}L_i[d_i], \Sigma^{r_j+r}L_j[d_j])\cong FTw(\mathcal{A})(\Sigma^{r_i}L_i[d_i], \Sigma^{r_j}L_j[d_j]).$$

d. The sum in the definition of $\mu_d^{Tw}$ is a finite sum, for basically the same reason that the Maurer-Cartan identity consists of a finite sum. For instance, in $\mu_1^{Tw}f$ for $f\in FTw(\mathcal{A})(\overline{L},\overline{L}')$ we sum elements of the form $\mu_{1+k_0+k_1}(q^{\otimes k_0},f,q'^{\otimes k_1})$, whose $(i,j)$-entry is $$\sum_{i_1,\ldots, i_{k_0}=1}^{|\overline{L}|}\sum_{i_{k_0+1},\ldots, i_{k_0+k_1}=1}^{|\overline{L}'|}\mu_{1+k_0+k_1}\left(q_{ii_1},\ldots, q_{i_{k_0-1}i_{k_0}},f_{i_{k_0}i_{k_0+1}},q'_{i_{k_0+1}i_{k_0+2}},\ldots, q'_{i_{k_0+k_1}j}\right)$$ so that, by lower triangularity of the differentials, $\mu_{1+k_0+k_1}(q^{\otimes k_0},f,q'^{\otimes k_1})=0$ whenever $k_0\geq |\overline{L}|$ or $k_1\geq |\overline{L}'|$. In particular, the longest possible summand is $$\mu_{|\overline{L}|+|\overline{L}'|-1}\left(q_{12},\ldots, q_{|\overline{L}|-1|\overline{L}|},f_{|\overline{L}|1},q'_{12},\ldots, q'_{|\overline{L}'|-1|\overline{L}'|}\right)$$ in the $(1,|\overline{L}'|)$-entry. It is easy to see that the same holds in $\mu_d^{Tw}(f^1,\ldots, f^d)$ whenever $k_i\geq |\overline{L_i}|$ for some $i$.

\end{rem}

\begin{dfn}
	Let $\overline{L}$ and $\overline{L}'$ be two filtered twisted complexes in $FTw(\mathcal{A})$ as above, $f\in FTw(\mathcal{A})(\overline{L}, \overline{L}')$ be a degree zero morphism such that $\mu_1^{FTw(\mathcal{A})}f=0$, and $\lambda\in \mathbb R$ such that $\lambda\geq \mathbb{A}(f)$. We define the $\lambda$-filtered mapping cone of $f$ as $$\text{Cone}^{\lambda}(f) := \left(\overline{L}'\oplus \Sigma^{\lambda}\overline{L}[1], \ q_{co}\right) \text{   where  }q_{co} := \begin{pmatrix}
		q' & f \\
		0 & \Sigma^\lambda q[1]
	\end{pmatrix}.$$
\end{dfn}
\begin{lem}
	$\text{Cone}^{\lambda}(f)$ is a filtered twisted complex.
\end{lem}
\begin{rem}
	Note that the subcategories $FTw^N(\mathcal{A})$ are not pre-triangulated, as the cone construction doesn't preserve the length filtration.
\end{rem}
\begin{prop}
	Let $\mathcal{A}$ be a filtered $A_\infty$-category. Then $H^0(FTw(\mathcal{A}))$ is a TPC.
\end{prop}

\begin{rem}
	Filtered twisted complexes might be described by allowing tensoring with graded filtered vector spaces (not persistence modules!) which are not one dimensional. Given a finitely and freely generated filtered graded vector $V := \bigoplus_{i\in I}\Lambda\cdot \gamma_i$ over $\Lambda$ and an object $L$ of $\mathcal{A}$, we can define the tensor $V\otimes L$ simply as the direct sum (or in other words, $0$-filtered mapping cone of the zero maps) of the elements $\Sigma^{|\gamma_i|}L[\deg(\gamma_i)] = \Lambda\cdot \gamma_i\otimes L$, where $\deg(\gamma_i)$ denotes the degree of $\gamma_i$ in $V$. We will often denote $\Sigma^{|\gamma_i|}L[\deg(\gamma_i)]$ simply as $L\cdot \gamma_i$. If $V$ is a filtered chain complex, i.e. it carries a differential $d_V$, then the tensor product inherits a deformed differential as follows: write $d_V\gamma_i = \sum_j\alpha^i_j\gamma_j$ for any $i\in I$, then the differential on $V\otimes L$ is the matrix with $(i,j)$ entry equal to $\alpha^i_je_L$, where $e_L\in \mathcal{A}(L,L)$ is the strict unit. This definition extends to the whole category $FTw(\mathcal{A})$ (i.e. it's not only well-defined for images of objects of $\mathcal{A}$) in an obvious way.  Another possibility (cfr. \cite[Section (3l)]{Se:book-fukaya-categ}) is to define the category $FTw$ as having as objects tensors of the form $\bigoplus V_i\otimes L_i$, where $V_i$ is a finite dimensional filtered vector space and $L_i$ is an object of $\mathcal{A}$, together with an upper triangular differential $q$. In this case morphisms between $\bigoplus_i V_i\otimes L_i$ and $\bigoplus_j W_j\otimes L'_j$ are elements of $$\bigoplus_{i,j}\text{Lin}(V_i,W_j)\otimes \mathcal{A}(L_i,L_j')$$ and the $A_\infty$-maps are modified by considering compositions of linear maps (see \cite[Equation (3.17)]{Se:book-fukaya-categ}). It is striaghtforward to see that the two descriptions are equivalent.
\end{rem}

%=======================
%=======================
%=======================

%\subsection{Retract approximability}
\subsubsection{The filtered extended Yoneda embedding}\label{extendedYoneda}
We extend the Yoneda embedding $\mathcal{Y}_\mathcal{A}\colon \mathcal{A}\to F\md_{\mathcal{A}}$ introduced above to a filtered $A_\infty$-functor $$\widetilde{\mathcal{Y}_\mathcal{A}}\colon FTw\mathcal{A}\to F
\md_{\mathcal{A}}$$ via $$\widetilde{\mathcal{Y}_\mathcal{A}}:= \mathcal{I}_\mathcal{A}^*\circ \mathcal{Y}_{FTw\mathcal{A}}$$ where $\mathcal{I}_\mathcal{A}^*$ is the pullback of the $A_\infty$-functor $\mathcal{I}_\mathcal{A}\colon \mathcal{A}\to FTw(\mathcal{A})$ introduced in Remark \ref{finitenesssum}b., $\mathcal{Y}_{FTw\mathcal{A}}$ is the filtered Yoneda embedding for the filtered $A_\infty$-category $FTw\mathcal{A}$ and $\circ$ is the usual composition of $A_\infty$-functors. As pullbacks of filtered $A_\infty$-functors preserve filtered $A_\infty$-modules and are themselves filtered, the functor $\widetilde{\mathcal{Y}_A}$ is indeed a filtered $A_\infty$-functor. Note that this is just an immediate extension to the filtered world of the definition of the extended Yoneda embedding in \cite[Section (3s)]{Se:book-fukaya-categ}.
\begin{prop}
	The induced functor $H(\widetilde{\mathcal{Y}_\mathcal{A}})\colon PD(\mathcal{A})\to H(F	\md_{\mathcal{A}})$ gives a TPC equivalence between $PD(\mathcal{A})$ and the image of $H(\widetilde{\mathcal{Y}_\mathcal{A}})$.
\end{prop}
\begin{proof}
	The fact that $H(\widetilde{\mathcal{Y}_\mathcal{A}})$ gives an equivalence of persistence categories is clear. We need to prove that it is a TPC functor, that is, it is compatible with the shift functors and that it is triangulated at the $0$-level. The fact that it is compatible with shift functors is a direct consequence of the fact that the shift functor on $FTw\mathcal{A}$ is compatible with modules by construction. The fact that the $0$-level is triangulated follows from the Yoneda embedding being an $A_\infty$-functor at the chain level.
\end{proof}
%===============================
%===============================
%==============================
\subsubsection{Filtered twistings}\label{filtwistings}
%Let $Y$ be an object of our filtered $A_\infty$-category $\mathcal{A}$ and $\overline{L}=(L,q_L)$ be a filtered twisted complex over $\mathcal{A}$. We can consider the filtered twisted complex $$Y\otimes FTw(Y,\overline{L})$$ with differential 
Let $X$ and $Y$ be objects of $\mathcal{A}$, and consider the filtered twisted complex $Y\otimes \mathcal{A}(X,Y)$ with differential $$e_Y\otimes \mu_1\in FTw(Y\otimes \mathcal{A}(X,Y),Y\otimes \mathcal{A}(X,Y)) = \mathcal{A}(Y,Y)\otimes \text{Lin}(\mathcal{A}(X,Y), \mathcal{A}(X,Y))$$ where $e_Y\in \mathcal{A}(Y,Y)$ is the strict unit of $Y$ and $\text{Lin}(-,-)$ denotes the space of linear maps between two vector spaces. In other words, given a basis $\vec{a}=(a_1,\ldots, a_n)$ of $\mathcal{A}(Y,X)$, the twisted complex above is $\bigoplus_{i=1}^n \Sigma^{|a_i|}Y[\text{deg}(a_i)]$ with differential given by the matrix $q_{ij} = \alpha_{ij}e_Y$, where $e_Y\in \mathcal{A}(Y,Y)$ is the strict unit for $Y$ and the $\alpha_{ij}\in \k$ are the coeffients in $\mu_1(a_i)=\sum_{j=1}^n\alpha_{ij}a_j$. We define the filtered morphism of filtered twisted complexes $$\xi\colon Y\otimes \mathcal{A}(Y,X)\to X$$ to be the map corresponding to the identity under the canonical filtered isomorphism of chain complexes $$FTw(Y\otimes \mathcal{A}(Y,X), X)\cong \text{Lin}(\mathcal{A}(Y,X),\mathcal{A}(Y,X)).$$ 
\begin{dfn}
	We define the twisting $T_YX$ of $X$ by $Y$ as the twisted complex given by the cone of $\xi$, i.e. $T_YX = \text{Cone}(\xi)$
\end{dfn}
We can write $$T_YX= X\oplus \bigoplus_{i=1}^n \Sigma^{|a_i|}Y[\text{deg}(a_i)+1]$$ with differential given by the matrix $\begin{pmatrix}
	0& \vec{a} \\ \vec{0}^T & (q_{ij})
\end{pmatrix}$.
Note that the strict unit $e_X\in \mathcal{A}(X,X)$ induces a morphism of twisted complexes $i\colon X\to T_YX$ as a vector $(e_X,0,\ldots,0)$.

We are interested in describing the image of $T_Y X$ under the extended filtered Yoneda embedding $\widetilde{\mathcal{Y}_\mathcal{A}}$ introduced in \S\ref{extendedYoneda}.
\begin{lem}
	We have that $$\tilde{\mathcal{Y}}_{T_YX} \cong \text{Cone}(\phi)$$
	where $\phi\colon \mathcal{Y}_Y\otimes \mathcal{A}(Y,X)\to \mathcal{Y}_X$ is the filtered  morphism of modules given by $$\phi_{l|1}(x_1,\ldots, x_l,a\otimes b):= \mu_{s+2}(x_1,\ldots, a,b)$$ and $\mathcal{Y}_Y\otimes \mathcal{A}(Y,X)$ has a module structure as defined in \cite[Section (3c)]{Se:book-fukaya-categ}.
\end{lem}
\begin{proof}
	It is easy to see, that as an element of $FTw(Y\otimes \mathcal{A}(Y,X), X) = \mathcal{A}(Y,X)\otimes \text{Lin}(A(Y,X),\Lambda)$, the morphism $\xi$ corresponds to the element $$\sum_{i=1}^na_i\otimes \psi_{a_i}$$ where $a_i\in \mathcal{A}(Y,X)$ is a basis element as above, and $\psi_{a_i}\colon A(Y,X)\to \Lambda$ is the map sending $a_i$ to $1$ and all the other basis elements to $0$. Under the Yoneda embedding, i.e. as a morphism of $A_\infty$-modules $\mathcal{Y}_Y\otimes \mathcal{A}(Y,X)\to \mathcal{Y}_X$, $\xi$ becomes the map \begin{align*}
		\phi_{l|1}(x_1,\ldots, x_l,b\otimes c) & = \sum_{i=1}^n \psi_{a_i}(c)\lambda(a_i)_{l|1}(x_1,\ldots, x_l,b)\\ & = \sum_{i=1}^n \psi_{a_i}(c)\mu_{l+2}(x_1,\ldots, x_l,b,a_i) = \mu_{l+2}(x_1,\ldots, x_l,b,c).
	\end{align*}
	This ends the proof.
\end{proof}
\newcommand{\Tau}{\mathcal{T}}
We can easily extend this construction to the case where $X$ is itself a filtered twisted complex over $\mathcal{A}$. In this case, under the extended Yoneda embedding, $X$ corresponds to a filtered $A_\infty$-module $\mathcal{M}$ and $T_YX$ corresponds to the $0$-cone of the full contraction map $$\phi\colon \mathcal{Y}_Y\otimes \mathcal{M}(Y)\to \mathcal{M}$$ (cfr. \cite[Section (5a)]{Se:book-fukaya-categ}). We will write the above cone as $$\Tau_Y\mathcal{M}:= \text{Cone}(\phi).$$ The analogous of the lemma above holds in this case too.

\begin{lem}
	Let $Y$ be an object of $\mathcal{A}$ and $X$ be a filtered twisted complex over $\mathcal{A}$. Then there is a filtered quasi-isomorphism of $A_\infty$-modules $$\widetilde{\mathcal{Y}}_{T_YX} \to \Tau_Y\widetilde{\mathcal{Y}}_X.$$
\end{lem}
\subsubsection{Functoriality of filtered twisted complexes}
Let $\mathcal{A}$ and $\mathcal{B}$ be filtered $A_\infty$-categories. Recall that, for any $s\geq 0$, we defined the filtered $A_\infty$-category $\text{fun}^{\text{LD};s}(\mathcal{A}, \mathcal{B})$ of $A_\infty$-functors with linear deviation rate $s$ in \S\ref{a:func-LD}.\\
We introduce the following notations: for $N\geq 1$ and $s\geq0$ we write \begin{equation}\label{eq:FTwNQ}
	FTw^N\mathcal{Q}_s= \text{fun}^{\text{LD};s}(FTw^N\mathcal{A}, FTw^N\mathcal{B})
\end{equation} as well as $$FTw\mathcal{Q}_s= \text{fun}^{\text{LD};s}(FTw\mathcal{A}, FTw\mathcal{B})$$ Recall that $FTw^N\mathcal{A}$ is the full $A_\infty$-subcategory ()of the $A_\infty$-category $FTw\mathcal{A}$ of filtered twisted complexes) admitting only twisted complexes of length $\leq N$ as objects (see \S\ref{ap:ftc}). In this section we define filtered $A_\infty$-functors $$FTw^N\colon \text{fun}^{\text{LD};s}(\mathcal{A}, \mathcal{B})\to FTw^N\mathcal{Q}_{Ns}$$ for any $s\geq 0$ and $N\geq 1$. Note that in the target category, the deviation is $Ns$. Unfortunately, these functors do not extend to a functor $FTw$ because the deviation of functors in the image depends on the length of the twisted complexes it is applied to, as will be apparent from the discussion below. The best we can do is to define functors $$\overline{FTw}\colon \text{fun}^{\text{LD};s}(\mathcal{A}, \mathcal{B})\to \overline{FTw\mathcal{Q}_s}$$ where $\overline{FTw\mathcal{Q}_s}$ is the $A_\infty$-category of $A_\infty$-functors from $FTw\mathcal{A}$ to $FTw\mathcal{B}$ such that, for any $N\geq 1$, when restricted to the subcategory $FTw^N\mathcal{A}$ they have linear deviation rate $Ns$.\\

Let $s\geq 0$ and $(\mathcal{F},s)$ be an $A_\infty$-functor with linear deviation $s$. Then $\mathcal{F}$ immediately extends to an $A_\infty$-functor $\mathcal{F}^\Sigma\colon \mathcal{A}^\Sigma\to \mathcal{B}^\Sigma$ with deviation $s$ between the shift completions of $\mathcal{A}$ and $\mathcal{B}$. We define the extension $(\mathcal{F},s)^\oplus\colon \mathcal{A}^\oplus\to \mathcal{B}^\oplus$ to the filtered additive enlargments as follows (we drop the $s$ from the notation, but the induced functor depends on this choice):
\begin{enumerate}
	\item on an object $\overline{L}=\bigoplus_{i=1}^n\Sigma^{r_i}L_i[d_i]$, $\mathcal{F}^\oplus$ acts as \begin{equation}\label{eq:FTw}
		\mathcal{F}^\oplus\left(\bigoplus_{i=1}^n\Sigma^{r_i}L_i[d_i]\right):= \bigoplus_{i=1}^n\Sigma^{r_i-(i-1)s}\mathcal{F}(L_i)[d_i]
	\end{equation}
	\item given objects $\overline{L}_0,\ldots, \overline{L}_d$ of $\mathcal{A}^\oplus$ and morphisms $f^i\in \mathcal{A}^\oplus(\overline{L_{i-1}},\overline{L_i})$ for any $i=1,\ldots, d$, we set $\mathcal{F}^\oplus_d(f^1,\ldots,f^d)\in \mathcal{B}^\oplus(\mathcal{F}\overline{L}_0,\mathcal{F}\overline{L}_d)$ to be the matrix with $(i,j)$-entry equal to $$\sum_{i_1=1}^{|\overline{L}_1|}\cdots \sum_{i_{d-1}=1}^{|\overline{L}_{d-1}|}\mathcal{F}_d\left(f^1_{ii_1},\ldots, f^d_{i_{d-1}j}\right)$$ for $i=1,\ldots, |\overline{L_0}|$ and $j=1,\ldots, |\overline{L_d}|$.
\end{enumerate}
We can then easily extend $\mathcal{F}$ to an $A_\infty$-functor $FTw\mathcal{F}\colon FTw\mathcal{A}\to FTw\mathcal{B}$ by setting $$FTw\mathcal{F}(\overline{L}):= \mathcal{F}^\oplus(\overline{L})\text{  and  } q_{Tw\mathcal{F}(\overline{L})} := \sum_{d\geq 1}\mathcal{F}^\oplus_d(q_{\overline{L}},\ldots, q_{\overline{L}})$$ for any object $(\overline{L},q_{\overline{L}})$ of $FTw\mathcal{A}$, and, given objects $(\overline{L}_0,q_0),\ldots, (\overline{L}_d,q_d)$ of $FTw\mathcal{A}$ and morphisms $f^i\in FTw\mathcal{A}(\overline{L_{i-1}},\overline{L_i})$ for any $i=1,\ldots, d$, we set
$$(FTw\mathcal{F})_d(f^1,\ldots, f^d) := \sum_{k_0,\ldots, k_d\geq 0}\mathcal{F}^\oplus_{d+k_0+\cdots + k_d}\left(q_0^{\otimes k_0}, f^1,q_1^{\otimes k_1},\ldots, q_{d-1}^{\otimes k_{d-1}},f^d, q_d^{\otimes k_d}\right)$$
\begin{rem}
	\begin{enumerate}
		\item The sums in the definition of $q_{Tw\mathcal{F}(\overline{L})}$ and $FTw\mathcal{F}_d$ are finite. 
		\item The $s$-shifts in the image of $\overline{L}$ in (\ref{eq:FTw}) is essential in order for the image $q_{Tw\mathcal{F}(\overline{L})}$ of $q_{\overline{L}}$ to be at filtration level $\leq 0$, as we prove in the next lemma. Moreover, we remark here again that the induced functors $\mathcal{F}^\oplus$ and $FTw\mathcal{F}$ obviously depend on the choice of deviation $s$, although this fact is not explicitly reflected by the notation.
	\end{enumerate}
\end{rem}
\begin{lem}
	Let $(\mathcal{F},s)$ be an object of $\text{fun}^{\text{LD};s}(\mathcal{A}, \mathcal{B})$. Let $(\overline{L},q_{\overline{L}})$ be a filtered twisted complex. Then its image $(FTw\mathcal{F}(\overline{L}),q_{Tw\mathcal{F}(\overline{L})})$ under $FTw\mathcal{F}$ is a filtered twisted complex as well.
\end{lem}
\begin{rem}
	The fact that the image $FTw\mathcal{F}(\overline{L})$ preserves the structure of a filtered twisted complex although $\mathcal{F}$ is not in general a filtered $A_\infty$-functor may seem a bit counterintuitive, but might be taken as an hint that functors with linear deviation behave well with respect to filtered homological constructions and are hence the right setting.
\end{rem}
\begin{proof}
	It is easy to see that $q_{Tw\mathcal{F}(\overline{L})}$ satisfies the Maurer-Cartan equation. It remains to check that it lies at non-positive filtration level. We write $q$ for $q_{\overline{L}}$. Let $d\in \{1,\ldots, |\overline{L}|-1\}$ and consider the $(i,j)$-entry of $\mathcal{F}^\oplus_d(q,\ldots, q)$, which equals $$\sum_{i<i_1<\cdots < i_{d-1}<j} \mathcal{F}_d(q_{ii_1},\ldots, q_{i_{d-1}j}).$$ Assume that $\mathcal{F}_d(q_{ii_1},\ldots, q_{i_{d-1}j})$ is non-zero for some $i_1,\ldots, i_{d-1}$; in particular, $j-i\geq d$, as otherwise, by triangularity of $q$, this term would be zero. Note that $\mathcal{F}_d(q_{ii_1},\ldots, q_{i_{d-1}j})$ lies in $$(\mathcal{B}^\Sigma)^{\leq ds} (\Sigma^{r_i}\mathcal{F}(L_i)[d_i], \Sigma^{r_j}\mathcal{F}(L_j)[d_j])\subset (\mathcal{B}^\Sigma)^{\leq (j-i)s} (\Sigma^{r_i}\mathcal{F}(L_i)[d_i], \Sigma^{r_j}\mathcal{F}(L_j)[d_j])$$ which is isomorphic to $$ (\mathcal{B}^\Sigma)^{\leq 0} (\Sigma^{r_i-(i-1)s}\mathcal{F}(L_i)[d_i], \Sigma^{r_j-(j-1)s}\mathcal{F}(L_j)[d_j]).$$ By our definition of the action of $FTw\mathcal{F}$ on objects of $FTw\mathcal{A}$, it follows directly that $q_{Tw\mathcal{F}(\overline{L})}$ lies at vanishing filtration level.
\end{proof}
\begin{rem}
	Since the induced functor $FTw\mathcal{F}$ preserves lengths of twisted complexes by definition, it restricts to a functor $Tw^N\mathcal{A}\to Tw^N\mathcal{B}$. We denote this functor by $FTw^N\mathcal{F}$, i.e. highlighting the $N$, because of the following proposition.
\end{rem}
\begin{prop}
	Let $N\geq 1$ and let $(\mathcal{F},s)$ be an object of $\text{fun}^{\text{LD};s}(\mathcal{A}, \mathcal{B})$. Then the induced functor $$Tw^N\mathcal{F}\colon FTw^N\mathcal{A}\to FTw^N\mathcal{B}$$ is an $A_\infty$-functor with deviation rate $N\cdot s$, that is, an object of $FTw^NQ_{Ns}$ (see (\ref{eq:FTwNQ})).
\end{prop}
\begin{proof}
	Consider objects $(\overline{L^0},q_0),\ldots, (\overline{L^d},q_d)$ of $FTw^N\mathcal{A}$ of length $n^0,\ldots, n^d\leq N$ respectively, and morphisms $f^i\in FTw\mathcal{A}^{\alpha_i}(\overline{L_{i-1}},\overline{L_i})$ for any $i=1,\ldots,d$. The "most shifted" (in general) non-zero contribution to the matrix $$FTw\mathcal{F}_d(f^1,\ldots, f^d)\in FTw\mathcal{B}(Tw\mathcal{F}\overline{L_0},Tw\mathcal{F}\overline{L_d})$$ is the $(1,n^d)$ entry of the summand $$\mathcal{F}^\oplus_{d + (n^0-1)+\cdots + (n^d-1)}(q_0^{n^0-1},f^1,\ldots, f^d,q_d^{n^d-1}).$$
	This term lies at filtration level $$\leq \sum_{i=1}^d\alpha_i + \sum_{j=0}^d(n^j-1)s+ds = \sum_{i=1}^d\alpha_i + \sum_{j=0}^dn^js-s$$ in $\mathcal{B}^\Sigma(\Sigma^{r_1}\mathcal{F}(L^0_1),\Sigma^{r_{n^d}}\mathcal{F}(L^d_{n^d}))$, i.e. at level $$\leq \sum_{i=1}^d\alpha_i + \sum_{j=0}^{d-1}n^js$$ in 
	$\mathcal{B}^\Sigma(\Sigma^{r^0_1}\mathcal{F}(L^0_1),\Sigma^{r^d_{n^d}-(n^d-1)s}\mathcal{F}(L^d_{n^d}))$. As by assumption $n^0,\ldots, n^d\leq N$, we have $\sum_{j=0}^{d-1}n^js\leq dNs$ and the claim follows.
	
	%$$\mathcal{F}_{1 + (|\overline{L_0}|-1)+\cdots + (|\overline{L_d}|-1)}(q_0^{|\overline{L_0}|-1},f^1,\ldots, f^d,q_d^{|\overline{L_d}|-1})$$ Fix $i\in \{1,\ldots, |\overline{L_0}|\}$ and $j\in \{1,\ldots, |\overline{L_d}|\}$ and consider the $(i,j)$-entry of some summand $\mathcal{F}^\oplus_{d+k_0+\cdots + k_d}\left(q_0^{\otimes k_0}, f^1,q_1^{\otimes k_1},\ldots, q_{d-1}^{\otimes k_{d-1}},f^d, q_d^{\otimes k_d}\right)$ of $Tw\mathcal{F}_d(f^1,\ldots, f^d)$. 
\end{proof}
\begin{rem}
	The fact that $FTw\mathcal{F}$ acts on twisted complexes shifting summand by a quantity depending on the order is needed in order to have induced functors with shift that is sublinear on the order. The fact that the shift depend on the order seems counterintuitive, but it is in line with upper triangularity of differentials, which of course depends on the order of the summands.
\end{rem}

So far, we defined the $A_\infty$-functor $FTw$ just at the level of objects. We now define the first order term $FTw_1$. Let $(\mathcal{F},s)$ and $(\mathcal{G},s)$ be objects of $\text{fun}^{\text{LD};s}(\mathcal{A}, \mathcal{B})$ and consider an $A_\infty$-pre-natural transformation $T$ from $(\mathcal{F},s)$ to $(\mathcal{G},s)$ with shift $\rho$ (see \S\ref{a:func-LD} for the definition). We define the '$A_\infty$-pre-natural transformation'\footnote{The quotation marks are there as this is formally a natural transformation between functors that are not well-defined.} $FTw_1T$ with $d$th term, $d\geq 0$, $$(FTw_1T)_d(f_1,\ldots, f_d):= \sum_{k_0,\ldots, k_d\geq 0}T^\oplus_{d+k_0+\cdots + k_d}\left(q_0^{\otimes k_0}, f^1,q_1^{\otimes k_1},\ldots, q_{d-1}^{\otimes k_{d-1}},f^d, q_d^{\otimes k_d}\right)$$ where $T^\oplus$ is the matrixwise extension of $T$, defined analogously to the case of functors above. As in the case of functors, we can restrict to subcategories of twisted complexes of length $\leq N$ and define $FTw_1^NT$ for any $N\geq 1$. Taking filtrations into account, we have the following result.
\begin{lem}
	Let $(\mathcal{F},s)$ and $(\mathcal{G},s)$ be objects of $\text{fun}^{\text{LD};s}(\mathcal{A}, \mathcal{B})$ and consider an $A_\infty$-pre-natural transformation $T\colon (\mathcal{F},s)\to(\mathcal{G},s)$ of shift $\rho$ between them. Let $N\geq 1$. Then $FTw^N_1T$ is an $A_\infty$-pre-natural transformation between $FTw^N\mathcal{F}$ and $FTw^N\mathcal{G}$ of shift $\leq \rho$, that is $$FTw^N_1T\in FTw^NQ_{Ns}(FTw^N\mathcal{F},FTw^N\mathcal{G})^\rho.$$
\end{lem}
\begin{rem}
	In other words, the above lemma tells us that $FTw_1$ sends $A_\infty$-pre-natural transformations with linear deviation $s$ and shift $\leq \rho$ to $A_\infty$-pre-natural transformations with linear deviation $Ns$ and shift $\rho$.
\end{rem}
%The same is true if we consider a natural transformation between objects $(\mathcal{F},s)$ and $(\mathcal{G},r)$ of $\mathcal{Q}$. \noindent This generalizes to:
We define the higher order terms $FTw_d$, $d\geq 2$, to be zero. In particular, we proved the following result, which was announced at the beginning of this section.
\begin{cor}
	For any $N\geq 1$ and $s\geq 0$, the $A_\infty$-functor $$FTw^N\colon \text{fun}^{\text{LD};s}(\mathcal{A}, \mathcal{B})\to \text{fun}^{\text{LD};Ns}(FTw^N\mathcal{A},FTw^N\mathcal{B}) = FTw^N\mathcal{Q}_{Ns}$$ is filtered.
\end{cor}

\bibliography{biblio_approx8.bib}
\end{document}